\documentclass[12pt,reqno]{amsart}
\usepackage{amsthm,amsfonts,amssymb,euscript,mathrsfs,graphics,color,amsmath,amssymb,latexsym,marginnote}%,mathabx}
\usepackage[dvips]{graphicx}
\usepackage[margin=1.0in]{geometry}

\usepackage{hyperref}

\def\C{{\mathbb C}}
\def\Z{{\mathbb Z}}

\def\R{{\mathbb R}}

\def\e{{\varepsilon}}
\def\g{{\gamma}}

\def\a{{\alpha}}
\def\b{{\beta}}
\def\d{{\delta}}

\def\what{\widehat}
\def\wt{\widetilde}

\def\jq{\langle q \rangle}
\def\js{\langle s \rangle}

\def\Fhat{\what{\mathcal{F}}}
\def\Ftil{\wt{\mathcal{F}}}

\def\sign{ \mbox{sign} }

\def\K{{\mathcal{K}}}

\def\supp{\,\mbox{supp}\,}
\def\id{\,\mbox{id} \,}

\def\jt{\langle t \rangle}
\def\js{\langle s \rangle}
\def\jx{\langle x \rangle}
\def\jy{\langle y \rangle}
\def\jk{\langle k \rangle}
\def\jell{\langle \ell \rangle}
\def\jm{\langle m \rangle}

\def\hat{\widehat}
\def\bar{\overline}
\def\pv{\mathrm{p.v.}}

\def\tofill{\vskip30pt $\cdots$ To fill in $\cdots$ \vskip30pt}

\numberwithin{equation}{section}

\title[Bilinear estimates with potentials and quadratic NLS in $3$d]{
Bilinear estimates in the presence of a large potential
and a critical NLS in 3d 
}

\author{Fabio Pusateri}
\address{University of Toronto%, Department of Mathematics, 40 St George Street, %Rm 6290, Toronto, ON, M5S 2E4, Canada.
}
\email{fabiop@math.toronto.edu}

\author{Avy Soffer}
\address{Rutgers University%,Department of Mathematics, 110 Frelinghuysen Road, Piscataway, NJ 08854, USA.
}
\email{soffer@math.rugers.edu}

\newtheorem{theorem}{Theorem}[section]
\newtheorem{lemma}[theorem]{Lemma}
\newtheorem{proposition}[theorem]{Proposition}

\newtheorem{definition}[theorem]{Definition}
\newtheorem{remark}[theorem]{Remark}

\numberwithin{equation}{section}

\begin{document}

{\normalsize

\begin{abstract}
We propose an approach to nonlinear evolution equations with
large and decaying external potentials %in $3$d,
that addresses %one of the outstanding open questions
the question of controlling globally-in-time %the behavior of 
the nonlinear interactions of localized waves in this setting. 
%under the influence of an external potential.
This problem arises %analysis appears to be broadly applicable 
when studying localized perturbations around (possibly non-decaying) special solutions of evolution PDEs,
and trying to control
the projection onto the continuous spectrum of the nonlinear radiative interactions.

One of our main tools is the Fourier transform adapted to the Schr\"odinger operator $H=-\Delta+V$, 
which we employ at a nonlinear level.
%and establish distorted Fourier analogues of weighted estimates.
%implying 
%and integrable-in-time decay. %for small solutions.
%This is the first result of his type in this context. To study precisely the dynamics in the presence of a potential, 
%
%A key aspect of our proof is the development of %novel 
%multilinear harmonic analysis techniques 
%in the setting of $3$d Schr\"odinger operators with a large potential.
%In particular, we propose a 
As a first step we analyze the spatial integral of the product of three 
generalized eigenfunctions of $H$, %- what we call the ``nonlinear spectral distribution'' (NSD) -
and determine the precise structure of its singularities. %in space.
This leads to study bilinear operators with certain singular kernels, for which
we derive product estimates of Coifman-Meyer type.
%We naturally come to study certain bilinear operators with singular kernels. %in frequency space.
This analysis can then be combined with multilinear harmonic analysis tools and 
the study of oscillations %for operators with singular kernels.
%and establish 
%The main The goal is to establish 
%This analysis can then be combined with the study of oscillations
to obtain (distorted Fourier space analogues of) weighted estimates for dispersive and wave equations.
%bilinear estimates in weighted 

%
%We consider the spatial integral of the product of three (or more) generalized eigenfunctions,
%which we call the ``nonlinear spectral distribution'' (NSD).
%This distribution depends on three or more 
%frequencies and dictates the characters of the nonlinear interactions %together with the superposition of the three linear oscillations
%we study the  and its singularities. 
%

As a first application we consider 
the nonlinear Schr\"odinger equation in $3$d in the presence of large decaying %(regular and decaying)
potential with no bound states, and with a $u^2$ non-linearity.
The main difficulty is that a quadratic nonlinearity in $3$d is critical with respect to the Strauss exponent;
moreover, this nonlinearity has non-trivial fully coherent interactions even when $V=0$.
We prove quantitative global-in-time bounds and scattering for small solutions.

%the radiative part of the solution. %such as those 
%that appear in the radiation eqeation.%ing

%In particular, we extract sufficient oscillations for large times by determining 
%``good directions'' for integration by parts in frequency space.

%We consider this as a first step towards the understanding the stability of non-trivial solutions 
%for dispersive and hyperbolic equations in $3$d.

\end{abstract}
}

\maketitle

\setcounter{tocdepth}{1}

\begin{quote}%\footnotesize
\tableofcontents
\end{quote}

%\section{Results}

%\section{Previous works}

\section{Introduction}
%\medskip
%\subsection{Motivation and set-up}
This work is motivated by %problems about
questions on the long-time stability %, or the instability, 
of large, and possibly non-localized, special solutions of nonlinear evolution equations.
%To approach these question 
%To tackle these problems, 
We propose a systematic approach to the study of equations of the form
\begin{align}\label{introeq}
i\partial_t u + L(D) u  + V(x) u = \mathcal{N}(u), \qquad u(t=0)=u_0,
\end{align}
where $u: (t,x) \in \R\times \R^d \rightarrow \C$, $d>1$,
$L(D)=L(-i\nabla_x)$ is a real-valued dispersion relation,
$\mathcal{N}$ is a nonlinear term in $(u,\bar{u})$,
and $V$ is a large real-valued decaying %\footnote{About time-dependence ...} 
potential.
The initial data is assumed to be sufficiently regular, small and localized,
and we are interested in the global existence and quantitative estimates of solutions.
Typical examples are nonlinear Schr\"odinger equations ($L(D)=-|D|^2=\Delta$),
nonlinear Klein-Gordon ($L=\sqrt{m^2+|D|^2}$) and wave equations ($L=|D|$).

The methods we develop are intended to be most 
relevant for the stability of special solutions in the following contexts: 

\smallskip
\setlength{\leftmargini}{2em}
\begin{itemize}
\item[(1)] localized special solutions (solitons and solitary waves) 
for equations with nonlinearities of low degree of homogeneity,
e.g. quadratic nonlinearities in dimensions $2$ and $3$, 
such as those appearing in water waves and plasma models;
%Gross-Pitaevskii ??
%EP, EM, Zak

\smallskip
\item[(2)] non-localized solutions (e.g. topological solitons) for equations with 
nonlinearities of higher degree; %  of homogeneity;
examples here include field theories in dimensions $2$ or higher,
such as Ginzburg-Landau-type theories and their vortices solutions,
and generalizations of $\phi^4$-type field theories with `domain walls' solutions 
(see \cite{bookManSut}). %in $\phi^4$-theories 

%Skirmions

\end{itemize}

%The general discussion will hopefully serve as an explanation
%however, it might look at little simplistic to for the experts.

%One of our main motivations is the study of the stability of special (large) solutions, e.g. 
%solitons and topological solitons, %or small solitary waves
%of dispersive and wave equations %from quantum field theory, general relativity etc\dots
%where models like the one we consider here appear in when perturbing around such solutions.

\subsubsection*{Linearized dynamics}
To explain the relevance of \eqref{introeq} consider a special solution of a nonlinear evolution equation, 
such as, for the sake of concreteness, a stationary soliton, $Q=Q(x)$.
The basic strategy to investigate its stability is to look at solutions of the full nonlinear problem
in the form $Q(x)+v(t,x)$, where $v$ is small (and localized) at the initial time $t=0$.
Disregarding for the sake of exposition the issue of modulating $Q$ by the symmetries of the equation 
(such as translations and phase rotations), 
one can immediately see that understanding the stability for $Q$ amounts to studying the long-time behavior of $v$.

The equation for the evolution of the perturbation $v$ presents two %distinct 
fundamental difficulties:
%
%is a nonlinear PDE whose linear part 

\smallskip
\setlength{\leftmargini}{2em}
\begin{itemize}

\item[(A)] %The first difficulty is that 
The linear part of the equation involves an added {\it effective potential} coming from $Q$;
we refer to this as the {\it perturbed linear operator},
as opposed to the ``unperturbed'' or ``flat'' operator, which is one with no potential term.
A basic example of such an operator is the operator\footnote{In
many cases the operators obtained upon linearization are more complicated than $L_V$ and
are not necessarily scalar or self-adjoint.} 
$L_V := L(D) + V$ from \eqref{introeq}.

\smallskip
\item[(B)] %the equation for the perturbation is that,
The equation for the perturbation will typically contain {\it `pure' quadratic nonlinear terms}
without additional localization, in both scenarios (1) and (2) above.
%It is well-known that these types of terms 

\end{itemize}

In many relevant applications the potential part in the perturbed operator %(which could be large)
is smooth and decaying, so we assume this to be the case 
in our discussion. %spectrum continuous...
Then, it is well established that
many quantitative properties of linear homogeneous solutions, %of constant coefficient operators,
such as for example time-decay and Strichartz estimates for solutions of $i\partial_tu + Lu = 0$,
still hold %can be extended to operators with potentials 
for solutions of $i\partial_tu + L_Vu=0$, when one projects onto %restricted to 
the continuous spectrum of $L_V$.
However, applications of classical nonlinear tools such 
as commuting vectorfields and normal forms transformations are almost 
immediately ruled out by the presence of the inhomogeneous potential term;
%see the discussion after \eqref{introD} below for more on this aspect.
%
%Indeed, a general potential kills many of the invariances of the equation, hence the effectiveness of 
%standard vectorfields\footnote{While it is certainly reasonable to... .},
%and, at the same time, impacts the linear flow in such a way 
%that normal forms transformations can no longer be determined algorithmically as in the flat case.
%
%We will discuss this aspect more after \eqref{introD}.
Furthermore, the perturbed operator may have eigenvalues below the continuous spectrum, or resonances at the edge.
This is an important aspect to consider, %and we will discuss it a little below, 
but it is not the main focus of the present paper.

The second main difficulty %pointed out above %that appears in the equation for the perturbation is that,
%in both cases (1) and (2) above, one typically obtains 
is the presence of {\it quadratic nonlinear terms} in the equation for $v$.
%(and might contain additional terms with higher homogeneity in $v$ and the coupling of $Q$ and $v$)
Quadratic nonlinearities are the most difficult to control due to the slow decay.
Moreover, in both cases (1) and (2) one does not expect 
localized coefficients in front of these quadratic terms 
- unlike in the case of solitons for models 
with nonlinearities which are at least cubic -
that can be leveraged through improved local decay.
As a result, tools from the linear theory of perturbed operators %$\partial_t + i L(D) + V$
%(e.g. dispersive and Strichartz type estimates) 
and energy methods are ineffective to treat these equations. %`pure' quadratic nonlinearities.
%could comment out

In the unperturbed case, $V=0$, quadratic models like \eqref{introeq} present similar issues.
Starting with seminal works by various authors in the `80s, including \cite{C,K,shatahKGE,KPV},
many techniques have been developed for the study of these classes of weakly nonlinear equations.
%See the paragraph at the end of \ref{ssecdFTintro} for references to 
%some of the more recent advances in this area.
Roughly speaking, when dispersive effects are weak,
one needs a refined analysis of the nonlinear interactions;
this can be done in various ways, including normal form analysis, 
commuting vectorfields methods, harmonic analysis tools and more.
%
%... Steady developments ... Recent achievements
%
%In recent years this line of research has culminated in the 
%
However, these techniques are hard to adapt to the perturbed case $V\neq 0$,
and, to our knowledge, %approaches that are as 
alternative systematic and robust approaches %as these classical ones 
have not been developed so far.
%To our knowledge %nonlinear equations with external potentials and low power nonlinearities,
%(e.g. at or below the Strauss exponent)
%have not been systematically approached %in the literature 
%so far, 
In fact, very little is known
about the long-time behavior of nonlinear equations with external potentials and low power nonlinearities
(e.g. at or below the Strauss exponent),
especially in comparison to the very rich theory for higher power nonlinearities, or the unperturbed cases.
Recently, there have been some results in this direction
in particular cases, such as the case of one spatial dimension \cite{N,DelortNLSV,GPR,ChPu},
the case of small potentials \cite{Leger1,Leger2}, and the case of non-resonant nonlinearities \cite{GHW};
see %\ref{introrefs} 
below for more on \cite{Leger1} and \cite{GHW}.
%which are the works most related to ours.

\subsubsection*{Nonlinear evolution with a potential}
The aim of this paper is to initiate a refined study of nonlinear interactions in the perturbed setting,
for models like \eqref{introeq}.
In particular, the method developed here addresses the combination of (A) and (B)
%Therefore, we will restrict our attention to 
in the simplest case of a potential such that $L_V$ has no eigenvalues or resonances.
%\footnote{
The understanding is that the analysis can be used in the cases (1) and (2) above,
when one restricts the nonlinear interactions to the continuous spectrum of the relevant operator.
The additional interplay with the discrete modes, %\footnoe{Note that discrete modes can be somewhat separated by projecting}, 
if any are present, needs to be dealt with separately.
%} %by other means.

%in other words, we look at the projection of the linear part of the equation for $v$
%onto the continuous spectrum, and disregard nonlinear interactions between

%The main motivation of this paper is that, 

%As mentioned above, to our knowledge, 

%Understanding the interaction of localized waves, say a product of two linear waves,
%through the uncertainty caused by the de-localizing interaction of each wave with potential.
One of the main difficulties in treating problems like \eqref{introeq}
is to understand how two (or more) localized waves interact
in the presence of an external potential.
As we will see more in details below, the potential, although smooth and localized, 
has a {\it ``delocalizing''} effect:
waves moving under its influence have a higher degree of uncertainty compared to `flat' waves
and therefore are harder to control precisely.
%comment more??
Moreover, the potential {\it ``decorrelates''} frequencies: %of interacting waves:
the sum of the frequencies of two interacting waves is not the frequency of the product, 
as it is in the unperturbed case.
The method that we are proposing here deals with these issues. %presented above.

More precisely, Sections \ref{SecLin}--\ref{secBE} and Section \ref{SecOther}
contain results that are generally applicable to Schr\"odinger operators
$H=-\Delta + V$, for a real, regular and decaying potential $V$.
In essence, we show how to use the Fourier transform adapted to $H$ 
- the so-called {\it distorted/perturbed Fourier transform} - 
to write in a fairly explicit way the product of two functions in distorted frequency space.
The key is to identify the singular structure of the product, %in frequency space,
which then permits a precise analysis of nonlinear oscillations;
see Section \ref{secidea} for a more detailed explanation.
These results can in principle be applied as a black-box, or with some modifications, 
to tackle the nonlinear analysis on the continuous spectrum for several
problems in the categories (1) and (2) above.

In Sections \ref{secdkL2} and \ref{ssecmu23Est} we give a first application of our general approach 
and study in detail the quadratic nonlinear Schr\"odinger equation 
\begin{align}
\label{NLSV0}
i\partial_t u  + (-\Delta + V)u = u^2, \qquad u(t=0) = u_0, \qquad x\in\R^3.
\end{align}
%for $V$ is a real potential, sufficiently regular and fast decaying, with no eigenvalues or resonances.
This is a prototypical model for a nonlinear equation with a potential;
we have chosen the Schr\"odinger equation
%We chose the Schr\"odinger equation due to 
for the simplicity of its dispersion relation, 
but our analysis could extend to the Klein-Gordon equation,
or the wave equation under some additional null form assumptions. %on the nonlinear terms.
Note that a quadratic nonlinearity in $3$d is critical with respect to the Strauss exponent.
The choice of the nonlinearity $u^2$, as opposed to $\bar{u}^2$ or $|u|^2$ is relevant.
The $\bar{u}^2$ nonlinearity was treated in \cite{GHW}, and we can easily included it in our analysis.
The important difference %with $u^2$, 
is that $\bar{u}^2$ admits a global normal form transformation which effectively
makes the equation sub-critical relative to the Strauss exponent. 
In particular, one does not need a precise analysis of spatially localized interactions.
For the $|u|^2$ nonlinearity the existence of global solutions  is still open even for $V=0$.
For $V=0$ the other nonlinearities $au^2 + b\bar{u}^2$ have been treated in \cite{HNNLS3d} and \cite{GMS1}.

%There are many interesting physical situations %see ... for more examples,
%where the nonlinear interactions that one encounters cannot 
%be treated by standard methods and relying tools from the linear theory, such as decay estimate
%and Strichartz estimate, or extensions of them.

%\medskip
\subsubsection*{Distorted Fourier Transform (dFT) and Nonlinear Spectral Distribution (NSD)}\label{ssecdFTintro}
Our general set-up is based on the use of the distorted Fourier Transform (dFT)
%that is, the Fourier transform 
adapted to the Schr\"odinger operator $H=-\Delta+V$.
For the moment, it suffices to say that under suitable decay and generic spectral assumptions on $V$,
%see \eqref{assV1} (or if one restricts to the the continuous spectrum)
the familiar formulas relating the Fourier transform and its inverse (in dimension $d=3$) 
hold if one replaces (up to constants) $e^{ik\cdot x}$ by the generalized eigenfunctions $\psi(k,x)$,
which solve $H \psi(x,k) = |k|^2 \psi(x,k)$, for all $k \in \R^3\smallsetminus\{0\}$.
That is, for any $g\in L^2$, there exists a unitary operator $\Ftil$ defined by
\begin{align}\label{dFT}
\begin{split}
& \wt{\mathcal{F}}g(k) := \widetilde{g}(k) = \frac{1}{(2\pi)^{3/2}}
  \int_{\R^3} \overline{\psi(x,k)} g(x)\, \mathrm{d}x, \qquad 
  \\ & \mbox{with} 
  \qquad \Ftil^{-1}g (x) := \frac{1}{(2\pi)^{3/2}}\int_{\R^3} {\psi(x,k)} g(k) \,\mathrm{d}k,
\end{split}
\end{align}
%where one can define the generalized eigenfunctions solving 
%\begin{align*}
%\forall k \in \R^3\smallsetminus\{0\}, \qquad %(-\Delta + V ) \psi(x,k) = |k|^2 \psi(x,k)$.
%\end{align*}
that diagonalizes the Schr\"odinger operator: $\Ftil H = |k|^2 \Ftil$.
See Theorem \ref{theodFT}.

For a solution $u$ of \eqref{NLSV0} - with the obvious modifications in the case of other
dispersion relations or other nonlinearities - 
we look at the `profile' or `interaction variable' %$f$ by
\begin{align}\label{introf}
f(t,x) := \big( e^{-it(-\Delta + V )} u(t,\cdot) \big) (x), 
  \qquad \wt{f}(t,k) = e^{-it|k|^2} \widetilde{u}(t,k),
\end{align}
which satisfies the equation $\widetilde{f}(t,k) = \widetilde{u_0}(k) - i \mathcal{D}(t)(f,f)$ where
\begin{align}\label{introD}
\begin{split}
%& \widetilde{f}(t,k) = \widetilde{u_0}(k) - i \mathcal{D}(t)(f,f),
%\\
& \mathcal{D}(t)(f,f) := \int_0^t \iint_{\R^3 \times \R^3} 
	e^{is (-|k|^2 + |\ell|^2 + |m|^2 )} \widetilde{f}(s,\ell) \widetilde{f}(s,m)
  \, \mu(k,\ell,m) \, \mathrm{d}\ell \mathrm{d}m\,\mathrm{d}s,
\end{split}
\end{align}
with
\begin{align}\label{intromu}
\mu(k,\ell,m) := \frac{1}{{(2\pi)}^{9/2}} \int_{\R^3} \overline{\psi(x,k) }\psi(x,\ell) \psi(x,m) \, \mathrm{d}x.
\end{align}
\eqref{introD} is Duhamel's formula for \eqref{NLSV0} in distorted Fourier space.
The distribution $\mu$ characterizes the interaction between the generalized eigenfunctions,
and we call it the ``{\it Nonlinear Spectral Distribution}'' (NSD). 

%comment/refs to overlap functions and other instances of this appearing?}

Note that in the unperturbed case $V=0$ the NSD is just a delta function $\delta(k-\ell-m)$.
%in particular the two input frequencies must sum to the output frequency.
In contrast with this, in equations \eqref{introD}-\eqref{intromu} %the difficulties mentioned above are already apparent:
%in \eqref{introD} 
all frequencies interact with each other without any a priori constraint.
%One could still think that if $k\pm\ell\pm m$ is far away from zero, then the contribution would be better;
%in a certain sense, this is true, but $\mu$ could still be singular
Looking at \eqref{introD} we see that the set where the integral has no oscillations in time $s$ 
is always larger than in the case $V=0$;
%(this is a manifestation of 
%The subset of $\R^9$ on which there are no oscillations (in $\ell,m$ or $s$)
this implies that time averaging and 
the standard theory of normal forms transformations %(or integration by parts in $s$) 
are less efficient. %in the presence of a potential. %cannot be applied.
%This is connected to the 
At the same time, as we shall see, $\mu(k,\ell,m)$ is singular on a much larger set compared to $\delta(k-\ell-m)$;
for example, it is singular when $|k-\ell|=|m|$ (see \eqref{i16}) or when $|k|=|\ell|+|m|$.
Then, even when the oscillatory exponential factor in \eqref{introD} is non-stationary
%(that is, its gradient is non-zero) 
one cannot directly obtain cancellations;
these are only possible if non-stationarity holds in the directions where $\mu$ is regular.
The singular behavior of $\mu$ is another manifestation of both the `uncertainty'
and the `loss of invariance properties'
caused by the external potential, and leads to the ineffectiveness of a direct 
application of methods based on vectorfields.

One of our main ideas is to study precisely the structure of $\mu$ and its singularities as a distribution of $\R^9$.
After identifying all the singularities, we are naturally led to look at bilinear operators %similar to \eqref{introD} 
with kernels that are singular on certain annuli in frequency space.
For all the relevant operators that appear we prove suitable bilinear estimates of H\"older (Coifman-Meyer) type.
See for example \eqref{i17}-\eqref{i18} for an informal statement of this type.
In the specific case of the NLS equation \eqref{NLSV0} we can then proceed 
to analyze in detail the integral \eqref{introD}. %, taking into account the singularities of $\mu$.
We point out that the general analysis of the NSD can be used for other equations as a black-box.
The ideas for the analysis of the nonlinear model \eqref{NLSV0} are also quite general,
and in particular the integration by parts using {\it ``good vectorfields''}
that are tangential to the singularities of $\mu$;
see \eqref{i41}-\eqref{i50} and the last part of \ref{ssecidea1} and for more on this.
%However, the details of the nonlinear analysis are more specific of the model at hand.

Overall, the present approach can be seen as an extension of the analysis 
put forward by \cite{GNTGP,GMS1,GMS2} for the case $V=0$,
where one uses the regular Fourier transform (in which case, recall, the NSD is a delta)
and analyzes the corresponding oscillatory integral \eqref{introD}.
In recent years, some deep advancements have been made 
starting from these types of basic ideas, and the resulting methods
%combined with deep advancements in related directions,
%combined with deep advancements in multilinear Harmonic analysis, 
have been
%methods inspired by this type of analysis have been proven to be 
quite successful in the study of global regularity and asymptotics for small solutions of
%(quasilinear)
dispersive and wave equations.
See for example  \cite{IP1,IoPu2,GIP,DIPP} and references therein
where the authors study the stability of `trivial' equilibria 
(e.g. a flat and still sea in the context of the free-boundary Euler equations, 
or a neutral plasma in the context of the Euler-Maxwell system).
%led to great advances on the stability (close to , e.g.
Our long-term hope is that, following the approach in this paper,
parallel developments can be made in the context of nonlinear 
equations with potentials, leading to advances in the study of 
the long-time dynamics around non-trivial equilibria.

%Much of our analysis appears to be quite general, and several results - especially those 
%about the structure of the product of generalized eigenfunctions - 
%can be used as a black-box for any other dispersive or wave equation with a large external potential.

\subsubsection*{More background and related works}\label{introrefs}
%The study of non-linear dispersive and wave equations
%and the stability of special solutions is a huge subject, 
%and a complete survey of it is beyond the scope of this paper (and beyond the ability of the author).
%For a more general discussion about the topic we refer to the surveys \cite{Sof06}, \cite{Taosurvey} 
%\cite{xxx}, and references therein.
%nicer sentence...

%In what follows we will only concentrate on 

%Before detailing our results, we discuss a few works connected to
%problems that our method aims to address, and some prior works on models similar to \eqref{introeq}.
%For more general overviews of classical results about solitons
%and nonlinear dynamics with external potentials we refer to the surveys \cite{Sof06,Taosurvey,WeinSur} 
%and references therein.

The first question one asks when studying equations with potentials such as \eqref{introeq},
is how much of the linear theory for solutions of $i\partial_t u + L(D) u=0$
can be carried onto perturbed/inhomogeneous linear solutions.
As an example, classical results for linear Schr\"odinger operators \cite{JSS,GolSch},
guarantee that pointwise decay estimates like
\begin{align*}
{\| e^{itH}P_c f\|}_{L^1\mapsto L^\infty} \lesssim |t|^{-d/2}, \qquad x\in \R^d,
\end{align*}
hold under mild assumption on the decay of $V$ (here $P_c$ is the projection onto the continuous spectrum). 
Generalizations to other dispersive and wave equations are also known.
For more on dispersive estimates, see the survey of Schlag \cite{Sch1},
the book \cite{KKbook} %the work \cite{} 
and references therein.
Strichartz estimates can also be derived for solutions of the perturbed linear problem under fairly general
assumptions on $V$.

For nonlinear problems such as \eqref{introeq},
the first attempt is to use linear estimates (and energy estimates) to control the flow for long times.
This is in parallel to the classical strategy that one would use without the potential;
see the seminal work of Strauss \cite{Strauss0} and \cite{StraussCBMS} and references therein.
Even in the absence of discrete spectrum this approach works only when the spatial dimension
and homogeneity of the nonlinearity are high enough, so that dispersive effects are sufficiently strong.
%for example, it would {\it not} work at or below the Strauss exponent, e.g. for quadratic Schr\"odinger
%or wave equations in $3$d, or for quadratic Klein-Gordon in $2$d. 
% Strauss exponent for KG ??
%In the presence of One major difficulty that can arise 
%
If the perturbed operator has discrete spectrum the situation is more delicate.
%even when the homogeneity of the nonlinearity and/or the dimension are large.
The main issue is the presence of linear and nonlinear bound states (time-periodic localized structures),
and their interaction with the radiative/dispersive part of the solution.
Among the many important works %that study the global behavior of nonlinear equations
%with large potentials in the presence of bound states,
in this direction, we mention
Soffer-Weinstein \cite{SofWeinA,SofWein2}
Tsai-Yau \cite{TsaiYau}, Gustafson-Nakanishi-Tsai \cite{GNT04},
Bambusi-Cuccagna \cite{Bambusi-Cuccagna}, Kirr-Zarnescu \cite{KirZar}.
%Buslaev-Perelman / Sulem in 1d... \cite{}
For more general overviews we refer the reader to the surveys \cite{Sof06,WeinSur} and reference therein.

%Concerning works where discrete spectrum present of discrete spectrum and associated bound states
%(time-periodic localized structures).
%In such cases, the analysis is much more difficult,
%Neverthless, the global behavior of solutions can be understood in some cases.

%From GHW: An interesting related problem is the situation where the
%strength of the dispersion is sufficient, but the potential exhibits linear, or nonlinear,
%bounds states. We refer for instance to [4, 21, 43, 46] for more on this direction of research.

%Concerning nonlinear equation with large potentials, 
%seminal contributions include 
%the beautiful work on resonances and radiation damping by Soffer-Weinstein \cite{SofWein2}, %\cite{SofWeinA}

%Cite others works where dispersion is sufficient (because of dimension, or high power nonlinearity)
%but there are other issues, like bound states etc

%On the other hand these works are able to tackle delicate issues such as linear or nonlinear bound states.

%one has that (a) the dispersion from the linear group is sufficiently strong, 
%and/or (b) the nonlinearity is of sufficiently high degree and,
%
%and (c) the structures are localized

%We believe that further developments of our method, 
%combined with some of the techniques and ideas in the above papers
%will give access to a large class of open problems.

%What we do here concerns the continuous part of the spectrum, 
%which needs to be addressed in any of the above motivating problems

The works cited above are characterized by strong dispersive effects due to the combination 
of the large spatial dimension and/or a high power (or highly localized) nonlinearity.
The situation is quite different, and much less is known,
%From GHW:
%All the works which we have discussed so far examine (as we doing for the rest of this manuscript) 
%the situation where the linear operator spectrum is continuous, but one
%cannot get global existence simply by means of the dispersive estimates, since the decay
%rate given by the linear group is too weak and the power of the nonlinearity is not high
%enough to compensate for it. 
when one cannot get global existence by means of dispersive estimates or energy methods,
even if the perturbed operator has only continuous spectrum. %We discuss this next.
%the power of the nonlinearity is not high enough to compensate for it
%the decay rate of the linear group is too weak and

An interesting work in this direction is the paper 
of Germain-Hani-Walsh \cite{GHW} who treated an NLS equation like \eqref{NLSV0} 
with a large, generic, and decaying potential, and with a $\bar{u}^2$ nonlinearity.
In \cite{GHW} the authors use the dFT %to write the equation in an analogous form to \eqref{introD},
and lay some groundwork for understanding the NSD \eqref{intromu}.
In particular, they establish Coifman-Meyer type bilinear estimates for operators whose kernel 
is given by $\mu$ times a Coifman-Meyer-type symbol, 
as well as estimate for their bilinear commutator with $\partial_k$.
%this latter is relevant for the control of
Since the $\bar{u}^2$ nonlinearity leads to 
a factor of $-(|k|^2+|\ell|^2+|m|^2)$ in the exponential phase in \eqref{introD},
in \cite{GHW} the authors can directly exploit time oscillations 
and do not need to analyze $\mu$ and its structure and regularity in further details,
nor need to exploit oscillations in distorted frequency space.
In general one does not expect a typical scenario to be as favorable
and, indeed, this is not the case for \eqref{NLSV0}.
%A work that is very close to ours and can be considered a starting point for our analysis,
More recently, L\'eger \cite{Leger1} was able to treat \eqref{introeq} 
under the assumption that the potential, which can also be mildly time-dependent, is small. 
The smallness of $V$ permits a more `perturbative' approach using the regular Fourier transform
but the problem is still hard due to the presence of coherent resonant interactions (which are absent for $\bar{u}^2$).
%The author uses a nice combination
%of structural results for wave operators \cite{BeSch} 
%and a refined study of oscillations via the regular Fourier transform to obtain decay and scattering.
%
%Coming to the nonlinear Schr\"odinger equation, in the case $V=0$, 
%global regularity was obtained %the model \eqref{introeq}
%by Hayashi-Naumkin \cite{HNNLS3d} and %a new proof, similar in spirit to our approach 
%Germain-Masmoudi-Shatah \cite{GMS2} for nonlinearities of the form $a u^2 + b \bar{u}^2$.
%\cite{GMS2} also contains generalizations to certain systems.
%
%Pertaining to the use of the dFT in nonlinear problems,
We also mention %we refer the reader 
the recent work of Kenig and Mendelson \cite{KenMen} where the authors 
obtain a soliton stability result with high probability for the focusing energy critical $3$d wave
using a randomization procedure through distorted Fourier projections.

%(adapted to the Schr\"odinger operator linearized around a soliton).
Finally, we point out that in the one dimensional case the distorted Fourier transform 
has already been used fairly effectively.
Indeed, in $1$d the generalized eigenfunctions of $H$ satisfy ODEs and are given as solutions 
of much simpler Volterra-type integral equations, as opposed to solutions of the 
Lippmann-Schwinger integral equation \eqref{i0};
in particular, their difference with the standard exponentials decays to zero at infinity as fast as the potential,
and the structure of the NSD is %simpler and 
more explicit.
For the $1$d cubic NLS this approach has been used by the author with Germain and Rousset \cite{GPR} 
and by the author and Chen \cite{ChPu}.
See also the related earlier works by Naumkin \cite{N} and Delort \cite{DelortNLSV} on the same problem,
of Cuccagna-Georgiev-Visciglia \cite{CGV} in the subcritical case;
other related works in $1$d include %of Donninger-Krieger 
\cite{DoKri} on wave equations,
%Lindblad-Soffer 
\cite{LS,LLS} on certain Klein-Gordon equations with non-constant coefficients,
and \cite{KowMarMun} on the stability of the $\phi^4$ kink.

%We also refer to the work in completion by the author and Germain \cite{GPupreprint} on quadratic Klein-Gordon
%where the presence of the potential brings a new coherent phenomenon which has not been treated 
%before in $1$d.

%\subsection{Open problems ??}

%Finally we would like to mention some problems that we believe can now be approached
%with our methods and some refinement of our techniques: %can be approached with these method
%
%\begin{itemize}
%
%\item Extension/Adaptation to non-constant coefficients
%
%\item Radiation damping for quadratic Klein-Gordon
%
%\item Stability of \dots
%
%\item Adapting this approach to dimension $d=2$.
%
%\end{itemize}

%We consider this as a first step towards the understanding the stability of non-trivial solutions 
%for dispersive and hyperbolic equations in $3$d.

\subsubsection*{Main result for quadratic NLS} %on global existence}
%\medskip 
%\subsection{The NLS model}
%Given a real potential $V$ with no bound states which is sufficiently smooth and localized, 
We consider the equation %Here is our result for the quadratic NLS model in $3d$
\begin{align}
\label{NLSV}
i\partial_t u  + (-\Delta + V)u = u^2
\end{align}
with an initial data $u(t=0) = u_0$.
%This is a model for the linearization around a non-decaying/topological soliton for a dispersive equation (Explain).
Under suitable assumptions on $V$ and $u_0$, 
our main result is the global-in-time existence of solutions, 
together with quantitative pointwise decay estimates and %information about the asymptotic behavior.
global bounds on certain weighted-type norms of the solution.
%\subsubsection*{Assumptions}
More precisely, let
\begin{align}\label{Nparam}
N = 2000, \quad N_1 = 200, %(?)
\end{align}
and assume that 
%\begin{align}\label{assV1}
%V \in W^{N_0,1}
%\end{align}
%and
\begin{align}\label{assV2}
V\in H^N, 
  \qquad \int_{\R^3} (1+|x|)^{N_1+10} |\nabla_x^{\alpha} V(x) | \, \mathrm{d}x < \infty, \qquad 0\leq |\a| \leq N_1+10,
\end{align}
and that
\begin{align}\label{assV1}
\mbox{$H=-\Delta+V$ has no eigenvalues or resonances. %$V$ is generic
}
\end{align}
This is our main result for \eqref{NLSV}.

\begin{theorem}\label{maintheo}
Consider \eqref{NLSV} under the assumption \eqref{Nparam}-\eqref{assV2}.
%We assign sufficiently smooth and localized initial data $u_0 = u(t=0)$ such that
Consider an initial data $u_0$ satisfying  
\begin{align}\label{data0}
{\| u_0 \|}_{H^{N}} + {\| \nabla_k \widetilde{u_0} \|}_{L^2} + {\| \nabla_k^2 \widetilde{u_0} \|}_{L^2} 
  \leq \e_0,
\end{align}
where $\widetilde{g} = \widetilde{\mathcal{F}}(g)$ 
denotes the distorted Fourier transform of $g$ as defined in Theorem \ref{theodFT}.
%\eqref{data0} and $V$.

Then, there exists $\bar{\e}$ small enough such that, for all $\e_0 \leq \bar{\e}$,
the equation \eqref{NLSV} %under the assumptions \eqref{Nparam}-\eqref{assV2}
admits a unique global solution $u\in C(\R; H^{N}(\R^3))$ with $u(t=0)=u_0$, satisfying
\begin{align}\label{maintheoconc}
{\|u(t)\|}_{H^{N}} + \jt^{1+\alpha}{\|u(t)\|}_{L^\infty} \lesssim \e_0
\end{align}
for some $\alpha > 0$.
\end{theorem}

Here is a few comments about the statement

\begin{itemize}

\item[$-$] {\it High regularity}: 
We consider very smooth solutions in a high Sobolev space $H^N$ (and therefore require 
the same amount of regularity for the potential).
Although \eqref{NLSV} is a semilinear equation we find it useful to control a large number derivatives
in many parts of our analysis;
%of the solution.
%This is helpful 
for example, when we perform various expansions that lose derivatives
or want to deal with non-standard symbols of multilinear operators that have some losses 
%whose regularity is non-uniform in the size of the frequencies
at high frequencies.
Thanks to the $H^N$ bound we can think of frequencies as being effectively bounded from above 
by a small power of time, and thus having minimal impact on all our estimates for the evolution.
These smoothness and decay hypotheses are clearly not optimal, and the values of $N$ and $N_1$ 
can certainly be improved. 
%; we have chosen some large enough values for the sake of convenience and to ...

\smallskip
\item[$-$] {\it Distorted weighted norm}: The last two norms in \eqref{data0} are distorted Fourier analogues
of the more standard $L^2(\jx^4 \mathrm{d}x)$ norms which are found in the literature on this
and similar types of problems.

\smallskip
\item[$-$] {\it Decay and scattering}:
Part of the conclusion of our nonlinear analysis, see the bootstrap Proposition \ref{proBoot},
is that we can control, up to some small power of time, distorted Fourier analogues
of weighted norms along the evolution. In particular we can prove that,
for $f=e^{-itH}u$, $\partial_k \wt{f}(t)$ is uniformly bounded in $L^2_k$,
and $\partial_k^2 \wt{f}(t)$ grows in $L^2_k$ at most like $\jt^{1/2+\delta}$ for $\delta>0$ small.
These bounds, combined with standard linear decay estimates, interpolation, 
and the boundedness of wave operators,
then imply that $u(t)$ decays pointwise at an integrable rate of $\jt^{-1-\alpha}$ for some $\alpha>0$.
In particular the solution {\it scatters} to a (perturbed) linear solution as $|t|\rightarrow \infty$.
%$\nabla_k \widetilde{u_0}$ and $\nabla_k^2 \widetilde{u_0}$
\end{itemize}

\medskip
\subsection*{Acknowledgments}
F.P. is supported in part by a startup grant from the University of Toronto, 
a Connaught Fund New Researcher grant, and NSERC grant RGPIN-2018-06487.

\noindent
A.S. is supported by  in part by NSF grant DMS-160074.

\medskip
\section{Main ideas and strategy}\label{secidea}
In this section we first give a summary of our strategy pointing out some important elements in our proofs.
Then we introduce the necessary notation, 
define the functional space in which we will work to prove global existence for \eqref{NLSV},
and state the main bootstrap proposition which will imply Theorem \ref{maintheo}.

\iffalse
formula expressed in the distorted Fourier space, using the Fourier transform adapted to the Schr\"odinger operator.
This gives an expression for the solution in terms of an oscillatory integral whose amplitude 
(related to the profiles of the solution) has limited smoothness, 
and which is taken with respect to a complicated singular kernel.
The aim is to obtain decay estimates and control the solution in the distorted Fourier analogue of weighted spaces.
One needs to perform a rather precise analysis because the nonlinearity is strong (only quadratic) and, even for $V=0$, 
it presents non-trivial coherent interactions of waves that are classically resonant (the interaction of two linear waves can produce a linear oscillation)
and, at the same time, also spatially concentrated (the velocity of the two interacting waves can be the same).
In particular, an important step in the analysis is to understand what we call the ``nonlinear spectral measure'', 
that is, the interaction
of three generalized eigenfunctions, obtaining precise asymptotic expansions for it in the form of explicit singular kernels.
\fi

%We remark that this is a novel approach to nonlinear equations with a potential.
%Although in part inspired by recent advances in the study of weakly nonlinear waves,
%this is the first attempt at developing
%key multilinear harmonic analysis tools, such as weighted bilinear estimates, in this context.

\subsection{Main steps}\label{ssecidea1}
%Here is a breakdown and some important elements of our strategy.

\subsubsection*{Step 1: Distorted Fourier Transform and the Nonlinear Spectral Distribution} %Duhamel's formula}
Under our assumptions on the potential $V$ we can define,
for $k \in \R^3\smallsetminus\{0\}$, a family of generalized eigenfunctions 
associated to $H=-\Delta + V$ as the unique solutions of the problem
\begin{align*}
(-\Delta + V ) \psi(x,k) = |k|^2 \psi(x,k), \qquad k \in \R^3\smallsetminus\{0\},
\end{align*}
with the asymptotic condition $v(x,k):=\psi(x,k) - e^{ix\cdot k} = O(|x|^{-1})$
and verifying the Sommerfeld radiation condition $r (\partial_r - i|k|) v(x,k) \rightarrow 0$, 
for $r=|x| \rightarrow \infty$.
%For $k$ in the continuous spectrum the generalized 
%eigenfunctions then satisfy the integral equation
These satisfy the integral equation
\begin{align}
\label{i0}
\psi(x,k) = e^{ix\cdot k} - \frac{1}{4\pi} \int_{\R^3} \frac{e^{i|k||x-y|}}{|x-y|} V(y) \psi(y,k) \, \mathrm{d}y.
\end{align}
The family $\{\psi(\cdot,k)\}$ forms a basis for the absolutely continuous spectrum of $H$
and thanks to classical results (see Theorem \ref{theodFT})
the familiar formulas relating the Fourier transform and its inverse in dimension $d=3$ hold if one replaces
(up to constants) $e^{ik\cdot x}$ by $\psi(k,x)$.
%see \eqref{dFT} and Theorem \ref{theodFT} below. 

%\begin{align*}
%\wt{\mathcal{F}}f(k) = \widetilde{f}(k) = \int_{\R^3} \overline{\psi(x,k)} f(x)\, \mathrm{d}x 
%  \qquad \mbox{and} \qquad f(x) = \int_{\R^3} {\psi(x,k)} \widetilde{f}(k) \,\mathrm{d}k. 
%\end{align*}

Recall that, given $u$ solution of \eqref{NLSV},
one has the distorted Duhamel's formula for the profile 
%we define the {\it profile} or {\it interaction variable} $f$ by
\eqref{introf}-\eqref{introD} with the nonlinear spectral distribution (NSD) defined by \eqref{intromu}. %:
%\begin{align}\label{i1}
%f(t,x) := \big( e^{-it(-\Delta + V )} u(t,\cdot) \big) (x), 
%  \qquad \wt{f}(t,k) = e^{-it|k|^2} \widetilde{u}(t,k).
%\end{align}
%This satisfies the equation
%\begin{align}\label{i2}
%\begin{split}
%& \widetilde{f}(t,k) = \widetilde{u_0}(k) - i \mathcal{D}(t)(f,f)
%\\
%& \mathcal{D}(t)(f,f) := \int_0^t \iint e^{is (-|k|^2 + |\ell|^2 + |m|^2 )} \widetilde{f}(s,\ell) \widetilde{f}(s,m)
%  \, \mu(k,\ell,m) \, \mathrm{d}\ell \mathrm{d}m\,\mathrm{d}s,
%\end{split}
%\end{align}
%
%
%\begin{align}\label{i3}
%\mu(k,\ell,m) := (2\pi)^{-9/2} \int \overline{\psi(x,k) }\psi(x,\ell) \psi(x,m) \, \mathrm{d}x.
%\end{align}

\subsubsection*{Step 2: Expansion of the generalized eigenfunctions}
To understand the global-in-time properties of solutions through \eqref{introD},
one needs a very precise understanding of the NSD, 
and the ability to exploit generalized frequencies oscillations.
We begin by separating the flat and the potential contributions to the generalized eigenfunction by setting
\begin{align}\label{i4}
\begin{split}
\psi(x,k) & = e^{ix\cdot k} - e^{i|k||x|} \frac{1}{4\pi|x|} \psi_1(x,k),
\\
\psi_1(x,k) & := \int_{\R^3} e^{i|k| [ |x-y| - |x| ]} \frac{|x|}{|x-y|} V(y) \psi(y,k) \, \mathrm{d}y,
\end{split}
\end{align}
and then expanding $\psi_1$ in negative powers of $|x|$:
\begin{align}\label{i5}
\psi_1(x,k) = \sum_{j=0}^{n} g_j(\omega,k) \, r^{-j} \jk^j  + R(x,k), \qquad r := |x|, \, \omega := \frac{x}{|x|},
\end{align}
%Here the sum involves a finite number of terms, 
where $R$ is a sufficiently regular remainder that decays faster than $r^{-n}$ with $n$ large enough,
and the coefficients $g_{j}(\omega,k)$ 
belong to a suitably defined symbol class whose prototypical element has the form 
\begin{align}\label{i6}
g(\omega,k) = \int_{\R^3} e^{-i|k| \omega\cdot y} f(y) \, \mathrm{d}y, %\qquad {(1+|y|)}^N f(y) \in L^1,
\end{align}
for a fast decaying $f$.
In particular, these are smooth functions of $\omega$ with some singularity in $k$.
For full details on the expansion \eqref{i5} see Subsection \ref{ssecpsiexp}, 
and in particular the statement of Lemma \ref{lemmapsi1}.
%For the sake of explanation, in the remaining of this introduction
%let us just think of them as smooth functions of $\omega$ times 
Notice that  $\omega$-derivatives of \eqref{i6} grow with $|k|$;
this causes some technical difficulties in dealing with high frequencies.
We resolve this by restricting the nonlinear analysis for the evolution %performed below 
to frequencies
%\begin{align}
$|k| \leq \jt^{\delta}$
%\end{align}
for a small $\delta>0$, and treating high frequencies $|k| \geq \jt^\delta$ by leveraging the high $H^{N}$
smoothness of solutions.
We do not discuss the estimate for high frequencies in this explanation, but refer the reader to \ref{secHF}.

\subsubsection*{Step 3: Asymptotics for the NSD}
The next step consists of plugging-in the (linear) expansions \eqref{i4}-\eqref{i5} 
into the expression \eqref{intromu} for $\mu$ to obtain an expansion of the NSD.
We see that for large $|x|$
\begin{align}
\label{i10}
\begin{split}
& \overline{\psi(x,k) }\psi(x,\ell) \psi(x,m) = e^{ix\cdot (-k+\ell+m)} 
  - \frac{1}{4\pi|x|} e^{i|m||x|}  e^{ix\cdot(-k+\ell)} g_0(\omega,m) 
  \\ & - \frac{1}{4\pi|x|} e^{i|\ell||x|}  e^{ix\cdot(-k+m)} g_0(\omega,\ell)
  - \frac{1}{4\pi|x|} e^{-i|k||x|} e^{ix\cdot(\ell+m)} \bar{g_0(\omega,k)} + O(|x|^{-2}),
\end{split}
\end{align}
%Then, from  \eqref{intromu} and \eqref{i10} we have 
so that (up to irrelevant constants)
\begin{align}\label{i11}
\begin{split}
\mu(k,\ell,m) \approx \delta(k-\ell-m) + \int_{\R^3} \frac{1}{|x|} e^{i|m||x|}  e^{ix\cdot(-k+\ell)} g_0(\omega,m)  
  \, \mathrm{d}x \\ + \,\, \mbox{``similar or better terms''}.
\end{split}
\end{align}
This leads us to study the behavior of oscillatory integrals of the form
\begin{align}\label{i12}
\nu(p,q) := \int_{\R^3} \frac{1}{|x|} e^{i|p||x|}  e^{ix\cdot q} g_0(\omega,p)  \, \mathrm{d}x.
\end{align}
In Proposition \ref{Propnu+} we give an expansion for this integral and, in particular, establish that
\begin{align}\label{i15}
\nu(p,q) \approx \frac{1}{|q|} \Big( \delta(|p|-|q|) + \pv \frac{1}{|p|-|q|} \Big) + \mbox{``better terms''}
\end{align}
up to some coefficient involving $g_0(\pm q/|q|,p)$. 
%that bring some losses for frequencies $|p|,|q|\gtrsim 1$ 
%(but recall that since we are only looking at $|p|,|q| \lesssim \jt^\delta$ these losses 
%can be tolerated in the estimates for the evolution problem).
%The structure of the singularities of the ``better terms'' in \eqref{i15} is 
%also studied in details in Proposition \ref{Propnu+}.
%The formula \eqref{i15} 
This gives an explicit expression for the leading order in the expansion of $\mu$:
we can essentially think that
\begin{align}\label{i16}
\mu(k,\ell,m) \approx \frac{1}{|\ell-k|} \pv \frac{1}{|\ell-k|-|m|} + \mbox{``similar or better terms''}.
\end{align}
%The structure of the singularities of the ``better terms'' in \eqref{i15} is 
%also studied in details in Proposition \ref{Propnu+}.

\subsubsection*{Step 4: Multiplier estimates}
%Before going back to the evolution problem, 
Next, we are going to establish some multiplier estimates for the terms in the expansion of $\mu$.
The typical statement will be an H\"older/Coifman-Meyer type estimate 
for the multiplier appearing on the right-hand side of \eqref{i16}.
More precisely, if we define non-standard pseudo-product operators of the form
\begin{align}\label{i17}
B(g,h)(k) = \iint_{\R^3\times\R^3} \wt{g}(\ell) \wt{h}(m) \, b(k,\ell,m) \, \frac{1}{|\ell-k|} \pv \frac{1}{|\ell-k|-|m|} \, \mathrm{d}m\mathrm{d}\ell,
\end{align}
where $b$ belongs to a suitable symbol class,
then, up to some small losses and some less important factors which we do not detail here,
\begin{align}\label{i18}
{\| B(g,h) \|}_{L^2} \lesssim {\| \mathcal{W}^\ast g \|}_{L^p} {\| \mathcal{W}^\ast h \|}_{L^q}, 
  \qquad \frac{1}{p}+\frac{1}{q}=\frac{1}{2},
\end{align}
where $\mathcal{W} = \what{\mathcal{F}}^{-1} \wt{\mathcal{F}}$ is the wave operator, see \eqref{W0}.
We refer the reader to Section \ref{secBE} and Theorem \ref{theomu1} for precise statements.
Similar estimates are also needed for all the other bilinear operators associated
to the other terms in the expansion of \eqref{i16}. %up to a sufficiently large order.
An interesting aspect is how these estimates are obtained
by establishing results on bilinear pseudo-product operators supported on thin annuli
(see Lemmas \ref{lemBEan} also \ref{lemBEan2}).

%Could put more comments/details here?

\subsubsection*{Step 5: Set up for the nonlinear analysis}
With the precise information obtained on $\mu$,
we can proceed to study our nonlinear equation through the distorted Duhamel's formula.
From \eqref{introD} 
and \eqref{i11}-\eqref{i16} we see that  $\wt{f}(t) = \wt{u}_0 + \mathcal{D}(t)(f,f)$ where, at leading order, we have
\begin{align}
\label{i20}
\begin{split}
\mathcal{D}(t)(f,f) \approx \int_0^t \iint e^{is (-|k|^2 + |\ell|^2 + |m|^2 )} \widetilde{f}(s,\ell) \widetilde{f}(s,m)
  \, \frac{1}{|\ell-k|} \pv \frac{1}{|\ell-k|-|m|} \, \mathrm{d}\ell \mathrm{d}m\,\mathrm{d}s,
\end{split}
\end{align}
Our aim is to estimate globally-in-time the solution $u$ through its representation via $\wt{f}$ above.
To do this we devise a proper functional framework and place the evolution in a space 
that is strong enough to guarantee global decay,
but also sufficiently weak to allow us the possibility of closing the estimates.

%Thanks to the previous point, at leading order this is \eqref{D12}.
%We then have a more {\it precise understanding of the oscillations and ``good directions''} for integrating by parts and extracting decay. 
%See, for example, the discussion at the beginning of Subsection \ref{ssecN_1}.

%\subsubsection*{Functional framework}

%The functional framework that we choose for our estimates is given by the following:
%we assume that the initial data satisfies
%\begin{align}\label{i29}
%{\| u_0 \|}_{H^N} + {\| \nabla_k \widetilde{u}_0 \|}_{L^2} + {\| \nabla_k^2 \widetilde{u}_0 \|}_{L^2} \leq \e
%% {\| \widetilde{u}(t) \|}_{L^\infty}
%\end{align}
As mentioned after Theorem \ref{maintheo}, we will prove the following bounds:
\begin{align}\label{i30}
{\| u(t) \|}_{H^N} \lesssim \e,
\qquad {\| \partial_k \widetilde{f}(t) \|}_{L^2} \lesssim \e,
\qquad {\| \partial_k^2 \widetilde{f}(t) \|}_{L^2} \lesssim \e \jt^{1/2+\delta},
%\qquad {\| \widetilde{u}(t) \|}_{L^\infty} \lesssim \e,
\end{align}
for some small $\delta>0$, and all $t\in\R$;
see also the bootstrap Proposition \ref{proBoot}.
%
%Let us comment on the various norms that we have chosen above.
%
%The assumption on high Sobolev regularity for the data is propagated over time by standard energy estimates 
%relying on an integrable-in-time decay bound for ${\|u(t)\|}_{L^\infty}$.
%
%Integrable-in-time pointwise decay is implied by the assumptions \eqref{i30} through the following
%distorted variant of $L^1-L^\infty$ decay and interpolation:
%\begin{align}\label{i31} 
%{\| e^{it(-\Delta+V)}f(t) \|}_{L^\infty} 
%  \lesssim |t|^{-3/2} \, \big\|\partial_k \wt{f}(t) \big\|_{L^2}^{1/2} \, \big\|\partial_k^2 \wt{f}(t) \big\|_{L^2}^{1/2}.
%\end{align}
%Here we see the appearance 
%of the distorted analog of the standard weighted norms $\| \langle x\rangle^\beta f \|_{L^2}$.
%
%The uniform-in-time estimate of ${\| \wt{u} \|}_{L^\infty}$ is
%conveniently used in various instances to take advantage of the smallness of the size of the support
%(in generalized frequency space) on which we are able to restrict the integrals in \eqref{i20}.

%Here $f$ is the profile associated to $u$, see \eqref{prof}.
The remaining part of the argument is dedicated to estimating a priori the nonlinear expressions
in the right-hand side of \eqref{i20} according to \eqref{i30}.
The most difficult estimate is the one for $\partial_k^2\wt{f}$, for which we have to allow 
a certain growth in time. This is ultimately due to the lack of invariances (such as scaling and gauge-invariance)
of the equation.

\subsubsection*{Step 6: Nonlinear estimates and ``good directions''}
For simplicity, let us concentrate on the leading order term in \eqref{i20}, that is %which is, up to a change of variables,
%and dropping the $\pv$ for easier notation,
\begin{align}
\label{i40}
\begin{split}
\mathcal{B}(f,f)(t,k) & := \int_0^t \iint e^{is \Phi(k,\ell,m)} \widetilde{f}(s,\ell+k) \widetilde{f}(s,m)
  \,  \frac{1}{|\ell|} \pv \frac{1}{|\ell|-|m|} \, \mathrm{d}\ell \mathrm{d}m\,\mathrm{d}s,
  \\
  \Phi(k,\ell,m) & := |\ell|^2 + 2\ell\cdot k + |m|^2.
\end{split}
\end{align}
$\Phi$ is the so-called `phase' or `modulation'.
Recall that we have H\"older-type bilinear estimates for such expressions, see \eqref{i10}-\eqref{i11},
and that we may restrict to frequencies $|k|+|\ell|+|m| \lesssim 1$.
%and that we need to establish three bounds corresponding to the three norms in \eqref{i30}.
Moreover, we may restrict our attention to the case 
\begin{align}\label{i40'}
\big| |\ell| - |m| \big| \ll |\ell| \approx |m|
\end{align}
where the $\pv$ is indeed singular. %Moreover, for simplicity we assume $|\ell|\approx 1$.
%The harder estimate is the one for the highest weighted norm $\|\partial_k^2\widetilde{f}\|_{L^2}$.

The main difficulty is that an application of $\partial_k^j$, $j=1,2$, to \eqref{i40} leads to terms like
\begin{align}
\label{i41}
\begin{split}
\int_0^t \iint (-2is\ell)^j \, e^{is \Phi(k,\ell,m)} \widetilde{f}(s,\ell+k) \widetilde{f}(s,m)
  \,  \frac{1}{|\ell|} \pv \frac{1}{|\ell|-|m|} \, \mathrm{d}\ell \mathrm{d}m\,\mathrm{d}s
\end{split}
\end{align}
which contain powers of $s$ in the integrand. 
To obtain good bounds on expressions like \eqref{i41} it is necessary to exploit
oscillations through integration by parts arguments.
This leads to several difficulties:

\begin{itemize}
\item[(1)] the amplitudes, i.e. the profiles $\widetilde{f}$, have limited smoothness, 
according to the a priori assumptions (see \eqref{i30}),

\item[(2)] the oscillating factor $s \Phi(k,\ell,m)$ is stationary in many directions, and 

\item[(3)] the integral is taken with respect to a singular kernel. %$\pv \frac{1}{|\ell|-|m|}\mathrm{d}\ell \mathrm{d}m$.
\end{itemize}

\noindent
We note that issues similar to (1) and (2) 
have been handled in various problems for equations %(even quasilinear ones) 
without external potential. %(where the kernel is just $\delta(k-\ell-m)$).
The third issue, and its combination with (2),
is however a new difficulty and, to our knowledge, appears here for the first time.
%Below we explain how we address it.

%\subsubsection*{The ``good directions''}
Due to the singularity of the kernel, several directions of integrations are forbidden, 
such as, for example, $\partial_\ell$ and $\partial_m$.
However, there is a natural choice of direction along which we are allowed to integrate, that is,
the ``{\it good direction}''
\begin{align}\label{i50}
X := \partial_{|\ell|} + \partial_{|m|}
\end{align}
which is tangential to the singularity.
%so that integrating by parts in its direction does not increase the singularity.
%The question then is whether the phase is stationary in the $(X,\partial_s)$ direction. 
A calculations gives %$X \Phi := 2k \cdot \frac{\ell}{|\ell|} + 2 |\ell| + 2|m|$, and we notice that 
\begin{align*}
\Phi(k,\ell,m) = (|m|-|\ell|)^2 + |\ell| X\Phi(k,\ell,m) - 2|\ell|^2,
\end{align*}
and it follows that, close to the singularity of the kernel,
\begin{align}\label{i52}
|X\Phi(k,\ell,m)| \not\approx |\ell| \quad \Longrightarrow  \quad |\Phi| \gtrsim |\ell| \max(|\ell|, |X\Phi|). 
\end{align}
In other words, one of the following three things happens:
either (a) the kernel is not very singular, or
(b) the phase is non-stationary in the $X$ direction (more precisely, $|X\Phi| \gtrsim |\ell|$)
%(here we assume $|\ell|$ is not too small, relative to a negative power of $s$)
or (c) the integrand of \eqref{i41} is non-stationary in the $s$ direction, 
(more precisely, $|\Phi| \gtrsim |\ell|^2$).
These facts turn out sufficient to obtain the desired weighted $L^2$ bounds.

%\subsubsection*{$\mathcal{F}L^\infty$ Estimate}
%\begin{align}
%\label{i50}
%\begin{split}
%\mathcal{B}(t)(f,f)(t,k) & := \int_0^t \iint e^{is \Phi(k,\ell,m)} \widetilde{f}(s,\ell+k) \widetilde{f}(s,m)
%  \,  \frac{1}{|\ell|} \pv \frac{1}{|\ell|-|m|} \, \mathrm{d}\ell \mathrm{d}m\,\mathrm{d}s,
%\end{split}
%\end{align}

%Let us give a few details on how to the above discussion is implemented in practice to obtain weighted estimates.
We refer the reader to Section \ref{secdkL2} for details of these weighted estimates for the main term \eqref{i40}.
The lower order terms corresponding to the ``similar and better terms'' in \eqref{i11} are 
estimated in Section \ref{ssecmu23Est}.

\subsection*{Notation}
We fix $\varphi: \R \to [0,1]$ an even smooth function supported in $[-8/5,8/5]$ and equal to $1$ in $[-5/4,5/4]$. 
For simplicity of notation, we also let $\varphi: \R^n \to [0,1]$ denote the corresponding radial function on $\R^n$. 
Let
\begin{equation}
\label{LP0}
\begin{split}
& \varphi_K(x) := \varphi(|x|/2^K)-\varphi(|x|/2^{K-1}) \quad \text{ for any } \quad K \in \Z,
  \qquad \varphi_I:= \sum_{M\in I\cap \Z} \varphi_M \text{ for any } I\subseteq \R,
\\
& \varphi_{\leq B}:=\varphi_{(-\infty,B]},\quad\varphi_{\geq B}:=\varphi_{[B,\infty)},
  \quad \varphi_{<B}:=\varphi_{(-\infty,B)}, \quad \varphi_{>B}:=\varphi_{(B,\infty)}.
\end{split}
\end{equation}
For any $A<B\in\mathbb{Z}$ and $J \in [A,B] \cap \Z$ we let
\begin{equation}
\label{LP2}
\varphi^{[A,B]}_J:=
\begin{cases}
\varphi_{J} \quad & \text{ if } A<J<B,
\\
\varphi_{\leq A} \quad & \text{ if } J=A,
\\
\varphi_{\geq B} \quad &\text{ if } J=B,
\end{cases} \qquad
\varphi^{(A)}_J:=
\begin{cases}
\varphi_{J} \quad & \text{ if } A<J,
\\
\varphi_{\leq A} \quad & \text{ if } J=A.
\end{cases}
\end{equation}
For simplicity of notation, we will also use $\varphi_{\sim K}$ to denote a generic smooth cutoff function 
which is one on the support of $\varphi_K$ and is supported in $[c_1 2^K, c_22^K]$ 
for some absolute constants $c_1 < 1 < c_2$.
We will also sometimes denote with $\varphi^\prime$ 
a generic cutoff function with support properties similar 
to those of $\varphi$, such as derivatives of $\varphi$.

$P_K$, $K\in \Z$, denotes the Littlewood--Paley projection operator defined by the (flat) Fourier multiplier
$\xi \to \varphi_K(\xi)$.  
$P_{\leq B}$ denotes the operator defined by the Fourier 
multiplier $\xi\to \varphi_{\leq B}(\xi)$.
Similarly we define $P_{< B}$, $P_{\geq B}$ and so on.
%(respectively $\xi\to \varphi_{>B}(\xi)$).

For any $x\in\mathbb{Z}$ let $x_{+}=\max(x,0)$ and $x_-:=\min(x,0)$.
For any number $p\in\R$ we will denote with $p+$, resp. $p-$, a number which is larger, resp. smaller, than $p$ 
but can be chosen arbitrarily close to $p$.

We use standard notation for functional spaces such as $L^p$, $W^{s,p}$ and $H^s$.
$\mathcal{S}$ denotes the Schwartz class.

\medskip
\subsection{The main bootstrap argument}
We place our evolution in the space $X = A \cap W$ defined by the following norms:
\begin{align}
\label{space}
\begin{split}
& {\| u(t) \|}_A := {\| \langle k \rangle^N \wt{f}(t) \|}_{L^2} %+ {\| \wt{f}(t) \|}_{L^\infty} 
  + {\| \partial_k \wt{f}(t) \|}_{L^2},
\\
& {\| u(t) \|}_W := {\| \partial_k^2 \wt{f}(t) \|}_{L^2},
\end{split}
\end{align}
where we recall that $u$ and $f$ are related by $u=e^{itH}f$. %\eqref{introf}.
%with
%\begin{align}
%\label{paramZ} 
%0 < p_0, \delta \ll 1
%\end{align}

The first  norm ${\| \langle k \rangle^N \wt{f}(t) \|}_{L^2}$ is the equivalent of the Sobolev norm ${\|u(t)\|}_{H^N}$.
%We choose $N$ very large so that we can control high frequencies in various expansions and bilinear estimates.
%The second component, ${\| \partial_k \wt{f}(t) \|}_{L^2}$ is the distorted Fourier analogue of 
%a weighted norm of the form ${\| xf(t) \|}_{L^2}$.
Since we will be able to prove integrable-in-time decay for $u$ the norm ${\| u(t) \|}_A$ 
can be uniformly bounded in $t$.
On the contrary, the highest weighted-type norm ${\| u(t) \|}_W$ is allowed to grow at a certain
(quite fast) rate of $\jt^{1/2+}$. 
%It is probably not possible to obtain a much better control of this, 
%essentially due to the lack of invariance of the equation.
This is still sufficient to obtain a certain control (with growth) of the $L^1_x$-norm of $\mathcal{W} f$,
and infer pointwise-in-$x$ time decay via a standard dispersive estimate.
%Besides pointwise decay, necessary to prove global existence, by controlling Sobolev norms, as well as scattering,
%we establish uniform bounds on the distorted Fourier transform. 
%This bound is especially helpful when trying to control the size of 
%multilinear integrals with singular kernels close to resonances.

The proof of Theorem \ref{maintheo} relies on the following bootstrap estimates.

\begin{proposition}[Main Bootstrap]\label{proBoot}
Let $u$ be a solution of \eqref{NLSV} on a time interval $[0,T]$, with initial data 
satisfying
\begin{align}\label{proBootdata}
{\| u_0 \|}_{H^N} + {\| \partial_k \widetilde{u_0} \|}_{L^2} + {\| \partial_k^2 \widetilde{u_0} \|}_{L^2} 
% + {\| \widetilde{u_0} \|}_{L^\infty} 
\leq \e_0.
\end{align}
With the definitions \eqref{space} assume the a priori bounds
\begin{align}
\label{apriori}
\sup_{t\in[0,T]} \Big( {\|u(t)\|}_{A} + \jt^{-1/2-\delta}{\|u(t)\|}_{W} \Big) \leq \e \leq \e_0^{2/3},
\end{align}
for some properly chosen\footnote{We can choose $\delta=240/(N-5)$ for example.} %To check
$\delta\in(0,1/4)$.
Then, we have the improved bounds
\begin{align}
\label{proBoot2}
\sup_{t\in[0,T]} \Big( {\|u(t)\|}_{A} + \jt^{-1/2-\delta}{\|u(t)\|}_{W} \Big) \leq \frac{\e}{2}.
\end{align}%{\| u \|}_A  + {\| u \|}_W \leq \e_0 + C\e^2.
\end{proposition}

Through a standard bootstrap argument this proposition gives us global solutions for \eqref{NLSV}.
The Sobolev bound in \eqref{maintheoconc} follows from 
${\| \jk^N \wt{f} \|}_{L^2} = {\| \jk^N \wt{u} \|}_{L^2} \approx {\| u \|}_{H^N}$.
The bounds \eqref{proBoot2} imply the poitnwise decay estimate stated in \eqref{maintheoconc}
with $\alpha=1/4-\delta/2$, via the linear estimate \eqref{linearinfty}
and the interpolation $\| \wt{f} \|_{L^\infty}^2 \lesssim \| \partial_k \wt{f} \|_{L^2}
\| \partial_k^2 \wt{f} \|_{L^2}$.

The proof of Proposition \ref{proBoot} is performed in \ref{secdkL2} (for the leading order terms)
and \ref{ssecmu23Est} (for all the lower order terms).

%\begin{lemma}\label{Lembootcons}
%Put here useful consequences of the a priori bounds \eqref{space}-\eqref{apriori}. 
%\end{lemma}

\medskip
\section{Linear Spectral Theory}\label{SecLin}

%Could put the more standard things in the Appendix, or just make references.
\subsection{Generalized eigenfunctions and distorted Fourier transform}

Given a potential $V:\R^3\rightarrow \R$, consider the Schr\"odinger operator $H=-\Delta+V$ associated to it.
If $V$ decays fast enough (e.g. it is `short range' in the sense of Agmon \cite{Agmon}) the spectrum of
$H$ consists of the absolutely continuous spectrum $[0,\infty)$ and a countable number of 
negative eigenvalues $0>\lambda_1 >\lambda_2 > \dots$ with finite multiplicity.
One has the orthogonal decomposition $L^2(\R^3)=L^2_{\mathrm{ac}}(\R^3) \oplus L^2_{\mathrm{p}}(\R^3)$
where $L^2_{\mathrm{ac}}(\R^3)$ is the absolutely continuous subspace for $H$
and $L^2_{\mathrm{p}}(\R^3)$ is the span of the eigenfunctions corresponding to the negative eigenvalues.

%We refer to Appendix \ref{secF} for a more detailed presentation of the distorted Fourier transform in $3$ dimensions,
%and admit for the moment the existence 

For any $k \in \R^3\smallsetminus\{0\}$ we have that $|k|^2$ is in the continuous spectrum of $H$ and 
the associated (generalized) eigenfunctions $\psi(x,k)$ are defined as solutions of
%the Schr\"odinger operator $H:= -\Delta + V$, 
\begin{align}\label{psixk}
(-\Delta + V ) \psi(x,k) = |k|^2 \psi(x,k), \qquad \forall \, k \in \R^3\smallsetminus\{0\},
\end{align}
with the asymptotic condition $\psi(x,k) - e^{ix\cdot k} = O(|x|^{-1})$ for $|x| \rightarrow \infty$,
and the Sommerfeld radiation condition
\begin{align*}
r (\partial_r - i|k|) v(x,k) \longrightarrow 0,
\end{align*}
as $r=|x| \rightarrow \infty$.
The functions $\psi(x,k)$ are `distorted' version of the plane waves $e^{ik\cdot x}$. %generalized eigenfunctions 
They satisfy the so-called Lippmann-Schwinger equation
\begin{align*}%\label{LSV}
\psi(x,k) = e^{ix\cdot k} - R_V(|k|^2)(V e^{ix\cdot k}),
\end{align*}
where $R_V(\lambda) = (H-\lambda)^{-1}$ is\footnote{This can be formally understood as 
$R_V(\lambda) := \lim_{\epsilon \rightarrow 0+}(H-\lambda+i\epsilon)^{-1}$,
%for $z\in\C$ and not in the spectrum of $H$,
where the limit is taken with respect to a proper operator norm topology,
say from $\langle x \rangle^{-s}L^2$ to $\langle x \rangle^s H^2$ for some $s>1/2$.}
the (perturbed) resolvent. %; %see for example \cite{Agmon}.
For our analysis it will actually be more convenient to write $\psi$ as a solution of the integral equation
\begin{align*}%\label{LS0}
\psi(x,k) = e^{ix\cdot k} - R_0(|k|^2) \big(V \psi(\cdot,k)\big),
\end{align*}
where $R_0(\lambda) = (-\Delta - \lambda)^{-1}$ is the flat/unperturbed resolvent;
more explicitly,
\begin{align}
\label{psi0}
\psi(x,k) = e^{ix\cdot k} - \frac{1}{4\pi} \int_{\R^3} \frac{e^{i|k||x-y|}}{|x-y|} V(y) \psi(y,k) \, \mathrm{d}y.
\end{align}

The following Theorem guarantees the existence of the Distorted Fourier Transform
and its inverse, under suitable assumptions on the potential.

\begin{theorem}[Distorted Fourier Transform]\label{theodFT}
Consider the Schr\"odinger operator $H=-\Delta + V$ with a fast decaying potential $V = O(|x|^{-1-})$
in dimension $d=3$, and assume \eqref{assV1}.
For $g\in\mathcal{S}$ define the distorted Fourier Transform (dFT) by
\begin{align}\label{dFTdef}
(\Ftil g)(k) := \wt{g}(k):= \frac{1}{(2\pi)^{3/2}} \lim_{R\rightarrow \infty} 
  \int_{|x|\leq R} \overline{\psi(x,k)} \, g(x) \, \mathrm{d}x.
\end{align}
Then, $\Ftil$ extends to an isometric isomorphism of $L^2(\R^3)$ with inverse
\begin{align}\label{dFTinv}
(\Ftil^{-1} g) (x):= \frac{1}{(2\pi)^{3/2}} \lim_{R\rightarrow \infty} 
  \int_{|x|\leq R} \psi(x,k) \, g(k) \, \mathrm{d}k.
\end{align}
Moreover, $\Ftil$ diagonalizes the Schr\"odinger operator: $\Ftil H \Ftil^{-1} = |k|^2$.
\end{theorem}

This theorem is due to several authors including
Ikebe \cite{Ikebe}, Alsholm-Schmidt \cite{AS}, and Agmon \cite{Agmon}.
%, Yajima \cite{Yajima}
We refer the interested reader to Section 2 of \cite{GHW} for a more extensive 
presentation of this topic, and a discussion about the validity of Theorem \ref{theodFT}
under weaker assumptions on the potential, such as those made in the references cited above.
In the case that \eqref{assV1} does not hold and $H_V$ has discrete spectrum in $(-\infty,0)$, %below the continuous one, 
then the generalized eigenfucntions will diagonalize $H_V$ restricted to the absolutely continuous subspace 
$L^2_{\mathrm{ac}}$, and the dFT is well-defined and invertible there.

An important object in the study of the flow associated to $H$ is the wave operator defined by
\begin{align}\label{W0}
\mathcal{W} = \mathrm{s-lim}_{t\to \infty} e^{itH} e^{it\Delta},
\end{align}
where the limit is in the strong operator topology.
The wave operator is unitary on $L^2$ and is connected to the dFT by the formula
\begin{align}\label{W01}
\mathcal{W} = \widetilde{\mathcal{F}}^{-1}  \widehat{\mathcal{F}},
\end{align}
where $\widehat{\mathcal{F}}$ is the regular/flat Fourier transform.
In particular, $\mathcal{W}^{-1} = \mathcal{W}^* = \widehat{\mathcal{F}}^{-1} \widetilde{\mathcal{F}}$,
and one has the following intertwining formulas for $H$ and $H_0=-\Delta$:
\begin{align}\label{Wint}
a(H) = \mathcal{W} a(H_0) \mathcal{W}^*.
\end{align}
Under relatively mild decay and regularity assumptions on $V$ (much weaker than our assumption \eqref{assV2})
and provided that $V$ is generic, that is, there are no solutions
of $H\psi = 0$ in $\langle x \rangle^{1/2+} L^2$ (no resonances),
we have that $\mathcal{W}$ and $\mathcal{W}^\ast$ are bounded on $W^{k,p}$.
See Yajima \cite{Yajima} and the discussion in \cite{GHW} and reference therein.

It is worth pointing out that while we use standard results on the $L^p$ boundedness of wave operators,
we do not rely on any specific structural property about them, such as those 
found in the literature\footnote{For 
convenience we use a result of \cite{GHW}, Proposition \ref{proGHW} below,
which relies on the structure of $\mathcal{W}$;
but it is should be possible  to use our approach to obtain this independently.}, 
see for example \cite{Yajima,GHW,BeSch}.
On the other hand, our approach does rely on analyzing the structure of `wave operator'- like quantities
at a nonlinear level.

\medskip
\subsection{Bounds on $\psi(x,k)$}
We begin our analysis by establishing some basic estimates on $\psi$ and its derivatives
in $x$ and $k$.

\begin{lemma}[Basic properties of $\psi$]\label{lemmapsi}
Let $\psi$ be defined as in \eqref{psi0} with $V$ satisfying \eqref{assV2}.
Then:
\begin{align}
\label{psiLinfty}
%& |\psi(x,k)| \lesssim 1
%\\
%\label{psi_xLinfty}
%& {\| \partial_x^\alpha \psi(x,k) \|}_{L^\infty} \lesssim (1+|k|)^{|\alpha|}, \qquad 1 \leq |\alpha| \leq N_0
%\\
%\label{psi_kLinfty}
%& {\| \partial_k^\alpha \psi(x,k) \|}_{L^\infty} \lesssim (1+|x|) \big(|x| + |k|^{-1}\big)^{|\alpha|-1}, \qquad 1\leq|\alpha|\leq N_1. %+1
& | \partial_x^\alpha \partial_k^\beta \psi(x,k) | 
  \lesssim ( \jk^{|\alpha|} + (|k|/\jk)^{1-|\beta|} ) \jx^{|\beta|}
  \qquad 0 \leq |\alpha|,|\beta| \leq N_1
\end{align}
\end{lemma}

\begin{proof}

%Let us write $v(x,k) = e^{-ix\cdot k} \psi(x,k)$ so that
%\begin{align}
%\begin{split}
%v(x,k) & = 1 - \frac{1}{4\pi} \int_{\R^3} \frac{e^{i|k||x-y|}}{|x-y|} e^{-i(x-y)\cdot k} V(y) v(y,k) \, \mathrm{d}y
%  \\ & = 1 + \int_{\R^3} M_k(x-y)V(y) v(y,k) \, \mathrm{d}y, \qquad M_k (z) := \frac{e^{i|k||z|}}{|z|} e^{-iz\cdot k}.
%\end{split}
%\end{align}
For $f \in L^\infty_{x,k}$ let us define the operator
\begin{align}
\label{lemmapsiT}
\big(T_k f(\cdot,k) \big)(x) := - \frac{1}{4\pi} \int_{\R^3} \frac{e^{i|k||x-y|}}{|x-y|} V(y) f(y,k) \, \mathrm{d}y.
\end{align}
Let us write 
\begin{align}\label{lemmapsiv}
\begin{split}
& v(x,k) := \psi(x,k) - e^{ix\cdot k}, 
\end{split}
\end{align}
so that \eqref{psi0} implies
\begin{align}\label{lemmapsiv2}
v(x,k) & = (T_k v(\cdot,k))(x) + (T_k e^{ix\cdot k})(x).
\end{align}
Note that $T_k$ is a compact operator from $L^\infty$ to $C_0$, where $C_0$
is the space of bounded continuous functions decaying to $0$ at infinity.
In particular, for any $g \in C_0$ there exists a unique $C_0$ solution to the integral equation $f = g + T_k f$
if and only if $f = T_k f$ admits only the trivial solution;
this is indeed the case since $T_k f = -R_0 (Vf)$, where $R_0$ is the flat resolvent
$(-\Delta - |k|^2)^{-1}$ and we are assuming absence of eigenvalues and resonances for $-\Delta + V$.
This and \eqref{lemmapsiv2} imply \eqref{psiLinfty} for $\alpha=\beta=0$.

Let us define
\begin{align}\label{lemmapsivab}
& v_{\alpha\beta} :=  \jk^{-|\alpha|}  \big( |k|/\jk \big)^{|\beta|-1} \jx^{-|\beta|} \partial_x^\alpha\partial_k^\beta v,
%\qquad v_{\beta}:=v_{0\beta}, 
\qquad |\beta|\geq 1.
\end{align}
The conclusion \eqref{psiLinfty} will follow from uniform bounds on $v_{\alpha\beta}$.
To obtain these bounds we will show, by induction, 
that $v_{\alpha\beta}$ satisfies an integral equation similar to the one satisfied $v$ \eqref{lemmapsiv2}, 
up to lower order terms.

To formalize this, let us define the class $\mathcal{T}^N$ of $k$ dependent operators as follows:
\begin{align}\label{lemmapsiclassT}
\begin{split}
T_k \in \mathcal{T}^N \quad \stackrel{def}{\Longleftrightarrow} \quad 
	& T_k f := \int_{\R^3} \frac{e^{i|k||x-y|}}{|x-y|} a(x,y) f(y) \, \mathrm{d}y, 
\\
& \mbox{with} \qquad \int_{\R^3} \jy^{N} \big( \langle \partial_x \rangle^N 
	+ \langle \partial_y \rangle^N \big) |a(x,y)| \, \mathrm{d}y \lesssim 1.
\end{split}
\end{align}
We think of operators in $\mathcal{T}^N$ as acting on function $f=f(x,k)$.
The operator $T_k$ in \eqref{lemmapsiT} belongs to $\mathcal{T}^{N_1}$ by the assumption \eqref{assV2}.
The following properties hold:

\begin{itemize}
 
\item[(i)] For $N>3$, operators in $\mathcal{T}^N$ are compact from $L^\infty$ to $C_0$.

\item[(ii)] We have
\begin{align}\label{lemmapsiT1}
\begin{split}
& \mathcal{T}^{N} \subset \mathcal{T}^{N'}, \quad N'\leq N, 
\\
& \jx^{-\ell} \mathcal{T}^{N} \subset \mathcal{T}^{N}, \quad \forall \, \ell \geq 0,
\\
& \mathcal{T}^{N}( \jy \cdot), \, \mathcal{T}^{N}(y \cdot) \subset \mathcal{T}^{N-1},
\\
& \jx^{-1} \partial_k \mathcal{T}^{N} \subset \frac{k}{|k|} \mathcal{T}^{N-1}.
\end{split}
\end{align}
Here, for $T_k \in \mathcal{T}_k$ we denote $\partial_kT_k$ 
the operator obtained by differentiating in $k$ the term $e^{i|k||x-y|}$ in \eqref{lemmapsiclassT}.

\item[(iii)] We have
\begin{align}\label{lemmapsiT2}
\begin{split}
[\partial_x, T_k] := \partial_x T_k - T_k \partial_x \subset \mathcal{T}^{N-1}.
\end{split}
\end{align}
This can be seen by applying directly $\partial_x$ to \eqref{lemmapsiclassT}, converting it into $-\partial_y$
and integrating by parts.

%\item[(iv)]

\end{itemize}

\medskip
{\it Claim}.
Let $v_{\alpha\beta}$ be defined as in \eqref{lemmapsivab} for $|\alpha| + |\beta| = N \leq N_1$.
The following identity holds true:
\begin{align}\label{lemmapsiclaim1}
v_{\alpha\beta} + T_{0,\beta}(v_{\alpha\beta}) = G_{\alpha\beta}
\end{align}
where 
\begin{align}\label{lemmapsiclaim1'}
T_{0,\beta} (f) :=  \frac{1}{\jx^{|\beta|}} \frac{1}{4\pi} \int_{\R^3} \frac{e^{i|k||x-y|}}{|x-y|} V(y) \jy^{|\beta|} f(y) \, \mathrm{d}y,
\end{align}
and $G_{\alpha\beta}$ is a linear combination of the form
\begin{align}\label{lemmapsiclaim2}
\begin{split}
& G_{\alpha\beta} = \sum a_\ell(k) T_\ell ( v_{\gamma\delta}) + \sum a_\ell'(k) T_\ell' (e^{ix\cdot k}), 
	\qquad T_\ell, T_\ell' \in \mathcal{T}^{N_1-N-1},
\end{split}
\end{align}
where the sums run over finitely many indexes $\ell$ and
\begin{align}\label{lemmapsiclaim2'}
\begin{split}
|\gamma|\leq |\alpha|, \quad |\delta| \leq |\beta|, \quad |\gamma|+|\delta| \leq N- 1, 
\end{split}
\end{align}
and with coefficients $a_\ell(k), \, a_\ell'(k)$ that are either 
(a) smooth and bounded with all their derivatives, or 
(b) $0$-homogeneous for $|k| \ll 1$ and otherwise smooth and bounded with all their derivatives.
%\begin{align}
%a_\ell(k), \, a_\ell'(k) \, \sim \, \mbox{$0$-homogeneous for} |k| \ll 1, \quad \mbox{smooth and bounded otherwise}.
%\end{align}

\medskip
{\it Proof of the Claim}.
%To prove \eqref{lemmapsiclaim1}-\eqref{lemmapsiclaim2}
We proceed by induction on $N$. 
The case $N=0$ is given by \eqref{lemmapsiv2}.
Let us assume that the claimed identity is true for $v_{\alpha\beta}$ with $|\alpha|+|\beta| = N$.
In order to prove it for $N+1$ we derive the corresponding identity for $v_{\alpha\beta'}$ with $\beta' = \beta + \beta_0$, 
$|\beta_0|=1$.
This will suffice since the case of $v_{\alpha'\beta}$ with $\alpha' = \alpha + \alpha_0$, $|\alpha_0|=1$,
is simpler and follows more directly by applying \eqref{lemmapsiT2}.

Assuming without loss of generality $\beta'=\beta+(1,0,0)$, 
from \eqref{lemmapsivab} and \eqref{lemmapsiclaim1}, denoting $T=T_{0,\beta}$, we have
\begin{align}\label{lemmapsiclaim3}
v_{\alpha\beta'} = \frac{1}{\jx} \frac{|k|}{\jk} \partial_{k_1} v_{\alpha\beta}  = 
  -\frac{1}{\jx}  \frac{|k|}{\jk} \partial_{k_1} T(v_{\alpha\beta}) + \frac{1}{\jx} \frac{|k|}{\jk} \partial_{k_1} G.
\end{align}
First we calculate
\begin{align*}
\frac{1}{\jx}  \frac{|k|}{\jk} \partial_{k_1} T(v_{\alpha\beta}) 
	& = \frac{1}{\jx} T\big( \frac{|k|}{\jk} \partial_{k_1} v_{\alpha\beta} \big)
 	+ \frac{1}{\jx}  \frac{|k|}{\jk} (\partial_{k_1} T)(v_{\alpha\beta})
 	\\
 	& = \frac{1}{\jx} T\big( \jy \, v_{\alpha\beta'}\big)
 	+ \frac{1}{\jx}  \frac{|k|}{\jk} (\partial_{k_1} T)(v_{\alpha\beta}).
% 	\\ & = \frac{1}{\jx} T\big( \jy \, v_{\alpha\beta'}\big) + \mathcal{T}^{N_1-(N+1)}
\end{align*}
%having used the definition \eqref{lemmapsivab}.
Since $T\in \mathcal{T}^{N_1-N}$, in view of the properties \eqref{lemmapsiT1}, this is of the form 
\[ T_1(v_{\alpha\beta'}) + \frac{k_1}{\jk} T_2(v_{\alpha\beta}), 
	\qquad T_1,T_2\in \mathcal{T}^{N_1-(N+1)}, \]
which is consistent with \eqref{lemmapsiclaim1}-\eqref{lemmapsiclaim2'}. 

Next, we look at the second term in \eqref{lemmapsiclaim3}, and consider the first contribution to $G$ 
from \eqref{lemmapsiclaim2}. 
For $|\gamma|+|\delta| \leq N-1$ and $\delta' = \delta +(1,0,0)$, proceeding similarly as above we have
\begin{align*}
& \frac{1}{\jx}  \frac{|k|}{\jk} \partial_{k_1} \big[ a_\ell(k) T_\ell (v_{\gamma\delta}) \big]
\\ & = a_\ell(k) \frac{1}{\jx} T_\ell \big( \jy v_{\gamma\delta'}\big)
+ \big( \frac{|k|}{\jk} \partial_{k_1} a_\ell(k) \big) \, \frac{1}{\jx}  T_\ell (v_{\gamma\delta}) 
+ \frac{|k|}{\jk} a_\ell(k) \frac{1}{\jx}  (\partial_{k_1} T_\ell) (v_{\gamma\delta}) 
\end{align*}
which, using the properties \eqref{lemmapsiT1}, is of the form
\[ a(k) T_1(v_{\gamma\delta'}) + b(k) T_2(v_{\gamma\delta}) + c(k) T_3(v_{\gamma\delta})
	\qquad T_1,T_2,T_3\in \mathcal{T}^{N_1-(N+1)}, \]
for some coefficients $a,b,c$ with the same properties of $a_\ell$.
This is consistent with \eqref{lemmapsiclaim2}-\eqref{lemmapsiclaim2'} with $N$ and $\beta$ replaced by $N+1$ and $\beta'$
as desired.
We can deal similarly with the second sum in \eqref{lemmapsiclaim2}, thus obtaining our induction step.

\medskip
{\it Conclusion}.
From \eqref{lemmapsiclaim1}-\eqref{lemmapsiclaim3} we can deduce inductively that $v_{\alpha\beta}$
is the unique bounded solution of the equation \eqref{lemmapsiclaim1}.
Indeed, the existence of $v_{\alpha\beta} \in L^\infty_{x,k}$ is given by the fact that $T_{0,\beta}$ 
is compact, and $G_{\alpha\beta} \in C_0$. 
Moreover, $v_{\alpha\beta}$ is the unique solution of \eqref{lemmapsiclaim1}
if and only if the equation $f + T_{0,\beta}f = 0$ admits only the trivial solution $f\equiv0$. 
To verify that this is the case, we notice that if $f + T_{0,\beta}f = 0$ for a bounded $f$, then
$g = \jx^{|\beta|} f$ is a polynomially bounded solution of $g = T_k g$;
this means that $g$ is in the spectrum of $-\Delta + V$ and thus has to be trivial \cite{SimonSpec}.

From \eqref{lemmapsiv} we see that
\begin{align*}
| \partial_x^\alpha\partial_k^\beta \psi(x,k)| 
& \lesssim  | \partial_x^\alpha\partial_k^\beta e^{ix\cdot k}| + \jk^{|\alpha|}  \big( |k|/\jk \big)^{1-|\beta|} \jx^{|\beta|} 
\\
& \lesssim  \big(\jk^{|\alpha|}  + \big( |k|/\jk \big)^{1-|\beta|} \big) \jx^{|\beta|} 
%\\& \lesssim  \big(|k|^{|\alpha|} + |k|^{-|\beta|} \big) \jx^{|\beta|} 
\end{align*}
which proves the estimates \eqref{psiLinfty}.
\end{proof}

\medskip
\subsection{Expansion of $\psi$}\label{ssecpsiexp}
From the formula \eqref{psi0} we write 
\begin{align}
\label{psipsi1}
\begin{split}
\psi(x,k) & = e^{ix\cdot k} - e^{i|k||x|} \frac{1}{4\pi|x|} \psi_1(x,k),
\\
\psi_1(x,k) & := \int_{\R^3} e^{i|k| [ |x-y| - |x| ]} \frac{|x|}{|x-y|} V(y) \psi(y,k) \, \mathrm{d}y.
\end{split}
\end{align}
$\psi_1$ is the key linear object that we want to study, and for which we want to obtain precise asymptotic expansions.

In the following Lemma we summarize some basic properties of $\psi_1$.

\begin{lemma}[Basic properties of $\psi_1$]\label{Lempsi1}
Under the assumption \eqref{assV2}, %for all $|x| \geq 1$,
the function $\psi_1$ defined by \eqref{psipsi1} satisfies, for $|x| \gtrsim 1$,
\begin{align}
\label{psi10}
\begin{split}
& \big| \jk^{-|\alpha|} |x|^{|\alpha|} \nabla^{\alpha}_x \psi_1(x,k) \big| \leq c_\alpha, \qquad |\alpha | \leq N_1,
\\
& \big| \jk^{-|\alpha|} |x|^{|\alpha|} \nabla^{\alpha}_x \nabla_k^\beta \psi_1(x,k) \big| 
	\leq c_{\alpha,\beta} \max \big(1, |k|^{1-|\beta|} \big),  
	\quad \quad |\alpha| + |\beta| \leq N_1, \quad \beta\neq 0.
\end{split}
\end{align}
In particular , if we define the angular derivative vectorfields
\begin{align}
\label{defOmega}
\Omega_x := x \wedge \nabla_x = (x_2\partial_{x_3}-x_3\partial_{x_2}, x_3\partial_{x_1}-x_1\partial_{x_3},
x_1\partial_{x_2}-x_2\partial_{x_1}) =: (\Omega_1, \Omega_2,\Omega_3)
\end{align}
%and will sometimes denote it just by $\Omega$ when there is no confusion.
one has
\begin{align}
\label{psi11}
\begin{split}
& \big| \jk^{-|\alpha|} \Omega^{\alpha}_x \psi_1(x,k) \big| \leq c_\alpha, \qquad |\alpha | \leq N_1,
\\
& \big| \jk^{-|\alpha|} \Omega^{\alpha}_x \nabla_k^\beta \psi_1(x,k) \big| \leq c_\alpha \max(1,|k|)^{1-|\beta|}, 
  \quad \quad |\alpha| + |\beta| \leq N_1, \quad \beta\neq 0.
\end{split}
\end{align}

\end{lemma}

\medskip
\begin{proof}
Let us decompose 
\begin{align}
\label{Lempsi1pr1}
\begin{split}
\psi_1(x,k) & = \psi_1^-(x,k) + \psi_1^+(x,k),
\\
\psi_1^-(x,k) & := \int_{\R^3} e^{i|k| [ |x-y| - |x| ]} \frac{|x|}{|x-y|} V(y) \psi(y,k) \, 
	\varphi_{\leq-10}(|y|/|x|) \, \mathrm{d}y,
\\
\psi_1^+(x,k) & := \int_{\R^3} e^{i|k| [ |x-y| - |x| ]} \frac{|x|}{|x-y|} V(y) \psi(y,k) \,
	\varphi_{>-10}(|y|/|x|) \, \mathrm{d}y;
\end{split}
\end{align}
recall the notation for cutoffs from \eqref{LP0}.
%It suffices to look at $|x|\gtrsim 1$.

\medskip
{\it Estimate of $\psi_1^-$}.
From Fa\'{a}-di Bruno's formula we see that $\partial_{x_1}^{\alpha_1} e^{i|k| [ |x-y| - |x| ]}$ 
is bounded by a linear combination of terms of the form
\begin{align}
\label{Lempsi1pr2.0}
\sup_{\substack{(p_1,\dots,p_a) \, : \, p_i\geq 0 \\ \sum_{a\geq 1} a p_a = \alpha_1}} 
	\quad \prod_{b=1}^{\alpha_1} \big( |k| \partial_{x_1}^b ( |x-y| - |x| ) \big)^{p_b}.
\end{align}
%for example using from Fa\'{a}-di Bruno's formula.
On the support of $\psi_1^-$ we have $|\partial_{x_1}^b ( |x-y| - |x| )| \lesssim |y| |x|^{-b}$
and therefore we see that
\begin{align}
\label{Lempsi1pr2}
\big|  \partial_{x_1}^{\alpha_1} e^{i|k| [ |x-y| - |x| ]} \big| \lesssim (|k|+|k|^{\alpha_1})
	\cdot (|y| + |y|^{\alpha_1}) |x|^{-\alpha_1}.
\end{align}
%Don't need $|y| \ll |x|$
Since we also have
\begin{align*}
\Big|\partial_{x_1}^{\alpha_1} \frac{|x|}{|x-y|}\Big| 
  + |\partial_{x_1}^{\alpha_1} \varphi_{\leq-10}(|y|/|x|) | \lesssim |x|^{-\alpha_1}
\end{align*}
the first inequality in \eqref{psi10} follows for the term $\psi_1^-$. %\Dets

To deal with derivatives in $k$ we apply again Fa\'{a}-di Bruno and estimate
\begin{align}\label{Lemmapsipr3}
\begin{split}
\big|  \partial_{k_1}^{\beta_1} e^{i|k| [ |x-y| - |x| ]} \big|
\lesssim \sup_{\substack{(p_1,\dots,p_b) \, : \, p_i\geq 0 \\ \sum_{b\geq 1} b p_b = \beta_1}} 
	\quad \prod_{b=1}^{\beta_1} \big| (\partial_{k_1}^b |k|) ( |x-y| - |x| )  \big|^{p_b}
	\lesssim (1+|k|^{1-\beta_1}) \cdot |y|/|x|
\end{split}
\end{align}
Arguing as before for the $x$-derivatives, and using the estimates 
\eqref{psiLinfty} to bound $\partial_{k_1}^{\beta_1} \psi$,
with the assumptions on $V$, %\eqref{assV2}, 
we obtain the second inequality in \eqref{psi10}. %\Dets
%The estimates for the angular derivatives follow immediately.

\medskip
{\it Estimate of $\psi_1^+$}.
First notice that since the potential satisfies \eqref{assV2}, $\psi_1^+$ decays very fast in $x$:
\begin{align*}
| \psi_1^+(x,k) | \lesssim {|x|}^{-N_1-5} \int_{\R^3} \frac{|x|}{|x-y|} {|y|}^{N_1+5} |V(y)| \varphi_{>-10}(|y|/|x|)
  \,\mathrm{dy}
  \\ \lesssim |x|^{-N_1-5} \, {\big\| \jx^{N_1+6} V \big\|}_{L^\infty \cap L^1}.
\end{align*}
To prove estimates on several $x$ derivatives however we need 
to take care of the singularity arising when differentiating the integrand.
Let us write
\begin{align}
\label{Lempsi1pr5}
\begin{split}
\psi_1^+(x,k) & = a(x,k) + b(x,k),
\\
a(x,k) & = \int_{\R^3} e^{i|k| [ |x-y| - |x| ]} \frac{|x|}{|x-y|} V(y) \psi(y,k) \, 
	\varphi_{>-10}(|y|/|x|) \varphi_{> 0 }(|x-y|)\, \mathrm{d}y,
\\
b(x,k) & = \int_{\R^3} e^{i|k| [ |x-y| - |x| ]} \frac{|x|}{|x-y|} V(y) \psi(y,k) \, 
	\varphi_{>-10}(|y|/|x|) \varphi_{\leq 0 }(|x-y|)\, \mathrm{d}y.
\end{split}
\end{align}

Arguing as in the proof of \eqref{Lempsi1pr2} above, on the support of $a(x,k)$, 
where $|y| \gtrsim |x|\gtrsim 1$ and $|y-x| \gtrsim 1$, we have
\begin{align}
\label{Lempsi1pr6}
\big|  \partial_{x_1}^{\alpha} e^{i|k| [ |x-y| - |x| ]} \big| \lesssim |k|+|k|^{\alpha}, %\cdot (1+|x|^{-\alpha}).
\end{align}
as well as 
\begin{align}
\label{Lempsi1pr7}
\Big|  \partial_{x_1}^{\alpha} \frac{|x|}{|x-y|} \Big| \lesssim 1+|y|. %(1+|x|^{-\alpha}).
\end{align}
Then we can estimate $|\partial_{x_1}^{\alpha_1} a(x,k)|$ by a linear combination of terms of the form
\begin{align}
%\partial_{x_1}^{\alpha_1} a(x,k) \lesssim 
I_{\alpha_2\alpha_3} = \Big| \int_{\R^3} 
  \Big[ \partial_{x_1}^{\alpha_2} e^{i|k| [ |x-y| - |x| ]} \Big] 
  \Big[ \partial_{x_1}^{\alpha_3} \frac{|x|}{|x-y|} \Big] V(y) \psi(y,k) \, 
  \varphi_{>-10}(|y|/|x|) \varphi_{> 0}(|x-y|)\, \mathrm{d}y \Big|
\end{align}
with $\alpha_2+\alpha_3 = \alpha_1$, 
plus easier terms arising when derivatives hit the cutoffs, which we disregard.
We then see that
\begin{align*}
%\partial_{x_1}^{\alpha_1} a(x,k) \lesssim 
I_{\alpha_2\alpha_3} & \lesssim (|k| + |k|^{\alpha_2}) %(1+|x|^{-\alpha_2}) \cdot (1+|x|^{-\alpha_3})
	\int_{\R^3} (1+|y|)|V(y)|\, \varphi_{>-10}(|y|/|x|)\,\mathrm{d}y
	\lesssim (1+|k|)^{\alpha_1} %(1+|x|^{-\alpha_1}) 
	(1+|x|)^{-N_1}
\end{align*}
which is consistent with the right-hand side of \eqref{psi10}.
To deal with the derivatives in $k$ we use
\begin{align*}%\label{Lempsi1pr6'}
\big|  \partial_{k_1}^{\beta_1} e^{i|k| [ |x-y| - |x| ]} \big| \lesssim (1+|k|^{1-\beta_1}) (1+|y|)^\beta,
\end{align*}
see \eqref{Lemmapsipr3},  and obtain that $\partial_{k_1}^{\beta_1} a(x,k)$ is
bounded by a linear combination of terms of the form
\begin{align*}
J_{\beta_2\beta_3} & = \int_{\R^3} \big| \partial_{k_1}^{\beta_2} e^{i|k| [ |x-y| - |x| ]} \big| \frac{|x|}{|x-y|}
  \big| V(y) \big|\, \big| \partial_{k_1}^{\beta_3} \psi(y,k) \big| \, \varphi_{>-10}(|y|/|x|) \,\mathrm{d}y  \Big| 
\\ 
& \lesssim \int_{\R^3} (1+|k|^{1-\beta_2}) (1+|y|)^{\beta_2}
  (1+|y|) \, |V(y)| \, (1+|y|)^{\beta_3} (1+|k|^{1-\beta_3}) \, \varphi_{>-10}(|y|/|x|)\,\mathrm{d}y
\end{align*}
for $\beta_2+\beta_3=\beta_1$, having used \eqref{psiLinfty}.
In view of \eqref{assV2} and $\beta_1 \leq N_1$, we have 
\begin{align*}
|J_{\beta_2\beta_3}| & \lesssim (1+|x|)^{-N_1+1}(1 + |k|^{1-\beta_1})
\end{align*}
which is sufficient for the second inequality in \eqref{psi10} when $\alpha=0$.
The same arguments can be used to obtain the full bound for $(x,k)$-derivatives. 

%\Dets 
 
%E.g by Fa\`{a} di Bruno's formula one has that
%\begin{multline}
%D^q|\xi-\eta|=\sum_{r=1}^q\sum_{\alpha_1+\cdots+\alpha_r=q-r}|\xi-\eta|^{1-2r}\prod_{j=1}^r
%D^{\alpha_j+1}|\xi-\eta|^2\\=\sum_{r=1}^q\sum_{\alpha_1+\cdots+\alpha_r=q-r}|\xi-\eta|^{1-2r}\prod_{\alpha_j\,\textrm{even}}(\xi_1\eta_2-\xi_2\eta_1)\prod_{\alpha_j\,\textrm{odd}}(\xi_1\eta_1+\xi_2\eta_2),
%\end{multline}again with coefficients omitted, therefore
%\begin{equation}\|D^q|\xi-\eta|\|_{L^\infty}\lesssim\sup_{r\geq 1}2^{(1-2r)k}2^{r(k+k_2)}\lesssim 2^{k_2}
%\end{equation}
%with constants depending on $q$.

To estimate the term $b(x,k)$ in \eqref{Lempsi1pr5} we need to take care of the singularity of 
high derivatives of $|x-y|^{-1}$. We first rewrite
\begin{align}\label{Lempsi1pr10}
b(x,k) = E_k^{-1}(x)  \int_{\R^3} E_k(x-y) V(y) \psi(y,k) \, \varphi_{>-10}(|y|/|x|)\, \mathrm{d}y,
	\qquad E_k(z) = \frac{e^{i|k| |z|}}{|z|} \varphi_{\leq 0}(z).
\end{align}
When applying $k$ derivatives we can use the same arguments as above.
For the spatial derivatives instead, 
the desired estimates can be easily seen to hold when derivatives hit $E_k^{-1}(x)$.
We may then just look at the cases when derivatives hit the integrand.
For such terms we convert $\partial_x$ hitting $E_k(x-y)$ into $-\partial_y$ and integrate by parts onto $V\psi$.
Using the assumptions \eqref{assV2} and \eqref{psiLinfty} we arrive at \eqref{psi10}. 
%\Dets
\end{proof}

%The next Lemma shows that $\psi_1$ admits an expansion in negative powers of $|x|$ with coefficients 
%belonging to $ \mathcal{G}^N$ classes.
The next lemma gives an expansion for $\psi_1$ in powers of $|x|^{-1}$.

\begin{lemma}\label{lemmapsi1}
Let $N_2 \in [1,N_1] \cap \Z$ where $N_1$ is as in \eqref{assV2}. 
Denoting $r=|x|$ and $\omega = x/|x|$, we have the expansion
\begin{align}\label{lemmapsi1exp}
\psi_1(x,k) = \sum_{j=0}^{N_2-1} g_{j}(\omega,k) \, r^{-j} \jk^j  + R_{N_2}(x,k),
\end{align}
for $r \geq 1$, where 
\begin{align}\label{lemmapsi1g}
g_0(\omega,k) := -\frac{1}{4\pi}\int_{\R^3} e^{-i|k|\omega\cdot y} V(y)\psi(y,k) \, 
	%\varphi_{\leq -10}(y/|x|)
	\mathrm{d}y, %\in \mathcal{G}^{N_1}, 
	%\qquad g_j \in \mathcal{G}^{N_1-j},
\end{align}
the coefficients $g_j$, $j=0,1,\dots,N_2-1$, satisfy

\begin{align}\label{lemmapsi1gj}
\big| \partial_\omega^\alpha \partial_k^\beta g_j(\omega,k) \big| \lesssim \jk^{|\alpha|} + (|k|/\jk)^{1-|\beta|},
  \qquad |\alpha| + |\beta| \leq N_1 - N_2,
\end{align}
and
\begin{align}\label{lemmapsi1R}
\big| \partial_k^\beta R_{N_2}(x,k) \big| \lesssim  r^{-N_2}  \big( \jk^{N_2} + (|k|/\jk)^{(1-|\beta|)} \big),
  \qquad  |\beta| \leq N_1 - N_2-1.
%\in  r^{-{N_2}} (1+|k|)^{N_2} \mathcal{G}^{N_1-N_2}.
\end{align}

%: NOTE: ONLY k-derivatives for the remainder ...

\end{lemma}

\smallskip
\begin{proof}
From the definition \eqref{psipsi1}, writing $x = r\omega$, $r=|x|$, we have
\begin{align}\label{lemmapsi1pr0}
\begin{split}
%\psi_1(x,k) =
\psi_1(r\omega,k) & = -\frac{1}{4\pi} \int_{\R^3} e^{i|k| r (|\omega-y/r| - 1)} \frac{1}{|\omega-y/r|} 
	V(y) \psi(y,k) \, \mathrm{d}y
\\
& = %-\frac{1}{4\pi}\int_{\R^3} e^{-i|k|  \omega \cdot y} V(y) \psi(y,k) \, \mathrm{d}y
g_0(\omega,k) + I_1(x,k) + I_2(x,k) + I_3(x,k),
\end{split}
\end{align}
where $g_0$ is defined in \eqref{lemmapsi1g} and 
\begin{align}
\label{lemmapsi1pr0.1}
I_1 & := -\frac{1}{4\pi} \int_{\R^3} \Big[ e^{i|k| r (|\omega-y/r| - 1)} - e^{-i|k|  \omega \cdot y} \Big]  \frac{1}{|\omega-y/r|} 
  V(y) \psi(y,k) \, \varphi_{\leq -10}(y/|x|) \, \mathrm{d}y,
\\
\label{lemmapsi1pr0.2}
I_2 & := -\frac{1}{4\pi} \int_{\R^3} e^{-i|k|  \omega \cdot y} \Big[  \frac{1}{|\omega-y/r|} - 1 \Big] V(y) \psi(y,k) \,
  \varphi_{\leq -10}(y/|x|) \, \mathrm{d}y,
\\
\label{lemmapsi1pr0.3}
I_3 & := -\frac{1}{4\pi} \int_{\R^3} \Big[ \frac{e^{i|k| r (|\omega-y/r| - 1)}}{|\omega-y/r|}  - e^{-i|k|  \omega \cdot y} \Big] 
  V(y) \psi(y,k) \,  \varphi_{> -10}(y/|x|) \, \mathrm{d}y.
\end{align}
We will expand the integrands in the first two terms in powers of $y/r$, while the third 
term is a remainder that can be absorbed into $R_{N_2}$ directly.

%As in the proof of Lemma \ref{Lempsi1}, see \eqref{Lempsi1pr1},  let us write
%\begin{align}
%\label{lemmapsi1pr0}
%\begin{split}
%\psi_1(r\omega,k) & = \psi_1^-(r\omega,k) + \psi_1^+(r\omega,k),
%\\
%\psi_1^-(r\omega,k) & := \int \exp \big( i|k|r [ |\omega-y/r| - 1 ] \big) \frac{1}{|\omega-y/r|} V(y) \psi(y,k) \, 
%	\varphi_{\leq -10}(y/|x|) \mathrm{d}y,
%\\
%\psi_1^+(r\omega,k)  &: =\int \exp \big( i|k|r [ |\omega-y/r| - 1 ] \big) \frac{1}{|\omega-y/r|} V(y) \psi(y,k) \, \varphi_{>-10}(y/|x|) \mathrm{d}y.
%\end{split}
%\end{align}

\medskip
{\it Estimate of \eqref{lemmapsi1pr0.1}}.
Observe that, for $|y| \leq r/2$ and arbitrary $n$ we can expand
\begin{align}
\label{lemmapsi1pr1}
\begin{split}
|\omega-y/r| & = \sqrt{1 + |y|^2 r^{-2} - 2\omega\cdot y/r} 
\\ & = 1 - \omega\cdot y/r + \sum_{j=2}^{n-1} r^{-j} \sum_{j_1+j_2= j}  a_{j_1j_2} |y|^{j_1} (\omega\cdot y)^{j_2} + 
  R_n(x,y),
\end{split}
\end{align}
for some coefficients $a_{j_1j_2}\in \C$, with
\begin{align}
\label{lemmapsi1pr2}
R_n(x,y) = (|y|/r)^{n}(1 + a(\omega,y)), \qquad |\partial_\omega^{\alpha}a(\omega,y)| \lesssim 1.
\end{align}
A similar expansion holds for $|\omega-y/r|^{-1}$.
Then, we can write %(expanding up to order $n+1$ in the formula above)
\begin{align}\label{lemmapsi1pr3}
\begin{split}
& X := r (|\omega-y/r| - 1) + \omega \cdot y = 
  \sum_{j=1}^{n-1} r^{-j} a_{j}(\omega,y) + R_n(x,y),
\\
& \mbox{with} \qquad | \partial_\omega^\alpha a_{j}(\omega,y) | \lesssim \jy^{j+1},
\qquad | \partial_\omega^\alpha R_n(x,y) | \lesssim \jy^{n+1} r^{-n}.
\end{split}
\end{align}
We look at the factor in the integrand of \eqref{lemmapsi1pr0.1} and write 
\begin{align}\label{lemmapsi1pr5.1}
\begin{split}
& \Big[ e^{i|k| r (|\omega-y/r| - 1)} - e^{-i|k|  \omega \cdot y} \Big] \frac{1}{|\omega-y/r|}  
  =  e^{-i|k| \omega \cdot y} \frac{1}{|\omega-y/r|} \big[ e^{i|k|X} - 1 \big]
%\\ e^{-i|k|  \omega \cdot y} \Big[ e^{i|k| r (|\omega-y/r| - 1)} - 1 \Big]  \frac{1}{|\omega-y/r|}  
\\ 
& = e^{-i|k|  \omega \cdot y} \Big[ \sum_{j=1}^{n-1} r^{-j} \sum_{1 \leq \ell\leq j} |k|^\ell a_{j,\ell}(\omega,y) + 
	\sum_{1 \leq \ell \leq n} |k|^\ell R_{n,\ell}(x,y) \Big] ,
\end{split}
\end{align}
where the coefficients and remainder terms satisfy
\begin{align}\label{lemmapsi1pr5.2}
\begin{split}
& | \partial_\omega^\alpha a_{j,\ell}(\omega,y) | \lesssim \jy^{j+1},
	\qquad | R_{n,\ell}(x,y) | \lesssim \jy^{n+1} r^{-n}.
\end{split}
\end{align}

Using the definition \eqref{lemmapsi1pr0.1} and the expansion \eqref{lemmapsi1pr5.1} we have
\begin{align}
I_1(x,k) = \sum_{j=1}^{N_2-1} b_j(\omega,k) \jk^j r^{-j} + R_{N_2}^1(x,k),
\end{align}
having defined %We then set $n=N_2$ and let
\begin{align}
\begin{split}
b_j(\omega,k) &:= \int_{\R^3} e^{-i|k|  \omega \cdot y} \frac{1}{\jk^j} 
  \sum_{1\leq \ell\leq j} |k|^\ell a_{j,\ell}(\omega,y) \, V(y) \psi(y,k) \, \mathrm{d}y,
\\
R_{N_2}^1(x,k) & := \int_{\R^3} e^{-i|k|  \omega \cdot y} 
  \sum_{1 \leq \ell\leq N_2} |k|^\ell R_{N_2,\ell}(x,y) \, V(y) \psi(y,k) \, \mathrm{d}y. %\frac{1}{\jk^n}
\end{split}
\end{align}
In view of the first estimate of \eqref{lemmapsi1pr5.2}, the integrability assumptions on $V$ 
from \eqref{assV2}, the constraints $|\alpha|+ |\beta| \leq N_1-N_2$,
%\[ \int_{\R^3} \jy^{N_2+1+ |\alpha|+ |\beta|} V(y) \,dy < \infty \] 
and the estimates \eqref{psiLinfty} giving 
$|\partial_k^\beta \psi(y,k)| \lesssim \jy^{|\beta|}(|k|/\jk)^{1-|\beta|}$ for $\beta\neq 0$,
we see that the coefficients $b_j$ satisfy estimates as in \eqref{lemmapsi1gj}.
Similarly the remainder $R^1_{N_2}$ satisfies estimates as in \eqref{lemmapsi1R}.
This gives an expansion of the desired form \eqref{lemmapsi1exp} for $I_1$.

\medskip
{\it Estimate of \eqref{lemmapsi1pr0.2}}.
The term $I_2$ is similar to \eqref{lemmapsi1pr0.1} so we can skip the details.

\medskip
{\it Estimate of \eqref{lemmapsi1pr0.3}}.
Since on the support of $I_3$ we have $|y| \gtrsim |x|$, 
we can use the weighted integrability of $V$ in \eqref{assV2} to show that this term is a remainder as in \eqref{lemmapsi1R}.
Using that for $|x| \lesssim |y|$, $\beta\neq 0$, we have the bounds
\begin{align*}
\big| \partial_k^\beta e^{i|k| r (|\omega-y/r| - 1)} \big| + \big| \partial_k^\beta e^{-i|k| \omega\cdot y} \big|
	+ | \partial_k^\beta \psi(y,k) | 
	\lesssim  \jy^{|\beta|} (|k|/\jk)^{1-|\beta|}),
\end{align*}
see \eqref{psiLinfty}, %on the $k$-derivatives of $\psi$,
we have, for $|\beta| \leq N_1-N_2-1$
\begin{align*}
|\partial_k^\beta I_3(x,k) | 
	& \lesssim  (|k|/\jk)^{1-|\beta|}) \,
	\int_{\R^3} \frac{|x|}{|x-y|} \, \jy^{|\beta|} V(y) \,  \varphi_{> -10}(y/r)\, \mathrm{d}y,
	\\	
	& \lesssim (|k|/\jk)^{1-|\beta|}) \, r^{-N_2} \cdot {\big\| \jx^{N_1} V \big\|}_{L^\infty\cap L^1}. 
\end{align*}
%To estimate also the $\omega$-derivatives we need some additional integration by parts
%to avoid the singularity from negative powers of $|\omega-y/r|$. This can be taken care of 
%as in the last part of the proof of Lemma \ref{Lempsi1} with a splitting as in \eqref{Lempsi1pr5} 
%and arguing as in \eqref{Lempsi1pr10} above.
This concludes the proof of the Lemma.
\end{proof}

\medskip
Motivated by \eqref{psi10} and the expansion \eqref{lemmapsi1exp} we define the following classes of symbols:

\medskip
\begin{definition}\label{Gclass}
For $N \in \mathbb{Z}_+$ we let $\mathcal{G}^N$ be the class of $L^\infty_{x,k}$ 
%functions $f : \R^3 \times \R^3 \mapsto \C$ such that
%\begin{align}\label{classG_N}
%& \big| \nabla_x^\alpha \nabla_k^\beta f(x,k) \big| \leq c_{\alpha,\beta} \jx^{-|\alpha|} 
%  \big(\jk^{|\alpha|} + (|k|/\jk)^{1-|\beta|} \big)
%  \qquad 1\leq|\alpha |+|\beta|\leq N. % \partial_{|x|} ??? 
%\end{align}
%With a slight abuse of notation we use the same name to denote the class $L^\infty_{\omega,k}$ 
functions $f : \mathbb{S}^2 \times \R^3 \mapsto \C$ such that
\begin{align}\label{classG_Nomega}
& \big| \nabla_\omega^\alpha \nabla_k^\beta f(\omega,k) \big| \leq c_{\alpha,\beta}
  \big(\jk^{|\alpha|} + (|k|/\jk)^{1-|\beta|} \big)
  \qquad 1\leq|\alpha |+|\beta|\leq N. % \partial_{|x|} ??? 
\end{align}
\end{definition}

%Comment: Can give the following definition, or avoid it if it is not such a useful notation
%Can avoid \eqref{classG_N} probably, and just keep \eqref{classG_Nomega}\dots}

\iffalse
\begin{definition}\label{Gclass}
For $N \in \mathbb{Z}_+$ we let $\mathcal{G}^N$ be the class of $L^\infty_{x,k}$ 
functions $f : \R^3 \times \R^3 \mapsto \C$ such that
\begin{align}\label{classG_N}
& \big| \nabla_x^\alpha \nabla_k^\beta f(x,k) \big| \leq c_{\alpha,\beta} (1+|x|)^{-|\alpha|} 
	% (1+|k|)^{|\alpha|} (1 +|k|^{-|\beta|}) %\min(1,|k|)^{1-|\beta|}, 
	\big(|k|^{|\alpha|} + |k|^{-|\beta|}\big)
	\qquad 1\leq|\alpha |+|\beta|\leq N. % \partial_{|x|} ??? 
\end{align}
\end{definition}

%We also defined related classes of symbols as follows
%\begin{definition}\label{Sclass}
%For $N \in \mathbb{Z}_+$ we define $\mathcal{S}_N$ to be the class of $L^\infty_{p,q}$ symbols
%$f : \R^3 \times \R^3 \mapsto \C$ such that
%\begin{align}\label{classG_N}
%& \big| \nabla_p^\alpha \partial_q^\beta f(p,q) \big| \leq c_{\alpha,\beta} 
%	|p|^{-|\alpha|} (1+|q|)^{|\alpha|} (1 +|q|^{-1})^{|\beta|} %\min(1,|k|)^{1-|\beta|}, 
%	\qquad 1\leq|\alpha| + |\beta|\leq N.
%\end{align}
%\end{definition}
\fi

To fix ideas one can think of functions in $ \mathcal{G}^N$ as functions of the form $\exp(i|k|x_1/|x|)$.
This is essentially how $\psi_1$ looks like, with the exception that its differentiability in $k$,
is limited by the integrability of $V$.
More precisely, one should think of the class $\mathcal{G}^N$ as functions of the form
\begin{align}
\int_{\R^3} e^{i|k| \frac{x}{|x|}\cdot y} f(y) \, \mathrm{d}y, \qquad {(1+|y|)}^N f(y) \in L^1.
\end{align}
%for $|x|\geq 1$.
Compare this with the formula for $g_0$ in \ref{lemmapsi1g}.
%The reason for the second definition \eqref{classG_Nomega} is 
%
%Note that if $f=f(x,k)$ satisfies \eqref{classG_N}, 
%then its restriction to the sphere $g(\omega,k) := f(\omega,k)$, $\omega=x/|x|$, satisfies \eqref{classG_Nomega}.
%Conversely, if $g(\omega,k)$ satisfies \eqref{classG_Nomega}, its homogeneous extension
%$f(x,k) := g(x/|x|,k)$ satisfies \eqref{classG_N}.
%
Functions in $ \mathcal{G}^N$ will often appear 
in the expressions for symbols of bilinear operators in our applications.
As symbols these are not standard ones (e.g., of bilinear Mihlin-H\"{o}rmander type),
for example because of losses when $k$ is large.

\medskip
\section{Preliminary bounds: Linear estimates and high frequencies}

In this section we first state some decay estimates for the linear evolution
and then show how to obtain the bootstrap estimate on the standard Sobolev norms in \eqref{proBoot2} 
using the decay and the a priori assumptions.
The rest of the section is then dedicated to a priori bounds for the nonlinear evolution when
one restricts the analysis to high frequencies that are large compared to time.

\medskip
\subsection{Linear Estimates}
We start by collecting some dispersive estimates for Schr\"odinger operators.

\begin{lemma}\label{LemLinear}
Under the assumptions \eqref{Nparam}-\eqref{assV2} on the potential $V$, 
with $\wt{f}$ defined as in Theorem \ref{theodFT}, we have 
\begin{align}
\label{linearinfty}
{\| e^{it(-\Delta+V)} f \|}_{L^\infty} 
  & \lesssim \frac{1}{|t|^{3/2}}{\| \wt{f} \|}_{L^\infty} + \frac{1}{|t|^{7/4}} {\| \partial_k^2 \wt{f} \|}_{L^2},
\end{align}
and
\begin{align}\label{LemLinearL6}
{\| e^{it(-\Delta+V)} f \|}_{L^6} \lesssim \frac{1}{|t|} {\| \partial_k \widetilde{f} \|}_{L^2}.
\end{align}
Interpolating \eqref{LemLinearL6} with the $L^2$ conservation we have
\begin{align}\label{linearLp<6}
{\| e^{it(-\Delta+V)} f \|}_{L^p} \lesssim \frac{1}{|t|^{(3/2)(1-2/p)}} 
  {\| \widetilde{f} \|}_{H^1_k}, \qquad 2\leq p \leq 6.
\end{align}
Moreover, for all $6 < p < \infty$,
\begin{align}
\label{linearLp>6}
\begin{split}
{\| e^{it(-\Delta+V)} f \|}_{L^p} & \lesssim 
  \frac{1}{|t|^{3/2(1-2/p)}} {\big\| \partial_k \wt{f} \big\|}_{L^2}^{1-\theta}
  {\| \partial_k^2 \wt{f} \|}_{L^2}^{\theta}, \qquad \theta = \frac{1}{2} - \frac{3}{p}. 
\end{split}
\end{align}

\end{lemma}

\begin{proof}
All these linear estimates can be deduced from 
the corresponding estimates involving the flat Fourier transform, 
and using the boundedness of the wave operator.
We recall that, under our assumptions, % $V\in W^{N,\infty}$
the wave operators, defined by 
$\mathcal{W}_{\pm} := \lim_{t\rightarrow \pm \infty} e^{it(-\Delta+V)} e^{it\Delta}$
are bounded on Sobolev spaces; see for example Yajima \cite{Yajima}. %,Yajima2}.
Moreover, as in see \eqref{W0}, $\mathcal{W} := \mathcal{W}_+ = \wt{\mathcal{F}}^{-1} \what{\mathcal{F}}$.

To prove \eqref{linearinfty}, recall first that
\begin{align}\label{linearinfty0}
{\| e^{-it\Delta} f \|}_{L^\infty} \lesssim \frac{1}{|t|^{3/2}}{\| \what{f} \|}_{L^\infty} 
	+ \frac{1}{|t|^{7/4}} {\| \partial_k^2 \what{f} \|}_{L^2},
\end{align}
see, for example, \cite{GMS2}. Then it suffices to write
\begin{align}\label{Wop}
e^{it(-\Delta+V)} f =  \mathcal{W} e^{-it\Delta} \mathcal{W}^\ast f %, \qquad \mathcal{W} :=  \wt{\mathcal{F}}^{-1} \what{\mathcal{F}},
\end{align}
%where $\mathcal{W}$ is the wave operator.
and, by the boundedness of $\mathcal{W}$ on $L^p$ spaces,
%Comment: Check boundedness on $L^\infty$: okay!}
\eqref{linearinfty0},
and the fact that $\what{\mathcal{F}}$ and $\wt{\mathcal{F}}$ are unitary on $L^2$, we obtain \eqref{linearinfty}.

Similarly, \eqref{LemLinearL6} can be obtained using the standard Klainerman-Sobolev type embedding
\begin{align*}
{\| e^{-it\Delta} f \|}_{L^6} \lesssim
\frac{1}{|t|} {\| x f \|}_{L^2}  \lesssim \frac{1}{|t|} {\| \partial_k \what{f} \|}_{L^2}
\end{align*}
%
%\begin{align}
%\begin{split}
%{\| e^{it(-\Delta+V)} f \|}_{L^6} & \lesssim {\| W_+ e^{-it\Delta} W_+^\ast f \|}_{L^6} \lesssim {\| e^{-it\Delta} W_+^\ast f \|}_{L^6}
%  \\ & \lesssim \frac{1}{t} {\| x W_+^\ast f \|}_{L^2} = \frac{1}{t} {\| \partial_k \widehat{\mathcal{F}}W_+^\ast f \|}_{L^2}
%  =  \frac{1}{t} {\| \partial_k \widetilde{\mathcal{F}}f \|}_{L^2}.
%\end{split}
%\end{align}
and \eqref{linearLp>6} using, for $q>6$ with $1/q+1/q'=1$, and $\theta = 1/2-3/p$, that
\begin{align*}
{\| e^{-it\Delta} f \|}_{L^q} \lesssim \frac{1}{|t|^{(3/2)(1-2/q)}} {\| f \|}_{L^{q'}} 
 	\lesssim \frac{1}{|t|^{(3/2)(1-2/q)}}  {\| x f \|}_{L^2}^{1-\theta} {\| x^2 f \|}_{L^2}^{\theta}.
\end{align*}
\end{proof}

Next, we use Lemma \ref{LemLinear} to obtain some a priori decay bounds 
as direct consequences of the a priori assumptions \eqref{apriori}.

\begin{lemma}\label{LemLinear2}
Let $u = e^{it(-\Delta+V)}f$ and assume the bounds \eqref{apriori} hold
with the definitions in \eqref{space}. Then,
\begin{align}
\label{aprioriL<6}
{\| e^{it(-\Delta+V)} f \|}_{L^p} & \lesssim \e \jt^{-3/2(1-2/p)},  \qquad  2\leq p \leq 6,
\\
\label{aprioriL>6}
{\| e^{it(-\Delta+V)} f \|}_{L^p} & \lesssim \e \jt^{-5/4 + 3/(2p) + \delta\theta},
	%\jt^{\theta(1/2+\delta)}, 
	\qquad p > 6, \quad \theta = \frac{1}{2} - \frac{3}{p}.
%\\
%{\| e^{it(-\Delta+V)} f \|}_{L^\infty} & \lesssim \e \jt^{-5/4+\delta}.
\end{align}
\end{lemma}

\begin{proof}
For $|t|\leq 1$ the estimates follow from the boundedness of wave operators,
Sobolev's embedding, and the a priori bound \eqref{apriori}:
\begin{align*}
{\| e^{it(-\Delta+V)} f \|}_{L^p} & \lesssim {\| e^{-it\Delta} \mathcal{W}^\ast f \|}_{L^p}
  \lesssim {\| e^{-it\Delta} \mathcal{W}^\ast f \|}_{H^2} 
  \\ & \lesssim {\| \mathcal{W}^\ast f \|}_{H^2} \lesssim  {\| \jk^2 \wt{\mathcal{F}}f \|}_{L^2}
  \lesssim \e.
\end{align*}
For $|t| \geq 1$ the estimate \eqref{aprioriL<6}, resp. \eqref{aprioriL>6}, 
is a direct consequences of \eqref{linearLp<6}, resp. \eqref{linearLp>6},
and the bounds on the weighted norms in \eqref{apriori}.
\end{proof}

%\medskip
%\subsection{High Frequencies}

\medskip
\subsection{Sobolev estimates}\label{secSob}
We now prove the bootstrap estimate \eqref{proBoot2} for the Sobolev-type norm
%the main bootstrap Proposition \ref{proBoot}
using energy estimates and the pointwise decay from Lemma \ref{LemLinear2}.

\begin{proposition}\label{proSob}
Under the a priori assumptions \eqref{apriori} we have
\begin{align}\label{proSobconc}
{\| u(t) \|}_{H^N} + {\| \langle k\rangle^N \wt{f} \|}_{L^2} \leq \e_0 + C \e^2.
\end{align}
\end{proposition}

\begin{proof}
First notice that %\eqref{wtF1} implies
\begin{align}
{\| |k|^j \wt{f} \|}_{L^2} =  {\| |k|^j \wt{u} \|}_{L^2} = c {\| (-\Delta+V)^{j/2} u \|}_{L^2}.
\end{align}
Moreover, by direct estimates (or also using the boundedness of wave operators)
for any $j \leq N/2$
\begin{align}\label{proSob1}
{\| (-\Delta+V)^j g \|}_{L^2} \lesssim {\| g \|}_{H^{2j}} \lesssim \sum_{\ell = 0}^j {\| (-\Delta+V)^\ell g \|}_{L^2}.
\end{align}
In particular, the two norms in \eqref{proSobconc} are equivalent so it suffices to bound the first one.
We use a standard energy estimate.
We let
\begin{align*}
u^j := (-\Delta + V)^j u, %\qquad {\| u^j \|}_{L^2} \approx {\| |k|^{2j} \wt{u} \|} \approx
\end{align*}
for $j=0,\dots,N$, and differentiate the equation \eqref{NLSV} using $-\Delta + V$ to obtain
\begin{align*}
i\partial_t u^j  + (-\Delta + V)u^j = (-\Delta + V)^j u^2.
\end{align*}
Therefore, using \eqref{proSob1} and standard product estimates, %using the boundedness of wave operators on Sobolev spaces,
\begin{align*}
\frac{d}{dt} {\|u^j\|}_{L^2} \lesssim {\big\| (-\Delta + V)^j u^2 \big\|}_{L^2}
  %\lesssim {\big\| |\xi|^{2j} \wt{\mathcal{F}} (u^2) \big\|}_{L^2} 
  %\lesssim {\| \what{\mathcal{F}}\wt{\mathcal{F}} (u^2) \big\|}_{\dot{H}^{2j}} \\ 
  \lesssim {\| u \|}_{H^{2j}} {\| u \|}_{L^\infty}.
\end{align*}
Using the apriori assumption \eqref{apriori} and the decay estimate \eqref{aprioriL>6} %assumption \eqref{apriori}, 
we get
\begin{align*}
{\|u^j(t)\|}_{L^2} - {\| u^j(0) \|}_{L^2} 
	& \lesssim \int_0^t {\| u(s) \|}_{H^{2j}} {\| u(s) \|}_{L^\infty} \, \mathrm{d}s
	\lesssim \int_0^t \e \cdot \e \js^{-5/4+\delta/2}  \, \mathrm{d}s \lesssim \e^2.
\end{align*}
Summing over $j\leq N/2$ %and using again \eqref{proSob1}, 
gives the desired conclusion.
\end{proof}

\medskip
%\subsection{The distorted Duhamel's formula}
\subsection{Weighted estimates for high frequencies}\label{secHF}

Recall that we define the profile of a solution $u$ of \eqref{NLSV} by
\begin{align}
\label{prof}
f(t,x) := \big( e^{-it(-\Delta + V )} u(t,\cdot) \big) (x), 
  \qquad \wt{f}(t,k) = e^{-it|k|^2} \widetilde{u}(t,k).
\end{align}
and that this satisfies the equation
%\begin{align}\label{Duhamel00}
%\begin{split}
%\partial_t \widetilde{f}(t,k) & = -i e^{-it|k|^2} \wt{\mathcal{F}}(u^2)
%\\
%& = -i\iint e^{it (-|k|^2 + |\ell|^2 + |m|^2 )} \widetilde{f}(t,\ell) \widetilde{f}(t,m) 
%  \, \mu(k,\ell,m) \,\mathrm{d}\ell \mathrm{d}m,
%\end{split}
%\end{align}
\begin{align}
\label{Duhamel0}
\begin{split}
& \widetilde{f}(t,k) = \widetilde{u_0}(k) - i \mathcal{D}(t)(f,f)
\\
& \mathcal{D}(t)(f,f) := \int_0^t \iint e^{is (-|k|^2 + |\ell|^2 + |m|^2 )} \widetilde{f}(s,\ell) \widetilde{f}(s,m)
  \, \mu(k,\ell,m) \, \mathrm{d}\ell \mathrm{d}m\,\mathrm{d}s,
\end{split}
\end{align}
where
\begin{align}
\label{mu0}
\mu(k,\ell,m) := {(2\pi)}^{-9/2} \int \overline{\psi(x,k) }\psi(x,\ell) \psi(x,m) \, \mathrm{d}x
\end{align}
%is the distribution which characterizes the interaction between the generalized eigenfunctions

We want to estimate the weighted norms in \eqref{space} as in Proposition \ref{proBoot} 
when frequencies are large relative to a (small) power of time.
This will be helpful later on in the analysis of 
the nonlinear spectral distribution and its asymptotic expansion.
%the proof of bilinear bounds for the associated multipliers, %(because we can disregard large frequencies there)
%and the nonlinear analysis and the bootstrap estimates of weighted norms.
More precisely, let us restrict \eqref{Duhamel0} to high frequencies by considering
\begin{align}
\label{HF1}
\begin{split}
\mathcal{D}_{HF}(t)(f,f) := 
  \int_0^t \iint e^{is (-|k|^2 + |\ell|^2 + |m|^2 )} \widetilde{f}(s,\ell) \widetilde{f}(s,m)
  \, \mu(k,\ell,m) \, \\ \times \varphi_{\geq0}((|\ell|^2+|m|^2+|k|^2)\js^{-2\delta_N})
  \, \mathrm{d}\ell\mathrm{d}m\,\mathrm{d}s,
\end{split}
\end{align}
%that is the component of $\mathcal{D}$ in \eqref{Duhamel0} with inputs satisfying
%$|\ell| + |m| \gtrsim \js^{\delta_N}$. Also define the low frequency contribution
%\begin{align}
%\mathcal{D}_{LF}(t)(f,f) = \mathcal{D}(t)(f,f) - \mathcal{D}_{HF}(t)(f,f).
%\end{align}
where
\begin{align}\label{d_N}
\delta_N := \frac{3}{N-5}  %Fix constant at the end
\end{align}
with $N$ the Sobolev regularity of our solution, see \eqref{Nparam} and \eqref{space}-\eqref{apriori}.
This is our main Proposition in this section:

\begin{proposition}[High frequencies estimates]\label{proHF}
Under the a priori assumptions \eqref{apriori} we have
\begin{align}\label{proHFconc}
{\| \partial_k \mathcal{D}_{HF}(t)(f,f) \|}_{L^2} +
  \jt^{-1/2-\delta} {\| \partial_k^2 \mathcal{D}_{HF}(t)(f,f) \|}_{L^2} \leq \e_0 + C \e^2.
\end{align} 
%Moreover
%\begin{align}
%{\| \partial_k \mathcal{D}_{HF}(t)(f,f) \|}_{L^2} +
%  \jt^{-1/2} {\| \partial_k^2 \mathcal{D}_{HF}(t)(f,f) \|}_{L^2} \leq \e_0 + C \e^2.
%\end{align}
\end{proposition}

To prove Proposition \ref{proHF} we are going to make use, among other things,
of Proposition \ref{proGHW} below, which can be deduced from \cite{GHW}.
%Indeed, in \cite{GHW} the authors study various formulas related to the derivative $\partial_k \mu$,
%as well as bounds on bilinear expression of the form \eqref{proGHW0}.
For convenience, and only for the purpose of stating Proposition \ref{proGHW} below 
and applying it to the proof of Proposition \ref{proHF}, we introduce a Coifman-Meyer type norm for symbols as in \cite{GHW}:
\begin{align}\label{CMnorm}
{\| n \|}_{CM_\delta} := \sup_{0 \leq |a| \leq 10} 
	\big| \big(|k|+|\ell|+|m|\big)^{\delta+|a|} \nabla^a n(k,\ell,m) \big|, \qquad \delta >0.
\end{align}

\begin{proposition}[Germain-Hani-Walsh \cite{GHW}]\label{proGHW}
Consider the bilinear operator
\begin{align}\label{proGHW0}
\begin{split}
\mathcal{B}_{n}(g,h)(k) := 
  \iint \widetilde{g}(\ell) \widetilde{h}(m) \, n(k,\ell,m) \, 
  \, \mu(k,\ell,m) \, \mathrm{d}\ell \mathrm{d}m,
\end{split}
\end{align}
where $n$ is a %(Coifman-Meyer type) 
symbol verifying the estimate
\begin{align}
\label{proGHW1}
{\| n \|}_{CM_\delta} %:= \sup_{0 \leq |a| \leq 10} \big| \big(|k|+|\ell|+|m|\big)^{\delta+|a|} \nabla^a n(k,\ell,m) | \big| 
	\leq A,
\end{align}
for some $\delta >0$. Then:

\setlength{\leftmargini}{2em}
\begin{itemize}

\medskip
\item[(i)] (\it{H\"older estimates})
For any $1<p,q,r,p',q'<\infty$, the following estimate holds
\begin{align}
\label{proGHWest}
\begin{split}
{\| \wt{\mathcal{F}}^{-1} \mathcal{B}_{n}(g,h) \|}_{L^r} \lesssim A \big( {\|g\|}_{L^p} {\|h\|}_{L^q}
  + {\|g\|}_{L^{p'}} {\|h\|}_{L^{q'}} \big), \qquad \frac{1}{p} + \frac{1}{q}=\frac{1}{r} < \frac{1}{p'} + \frac{1}{q'}.
\end{split}
\end{align}
In particular, for  $p'\in (1,p)$, we have 
\begin{align}\label{proGHWest'}
{\| \wt{\mathcal{F}}^{-1} \mathcal{B}_{n}(g,h) \|}_{L^r} \lesssim A{\|g\|}_{L^p \cap L^{p'}} {\|h\|}_{L^q}, 
  \qquad \frac{1}{p} + \frac{1}{q} = \frac{1}{r},
\end{align}
and a similar estimate exchanging the roles of $g$ and $h$.

\medskip
\item[(ii)] (\it{Algebraic identity for the weights})
The following identities hold:
\begin{align}
\label{proGHW2}
\begin{split}
& \partial_k \mathcal{B}_{n}(g,h)(k) 
\\ & = \mathcal{B}_{\partial_k n}(g,h)(k) 
  + \mathcal{B}'_n \big(\Ftil^{-1} \partial_\ell \wt{g},h\big)(k) + \mathcal{B}'_{\partial_\ell n} (g,h)(k) 
  + \mathcal{B}'_n(g,h)(k)
  \\
  &= \mathcal{B}_{\partial_k n}(g,h)(k) + \mathcal{B}_n'\big(g,\Ftil^{-1}\partial_m \wt{h}\big)(k) 
  + \mathcal{B}_{\partial_m n}'(g,h)(k) + \mathcal{B}_n'(g,h)(k),
\end{split}
\end{align}
where we use the `prime' notation $\mathcal{B}'_n$ to denote a generic operator of the form \eqref{proGHW0},
with $\mu$ replaced by a slightly different expression $\mu'$,
and which satisfy the same H\"older bounds \eqref{proGHWest}-\eqref{proGHWest'} satisfied by $\mathcal{B}_n$.

Moreover, one can iterate formula \eqref{proGHW2} to obtain
\begin{align}
\label{proGHW2'}
\begin{split}
& \partial_k^2 \mathcal{B}_{n}(g,h)(k) 
%\\ & 
= \sum_{a,b\geq 0, \, a+b\leq 2} \mathcal{B}'_{\partial_{(k,\ell)}^a n}\big(\Ftil^{-1}\partial_\ell^b \wt{g},h)(k) 
%  + \mathcal{B}'_{\partial_k n} (\partial_\ell g,h)(k)
%  + \mathcal{B}'_n (\partial_\ell^2 g,h)(k) 
%  + \mathcal{B}'_{\partial_\ell^2 n} (g,h)(k) + \mathcal{B}_n'(g,h)(k)
%  \\ & 
  = \sum_{a,b\geq 0, \, a+b\leq 2} \mathcal{B}'_{\partial_{(k,m)}^a n}\big(g,\Ftil^{-1}\partial_m^b \wt{h}\big)(k),
\end{split}
\end{align}
where $\mathcal{B}'$ are operators as above, and where we denote $\partial_{(k,\ell)}$ a generic 
derivative in $k$ and/or $\ell$, and similarly for $\partial_{(k,m)}$.

%\medskip
%\item[(iii)] (\it{Weighted estimates})
%For any $2 < p,q < \infty$ with $1/p+1/q=1/2$, and $q'<q$, we have
%\begin{align}\label{proGHWestw}
%\begin{split}
%{\| \partial_k \mathcal{B}_{n}(g,h) \|}_{L^2} 
%	\lesssim & \big({\| n\|}_{CM_\delta} + {\| \partial_k n\|}_{CM_\delta}\big) 
%	\min\big( {\|g\|}_{L^p} {\|h\|}_{L^q\cap L^{q'}},  {\|g\|}_{L^q \cap L^{q'}} {\|h\|}_{L^p} \big)
%	\\ 
%	+ & {\| n\|}_{CM_\delta} \min \big( {\|\partial_k g\|}_{L^p} {\|h\|}_{L^q\cap L^{q'}}, 
%		{\|g\|}_{L^q\cap L^{q'}} {\|\partial_k h\|}_{L^p} \big)
%\end{split}
%\end{align}

\end{itemize}

\end{proposition}

\smallskip
Let us make a few remarks:

\setlength{\leftmargini}{2em}
\begin{itemize}
 
\item Proposition \ref{proGHW} is contained in \cite{GHW}. Below we give some more precise references and 
a few elements of the proof for completeness.
%by recalling several
In \cite{GHW} the analysis is actually carried out for a slightly different $\mu$,
with $\psi(x,k)$ instead of $\bar{\psi(x,k)}$ in \eqref{mu0}.
However, this has no impact on the structure of $\mu$ that is used in the proofs 
and on the final estimates stated in the theorem.

\item Part (i) gives product H\"older-type estimates for $\wt{\mathcal{F}}^{-1} \mathcal{B}_{n}(g,h)$.
When $\mu(k,\ell,m) = \delta(k-\ell-m)$
the expression $\widehat{\mathcal F}^{-1} \mathcal{B}_{n}(g,h)$ is usually called a pseudo-product,
and H\"older-type estimates
under conditions similar to \eqref{proGHW1} (with $\delta=0$) are due to Coifman-Meyer \cite{CM}. 

%One could avoid the use of the space $L^{p'}$ provided the distorted Riesz transform is bounded on $L^p$.

\item Part (ii) is a commutation formula %for $\mathcal{B}_{n}(g,h)(k)$ and $\partial_k$.
which essentially states that the bilinear commutator between $\mathcal{B}_{n}(g,h)(k)$ and $\partial_k$
%(this is a trilinear commutator if we see $(g,h,n)$ as arguments)
is given by pseudo-products of the same form as $\mathcal{B}_{n}$ and where the symbol gets differentiated.

\item The assumption on the symbols \eqref{proGHW1} is probably not optimal, but it suffices for our purposes.
%so we use it for convenience.

%\item The weighted estimate \ref{proGHWestw} in part (iii) follows directly from parts (i) and (ii).
%We will use it in the Proof of Proposition \ref{proHF} by choosing indexes $p$ slightly larger than $2$ ($p=2$ is not allowed)
%and $q'$ arbitrarily close to $q$ very large.

\end{itemize}

\begin{proof}[Proof of Proposition \ref{proGHW}]

(i) The estimate \eqref{proGHWest} is the content of Theorem 1.1 in \cite{GHW}.

(ii) To explain \eqref{proGHW2}, let us first introduce some notation.
Let $\mathcal{W} = \wt{\mathcal{F}}^{-1}\what{\mathcal{F}}$ be the wave operator as in \eqref{Wop}, 
%$T^\ast$ denotes the adjoint operator of $T$,
let $R_{k_i} := \partial_{k_i} /|\nabla_k|$, $i=1,2,3$, be the standard Euclidean Riesz transform and denote
\begin{align}\label{GHW9}
E := \wt{\mathcal{F}} [|x|, \mathcal{W}] \mathcal{W}^\ast \wt{\mathcal{F}}^{-1}.
\end{align}
$E_k$ will be used to denote the operator $E$ acting on the variable $k$.
With our notation \eqref{mu0} for $\mu$, the formula (3.57)
derived on pages 8523-8524 of \cite{GHW}, to be understood in the sense of distributions, reads
\begin{align}\label{GHW10}
\partial_{k_i} \mu(k,\ell,m) = R_{k_i} R_{\ell} \cdot \partial_\ell \mu(k,\ell,m)
  + R_{k_i} E^\ast_{\ell} \mu(k,\ell,m) - R_{k_i} E^\ast_k \mu(k,\ell,m),
\end{align}
or, equivalently, 
\begin{align}\label{GHW10'}
\begin{split}
|\nabla_{k}| \mu(k,\ell,m) & = |\nabla_\ell| \mu(k,\ell,m)
  + E^\ast_{\ell} \mu(k,\ell,m) - E^\ast_k \mu(k,\ell,m)
  \\ & = |\nabla_m| \mu(k,\ell,m)
  + E^\ast_m \mu(k,\ell,m) - E^\ast_k \mu(k,\ell,m).
\end{split}
\end{align}
%The same formulas also hold with $m$ replacing $\ell$.

Applying \eqref{GHW10} to an expression like \eqref{proGHW0} and integrating by parts in $\ell$ gives
\begin{align}
\label{GHW11}
\begin{split}
\partial_k \mathcal{B}_{n}(g,h) & = \mathcal{B}_{\partial_k n}(g,h) + A + B + C + D,
\\
A & := - \iint \widetilde{g}(\ell) \widetilde{h}(m) \, \partial_\ell n(k,\ell,m) 
  \cdot R_{\ell} R_k \mu(k,\ell,m) \, \mathrm{d}\ell \mathrm{d}m, 
\\
B  & := - \iint \partial_\ell \widetilde{g}(\ell) \widetilde{h}(m) \,n(k,\ell,m)
  \cdot R_{\ell} R_k \mu(k,\ell,m) \, \mathrm{d}\ell \mathrm{d}m, 
\\
C & := \iint \widetilde{g}(\ell) \widetilde{h}(m) \, n(k,\ell,m) \, 
  \, R_k E^\ast_\ell \mu(k,\ell,m) \, \mathrm{d}\ell \mathrm{d}m, 
\\
D & := -\iint \widetilde{g}(\ell) \widetilde{h}(m) \, n(k,\ell,m) \, 
  \, R_k E^\ast_k \mu(k,\ell,m) \, \mathrm{d}\ell \mathrm{d}m.
\end{split}
\end{align}

One is then led to study the operators %$\Lambda_i$, $i=1,2,3$
\begin{align}\label{GHW12} 
\begin{split}
\Lambda_1(g,h)  & := \iint \widetilde{g}(\ell) \widetilde{h}(m) \, n(k,\ell,m) \, 
 R_{\ell_i} R_k \mu(k,\ell,m) \, \mathrm{d}\ell \mathrm{d}m, 
\\
\Lambda_2(g,h) & := \iint \widetilde{g}(\ell) \widetilde{h}(m) \, n(k,\ell,m) \, 
  \, R_k E^\ast_\ell \mu(k,\ell,m) \, \mathrm{d}\ell \mathrm{d}m, 
\\
\Lambda_3(g,h) & := \iint \widetilde{g}(\ell) \widetilde{h}(m) \, n(k,\ell,m) \, 
  \, R_k E^\ast_k \mu(k,\ell,m) \, \mathrm{d}\ell \mathrm{d}m,
\end{split}
\end{align}
corresponding to the three different bilinear operators in \eqref{GHW11}, where $n$ denotes a generic symbol. 
%satisfying \eqref{proGHW1}.
Theorem 3.13 of \cite{GHW}, gives H\"older estimates, with small losses in the Lebesgue exponents as in \eqref{proGHWest},
on the operators \eqref{GHW12}; more precisely, 
under the assumption \eqref{proGHW1} on the symbol $n$, one has, for $j=1,2,3$,
\begin{align}
{\| \wt{\mathcal{F}}^{-1} \Lambda_j(g,h) \|}_{L^r} \lesssim {\|g\|}_{L^p} {\|h\|}_{L^q}
  + {\|g\|}_{L^{p'}} {\|h\|}_{L^{q'}}, \qquad \frac{1}{p} + \frac{1}{q}=\frac{1}{r} < \frac{1}{p'} + \frac{1}{q'}.
\end{align}
This estimate and \eqref{GHW11} give the claimed identity \eqref{proGHW2}.

By iterating the application of \eqref{GHW10} to the expressions in \eqref{GHW11},  we can derive
the second identity \eqref{proGHW2'}.
%The final weighted estimate is obtained by combining...
\end{proof}

\smallskip
\begin{proof}[Proof of Proposition \ref{proHF}]
By symmetry we may assume $|m| \geq |\ell|$ on the support of \eqref{HF1}.

%\begin{align}
%\begin{split}
%\mathcal{D}_{HF,1}(t)(f,f) := 
%  \int_0^t \iint e^{is (-|k|^2 + |\ell|^2 + |m|^2 )} \widetilde{f}(s,\ell) \widetilde{f}(s,m)
%  \, \mu(k,\ell,m) \, \\ \times \varphi_{\geq -10}((|\ell|+|m|)\js^{-\delta_N})\, \mathrm{d}\ell \mathrm{d}m\,\mathrm{d}s,
%\\
%\mathcal{D}_{HF,2}(t)(f,f) := 
%  \int_0^t \iint e^{is (-|k|^2 + |\ell|^2 + |m|^2 )} \widetilde{f}(s,\ell) \widetilde{f}(s,m)
%  \, \mu(k,\ell,m) \, \\ \times 
%  \varphi_{\geq 0}((|k|)\js^{-\delta_N}) \varphi_{< -10}((|\ell|+|m|)\js^{-\delta_N})\, \mathrm{d}\ell \mathrm{d}m\,\mathrm{d}s,
%\end{split}
%\end{align}

\medskip
{\it Case $|k| \lesssim |m|$.}
In this case we have $|m| \approx |m|+|\ell|+|k| \gtrsim \js^{\delta_N}$ on the support of $\mathcal{D}_{HF}$.
We want to apply the identity \eqref{proGHW2} to \eqref{HF1}.
To shorten our formulas we will often omit the argument of the cutoff $\varphi_{\geq 0}$ or the measure $\mu$,
or other arguments when there is no risk of confusion.
%Also, we will slightly abuse notation by denoting the operators $\mathcal{B}_n^{(i)}$, $i=1,2$ in \eqref{proGHW2}
%using just $\mu$ instead of $\mu^{(i)}$. Since all these operators satisfy the same H\"older bounds this is harmless.
Formula \eqref{proGHW2} applied to \eqref{HF1} gives
\begin{align}\label{prHF1}
\begin{split}
\partial_k \mathcal{D}_{HF} &= \mathcal{D}_{HF,1} + \mathcal{D}_{HF,2} + \mathcal{D}_{HF,3},
\\
\mathcal{D}_{HF,1}(f,f)(t) &= \int_0^t 
  \iint e^{is \Phi(k,\ell,m)} \widetilde{f}(s,\ell) \widetilde{f}(s,m)
  \,\varphi_{\geq0} \, \big( -is k \mu+ is \ell \mu' + \mu'  \big)
  \, \mathrm{d}\ell\mathrm{d}m\,\mathrm{d}s,
\\
\mathcal{D}_{HF,2}(f,f)(t) &= \int_0^t \iint e^{is \Phi(k,\ell,m)} 
  \partial_\ell \widetilde{f}(s,\ell) \widetilde{f}(s,m)
  \, \varphi_{\geq0} \, \mu' \, \mathrm{d}\ell\mathrm{d}m\,\mathrm{d}s,
\\
\mathcal{D}_{HF,3}(f,f)(t) &= \int_0^t \iint e^{is \Phi(k,\ell,m)} 
  \widetilde{f}(s,\ell) \widetilde{f}(s,m)
  \, \big( \partial_k \varphi_{\geq0} \, \mu + \partial_\ell \varphi_{\geq0} \, \mu' \big)
  \, \mathrm{d}\ell\mathrm{d}m\,\mathrm{d}s,
\end{split}
\end{align}
where $\mu'$ is as in the statement of Proposition \ref{proGHW}(ii).
Notice that thanks to the two identities in \eqref{proGHW2} the above expressions
do not involve $\partial_m \wt{f}$, which is the function with largest frequency. 
%and cannot afford losses of derivatives. %despite our assumption $|m|\geq |\ell|$.

For the first term in \eqref{prHF1} it suffices to show how to estimate the contribution involving $\mu$,
since the other two contributions involving $\mu^\prime$ are similar. 
We rewrite it as
\begin{align}\label{prHF1'}
\begin{split}
\mathcal{D}_{HF,1}' & := \int_0^t is \, e^{is|k|^2} 
  \iint  \langle\ell\rangle^{-2}\widetilde{u}(s,\ell) \, e^{is|m|^2} |m|^4\widetilde{f}(s,m)
  \, n_1(k,\ell,m)\, \mu(k,\ell,m)\, \mathrm{d}\ell\mathrm{d}m\,\mathrm{d}s,
  \\
  n_1(k,\ell,m) &:= \frac{-k\langle\ell\rangle^2}{|m|^4}\varphi_{\geq0}((|\ell|^2+|m|^2+|k|^2)\js^{-2\delta_N}) \,.
\end{split}
\end{align}
It is easy to check that since $|m|\gtrsim \max(\js^{\delta_N},|\ell|,|k|)$, we have, see \eqref{CMnorm},
\begin{align*}
{\| n_1 \|}_{CM_\delta} \lesssim 1.
\end{align*}
The assumptions \eqref{proGHW1} is then verified with $A \lesssim 1$.
For small $\gamma$ we let 
\begin{align*}
2+:= (1/2 - \gamma)^{-1}, \quad M:= \gamma^{-1}, %\quad M-:= M-\gamma,
\end{align*}
and apply \eqref{proGHWest'} to estimate
\begin{align*}
{\| \mathcal{D}_{HF,1}'(f,f)(t) \|}_{L^2} \lesssim \int_0^t s & \cdot
  {\| (-\Delta+V + 1)^{-1} u(s) \|}_{L^M \cap L^{M-\gamma}}
  \\ & \cdot %{\| |m|^4 \varphi_{\geq 0}(|m|\js^{-\delta_N}) \widetilde{f}(s,m) \|}_{L^{2}} 
  {\| \wt{\mathcal{F}}^{-1} (|m|^4 \varphi_{\geq -10}(|m|\js^{-\delta_N}) \wt{u}(s) ) \|}_{L^{2+}} \, \mathrm{d}s.
\end{align*}
%where `$\infty-$' denotes an arbitrarily large, but finite, exponent.
%where we denoted $\wt{P}$ the distortef Littlewood-Paley projection.
Using Sobolev's embeddings, and the apriori Sobolev bound \eqref{apriori} we deduce
\begin{align*}
& {\| (-\Delta+V + 1)^{-1} u(s) \|}_{L^M \cap L^{M-\gamma}} \lesssim {\| u(s) \|}_{L^2} \lesssim \e,
\end{align*}
and, using also the boundedness of $\wt{\mathcal{F}}^{-1} \what{\mathcal{F}}$ and see \eqref{d_N},
\begin{align*}
{\| \wt{\mathcal{F}}^{-1} (|m|^4 \varphi_{\geq 0}(|m|\js^{-\delta_N}) \wt{u}(s) ) \|}_{L^{2+}}
  & \lesssim {\| |m|^5 \varphi_{\geq 0}(|m|\js^{-\delta_N}) \widetilde{f}(s) \|}_{L^2} 
  \\ & \lesssim {\| |m|^N \widetilde{f}(s) \|}_{L^2} \js^{-\delta_N (N-5)}
  \lesssim \e \js^{-3}.
\end{align*}
It follows that
\begin{align*}
{\| \mathcal{D}_{HF,1}(f,f)(t) \|}_{L^2} \lesssim \int_0^t s \cdot \e \cdot \e \js^{-3} \, \mathrm{d}s
  \lesssim \e^2.
\end{align*}
The terms $ \mathcal{D}_{HF,2}$ and $\mathcal{D}_{HF,3}$ are easier and can be dealt with similarly,
using also the a priori bound \eqref{apriori} for $\partial_\ell\wt{f}$.

To estimate $\partial_k^2 \mathcal{D}_{HF}$ 
%we need to apply $\partial_k$ to the three expression in \eqref{prHF1}. 
we use the identity \eqref{proGHW2'}.
This leads to many terms, but since all of them can be treated similarly, we only give details for some.
Among the many terms generated by applying \eqref{proGHW2'} the main ones are
\begin{align}\label{prHF20}
\begin{split}
\mathcal{D}_{HF,4}(f,f)(t) &= \int_0^t
  \iint s^2 \ell^2 \, e^{is \Phi(k,\ell,m)} \widetilde{f}(s,\ell) \widetilde{f}(s,m)
  \,\varphi_{\geq0} \, \mu'(k,\ell,m) \, \mathrm{d}\ell\mathrm{d}m\,\mathrm{d}s,
\\
\mathcal{D}_{HF,5}(f,f)(t) &= \int_0^t \iint s\ell \, e^{is \Phi(k,\ell,m)}
  \partial_\ell \widetilde{f}(s,\ell) \widetilde{f}(s,m)
  \, \varphi_{\geq0} \,  \mu'(k,\ell,m) \, \mathrm{d}\ell\mathrm{d}m\,\mathrm{d}s,
\\
\mathcal{D}_{HF,6}(f,f)(t) &= \int_0^t 
  \iint e^{is \Phi(k,\ell,m)} \partial_\ell^2 \widetilde{f}(s,\ell) \widetilde{f}(s,m)
  \,\varphi_{\geq0} \,  \mu'(k,\ell,m) \, \mathrm{d}\ell\mathrm{d}m\,\mathrm{d}s.
\end{split}
\end{align}
Similarly to \eqref{prHF1'} we rewrite
\begin{align*}
\mathcal{D}_{HF,4} & := \int_0^t s^2 \, e^{is|k|^2} 
  \iint  \langle\ell\rangle^{-2}\widetilde{u}(s,\ell) \, |m|^5\widetilde{u}(s,m)
  \, n_4(k,\ell,m)\, \mu'(k,\ell,m)\, \mathrm{d}\ell\mathrm{d}m\,\mathrm{d}s,
  \\
  n_4(k,\ell,m) &:= \frac{-\ell^2 \langle\ell\rangle^2}{|m|^5}
  	\varphi_{\geq0}((|\ell|^2+|m|^2+|k|^2)\js^{-2\delta_N}),
\end{align*}
and, using ${\| n_4 \|}_{CM_\delta} \lesssim 1$ and the same notation above for the indexes, we can estimate
\begin{align*}
{\| \mathcal{D}_{HF,4}(f,f)(t) \|}_{L^2} \lesssim \int_0^t s^2 & \cdot
  {\| (-\Delta+V + 1)^{-1} u(s) \|}_{L^M \cap L^{M-}}
  \\ & \cdot
  {\| \wt{\mathcal{F}}^{-1} (|m|^5 \varphi_{\geq 0}(|m|\js^{-\delta_N}) \wt{u}(s) ) \|}_{L^{2+}} \, \mathrm{d}s
\\
& \lesssim \int_0^t s^2 \cdot \e \cdot \e \js^{-\delta_N (N-6)} \, \mathrm{d}s
  \lesssim \e^2 \js^{1/2}.
\end{align*}
Similarly, using the apriori bounds \eqref{apriori}
\begin{align*}
{\| \mathcal{D}_{HF,5}(f,f)(t) \|}_{L^2} \lesssim \int_0^t s & \cdot
  {\| (-\Delta+V + 1)^{-1} \wt{\mathcal{F}}^{-1} e^{it|\ell|^2} \partial_\ell \wt{f}(s) \|}_{L^M \cap L^{M-}}
  \\ &
  \cdot {\| \wt{\mathcal{F}}^{-1} (|m|^4 \varphi_{\geq 0}(|m|\js^{-\delta_N}) \wt{u}(s) ) \|}_{L^{2+}} \, \mathrm{d}s
\\
\lesssim \int_0^t s & \cdot {\| \partial_\ell \wt{f}(s) \|}_{L^2}
  \cdot \js^{-\delta_N (N-5)} {\| |m|^N \wt{u}(s) ) \|}_{L^2} \, \mathrm{d}s
\\
\lesssim \int_0^t & s \cdot \e \cdot \e \js^{-3} \, \mathrm{d}s \lesssim \e^2,
\end{align*}
and
\begin{align*}
{\| \mathcal{D}_{HF,6}(f,f)(t) \|}_{L^2} \lesssim \int_0^t &
  {\| (-\Delta+V + 1)^{-1} \wt{\mathcal{F}}^{-1} e^{it|\ell|^2} \partial_\ell^2 \wt{f}(s) \|}_{L^M \cap L^{M-}}
  \\ &
  \cdot {\| \wt{\mathcal{F}}^{-1} (|m|^3 \varphi_{\geq 0}(|m|\js^{-\delta_N}) \wt{u}(s) ) \|}_{L^{2+}} \, \mathrm{d}s
\\
\lesssim \int_0^t & {\| \partial_\ell^2 \wt{f}(s) \|}_{L^2}
  \cdot \js^{-5} {\| |m|^N \wt{u}(s) ) \|}_{L^2} \, \mathrm{d}s
\\ \lesssim \int_0^t & \e \js^{1/2+\delta} \cdot \e \js^{-3} \, \mathrm{d}s
  \lesssim \e^2.
\end{align*}

\medskip
{\it Case $|k| \gg |m|$.}
In this case $|k| \gg |m| + |\ell|$ and $|k| \gtrsim \js^{\delta_N}$ on the support of $\mathcal{D}_{HF}$.
Our strategy is to first integrate by parts in $s$ and then estimate the weighted norms of the resulting expression.
More precisely, we use $e^{is\Phi} = (i\Phi)^{-1}\partial_s e^{is\Phi}$ and write
\begin{align}\label{prHF2}
\begin{split}
\mathcal{D}_{HF}(t)(f,f) = A_{HF}(t)(f,f) - A_{HF}(t)(f,f)(0) - \int_0^t A_{HF}(s)(\partial_s f,f)\,\mathrm{d}s
  \\ - \int_0^t A_{HF}(s)(f,\partial_s f)\,\mathrm{d}s + \mbox{easier terms}
\end{split}
\end{align}
where
\begin{align}
A_{HF}(s)(g,h) :=  \iint e^{is\Phi(k,\ell,m)} \widetilde{g}(s,\ell) \widetilde{h}(s,m)
  \frac{1}{i\Phi(k,\ell,m)} \, \mu(k,\ell,m) \, \varphi_{\geq0} \, \mathrm{d}\ell\mathrm{d}m
\end{align}
and the ``easier terms'' are those where $\partial_s$ hits the cutoff which gives 
$\partial_s \varphi_{\geq0}= s^{-1} \varphi_{\sim0}$ and an easier term to treat.

We apply $\partial_k$ and $\partial_k^2$ to the terms in \eqref{prHF2}
using the identities \eqref{proGHW2}-\eqref{proGHW2'} %and estimate the resulting expressions in $L^2$.
obtaining many different contributions. In the case of one derivative, we have (omitting irrelevant constants)
\begin{align}\label{prHF3}
\begin{split}
& \partial_k \mathcal{D}_{HF}(f,f) = A_1 + A_2 + A_3 + \mbox{easier terms}
\\
& A_1 := t k \iint e^{it\Phi(k,\ell,m)} \widetilde{f}(t,\ell) \widetilde{f}(t,m)
  \frac{1}{\Phi(k,\ell,m)} \, \mu(k,\ell,m) \, \varphi_{\geq0} \, \mathrm{d}\ell\mathrm{d}m,
\\
& A_2 := \int_0^t s k \iint e^{is\Phi(k,\ell,m)} \partial_s \widetilde{f}(s,\ell) \widetilde{f}(s,m)
  \frac{1}{\Phi(k,\ell,m)} \, \mu(k,\ell,m) \, \varphi_{\geq0} \, \mathrm{d}\ell\mathrm{d}m \,\mathrm{d}s,
\\
& A_3 := \int_0^t \iint e^{is\Phi(k,\ell,m)} \partial_s \widetilde{f}(s,\ell) \, \partial_m \widetilde{f}(s,m)
  \frac{1}{\Phi(k,\ell,m)} \, \mu(k,\ell,m) \, \varphi_{\geq0} \, \mathrm{d}\ell\mathrm{d}m \,\mathrm{d}s.
\end{split}
\end{align}
To estimate these terms we can use that ${\| k/\Phi \|}_{CM_\delta} \lesssim 1$, and
%$\wt{\mathcal{F}}^{-1} e^{it|\ell|^2} \partial_s \widetilde{f}(s)  = u^2(s)$.
\begin{align}\label{prHFdsf}
\wt{\mathcal{F}}^{-1} e^{is|k|^2} \partial_s \widetilde{f}(s)  = u^2(s).
%{\| \wt{\mathcal{F}}^{-1} e^{it|\ell|^2} \partial_s \widetilde{f}(s)  \|}_{L^p}
%	= {\| u^2(s)  \|}_{L^p} \lesssim \e^2  
\end{align}
Using the H\"older estimate \eqref{proGHWest'}, and the decay estimate \eqref{aprioriL<6} in Lemma \ref{LemLinear2},
we can bound
\begin{align*}
{\| A_1 \|}_{L^2} & \lesssim t {\| u(t) \|}_{L^6 \cap L^{6-}} {\| u(t) \|}_{L^3}
 \lesssim t \cdot \e\jt^{-(1-)} \cdot \e \jt^{-1/2} \lesssim \e^2.
\end{align*}
Similarly, using also \eqref{prHFdsf} above, we can estimate
\begin{align*}
{\| A_2 \|}_{L^2} & \lesssim \int_0^t s \cdot {\| u^2(s) \|}_{L^3} {\| u(s) \|}_{L^6\cap L^{6-}} \, \mathrm{d}s
 \lesssim \int_0^t s \cdot \e^2 \js^{-2} \cdot \e \js^{-(1-)} \, \mathrm{d}s \lesssim \e^2,
\end{align*}
and
\begin{align*}
{\| A_3 \|}_{L^2} & \lesssim \int_0^t s \cdot {\| u^2(s) \|}_{L^3\cap L^{3-}} 
	{\| \wt{\mathcal{F}}^{-1} e^{is|m|^2} \partial_m \wt{f}(s) \|}_{L^6} \, \mathrm{d}s
\\
& \lesssim \int_0^t s \cdot \e^2 \js^{-(2-)} \cdot \e s^{-1} \js^{1/2+\delta} \, \mathrm{d}s
\lesssim \e^2
\end{align*}
having used also \eqref{LemLinearL6} for the second inequality.

When applying $\partial_k^2$ we again obtain several terms.
Omitting irrelevant constants, we can write schematically
\begin{align}\label{prHF4}
\begin{split}
& \partial_k^2 \mathcal{D}_{HF}(f,f) = B_1 + B_2 + B_3 + \mbox{easier terms}
\\
& B_1 = t^2 \iint e^{it\Phi(k,\ell,m)} \widetilde{f}(t,\ell) \widetilde{f}(t,m)
  \frac{k^2}{\Phi(k,\ell,m)} \, \mu(k,\ell,m) \, \varphi_{\geq0} \, \mathrm{d}\ell\mathrm{d}m,
\\
& B_2 = t \iint e^{it\Phi(k,\ell,m)} \widetilde{f}(t,\ell) \, \partial_m \widetilde{f}(t,m)
  \frac{k}{\Phi(k,\ell,m)} \, \mu(k,\ell,m) \, \varphi_{\geq0} \, \mathrm{d}\ell\mathrm{d}m,
\\
& B_3 = \int_0^t s^2 \iint e^{is\Phi(k,\ell,m)} \partial_s \widetilde{f}(s,\ell) \widetilde{f}(s,m)
  \frac{k^2}{\Phi(k,\ell,m)} \, \mu(k,\ell,m) \, \varphi_{\geq0} \, \mathrm{d}\ell\mathrm{d}m \,\mathrm{d}s,
\\
& B_4 = \int_0^t s \iint e^{is\Phi(k,\ell,m)} \partial_s \widetilde{f}(s,\ell) \partial_m \widetilde{f}(s,m)
  \frac{k}{\Phi(k,\ell,m)} \, \mu(k,\ell,m) \, \varphi_{\geq0} \, \mathrm{d}\ell\mathrm{d}m \,\mathrm{d}s.
\end{split}
\end{align}
The ``easier terms'' contain those terms where two derivatives fall on the profile $f$, which can be treated by 
an H\"older type inequality using the integrable-in-time $L^p$ decay, $p>6$, %for $p$ sufficiently large,
and other terms where derivatives hit the cutoff. %There is no $\partial_s \partial_m$

In $B_1$ we write $k^2/\Phi(k,\ell,m) = -1 + (|\ell|^2+|m|^2)/\Phi(k,\ell,m)$
The contribution corresponding to the symbol $-1$ is estimated using \eqref{proGHWest'}:
\begin{align*}%\label{prHF31}
%& {\|B_1\|}_{L^2} \lesssim 
t^2 {\| u(t) \|}_{L^6 \cap L^{6-}} {\| u(t) \|}_{L^3}
	\lesssim t^2 \cdot \e\jt^{-(1-)} \cdot \e \jt^{-1/2} \lesssim \e^2 \jt^{1/2+},
\end{align*}
where $1/2+$ denotes here a number larger, but arbitrarily close to, $1/2$.
This is consistent with \eqref{proHFconc}.
For the other contribution we note that 
%$(|\ell|^2+|m|^2)/\Phi(k,\ell,m) = (|\ell|^2+|m|^2)^{\beta} n(k,\ell,m)$ for small $\beta$,
${\| 1/\Phi(k,\ell,m)  \|}_{CM_\beta} \lesssim 1$, $\beta>0$, and estimate
the bilinear term by %(using the symmetry in $\ell$ and $m$)
\begin{align}\label{prHF5}
%& {\|B_1\|}_{L^2} \lesssim 
t^2 {\| u(t) \|}_{L^6 \cap L^{6-}} {\| \Delta u(t) \|}_{L^3}
 \lesssim t^2 \cdot \e\jt^{-(1-)} \cdot \e |t|^{-1/2} \jt^{\rho} \lesssim \e^2 \jt^{1/2+\rho+}
\end{align}
having used 
\begin{align*}
{\| \Delta u(t) \|}_{L^3} & \lesssim {\| P_{\leq K_0}\Delta u(t) \|}_{L^3} + {\| P_{> K_0}\Delta u(t) \|}_{L^3}
\\ & \lesssim 2^{2K_0} {\| u(t) \|}_{L^3} + {\| u(t) \|}_{H^N} 2^{-K_0(N-3)}
\lesssim \e \big( 2^{2K_0} |t|^{-1/2} + 2^{-K_0(N-3)} \big)
\end{align*}
with $2^{K_0} = |t|^{\rho/2}$ and $\rho= 1/(N-3)$.
Imposing  that $\rho < \delta$, the bound \eqref{prHF5} is consistent with 
the desired \eqref{proHFconc}.

For the second term in \eqref{prHF4} we use ${\| k/\Phi(k,\ell,m)  \|}_{CM_\delta} \lesssim 1$,
the decay estimate \eqref{LemLinearL6}, and the apriori bounds, to obtain
\begin{align*}
{\|B_2\|}_{L^2} & \lesssim 
t {\| u(t) \|}_{L^3 \cap L^{3-}} {\| \wt{\mathcal{F}}^{-1} e^{it|m|^2} \partial_m \wt{f}(t) \|}_{L^6}
\\ & \lesssim t \cdot \e\jt^{-(1/2-)} \cdot |t|^{-1} {\| \partial_m^2 \wt{f}(t) \|}_{L^2}
\\ & \lesssim t \cdot \e\jt^{-(1/2-)} \cdot \e |t|^{-1} \jt^{1/2+\delta} \lesssim \e^2 \jt^{2\delta}.
\end{align*}

The term $B_3$ is similar to the term $B_1$ as it contains a time integration but one profile is differentiated in time.
One can then proceed similarly with an $L^3\times L^{6-}$ estimate using, see \eqref{prHFdsf}, 
\begin{align*}
{\| \wt{\mathcal{F}}^{-1} e^{it|k|^2} \partial_s \widetilde{f}(s) \|}_{L^3}
\lesssim {\| u(s) \|}_{L^6}^2 \lesssim \e^2 \js^{-2}.
\end{align*}
%the contribution with symbol  $(|\ell|^2+|m|^2)\Phi$. %${\| 1/\Phi(k,\ell,m)  \|}_{CM_\delta} \lesssim 1$ 
%\begin{align*}
%{\| B_3 \|}_{L^2} & \lesssim \int_0^t s^2 \cdot {\| u^2(s) \|}_{L^3\cap L^{3-}} {\| u(s) \|}_{L^6} \, \mathrm{d}s
%\end{align*}
The term $B_4$ is easier and can be treated similarly to the previous ones by an 
$L^3\cap L^{3-} \times L^6$ estimate, so we skip it.
\end{proof}

%%%%%%%%%%%%%%%%%%%%%%%%%%%%%%%%%%%%%%%%%%%%%%%%%%
%%%%%%%%%%%%%%%%%%%%%%%%%%%%%%%%%%%%%%%%%%%%%%%%%%
%%%%%%%%%%%%%%%%%%%%%%%%%%%%%%%%%%%%%%%%%%%%%%%%%%
%%%%%%%%%%%%%%%%%%%%%%%%%%%%%%%%%%%%%%%%%%%%%%%%%%
%%%%%%%%%%%%%%%%%%%%%%%%%%%%%%%%%%%%%%%%%%%%%%%%%%
%%%%%%%%%%%%%%%%%%%%%%%%%%%%%%%%%%%%%%%%%%%%%%%%%%
%%%%%%%%%%%%%%%%%%%%%%%%%%%%%%%%%%%%%%%%%%%%%%%%%%
%%%%%%%%%%%%%%%%%%%%%%%%%%%%%%%%%%%%%%%%%%%%%%%%%%

\medskip
\section{Analysis of the NSD I: structure of the leading order}\label{secNSM}
%Nonlinear spectral theory

\subsection{Expansion of $\mu$ and leading order nonlinear terms}
We expand the integrand in \eqref{mu0} in negative powers of $|x|$ according to the 
relation between $\psi$ and $\psi_1$ in \eqref{psipsi1}, and write
\begin{align}\label{mudec}
\begin{split}
{(2\pi)}^{9/2}\mu(k,\ell,m) & = {(2\pi)}^{3/2}\delta_0(k-\ell-m) 
\\ &  -\frac{1}{4\pi} \mu_1(k,\ell,m) + \frac{1}{{(4\pi)}^2} \mu_2(k,\ell,m) - \frac{1}{{(4\pi)}^3}\mu_3(k,\ell,m),
\end{split}
\end{align}
where the integrand in $\mu_\rho$ is $O(|x|^{-\rho})$; more precisely

\medskip
\begin{itemize}

\item The distribution $\mu_1$ is given by
\begin{align}
\label{mu1}
\mu_1(k,\ell,m) & = \nu_1(-k+\ell,m) + \nu_1(-k+m,\ell) + \bar{\nu_1(-\ell-m,k)}
\end{align}
where
\begin{align}
\label{nu_1}
\nu_1(p,q) & := \int e^{ix \cdot p} \, \frac{e^{i|q||x|}}{|x|} \psi_1(x,q)\, \mathrm{d}x;
\end{align}

\item The measure $\mu_2$ is given by
\begin{align}
\label{mu2}
\mu_2(k,\ell,m) & = \nu_2^1(k,\ell,m) + \nu_2^2(k,\ell,m) + \nu_2^2(k,m,\ell)
\end{align}
where
\begin{align}
\label{nu_2}
\begin{split}
\nu_2^1(k,\ell,m) & := \int e^{-ix \cdot k} \frac{e^{i(|\ell|+|m|)|x|}}{|x|^2} \psi_1(x,\ell) \psi_1(x,m)\, \mathrm{d}x,
\\
\nu_2^2(k,a,b) & := \int e^{ix\cdot a} \frac{e^{i|x|(-|k|+|b|)}}{|x|^2} \bar{\psi_1(x,k)} \psi_1(x,b) \, \mathrm{d}x;
\end{split}
\end{align}

\item  The measure $\mu_3$ is given by
\begin{align}
\label{mu3}
\begin{split}
\mu_3(k,\ell,m) = \int \frac{e^{i(-|k|+|\ell|+|m|)|x|}}{|x|^3} \bar{\psi_1(x,k)} \psi_1(x,\ell) \psi_1(x,m) \, \mathrm{d}x.
\end{split}
\end{align}

\end{itemize}

We are interested in the regularity of these distributions, and therefore
are mostly concerned with the behavior for large $|x|$ of the integrands.

\medskip
According to \eqref{mudec} we decompose the nonlinear interaction in Duhamel's formula \eqref{Duhamel0} as
\begin{align}
\label{Duhameldec}
\begin{split}
{(2\pi)}^{9/2}\mathcal{D}(t)(f,f) &= {(2\pi)}^{3/2}\mathcal{D}_0(t)(f,f)
  \\ & - \frac{1}{4\pi}\mathcal{D}_1(t)(f,f) + \frac{1}{(4\pi)^2}\mathcal{D}_{2}(t)(f,f) 
  - \frac{1}{(4\pi)^3}\mathcal{D}_{3}(t)(f,f),
\\
\mathcal{D}_0(t)(f,f) & := \int_0^t \iint e^{is (-|k|^2 + |\ell|^2 + |k-\ell|^2 )} \widetilde{f}(s,\ell) 
  \widetilde{f}(s,k-\ell) \, \mathrm{d}\ell \,\mathrm{d}s,
\\
\mathcal{D}_\ast(t)(f,f) & := \int_0^t \iint e^{is (-|k|^2 + |\ell|^2 + |m|^2 )} \widetilde{f}(s,\ell) \widetilde{f}(s,m)
  \, \mu_\ast(k,\ell,m) \, \mathrm{d}\ell \mathrm{d}m \,\mathrm{d}s.
\end{split}
\end{align}

In this section we will analyze in details $\mu_1$,
and postpone the analysis of the lower order terms \eqref{mu2}-\eqref{mu3} 
to Subsection \ref{ssecmu2} and \ref{ssecmu3}. %\ref{SecOther}.
%\medskip
%\subsection{Set up of the nonlinear analysis}
We write more explicitly the leading order terms in \eqref{Duhameldec} using \eqref{mu1}
\begin{align}
\label{D1}
\begin{split}
\mathcal{D}_1(t)(f,f) & = \int_0^t 
  \iint e^{is (-|k|^2 + |\ell|^2 + |m|^2 )} \widetilde{f}(s,\ell) \widetilde{f}(s,m)
  \, \nu_1(-k+m,\ell) \, \mathrm{d}\ell \mathrm{d}m\,\mathrm{d}s
\\
& + \int_0^t \iint e^{is (-|k|^2 + |\ell|^2 + |m|^2 )} \widetilde{f}(s,\ell) \widetilde{f}(s,m)
  \, \nu_1(-k+\ell,m) \, \mathrm{d}\ell \mathrm{d}m\,\mathrm{d}s
\\
& +
\int_0^t \iint e^{is (-|k|^2 + |\ell|^2 + |m|^2 )} \widetilde{f}(s,\ell) \widetilde{f}(s,m)
  \, \bar{\nu_1(-\ell-m,k)} \, \mathrm{d}\ell \mathrm{d}m\,\mathrm{d}s
\end{split}
\end{align}
which, using the symmetry in $\ell$ and $m$ and changing variables, we may rewrite as 
\begin{align}
\label{D12}
\begin{split}
\mathcal{D}_1(t)(f,f) & = 
  2 \int_0^t \iint e^{is (|\ell|^2 + 2 k\cdot \ell + |m|^2 )} \widetilde{f}(s,\ell+k) \widetilde{f}(s,m)
  \, \nu_1(\ell,m) \, \mathrm{d}\ell \mathrm{d}m \,\mathrm{d}s
\\
& + \int_0^t \iint e^{is (-|k|^2 + |\ell|^2 + 2\ell \cdot m + 2|m|^2 )} 
  \widetilde{f}(s,-\ell-m) \widetilde{f}(s,m)
  \, \bar{\nu_1(\ell,k)}  \, \mathrm{d}\ell \mathrm{d}m \,\mathrm{d}s.
\end{split}
\end{align}
Note that the two integrals in \eqref{D12} are somewhat similar but have slightly different structure as, for example,
the measure in the second one is $k$ dependent.
This will require a slightly different treatment in some of the estimates that will follow in Sections \ref{secBE}
and \ref{secdkL2}.

%We will therefore concentrate on the first one and dedicate most of Section \ref{secdkL2} to estimating it.
%In Subsection \ref{ssecdkN_2} we discuss how to deal with the second integral in \eqref{D12}.

%Section \ref{SecOther} contains the estimates for the other terms in \eqref{Duhameldec}.

%By slightly abusing notation we still denote by
%\begin{align}\label{D13}
%\begin{split}
%\mathcal{D}_1(t)(f,f) = 
%  \int_0^t \iint e^{is (|\ell|^2 + 2 k\cdot \ell + |m|^2 )} \widetilde{f}(s,\ell+k) \widetilde{f}(s,m)
%  \, \nu_1(\ell,m) %\, \chi_{LF}(k,\ell,m) 
%  \,\mathrm{d}\ell \mathrm{d}m \,\mathrm{d}s,
%  %\\ \chi_{LF}(k,\ell,m) = \varphi_{<0}((|k|+|\ell|+|m|)\js^{\delta_N}).
%\end{split}
%\end{align}
%the bilinear term that needs to be estimated, under the assumption \eqref{HFrestr}.

\medskip
\subsection{Structure and properties of $\nu_1$}\label{ssecML}
Motivated by \eqref{D12} we need to study $\nu_1$ as in \eqref{nu_1}.
Our main aim in this section is to prove that a precise version of the following 
approximate identity:
\begin{align*}
\nu_1(p,q) = \frac{1}{|p|} \Big[  \delta(|p|-|q|) + \pv \frac{1}{|p|-|q|} \Big] %\cdot \big( 1 + O(2^{-A}) \big) 
  m_0(p,q) + R(p,q)
\end{align*}
where $m_0$ is a ``nice'' symbol of Coifman-Meyer type up to some losses,
%\begin{align}
%\Big| \partial_p^{\alpha} \partial_q^{\beta} \big( \nu_0(p,q) \varphi_{P}(p) \varphi_Q(q) \big) \Big| \lesssim
%	2^{-|\alpha|P} 2^{-|\beta|Q} \cdot 2^{A(|\alpha| + |\beta|)}.
%\end{align}
%for $|\alpha|+|\beta|$ large enough,
and $R$ is a better behaved remainder. %in a sense to be specified below.
%Recall the definition of $\psi_1$ in \eqref{psi11}, 
%and that $\psi_1 \in \mathcal{G}^{N_1}$, see Definition \ref{Gclass}.
%\medskip
%\subsection{Classes of measures}
%Here we define more precisely the building blocks of the measure $\mu_\rho$.
%
%For any $g \in \mathcal{G}$, see \eqref{classG}, we let $\mathcal{V}_1$ be the class of distributions of the form
%\begin{align}
%\int e^{ix \cdot p} \, \frac{e^{i|q||x|}}{|x|} g(x,q) \, \mathrm{d}x,
%\end{align}
%Then $\nu_1$ in \eqref{mu1} belongs to $\mathcal{V}_1$, and so does $\mu_1$.
%We will show asymptotics for such expressions and study their mapping properties $L^p\times L^q \rightarrow L^2$.
%Also, we need information about the derivatives of $\mu_1$ relative to $k$. 
%Schematically, the main property that we are going to show is
%\begin{align}
%\label{V_1dk}
%\partial_k \mathcal{V}_1 = \partial_\ell \mathcal{V}_1 + \partial_m \mathcal{V}_1 + \mathcal{V}_1 
%\end{align}
%in the sense that for any $\nu \in \mathcal{V}_1$, we have $\partial_k \nu_1$ is a suitable linear combination
%of elements belonging to the spaces in the right-hand of \eqref{V_1dk}.
%
%Similarly we let $\mathcal{V}_2$ be the class of distributions of the form
%\begin{align}
%\int e^{ix \cdot p} \, \frac{e^{i(|a| \pm |b|)|x|}}{|x|^2} g_1(x,a)g_2(x,b) \, \mathrm{d}x,
%\end{align}
%for $g_1,g_2 \in \mathcal{G}$.
%Then, the distributions $\nu_2^1, \nu_2^2$ in \eqref{mu_2} belong to $\mathcal{V}_2$, and so does $\mu_2$.
%Properties of $\partial_k \mathcal{V}_2$ will be similar to \eqref{V_1dk}.
This is the content of the following main proposition:

\medskip
\begin{proposition}[Structure of $\nu_1$]\label{Propnu+}
Let $\nu_1$ be the distribution defined in \eqref{nu_1}, with %$\psi_1 \in \mathcal{G}^{N_1}$, see Definition \ref{Gclass}. 
$\psi_1$ defined by \eqref{psipsi1}. 
Fix $N_2 \in [5, N_1/4] \cap \Z$. %$N_2 \in (5,\infty)\cap\Z$.
Let $p,q\in\R^3$ with $|p|\approx 2^P, |q|\approx 2^Q$, and assume that $P,Q \leq A$ for some $A>0$.
Then we can write
\begin{align}
\label{Propnu+1}
\nu_1(p,q) = \nu_0(p,q) + \nu_L(p,q) + \nu_R(p,q),
\end{align}
where:

\medskip
\begin{itemize}
\item[(1)] The leading order is
\begin{align}
\label{Propnu+1.0}
\nu_0(p,q) :=  \frac{b_0(p,q)}{|p|} \Big[ i\pi \, \delta(|p|-|q|) + \pv \frac{1}{|p|-|q|} \Big]
\end{align}
with
$b_0$ satisfying the bounds
\begin{align}
\label{Propnu+1.1}
%  \big| \nabla_p^\alpha \nabla_q^\beta \big( \varphi_P(p) \varphi_Q(q) b_0(p,q) \big) \big| \lesssim 
%  2^{-|\a|P} 2^{-|\b|Q} \cdot 2^{(|\a|+|\b|)A}
\big| \varphi_P(p) \varphi_Q(q) \, \nabla_p^\alpha \nabla_q^\beta b_0(p,q) \big| 
  \lesssim 2^{-|\a|P} \big(2^{|\a|Q} + 2^{(1-|\b|)Q_-}\big) \cdot \mathbf{1}_{\{|P-Q|< 5\}},
\end{align}
for all $P,Q \leq A$, $|\a|+|\b| \leq N_1$.
Recall our notation $Q_- = \min(Q,0)$.

\medskip
\item[(2)] 
%With 
%\begin{align}\label{NSM2.0}
%\mathcal{J}(A,P,Q) := \mathcal{J}:=\{J \geq 4A, \, J\geq -\min(P,Q)+ 4A\},
%\end{align}
The lower order terms $\nu_L(p,q)$ can be written as %finite combination?
\begin{align}
\label{Propnu+2}
\nu_L(p,q) = \frac{1}{|p|} 
  \sum_{a=1}^{N_2} \sum_{J\in\Z} b_{a,J}(p,q) \cdot 2^J K_a\big(2^J(|p|-|q|)\big) 
\end{align}
with %$b_a \in \mathcal{G}^{N_1-a}$ 
$K_a \in \mathcal{S}$ %(Schwartz class)
and $b_{a,J}$ satisfying
\begin{align}
\label{Propnu+2.1}
\sum_{J\in\Z} 
  \big|  \varphi_P(p) \varphi_Q(q) \nabla_p^\alpha \nabla_q^\beta  b_{a,J}(p,q) \big| \lesssim 
  2^{-|\a|P} \big(2^{|\a|Q}  + 2^{(1-|\beta|)Q_-}\big) \cdot \mathbf{1}_{\{|P-Q|< 5\}},
\end{align}
for all $P,Q \leq A$, $|\a|+|\b| \leq N_2$. %N_1-N_2$.

\medskip
\item[(3)] The remainder term $\nu_R$ satisfies the estimates %(note the cutoff before the differentiation)
\begin{align}\label{Propnu+3}
\begin{split}
\big| \varphi_P(p) \varphi_Q(q) \nabla_p^\alpha \nabla_q^\beta \nu_R(p,q) \big| 
  \lesssim 2^{-2\max(P,Q)} \cdot  2^{-(|\a|+|\b|)\max(P,Q)} \cdot 2^{(|\a|+|\b|+2)5A}
\end{split}
\end{align}
for $|P-Q| < 5$, and 
\begin{align}\label{Propnu+3imp}
\begin{split}
\big| \varphi_P(p) \varphi_Q(q) \nabla_p^\alpha \nabla_q^\beta \nu_R(p,q) \big| 
  \lesssim 2^{-2\max(P,Q)} \cdot 2^{-|\a|\max(P,Q)} \max(0,2^{(1-|\b|)Q_-}) \cdot 2^{(|\a|+|\b|+2)5A}
\end{split}
\end{align}
for $|P-Q| \geq 5$,
for all $|\a|+|\b| \leq N_2/2-3$.

%The version 
%\begin{align*}%\label{Propnu+3}
%\big| \varphi_P(p) \varphi_Q(q) \nabla_p^\alpha \nabla_q^\beta \nu_R(p,q) \big| 
%  \lesssim 2^{-2\max(P,Q)} \cdot  2^{-|\a|P} \max(0,2^{-(|\b|-1)Q_-}) \cdot 2^{(|\a|+|\b|+2)5A}
%\end{align*}
%is better than \eqref{Lemma1est2}, and maybe not true in general\dots
%
%But is also seems not needed\dots

\end{itemize}

\end{proposition}

Here are some comments

\begin{itemize}

\item 
One should think of Proposition \ref{Propnu+} as the statement that 
\begin{align}
\mu(p,q) \approx \nu_1(p,q) \approx \frac{1}{|p|} \delta(|p| - |q|) + \frac{1}{|p|} \, \pv \frac{1}{|p| - |q|},
\end{align}
up to small losses.
This is clearly most relevant when $||p|-|q|| \ll |p| \approx |q|$
(note the indicator functions in \eqref{Propnu+1.1} and \eqref{Propnu+2.1})
and is essentially exact when $|p|\approx|q| \lesssim 1$. 
It is important to notice how Proposition \ref{Propnu+} 
singles out the singularity of $\nu_1$, its strength and its structure up to a sufficiently high order,
after which the measure is essentially smooth.
%When $||p|-|q|| \gtrsim |p|$ the behavior is like
%a Coifman-Meyer symbol times $|p|^{-1} \max\{|p|,|q|\}^{-1}$.

\smallskip
\item There are some losses in our estimates when frequencies are large,
see the factors of $2^A$ in \eqref{Propnu+3}-\eqref{Propnu+3imp}.
These are coming from the various expansions, such as the one in \eqref{lemmapsi1exp},
where we allow growing factors of the frequency.
%When frequencies are high (larger than $1$) $\mu$ there are some losses as one can see, for example,
In the evolution problem we will handle these by comparing the size of frequencies and time,
using the high Sobolev regularity; see Proposition \ref{proHF} where `large' frequencies are treated,
leaving us only with frequencies of size $\lesssim \js^{\delta_N}$, see \eqref{d_N}, in the integrals in \eqref{D12}.
%We then can essentially think that $2^A \approx \js^{\delta_N}$ is a small loss for $N$ large.

%we may assume that the frequencies in the integrals in \eqref{D12} satisfy
%\begin{align}\label{HFrestr}
%|\ell| + |m| + |k| \lesssim \js^{\delta_N} 
%\end{align}

%in other words, we will chose $2^A$ as a small (negligible) power of time, so that the various $2^A$
%factors appearing in \eqref{Propnu+2.1}...

\smallskip
\item In the estimates \eqref{Propnu+1.1} and \eqref{Propnu+2.1} we have $|p| \approx |q|$
so the factors on the right-hand sides could be simplified a bit.
Nevertheless, we have decided to leave the explicit dependence on $2^P$ and $2^Q$ to highlight
the different roles played by the two variables,
such as the fact that the integrand in $\nu_1$ is smooth in $p$ but not in $q$.
\end{itemize}

\medskip
Proposition \ref{Propnu+} is proven in Subsection \ref{ssecprnu+}
using as key step Lemma \ref{Lemmanu0} below, 
which gives asymptotics for a basic ``building block'' \eqref{Lemmanu0def}.

\smallskip
\begin{lemma}[Asymptotic expansion for the ``building block'']\label{Lemmanu0}
Let $p,q\in\R^3$ with $|p|\approx 2^P, |q|\approx 2^Q$ and $|P-Q|<5$. 
Assume that for some for some $A>0$ we have $P,Q \leq A$, 
and let \begin{align}\label{NSM2}
\mathcal{J}(A,P,Q) := \mathcal{J}:=\{%J \geq 4A, \, 
  J\geq -\min(P,Q,0)+ 4A\}.
\end{align}
Consider the function
\begin{align}
\label{Lemmanu0def}
\K_J(p,q) := \int_{\R^3} e^{ix \cdot p} \frac{e^{i|x||q|}}{|x|} g(\omega,q) \, \varphi(x 2^{-J}) 
  \, \mathrm{d}x, \qquad J\in \mathcal{J},
\end{align}
where $\omega:=x/|x|$, with $g \in \mathcal{G}^N$ (see Definition \ref{Gclass})
and some smooth compactly supported $\varphi$. 
%and assume that 
%\begin{align}\label{Lemmanu0ass}
%2^j \max(|p|,|q|) \gg 1.
%\end{align}
Then, for any fixed $M\in(10,N) \cap \Z$ we have the expansion
\begin{align}
\label{Lemma1conc1}
\begin{split}
\K_J(p,q) & = \frac{a_0(p,q)}{|p|} \, 2^J \chi_0(2^J(|p| - |q|)) 
  \\ & + \sum_{\ell=1}^{M-1} \frac{a_\ell(p,q)}{|p|} \cdot 2^J \chi_\ell(2^J(|p| - |q|)) \cdot C_J
  %\cdot (|p|2^J)^{-\ell/2} 
  + R_{J,M}(p,q) \cdot C_J^R,
\end{split}
\end{align}
where:

\begin{itemize}

\medskip
\item $\chi_\ell$ are Schwartz functions; 
%\begin{align}\label{Lemma1conc1.2}
%\chi_\ell := \what{r^{-\ell/2}\varphi},
%\end{align}

\medskip
\item $a_0,a_1,\dots$ are smooth functions of $(p,q)\neq (0,0)$  with
\begin{align}
\label{Lemma1conc1.1}
a_0 := 2\pi i g(-p/|p|,q), %\qquad a_\ell = c_\ell (\partial^{\ell} g)(-p/|p|,q) 
\end{align}
and satisfying
\begin{align}%\label{Lemma1est1}
\label{Lemma1conc1.2}
%a_{\ell}(p,q) \in \mathcal{G}^{N-\ell}, \qquad \ell=1,\dots,M-1;
%
|\nabla_p^\alpha \nabla_q^\beta a_{\ell}(p,q)| \lesssim %2^{-|\alpha|P} 2^{-|\beta|Q} \cdot 2^{A(|\a|+|\b|)}
  2^{-|\alpha|P}\big(2^{Q|\alpha|}  + 2^{(1-|\beta|)Q_-}\big) \cdot \mathbf{1}_{\{|P-Q|< 5\}}
\end{align} 
for $|\alpha|+|\beta| \leq N-\ell$. %, with $Q_- := \min(Q,0)$;

\medskip
\item The remainder satisfies
\begin{align}
\label{Lemma1est2}
%| \nabla_p^\alpha \nabla_q^\beta R_{j,M}(p,q) | \lesssim (1+|p|^{-2}) \cdot  (1+|q|)^{|\alpha|}
% \cdot 2^{j(|\a|+|\b|)} %\max(2^j,|q|^{-1})^{|\beta|} 
%  \frac{(1+|q|)^{M+1}}{(2^j|p|)^{M/2+1}}
|\nabla_p^\alpha \nabla_q^\beta R_{J,M}(p,q) | \lesssim 2^{-2\max(P,Q)}
	\cdot
	2^{-(|\alpha|+|\beta|)\max(P,Q)}\cdot 2^{A(|\a|+|\b|+2)}
\end{align}
for $|\alpha| + |\beta| \leq \min(N-M-1,M/4-2)$;

%Check on this\dots it seems like one needs better for \eqref{Propnu+3}}

\medskip
\item The coefficients $C_J, C_J^R \geq 0$ satisfy %Could have $C_{\ell,J}$
\begin{align}\label{Lemma1conc1.3}
\sum_{J\in \mathcal{J}} C_J + C_J^R \leq 1.
\end{align}
% to make it sum later need $\leq \min(M/2-2,N-M-1)$. 
%choose $M \leq N/2$ so just have first term

% Could also use symbol notation here...

\end{itemize}
\end{lemma}

\medskip
\begin{remark}
Notice that \eqref{Lemma1conc1.2} implies the %but more symmetric, 
symbol-type bound (with losses)
\begin{align}\label{a_lbound2}
\big| \varphi_P(p)\varphi_Q(q) \nabla_p^{\alpha} \nabla_q^{\beta} \big[ \nabla_q a_{\ell}(p,q) \big] \big|
  & \lesssim 2^{-|\alpha|P} 2^{-|\beta|Q} 2^{A(|\alpha|+|\beta|)},
\end{align}
$|\alpha|+|\beta| \leq N-\ell-1$.
This is an estimate for $\nabla_q a_\ell$ as well.
This symbol-type bounds will be more convenient to use in some cases than \eqref{Lemma1conc1.2}.
\end{remark}

%This Lemma is proven in \ref{ssecprnu0}.

%The reason for this definitions and decompositions is the following:

\iffalse 
\begin{itemize}
\item In the application to the nonlinear problem $A$ will be related to
 the time parameter $s$ via $2^A = \js^{\delta_N}$. See \eqref{HFrestr} above. 
 The relevant frequencies will then be all $\lesssim 2^A$, hence the restriction to $P,Q \leq A$.
 
 \item In the measure $\nu_S$ we have $|x| \gtrsim 2^{2A} \gtrsim 2^{A}(|p|+|q|+1)$
 and $|x| \gtrsim 2^A \min(|p|,|q|)^{-1}$. In particular 
 \begin{align}
 |x| \gtrsim 2^A \max(|p|,|q|,|p|^{-1},|q|^{-1})
 \end{align}
 which will allow us to control more easily
 various remainders from expansions in negative powers of $|x|$ %, see for example ...
 and stationary phase arguments.
 
 \item On the support of the measure $\nu - \nu_S$ we either have 
 (a) $|x| \lesssim 2^{2A}$,
 in which case the measure is essentially smooth (up to some small negligible losses in powers of $\js$),
 or (b) $|x| \lesssim 2^A \min(|p|,|q|)^{-1}$, in which case the measure essentially behaves like
 $\min(|p|,|q|)^{-2}$ (up to negligible factors of $2^A$)
 which we will see is a much less singular behavior than the one of $\nu_S$.
 
\end{itemize}
 \fi

\medskip
\subsection{Proof of Lemma \ref{Lemmanu0}}\label{ssecprnu0}
\iffalse
The following definition will help us describe the precise structure of the measures $\nu_i$.
\begin{definition}\label{DefML1}
Given a decreasing sequence $\{f_\ell\}$, we say that a function $f$ satisfies 
\begin{align}\label{DefML1.1}
f(\lambda) \sim \sum_{\ell=0}^\infty f_\ell(\lambda)
\end{align}
if, for all non-negative $N$ and $a$, one has
\begin{align}\label{DefML1.2}
\big(\frac{d}{d\lambda}\big)^a \Big[ f(\lambda) - \sum_{\ell=0}^N f_\ell(\lambda) \Big] = O(f_{N+1}(\lambda) \lambda^{-a}).
\end{align}
\end{definition}
%Too general, want to adjust after the proof is complete
\fi
Recall that by Definition \ref{Gclass} we have
\begin{align}
\label{prLemma1g1}
& {\| \partial_\omega^{\alpha} \partial_q^\beta g(\cdot,q) \|}_{L^\infty_\omega} 
	\lesssim_{|\alpha|,|\beta|} \jq^{|\alpha|} + (|q|/\jq)^{1-|\beta|} %(1+|q|)^{|\alpha|} \min(1,|q|)^{1-|\beta|}
\end{align}
for all $1 \leq |\alpha|+|\beta| \leq N$. Note that, using
the assumption $|q|\approx 2^Q \lesssim 2^A$, $A>0$ we also have the (slightly less precise) bound
\begin{align}
\label{prLemma1g1'}
& {\| \partial_\omega^{\alpha} \partial_q^\beta g(\cdot,q) \|}_{L^\infty} 
	\lesssim 2^{-Q|\beta|} 2^{A(|\alpha|+|\beta|)}.
\end{align}
We write in polar coordinates, $x=r\omega$,
\begin{align}
\label{prLemma1pol}
4\pi \K_J(p,q) = \int_0^\infty \Big( \int_{\mathbb{S}^2} e^{ir\omega \cdot p} g(\omega,q) \mathrm{d}\omega \Big) 
  \, e^{ir |q|} \varphi(r2^{-J}) \, r \mathrm{d}r.
\end{align}
%To derive asymptotics for this expression we will first apply a stationary phase argument in the angular variable.
%%Need to distinguish $|P-Q|<5$ in the proof.

\medskip
{\it Asymptotics in the angular variable}.
We look at the spherical integral in \eqref{prLemma1pol} and, using the rotation invariance, without loss of generality
we reduce matters to considering $p=|p|e_3$ and the integral
\begin{align}
\label{prLemma1I}
I(X;q) := \int_{\mathbb{S}^2} e^{iX \omega \cdot e_3} g(\omega,q) \mathrm{d}\omega, \qquad X := r|p|,
\end{align}
where $X \approx 2^J |p| \gg 1$ by assumption.
Writing in standard spherical coordinates
$\theta \in [0,2\pi],\phi\in[0,\pi]$, $\omega=(\cos\theta\sin\phi,\sin\theta\sin\phi,\cos\phi)$, we see that
\begin{align}
\label{prLemma1Isplit}
\begin{split}
I & = \int_0^{2\pi} \int_0^\pi e^{iX \cos\phi} g(\omega,q)\, \sin\phi \, \mathrm{d}\phi \mathrm{d}\theta = I_0 + II
  \\ 
I_0 & = \frac{2\pi}{iX} \big[ e^{iX} g(e_3,q) - e^{-iX} g(-e_3,q) \big]
\\
II & = \frac{1}{iX} \int_0^{2\pi} \int_0^\pi e^{iX \cos\phi} \partial_\phi g(\omega,q) \, \mathrm{d}\phi \mathrm{d}\theta.
\end{split}
\end{align}

The contribution of the leading order term $I_0$ to \eqref{prLemma1pol} is
\begin{align}
\nonumber
& \int_0^\infty  \frac{2\pi}{iX} \big[ e^{iX} g(e_3,q) - e^{-iX} g(-e_3,q) \big]
  \, e^{ir|q|} \varphi(r2^{-J}) \, r \mathrm{d}r
\\ 
\nonumber
& = \frac{2\pi}{i|p|}  g(p/|p|,q) \int_0^\infty e^{ir (|q|+|p|)} \varphi(r2^{-J}) \,\mathrm{d}r
  - \frac{2\pi}{i|p|} g(-p/|p|,q) \int_0^\infty e^{ir (|q|-|p|)} \varphi(r2^{-J}) \,\mathrm{d}r
\\ \label{prL10}
& = \frac{2\pi}{i|p|}  g(p/|p|,q) \, 2^J \check{\varphi}(2^J(|q|+|p|))
  - \frac{2\pi}{i|p|} g(-p/|p|,q) \, 2^J \check{\varphi}(2^J(|q|-|p|)).
\end{align}
The second term in \eqref{prL10} coincides with 
the first term on the right-hand side of \eqref{Lemma1conc1} with $a_0$ in \eqref{Lemma1conc1.1}.
The first term of \eqref{prL10} can instead be absorbed into the remainder $R_{J,M}$ as we explain below.
In view of \eqref{prLemma1g1}, %$\partial_p g(p/|p|,q) = \partial_p (p/|p|) \partial_\omega g$ 
we have
\begin{align}
\label{prLemma1g2}
\begin{split}
|\partial_p^{\alpha} \partial_q^\beta g(p/|p|,q) |
  \lesssim_{\alpha,\beta} |p|^{-|\alpha|} (\jq^{|\alpha|} + (|q|/\jq)^{1-|\beta|})
  \\ \lesssim 2^{-P|\alpha|} 2^{-Q|\beta|} 2^{A(|\alpha|+|\beta|)}.
\end{split}
\end{align}
Given our localization of the variables and \eqref{NSM2} we have $2^{-\min(P,Q)} \lesssim 2^{J}$, so that
%\begin{align}\label{prL10.0}
%|p|^{-1} + |q|^{-1} \lesssim 2^J (|p|+|q|) 2^{-\max(P,Q)}.
%\end{align}
%This implies, 
for arbitrarily large $\rho$,
\begin{align}
\label{prL10.1}
&  \big( 2^J (|p|+|q|) \big)^\rho \, |\partial_p^{\alpha} \partial_q^\beta 
  \, 2^J \check{\varphi}(2^J(|p|+|q|))|
  \lesssim_{\alpha,\beta} 2^{-\max(P,Q)(1+|\alpha|+|\beta|)}.
\end{align}
Together with \eqref{prLemma1g2} this gives
\begin{align}\label{prL10.2}
\begin{split}
& \Big|\partial_p^{\alpha} \partial_q^\beta \Big(\frac{1}{|p|} g(p/|p|,q) 2^J \check{\varphi}(2^J(|p|+|q|)) \Big)\Big|
  \\ & \lesssim 2^{-2\max(P,Q)} \cdot 2^{-\max(P,Q)(|\alpha|+|\beta|)} 2^{A(|\alpha|+|\beta|)} \cdot c_J
\end{split}
\end{align}
with $c_J:=2^{-(P+J)}$. %and $\chi$ a Schwartz function. 
Thus the property \eqref{Lemma1est2} with \eqref{Lemma1conc1.3} holds true for this term.

Next we analyze the contribution from the lower order term $II$ in \eqref{prLemma1Isplit}, that is,
\begin{align}\label{prL15}
\begin{split}
& \int_0^\infty II(X;q) e^{ir |q|} \varphi(r2^{-J}) \, r \mathrm{d}r
\\ 
& = \frac{1}{i|p|} \int_0^\infty \Big(\int_0^{2\pi} \int_0^\pi e^{iX \cos\phi} \partial_\phi g(\omega,q) \, \mathrm{d}\phi \mathrm{d}\theta \Big) 
  \, e^{ir |q|} \varphi(r2^{-J}) \,\mathrm{d}r.
\end{split}
\end{align}
To deal with the innermost integral we want to apply the following stationary phase 
expansion:

%We will need the following 
\medskip
{\it Claim}: Suppose $f$ is a smooth function with $f(x_0) = f^\prime(x_0) = 0$ and $f^{\prime\prime}(x_0) \neq 0$,
and let $F$ be a function supported in a sufficiently small neighborhood of $x_0$
where $f$ does not have any other critical point. Define
\begin{align}\label{ST1}
I(\lambda) := \int e^{i\lambda f(x)} F(x) \, \mathrm{d}x.
\end{align}
Then, for $\lambda \gtrsim 1$, there exist coefficients $a_\ell$
%Comment on calues of $a_\ell$ (and relation to $b_\ell$ below\dots)
such that
\begin{align}\label{ST2} 
\Big| \Big(\frac{d}{d\lambda}\Big)^r \big[ I(\lambda)
  - \lambda^{-1/2} \sum_{\ell=0}^M a_\ell \, \lambda^{-\ell/2} \big] \Big| \lesssim \lambda^{-r-(M+1)/2}.
\end{align}
The coefficients $a_\ell$ depend linearly on the first $\ell$ derivatives of $F$ at the point $x_0$;
they also depend on a lower bound for $f^\prime$ on the support of $F$, and on higher-order derivatives of $f$.
In our application, we will have that all the quantities involving the phase $f$ 
are uniformly bounded by some absolute constant.
Then, we slightly abuse notation and let $a_\ell = c_\ell \partial_x^\ell F(x_0)$, for some $c_\ell\in\R$,
while this coefficient should technically be of the form $\sum_{k=0}^\ell c_k\partial_x^kF(x_0)$.
Finally, the implicit constant in \eqref{ST2} is upperbounded by the $L^\infty$ norm of at most $M+1$ derivatives of $F$;
we will similarly abuse notation and assume the constant is $C{\| \partial_x^{M+1}F\|}_{L^\infty}$.

The above claim is a classical statement, see for example Stein's book \cite[Proposition 3, p. 334]{SteinBook1}.

%Comment: Could state the more precise version that is applied.
%
%Could also prove it more quantitatively (in appendix say) 
%showing the exact dependence of the $a_\ell$ on the derivatives of $f$, 
%which is $\partial_\phi g(\omega,q)$ in the application.

\medskip
We isolate the stationary points $x_0 = 0$ and $\pi$ in the $d\phi$ integral in \eqref{prL15}
by defining
\begin{align*}
& 2\int_0^\pi e^{iX \cos\phi} \partial_\phi g(\omega,q) \, \mathrm{d}\phi = J_+(X,\theta,q) + J_-(X,\theta,q),
\\
& J_+ := \int_{-\pi}^\pi e^{iX \cos\phi} \partial_\phi g(\omega,q) \, \varphi_1(\phi)
  \, \mathrm{d}\phi,
\\
& J_- := \int_{-\pi}^\pi e^{iX \cos\phi} \partial_\phi g(\omega,q) \, \varphi_2(\phi) \, \mathrm{d}\phi,
\end{align*}
where $0\leq \varphi_1\leq 1$ is a smooth cutoff around $\phi=0$ and $\varphi_2 = 1 -\varphi_1$.

With $\lambda = X \gg 1$, the non-degenerate phase $\cos\phi - 1$, the stationary point $x_0 = 0$,
we deduce from the claim above and \eqref{ST2} that
\begin{align}
\label{prL1asymain1}
\Big| \Big(\frac{d}{dX}\Big)^\alpha \partial_q^\beta 
  \Big[ J_+(X) - e^{iX} X^{-1/2} \sum_{\ell=0}^{M-1} b_\ell^+ X^{-\ell/2} \Big] \Big| 
  \lesssim X^{-M/2-\alpha} \sup_\phi|\partial_\phi^{M+1} \partial_q^\beta g(\omega,q)|
\end{align}
and, similarly, with the phase $\cos\phi + 1$ and the stationary point $x_0 = \pi$, we get
\begin{align}
\label{prL1asymain2}
\Big| \Big(\frac{d}{dX}\Big)^\alpha \partial_q^\beta \Big[ J_-(X) - e^{-iX} X^{-1/2} \sum_{\ell=0}^{M-1} b_\ell^- X^{-\ell/2}\Big] \Big| 
  \lesssim X^{-M/2-\alpha} \sup_\phi|\partial_\phi^{M+1} \partial_q^\beta g(\omega,q)|
\end{align}
where we have defined
\begin{align}
\label{prL13}
b_\ell^\pm(p,q) := c_\ell \partial_\phi^{\ell+1} g(\pm p/|p|,q),
\end{align}
for some absolute constants $c_\ell$.

Plugging the asymptotics \eqref{prL1asymain1}-\eqref{prL1asymain2} into \eqref{prL15} we see that
\begin{align}
\nonumber
& \int_0^\infty II(X;q) e^{ir |q|} \varphi(r2^{-J}) \, r \mathrm{d}r 
\\
\label{prL17.1}
  & = \sum_{\ell=0}^{M-1} \frac{2\pi b_\ell^-(p,q)}{i|p|^{3/2 + \ell/2}} \,
  \int_0^\infty e^{-ir|p|} \, e^{ir |q|} \, r^{-(\ell+1)/2} \varphi(r2^{-J}) \,\mathrm{d}r
\\
\label{prL17.2}
& + \sum_{\ell=0}^{M-1} \frac{2\pi b_\ell^+(p,q)}{i|p|^{3/2 + \ell/2}} \,
  \int_0^\infty e^{ir|p|} \, e^{ir |q|} \, r^{-(\ell+1)/2} \varphi(r2^{-J}) \,\mathrm{d}r
\\
& + \frac{1}{|p|}R_{J,M}^-(p,q) + \frac{1}{|p|}R_{J,M}^+(p,q)
\label{prL17.3}
\end{align}
with
\begin{align}
\label{prL18.1}
& R_{J,M}^\pm(p,q) = \int_0^\infty e^{\pm ir|p|} \, e^{ir |q|} \, \varphi(r2^{-J}) B^\pm(r|p|;q) \,\mathrm{d}r,
\end{align}
where 
\begin{align}\label{prL18.B}
B^\pm(X;q) := J_\pm(X) - e^{\pm iX} X^{-1/2} \sum_{\ell=0}^{M-1} b_\ell^\pm X^{-\ell/2} 
\end{align}
satisfies
\begin{align}\label{prL18.2}
\begin{split}
\Big| \Big(\frac{d}{dX}\Big)^\alpha \partial_q^\beta B^\pm(X;q) \Big| & \lesssim X^{-M/2-\alpha} 
  \sup_{\omega\in\mathbb{S}^2}|\partial_\phi^{M+1} \partial_q^\beta g(\omega,q)|.
\end{split}
\end{align}
Moreover, we can see that
\begin{align}\label{prL18.3}
\begin{split}
\Big| \partial_p^\alpha \partial_q^\beta B(r|p|;q) \Big| & \lesssim 
  \sup_{\alpha_1 + \alpha_2 = \alpha} 
  X^{-M/2-|\alpha_1|} r^{|\alpha_1|} \cdot |p|^{-|\alpha_2|}
  \sup_{|\alpha'| \leq |\alpha_2|, \, \omega\in\mathbb{S}^2}|\partial_\phi^{M+1} \partial_\omega^{\alpha'} 
  \partial_q^\beta g(\omega,q)|
  \\ 
  & \lesssim %\sup_{\alpha_1 + \alpha_2 = \alpha} 
  X^{-M/2} 2^{-P|\alpha|} \cdot (2^{Q(M+1+|\alpha|)} + 2^{-(|\beta|-1)Q}).
\end{split}
\end{align}
%%%%%%
%% Explain details with the more precise stationary phase
%%%%%%

\medskip
{\it Contribution of \eqref{prL17.1}}.
The sum in \eqref{Lemma1conc1}, with the claimed properties, 
arises from the sum in \eqref{prL17.1}. To see this we write it as
\begin{align}
\label{prL19}
\begin{split}
& \sum_{\ell=1}^{M} \frac{b_{\ell-1}^-(p,q)}{i|p|^{1 + \ell/2}} \,
  \int_0^\infty e^{-ir|p|} \, e^{ir |q|} \, r^{-\ell/2} \varphi_J(r) \,\mathrm{d}r 
  \\ & = \sum_{\ell=1}^{M} \frac{b_{\ell-1}^-(p,q)}{i|p|}
  \, (2^{J} |p|)^{-\ell/2} \cdot 2^J \what{(r^{-\ell/2} \varphi)}(2^J(|p|-|q|))
  \\ & = \sum_{\ell=1}^{M} \frac{a_\ell(p,q)}{|p|}
  \, C_J \, 2^J \what{\chi_\ell}(2^J(|q|-|p|))
\end{split}
\end{align}
where we define, according to the notation in \eqref{Lemma1conc1} (see also \eqref{prL13})
\begin{align}\label{prL19.1}
\begin{split}
a_{\ell}(p,q) %:=  (2^{-P} |p|)^{-\ell/2} b_{\ell-1}^-(p,q) 
	& := (2^{J+P})^{-\ell/2 +\delta} (2^{-P}|p|)^{-\ell/2} \, c_{\ell-1} \partial_\phi^{\ell} g(-p/|p|,q),
\\
\chi_\ell & := \what{r^{-\ell/2} \varphi},
\\
C_J & := (2^{J+P})^{-\delta},
\end{split}
\end{align}
for some small fixed $\delta>0$.
To verify the estimates \eqref{Lemma1conc1.2} %\eqref{Lemma1est1} 
on the coefficients $a_\ell$ we use the fact that our parameters satisfy $J+P \geq 4A \geq 4Q$, see \eqref{NSM2},
and the estimates \eqref{prLemma1g1} on $g$: for $|p|\approx 2^P, |q|\approx 2^Q$
\begin{align*}
| \nabla_p^{\alpha} \nabla_q^{\beta} a_{\ell}(p,q) |
  & \lesssim (2^{J+P})^{-\ell/2 +\delta} 
  \sup_{\alpha_1+\alpha_2=\alpha}
  |p|^{-|\alpha_2|} \big| \nabla_p^{\alpha_1} \nabla_q^{\beta} \partial_\phi^{\ell} g(-p/|p|,q) \big|
\\
& \lesssim (2^{4A})^{-\ell/2 +\delta} 
  \sup_{\alpha_1+\alpha_2=\alpha}
  2^{-|\alpha_2|P} 2^{-|\alpha_1|P} \sup_{\rho \leq |\alpha_1|+\ell} 
  {\big\| \nabla_q^{\beta} \nabla_\omega^{\rho} g(\cdot,q) \big\|}_{L^\infty}
\\
& \lesssim (2^{4A})^{-\ell/2 +\delta} 
  2^{-|\alpha|P} (\jq^{|\alpha|+\ell} + (|q|/\jq)^{1-|\beta|})
%\\
%& \lesssim 2^{-|\alpha|P} 2^{(1-|\beta|)Q} 2^{A(|\alpha|+|\beta|-1)}
.
\end{align*}
For $Q\geq 0$ this gives the bound
$2^{-|\alpha|P} \jq^{|\alpha|} \lesssim 2^{-|\alpha|P} 2^{A|\alpha|}
  %\lesssim 2^{-|\alpha|P} 2^{A|\alpha|} 2^{(1-|\beta|)Q} 2^{|\beta|A},
$
while for $Q\leq 0$ it implies a bound of
%\begin{align*}
$2^{-|\alpha|P} (1 + 2^{(1-|\beta|)Q_-})%\lesssim 2^{-|\alpha|P} (2^{(1-|\beta|)Q} + 2^A)
$;
%\end{align*}
these are consistent with the desired bounds \eqref{Lemma1conc1.2}.
The definition of $C_J$ in \eqref{prL19.1} gives the property \eqref{Lemma1conc1.3}, 
since $J+P \gg 1$ in the set $\mathcal{J}$.
%the definition of the class $\mathcal{G} \ni g_0$, 

To complete the proof of the lemma it suffices to show how the three remaining terms in \eqref{prL17.2}-\eqref{prL17.3} 
satisfy the estimates \eqref{Lemma1est2} and can therefore be absorbed into the remainder $R_{J,M}$.

\medskip
{\it Contribution of \eqref{prL17.2}}.
Similar to \eqref{prL19}, the sum in \eqref{prL17.2} is given by
\begin{align}
\label{prL20}
\begin{split}
& \sum_{\ell=1}^{M} \frac{b_{\ell-1}^+(p,q)}{i|p|^{1 + \ell/2}} \,
  \int_0^\infty e^{ir|p|} \, e^{ir |q|} \, r^{-\ell/2} \varphi(r2^{-J}) \,\mathrm{d}r
\\ 
& = \frac{1}{i|p|(|p|+|q|)} \sum_{\ell=1}^{M} d_\ell(p,q) % \, (2^{P+J})^{-\delta} 
  \, 2^J (|q|+|p|)\what{(r^{-\ell/2} \varphi)}(2^J(|q|+|p|) \cdot C_J
\end{split}
\end{align}
having defined $d_{\ell} %= (2^{P+J})^{-\ell/2+\delta} (2^{-P}|p|)^{\ell/2} b_{\ell-1}^+ 
:= (2^{P+J})^{-\ell/2+\delta} (2^{-P}|p|)^{\ell/2} c_{\ell-1} \partial_\phi^{\ell} g(p/|p|,q)$, similarly to $a_\ell$,
and $C_J$ as in \eqref{prL19.1}.
Since $2^{P+J} \gtrsim 2^{4A}$, %see \eqref{NSM2}, 
and $r^{-\ell/2} \varphi$ is a Schwartz function, 
%and recalling also \eqref{prL10.0},
it follows that
\begin{align}
\label{prL21}
\begin{split}
\Big| & \sum_{\ell=1}^{M} \frac{b_{\ell-1}^+(p,q)}{i|p|^{1 + \ell/2}} \,
  \int_0^\infty e^{ir|p|} \, e^{ir |q|} \, r^{-\ell/2} \varphi(r2^{-J}) \,\mathrm{d}r \Big|
  \\ & \lesssim \frac{1}{(|p|+|q|)^2} \cdot \sup_{\ell =1,\dots,M} {\|d_\ell\|}_{L^\infty} \cdot C_J
  \\ & \lesssim \frac{1}{2^{2\max(P,Q)}}  \cdot 2^{3A (-\ell/2+\delta)} 
    \sup_{\ell =1,\dots,M} \sup_\omega \big| \partial^\ell_\omega g(\omega,q) \big| \cdot C_J,
\end{split}
\end{align}
having used the definitions \eqref{prL13}. 
In view of the inequality \eqref{prLemma1g1} and $|q| \lesssim 2^A$,
we see the validity of \eqref{Lemma1est2} for $\alpha=\beta=0$.
The general estimate \eqref{Lemma1est2} for $|\alpha|+|\beta| \geq 1$
follows similarly after differentiating the first line of \eqref{prL20} 
and using as before \eqref{prL13}, the estimate \eqref{prLemma1g1}, 
and our localization and restrictions on the parameters \eqref{NSM2}.
Note in particular how each differentiation of the integral in $p$ or $q$ can cost a potentially
dangerous factor of $2^J$, which can be traded for a factor of 
$(|p|+|q|)^{-1}$ since $\varphi$ is Schwartz.

\medskip
{\it Contribution of the remainders \eqref{prL17.3}}.
%Here we distinguish two cases: $|P-Q| \leq 5$ and $|P-Q|>5$.
Since $|P-Q| \leq 5$, we see from \eqref{prL18.1} and \eqref{prL18.2} that
\begin{align}\label{prL22}
\begin{split}
\frac{1}{|p|} |R_{J,M}^{\pm}(p,q)| & \lesssim 2^{-P} \sup_{\omega\in\mathbb{S}^2}|\partial_\phi^{M+1} g(\omega,q)|
  \int_0^\infty \varphi(r2^{-J}) (r|p|)^{-M/2} \, \mathrm{d}r 
  \\
  & \lesssim 2^{-2P} (1+ 2^{(M+1)Q}) \cdot 2^{-(J+P) (M/2-1)} 
  \\
  & \lesssim 2^{-2P} 2^{-\delta(P+J)},
\end{split}
\end{align}
for $\delta>0$ small enough, since $P+J \geq 4A \geq 4Q$ and $M>5$.
This is consistent with \eqref{Lemma1est2} when $\alpha=\beta=0$.
The general estimate \eqref{Lemma1est2} %in the case $|P-Q|\leq 5$ 
follows by differentiating the formula \eqref{prL18.1}:
%Using \eqref{prL18.1} we have
\begin{align*}
& \frac{1}{|p|} \big| \nabla_p^\alpha \nabla_q^\beta R_{J,M}^\pm(p,q) \big| 
  \\ & \lesssim \sup_{\substack{\alpha_1+\alpha_2=\alpha \\ \beta_1+\beta_2=\beta}}
  \Big| \int_0^\infty \nabla_p^{\alpha_1} e^{\pm ir|p|} \,\nabla_q^{\beta_1}  e^{ir |q|} \,
    \nabla_p^{\alpha_2} \nabla_q^{\beta_2} B(r|p|;q)  \, \varphi(r2^{-J}) \,\mathrm{d}r \Big|.
\end{align*}
On the support of the integral, since $r \approx 2^J \gtrsim 2^{-P} \approx 2^{-Q}$, we have
\begin{align*}
\big| \nabla_p^{\alpha_1} e^{\pm ir|p|} \nabla_q^{\beta_1} e^{ir |q|} \big|
\lesssim 2^{(|\alpha_1|+|\beta_1|)J}. %\cdot (2^{|\beta_1|J} + 2^{-|\beta_1|Q}).
\end{align*}
Using \eqref{prL18.3} we obtain
\begin{align}\label{prL25}
\begin{split}
& \frac{1}{|p|} \big| \nabla_p^\alpha \nabla_q^\beta R_{J,M}^\pm(p,q) \big| 
\\ & \lesssim 2^{-2P} \sup_{\substack{\alpha_1+\alpha_2=\alpha \\ \beta_1+\beta_2=\beta}}
  2^{(|\alpha_1|+|\beta_1|)J} %2^{|\alpha_1|J} \cdot (2^{|\beta_1|J} + 2^{-|\beta_1|Q})
  \cdot (2^{J+P})^{-M/2+1} 2^{-P|\alpha_2|} \cdot \big(2^{Q(M+1+|\alpha_2|)} + 2^{-(|\beta_2|-1)Q}\big)  
\end{split}
\end{align}
If $Q\leq 0$ for each term above we have a bound of 
\begin{align*}
2^{-2P} 2^{-(J+P)(M/2-1-|\alpha_1|-|\beta_1|)} \cdot 2^{-(|\a|+|\beta_1|)P} 2^{-|\beta_2|Q} 
\end{align*}
which is consistent with \eqref{Lemma1est2}. %(recall $|P-Q|\leq 5$).
For $Q\geq 0$ instead each term in \eqref{prL25} is bounded by
\begin{align*}
& 2^{-2P} \cdot 2^{(|\alpha_1|+|\beta_1|)J}
  \cdot (2^{J+P})^{-M/2+1} 2^{-P|\alpha_2|} \cdot 2^{Q(M+1+|\alpha_2|)}
\\
& \lesssim 2^{-2P} \cdot (2^{J+P})^{-M/2+1+|\alpha_1|+|\beta_1|} 
	\cdot 2^{Q(M+1)} 2^{|\alpha_1|A} \cdot 2^{-|\b_1|Q} 2^{-P(|\alpha_1|+|\alpha_2|)}
\\
& \lesssim 2^{-2P} \cdot  2^{-P|\alpha|} 2^{-Q|\beta|} \cdot 2^{A(|\alpha|+|\beta|)} \cdot 2^{-\delta(J+P)}
\end{align*}
having used $J+\min(P,Q) \geq 4A \geq 4Q$ and $|\alpha|+|\beta| \leq M/4-2$.
\iffalse
To conclude we analyze the case $|P-Q|>5$. Since $||p|-|q| | \gtrsim \max(|p|,|q|)$
the idea in this case is to integrate by parts in $r$ in the integrals \eqref{prL18.1}. 
For $D = 1,2,\dots$ we can write
\begin{align}\label{prL30}
R_{J,M}^\pm(p,q) = C \int_0^\infty \frac{1}{(\pm|p|+|q|)^D} e^{\pm ir|p|} \, e^{ir |q|} \, 
  \Big(\frac{d}{dr}\Big)^D \Big[ \varphi(r2^{-J}) B^\pm(r|p|;q) \Big] \,\mathrm{d}r.
\end{align}
In view of \eqref{prL18.2}
\begin{align}\label{prL31}
\Big| \Big(\frac{d}{dr}\Big)^D \Big[ \varphi(r2^{-J}) B^\pm(r|p|;q) \Big] \Big| 
  \lesssim 2^{-J D} \cdot 2^{-(J+P)M/2} 2^{(M+1)Q} \lesssim 2^{-JD} \cdot C_J.
\end{align}
In other words, each integration by parts gains a factor of $2^{-\max(P,Q)-J}$.
In particular
\begin{align*}
\Big| \frac{1}{|p|} R_{J,M}^\pm(p,q) \Big| \lesssim 
  2^{-P} \cdot \frac{1}{(\pm|p|+|q|)^D} 2^{-J D} C_J \cdot 2^J 
\end{align*}
which for $D=2$ gives \eqref{Lemma1est2} when $\alpha=\beta=0$.
For the general estimate, we apply $\nabla_p^\alpha \nabla_q^\beta$ to \eqref{prL30} 
use the bounds on the derivatives of $B^\pm$ in \eqref{prL18.2}-\eqref{prL18.3} and repeat 
sufficiently many times integration by parts in $r$.
In particular, observe that each $(p,q)$-derivative costs at most a factor of $2^J$
(recall that $2^{-P} + 2^{-Q} \lesssim 2^J 2^{-4A}$);
each such factor is offset by an integration by parts in $r$, that at the same time 
also introduces the desired factor of $2^{-\max(P,Q)}$.
%\Dets 
\fi
This concludes the proof of \eqref{Lemma1est2} and the Lemma. $\hfill \Box$

\medskip
We now combine Lemma \ref{Lemmanu0}
and  the asymptotic expansion \eqref{lemmapsi1exp} for $\psi_1$ in Lemma \ref{lemmapsi1}
to prove Proposition \ref{Propnu+}.

\medskip
\subsection{\bf Proof of Proposition \ref{Propnu+}}\label{ssecprnu+}
We consider $|p| \approx 2^P$ and $|q| \approx 2^Q$ with $P,Q \leq A$, $A\gg 1$ and write
\begin{align}\label{pr1}
\begin{split}
\nu(p,q) & = \nu^+(p,q)\mathbf{1}_{\{|P-Q| < 5\}} + \nu^-(p,q),
  \qquad \nu^+(p,q) := 
  \sum_{J \in \mathcal{J}} \nu^J(p,q), %\qquad \nu^- := \sum_{J \in \mathcal{J}^c} \nu^J(p,q),
\\
\nu^J(p,q) & := \int e^{ix \cdot p} \, \frac{e^{i|q||x|}}{|x|} \psi_1(x,q) \, \varphi_J^{(0)}(x) \,\mathrm{d}x,
\end{split}
\end{align}
where the cutoff $\varphi_J^{(0)}$ is defined in \eqref{LP0}-\eqref{LP2} and we recall
\begin{align}\label{prJ}
\mathcal{J} := \{ J \geq 4A, \, J \geq -\min(P,Q) + 4A\}. 
\end{align}
The term $\nu_+\mathbf{1}_{\{|P-Q| < 5\}}$ will give rise to the leading order terms (and some remainder terms), 
while all the terms in $\nu_-$ are lower order remainders.

%Let us define the cutoff functions 
%\begin{align}
%\label{nudec0}
%\chi_+(x) := H \ast \frac{1}{c} \varphi = \frac{1}{c} \int_{-\infty}^x \varphi(y)\,\mathrm{d}y,  \qquad \chi_-(x) = 1-\chi_+(x),
%\end{align}
%where $H=(1+\sign)/2$ is the Heavyside function, $\varphi$ the standard cutoff from \eqref{LP0}, $c = \int\varphi$.
%Decompose $\nu_1$ accordingly as
%\begin{align}\label{nudec1}
%\begin{split}
%& \nu_1(p,q) := \nu^{+}(p,q) + \nu^{-}(p,q),
%\\
%& \nu^\ast(p,q) := \int e^{ix \cdot p} \frac{e^{i|x||q|}}{|x|} \psi_1(x,q) \chi_\ast(|x||p|) \, \mathrm{d}x.
%\end{split}
%\end{align}

%Notice that $\chi_+(0) = 0$, and $\chi_+(x) = 1$ for $|x| \gg 1$, 
%while $\chi_-$ is a smooth compactly supported function.
%As we will see, $\nu^+$ captures the singular behavior of $\nu$ while $\nu^-$ is a less singular term. 
%For the term $\nu^+$ we will show a precise asymptotic expansion highlighting the exact structure of the singularity, see Lemma \ref{Propnu+}.
%For $\nu^-$ instead we will only need proper ``bilinear symbol'' type estimates.

%%Need to distinguish $|P-Q|<5$ in the proof.

\medskip
\noindent
{\it Analysis of $\nu_+\mathbf{1}_{\{|P-Q|\leq 5\}}$}.
For $g \in \mathcal{G}^N$, $N> N_2$, see \eqref{classG_Nomega}, let us denote %, similarly to \eqref{Lemmanu0def},
\begin{align}
\label{prnu+K}
\K_{J,n}(p,q)[g] := \int_{\R^3} e^{ix \cdot p} \frac{e^{i|x||q|}}{|x|} g(\omega,q) \, 
  (|\cdot|^{-n}\varphi)(x/2^J)  \mathrm{d}x.
\end{align}
%where $\rho$ is a generic Schwartz function which might change from line to line.
From the definition of $\nu^+$ in \eqref{pr1}, and using the expansion \eqref{lemmapsi1exp}
from Lemma \ref{lemmapsi1}, we can write
\begin{align}\label{prnu+exp}
\begin{split}
\nu^+(p,q) %& = \sum_{J \in \mathcal{J}} 
  %\int e^{ix \cdot p} \frac{e^{i|x||q|}}{|x|} \psi_1(x,q) \varphi(|x|/2^J) \, \mathrm{d}x
  %\\ 
  & = \sum_{J \in \mathcal{J}} \sum_{n=0}^{N_2-1} \K_{J,n}(p,q)[g_n] \, 2^{-J n} \jq^n + R(p,q)
\end{split}
\end{align}
where
%\begin{align}\label{prnu+jn}
%\nu^+_{j,n}(p,q) := \K_j(p,q)[g_n] 
%\end{align}
%and
\begin{align}\label{prnu+R}
R(p,q) = \sum_{J \in \mathcal{J}} 
  \int_{\R^3} e^{ix \cdot p} \frac{e^{i|x||q|}}{|x|} R_{N_2}(x,q) \, \varphi(x/2^J) \, \mathrm{d}x 
\end{align}
for $R_{N_2}(x,q)$ satisfying estimates as in \eqref{lemmapsi1R}.

%: HERE RIGHT NOW

\medskip
{\it Leading order contribution}.
Let us look first at the term with $n=0$ in \eqref{prnu+exp}, 
that is $\sum_{J\in\mathcal{J}} \K_{J,0}(p,q)[g_0]$.
Since $|P-Q| < 5$, in view of the result of Lemma \ref{Lemmanu0}, and adopting the same notation, 
we have that 
\begin{align}\label{prnu+1}
\sum_{J \in \mathcal{J}} & \K_{J,0}(p,q)[g_0] = A + \sum_{\ell=1}^{M-1} B_\ell + C,
\\ 
\label{prnu+11}
A & := \frac{a_0(p,q)}{|p|} \sum_{J \in \mathcal{J}} 2^J \hat{\varphi}(2^J(|p| - |q|)),
\\
\label{prnu+12}
B_\ell & := \frac{a_\ell(p,q)}{|p|} \sum_{J \in \mathcal{J}}  
	2^J \chi_\ell (2^J(|p| - |q|)) \cdot C_J %(|p|2^J)^{-\ell/2},
\\ 
\label{prnu+13}
C & := \sum_{J \in \mathcal{J}} R_{J,M}(p,q)\cdot C_J^R,
\end{align}
where the coefficients $a_\ell$ %\in \mathcal{G}^{N_1-\ell}$
satisfy the estimates \eqref{Lemma1conc1.2},
$\sum_{J\in\mathcal{J}} |C_J| + |C_J^R| < \infty$, %see \eqref{Lemma1conc1.3}
and the estimates \eqref{Lemma1est2} hold for the remainder $R_{J,M}$.

Using the properties of the standard cutoff $\varphi$, see \eqref{LP0}, 
we can define 
\begin{align}\label{J_0}
J_0 : = \max(4A, -\min(P,Q) + 4A) = 4A - \min(0,P,Q)
\end{align}
and rewrite the term \eqref{prnu+11} as
\begin{align}\label{prnu+111}
\begin{split}
A = \frac{a_0(p,q)}{|p|}  \sum_{J \geq J_0} \what{\varphi(\cdot \, 2^{-J})} (|p|-|q|) 
	& = \frac{a_0(p,q)}{|p|} \what{\varphi_{\geq 1}(\cdot \, 2^{-J_0})} (|p|-|q|).
\end{split}	
\end{align}
Writing
$\varphi_{\geq 1}(x) = \int_{-\infty}^x \psi(y)\, \mathrm{d}y = \mathbf{1}_{\{x>0\}} \ast \psi$, 
for a smooth $\psi \geq 0$ which is compactly supported in $[1/8,2]$ and with integral $1$,
we deduce the formula
\begin{align*}
\what{\varphi_{\geq 1}}(\xi) = \mathcal{F}\big( (1+\sign x)/2 \ast \psi \big)(\xi) 
	& = \sqrt{2\pi} \mathcal{F}\big( (1+\sign x)/2 \big)(\xi)  \cdot \what{\psi}(\xi)
	\\ & = \sqrt{\frac{\pi}{2}} \delta_0(\xi) + \pv \frac{\what{\psi}(\xi)}{i \xi}.
\end{align*}
%where $H$ denotes the standard Heavyside function.
It follows that 
\begin{align}\label{prnu+112}
\begin{split}
A & = \frac{a_0(p,q)}{|p|} 2^{J_0} \what{\varphi_{\geq 1}}((|p|-|q|) 2^{J_0})  
	\\ & =  \frac{a_0(p,q)}{|p|} \frac{1}{i\sqrt{2\pi}} \Big[ i\pi \delta(|p|-|q|) 
	+ \sqrt{2\pi} \, \pv \frac{\what{\psi}(2^{J_0}(|p|-|q|)}{|p|-|q|} \Big].
\end{split}	
\end{align}
%up to redefining $a_0$ to take into account factors of $\pi$.
Up to slightly redefining $a_0$, this gives us the first terms in the right hand-side of \eqref{Propnu+1}
with \eqref{Propnu+1.0}-\eqref{Propnu+1.1},
%up to a term that can be absorbed into the ``smooth'' remainder component $\nu_R$ in \eqref{Propnu+3}.
provided we show that the term (we use that $\what{\psi}(0) = 1/\sqrt{2\pi}$)
\begin{align}\label{prnu+113}
\begin{split}
A_R & := \frac{a_0(p,q)}{|p|} \pv \frac{\sqrt{2\pi} \,\what{\psi}(2^{J_0}(|p|-|q|)) - 1}{|p|-|q|}
  \\ & \, = \frac{a_0(p,q)}{|p|} \sqrt{2\pi} \, 2^{J_0} \int_0^1 {\hat{\psi}}^\prime(2^{J_0}(|p|-|q|) t) \, \mathrm{d}t
\end{split}	
\end{align}
%is an acceptable remainder, in that is 
satisfies estimates as in \eqref{Propnu+3}. %For this we distinguish a few cases.
To see that this is the case, let us consider first the case $J_0 = 4A$, $\min(P,Q) \geq 0$. 
It is not hard to see, using 
the estimates for $a_0$ from \eqref{Lemma1conc1.3}, and $2^{\max(P,Q)} \leq 2^A \leq 2^{J_0}$, that
\begin{align*}
|\nabla^{\alpha}_p \nabla^{\beta}_q A_R | & \lesssim 2^{J_0}2^{-P} \cdot 2^{J_0(|\alpha|+|\beta|)}
  \\ & \lesssim 2^{-2\max(P,Q)} 2^{2A} \cdot 2^{-|\alpha|P} 2^{(|\alpha|-1)A} \cdot 2^{4A(|\alpha|+|\beta|+1)} 
\end{align*}
which is acceptable.
The case $J_0 = 4A-\min(P,Q)$ (i.e. $\min(P,Q) \leq 0$) %with $|Q-P|<5$ 
is similar, using again that $\what{\psi}'$ is Schwartz, and that $2^P \approx 2^Q$.

To verify that the terms $B_\ell$ in \eqref{prnu+12} are of the form \eqref{Propnu+2}-\eqref{Propnu+2.1},
it suffices to recall that %$a_\ell \in \mathcal{G}^{N_1-\ell}$, 
$a_\ell$ satisfies \eqref{Lemma1conc1.2} for $1\leq \ell < M$, and $M \leq N_2$.

For the term in \eqref{prnu+13} we can use directly the estimate \eqref{Lemma1est2}-\eqref{Lemma1conc1.3}, 
to see that this satisfies bounds like those in \eqref{Propnu+3}.
%\eqref{Lemma1est2} are worse than this\dots

%\begin{align*}
%& |\nabla_p^\alpha \nabla_q^\beta C(p,q)|  \lesssim 
%	\sum_{j \,:\, 2^j|p|\gtrsim1} | \nabla_p^\alpha \nabla_q^\beta R_{j,M}(p,q) |
%\\ & \lesssim (1+|p|^{-2}) (1+|q|)^{|\alpha| +M+1}
%	\cdot \sum_{j \,:\, 2^j|p|\gtrsim1} (2^j|p|)^{-M/2+1} 2^{j|\alpha|} (2^j + |q|^{-1})^{|\beta|}
%\\ & \lesssim (1+|p|^{-2})  |p|^{-|\alpha|} \max(1,|p|^{-1},|q|^{-1})^{|\beta|} 
%	(1+|q|)^{|\alpha|+M+1}.
%\end{align*}
%provided $|\alpha|+|\beta| < M/2-1$. Choosing $M=N_2$ we have that the term $C$ 
%can be absorbed into $\nu_R$.

\medskip
{\it Lower order contributions}.
Let us consider the contribution to the sums in \eqref{prnu+exp} with $n = 1,\dots,N_2$.
We consider a term of the form
\[
I_n = \sum_{J\in\mathcal{J}} \K_{J,n}(p,q)[g_n] \cdot  2^{-J n} \jq^n,
	\qquad g_n \in \mathcal{G}^{N_1-n},
\]
and want to apply Lemma \ref{Lemmanu0} %\ref{lemmapsi1} 
with $M \leq N_1-n$.
Since $|q|\approx 2^{Q} \lesssim 2^A$ and for $J\in\mathcal{J}$ 
we must have $(1+|q|)2^{-J} \leq 2^{-J/2}$, the conclusion of Lemma \ref{Lemmanu0} gives
\begin{align*}
%\label{prnu+20}
I_n &= \sum_{\ell=0}^{M-1}  B^{(n)}_\ell + C^{(n)},
\end{align*}
where
\begin{align}
\label{prnu+21}
B^{(n)}_\ell & := \sum_{J\in\mathcal{J}}  
	\frac{a_\ell^{(n)}(p,q)}{|p|} \cdot 2^J \widehat{\chi_\ell^{(n)}}(2^J(|p| - |q|)) \cdot 2^{-J n/2},
\\ 
\label{prnu+22}
C^{(n)} & := \sum_{J\in\mathcal{J}}  R_{J,M}^{(n)}(p,q) \cdot 2^{-J n/2},
\end{align}
and we have that 

\begin{itemize}

\smallskip
\item[(1)] $a_\ell^{(n)}(p,q) \in \mathcal{G}^{N_1-n-\ell} \subset \mathcal{G}^{N_1-N_2-M}$,
for $0\leq \ell < M$,  %\leq N_2-n
and 

\smallskip
\item[(2)] $R_{j,M}^{(n)}(p,q)$ satisfy estimates like those in \eqref{Lemma1est2} with $|\alpha|+|\beta| 
\leq \min(N_1-n-M-1,M/4-2)$.
%Adjust here and below the range of $|\a|+|\b|$

\end{itemize}

%\noindent
%In particular, for any $n = 1,\dots,N_2$ we have
%\[ a_\ell^{(n)} (1+|q|)^n \in \mathcal{G}^{N_1-N_2} (1+|q|)^{N_2};\]
In particular, we see that the terms $B^{(n)}_\ell$ are of the form \eqref{Propnu+2}-\eqref{Propnu+2.1}.
%For 
The remainder term \eqref{prnu+22} 
%we use the estimates \eqref{Lemma1est2}, see (2) above, to deduce that
%\begin{align*}
%& |\nabla_p^\alpha \nabla_q^\beta C^{(n)}(p,q)|  \lesssim 
%	(1+|q|)^n \sum_{j \,:\, 2^j|p|\gtrsim1} 2^{-j n} | \nabla_p^\alpha \nabla_q^\beta R_{j,M}^{(n)}(p,q) |
%\\ & \lesssim (1+|p|^{-2}) (1+|q|)^{|\alpha| +n+M+1}
%	\cdot \sum_{j \,:\, 2^j|p|\gtrsim1} (2^j|p|)^{-M/2+1} 2^{j|\alpha|} (2^j + |q|^{-1})^{|\beta|} \cdot 2^{-jn} 
%\\ & \lesssim (1+|p|^{-2})  |p|^{-|\alpha|} \max(1,|p|^{-1},|q|^{-1})^{|\beta|} 
%	(1+|q|)^{|\alpha|+M+n+1}.
%\end{align*}
%provided we have $|\alpha| + |\beta| < M/2 -1 + n$.
%These 
satisfies estimates which are consistent with \eqref{Propnu+3} 
%for indexes $|\alpha|+|\beta|$ as large as $N_2/2-2$,
since we can choose $M=N_1-N_2-n$,
%and thus $M/2 -1 + n \geq N_2/2-1$.
%and absorb the terms $C^{(n)}$ into $\nu_R$.

%Adjust/comment parameters $N_2,M$ etc.

\medskip
{\it The remainder $R$ in \eqref{prnu+R}}.
%Can add details for $R_{N_2}$  (using Faa-di Bruno to be precise)
From \eqref{lemmapsi1R} we know that, for $|x| \approx 2^J$ and $|q| \approx 2^Q$
\begin{align}\label{prnu+30}
| \nabla_q^\beta R_{N_2}(x,q) | \lesssim 2^{-N_2 J} \big(2^{N_2 Q} + 2^{(1-|\beta|)Q_-} \big)
\end{align}
for all $|\beta|\leq N_1-N_2$. 
Differentiating \eqref{prnu+R}, using \eqref{prnu+30}, and $Q \leq J/4$, %on the support of the integral,
we see that, as long as $|\alpha|+|\beta| \leq N_2/2-3$,
\begin{align}\label{prnu+31}
|\nabla_p^{\alpha} \nabla_q^\beta R(p,q) | \lesssim 2^{(1-|\beta|)Q_-}.
\end{align}
This is upper bounded by the right-hand side of \eqref{Propnu+3}.
We can therefore absorb the term $R$ into $\nu_R$.

\medskip
{\it The remainder $\nu^-$ in \eqref{pr1}}.
Finally we show that the term $\nu_-$ can also be absorbed into the remainder $\nu_R$.
We look at the case $\min(P,Q)\leq 0$, since the other case is easier. %Check
By definition
\begin{align}\label{prnu+40}
\begin{split}
& \nu^-(p,q) = \sum_{J \in \mathcal{J}^c} \nu_J(p,q) \mathbf{1}_{\{|P-Q|<5\}}
   + \sum_{J \in \Z_+} \nu_J(p,q) \mathbf{1}_{\{|P-Q|\geq 5\}}, 
\\
%+ \sum_{J \in \mathcal{J}_2} \nu_J(p,q), 
& \mathcal{J}^c = \{ J \in \Z_+ \, : \, J < -\min(P,Q) + 4A\}.
\\
%& \mathcal{J}_1 := \{ J \in \Z_+ \, : \, J < 4A\}, 
  %\qquad \mathcal{J}_2 := \{ J \in \Z_+ \, : \, J < -\min(P,Q) + 4A\}.
\end{split}
\end{align}

Let us look first at the term with $|P-Q| < 5$ and $J\in\mathcal{J}^c$.
We inspect the formula \eqref{pr1} for $\nu_J$
to see that $\partial_p^\alpha \partial_q^\beta \nu^-(p,q)$  is a linear combination of terms of the form
\begin{align}\label{prnu+41}
I_{\alpha,\beta_1,\beta_2} := \int_{\R^3} (ix)^{\alpha} e^{ix \cdot p} \, \partial_q^{\beta_2} \big( e^{i|x||q|} \big) 
  \frac{1}{|x|} \, \partial_q^{\beta_1} \psi_1(x,q) \, %\partial_p^{\alpha_2} chi_-(|x||p|) 
  \varphi(x/2^J) \, \mathrm{d}x
\end{align}
for %$\alpha_1+\alpha_2=\alpha$, 
$\beta_1+\beta_2 =\beta$. %with $|\alpha|+|\beta| \leq N_1$.
The estimates \eqref{psi10} for $\psi_1$ give us
\begin{align}\label{prnu+42}
& \big| \nabla_q^{\beta_1} \psi_1(x,q) \big| 
  %\lesssim \jq^{|\alpha|} + (|q|/\jq)^{1-|\beta_1|} 
  \lesssim 1+ 2^{(1-|\beta_1|)Q_-}.
  %\qquad \quad 1\leq |\alpha| + |\beta| \leq N_1, \quad \beta\neq 0
\end{align}
Since we also have 
\begin{align}\label{prnu+43}
\begin{split}
\big| \nabla_q^{\beta_2} e^{i|x||q|} \big| \lesssim \sum_{a + b = |\beta_2|-1} 2^J \cdot 2^{Ja} \cdot 2^{-bQ} 
  \\
  \lesssim 2^J 2^{(1-|\beta_2|)Q_-} + 2^{J|\beta_2|}
\end{split}
\end{align}
we see that %for $|P-Q| \leq 5$ we have
\begin{align}\label{prnu+44}
|I_{\alpha,\beta_1,\beta_2}| & \lesssim 2^{2J} \cdot 2^{J|\alpha|}
  \cdot \big(2^J 2^{(1-|\beta_2|)Q_-} + 2^{J|\beta_2|} \big) \cdot (1+2^{(1-|\beta_1|)Q_-})
  \\
\nonumber
& \lesssim 2^{-2\max(P,Q)} \cdot 2^{-\max(P,Q)(|\alpha|+|\beta|)}  \cdot 2^{4A(|\alpha|+|\beta|+2)},
\end{align}
since %$J\leq 4A$ or 
$J\leq -\max(P,Q) + 4A + 5$. %and $P,Q \leq A$.
This is consistent with the desired bound \eqref{Propnu+3}.

For the elements in the second sum in \eqref{prnu+40} 
we can resort to an integration by parts argument using that $||p|-|q|| \gtrsim \max(|p|,|q|)$.
%We assume that $J \geq 5$, for otherwise one can proceed as in the case above.
Notice that, for any integer $\rho>0$, we can write
\begin{align}\label{prnu+50}
e^{i(x\cdot p + |x||q|)} =  T^\rho e^{i(x\cdot p + |x||q|)}, \qquad 
  T:= \frac{p+(x/|x|)|q|}{i\big|p+(x/|x|)|q|\big|^2} \cdot \nabla_x.
\end{align}
Since $|p+(x/|x|)|q|| \gtrsim 2^{\max(P,Q)}$, 
and for any $|\gamma|\geq 1$ we have $|\nabla_x^\gamma (p+(x/|x|)|q|)| \lesssim 2^{-|\gamma| J} 2^{Q}$
and 
\begin{align*}%\label{prnu+51}
& \big| \nabla_x^\gamma \psi_1(x,q) \big| \lesssim 2^{-|\gamma|J} 2^{|\gamma|Q_+},
\end{align*}
see \eqref{psi10}, we obtain
\begin{align}\label{prnu+51}
\begin{split}
|\nu_J(p,q)| = \Big| \int e^{i(x \cdot p + |q||x|)} (T^\ast)^\rho 
  \Big[ \frac{1}{|x|} \psi_1(x,q) \, \varphi_J^{(0)}(x) \Big] \,\mathrm{d}x \Big|
  \\ \lesssim 2^{2J} 2^{-J \rho} 2^{-\max(P,Q) \rho} 2^{A\rho}.
\end{split}
\end{align}
With $\rho=2$ this gives us \eqref{Propnu+3} for $\alpha=\beta=0$.
For $|\alpha|+|\beta|\geq 1$ we apply derivatives obtaining terms as in \eqref{prnu+41}
and then use integration by parts as above.
Using also the first line of \eqref{prnu+43}, we get the following improvement of \eqref{prnu+44}:
$|I_{\alpha,\beta_1,\beta_2}|$ is bounded by a linear combination of terms of the form
\begin{align}\label{prnu+52}
2^{2J} \cdot 2^{-J \rho} 2^{-\max(P,Q) \rho} 2^{A\rho} 
  \cdot 2^{J|\alpha|}
  \cdot ( 2^J 2^{Ja} 2^{-bQ} ) \cdot (1+2^{(1-|\beta_1|)Q_-}).
\end{align}
for $a+b=|\beta_2|-1$ (with the understanding that if $|\beta_2|=0,1$ then the whole term 
  involving $a$ and $b$ is absent).
%If $J\leq 4A$ it suffices to use $\rho=2$ in \eqref{prnu+52} to get
%\begin{align*}
%|I_{\alpha,\beta_1,\beta_2}| & \lesssim 2^{-2\max(P,Q)} 2^{2A} \cdot 2^{4A(|\alpha|+|\beta_2|)}
%  \cdot 2^{(1-|\beta|)Q_-},
%\end{align*}
%which suffices.
%If $J\leq 4A - \min(P,Q)$ 
We then use \eqref{prnu+52} with $\rho = |\alpha|+a+3$ to get that 
$|I_{\alpha,\beta_1,\beta_2}|$ is bounded by a linear combination of factors
\begin{align*}%\label{prnu+53}
\begin{split}
%|I_{\alpha,\beta_1,\beta_2}| & \lesssim 
& 2^{-2\max(P,Q)} 2^{-|\alpha|\max(P,Q)}
  \cdot 2^{-\max(P,Q)(a+1)} 2^{-bQ} \cdot 2^{(1-|\beta_1|)Q_-} \cdot 2^{A(|\alpha|+a+3)} 
  \\
  & \lesssim 2^{-2\max(P,Q)} 2^{-|\alpha|\max(P,Q)} \cdot 2^{-Q(a+b+1)}  
  \cdot 2^{(1-|\beta_1|)Q_-} \cdot 2^{A(|\alpha|+a+3)} 
  \\
  & \lesssim 2^{-2\max(P,Q)} 2^{-|\alpha|\max(P,Q)}
  \cdot 2^{(1-|\beta|)Q_-} \cdot 2^{A(|\alpha|+|\beta|+3)}
\end{split}
\end{align*}
consistently with \eqref{Propnu+3imp}.
$\hfill \Box$

%%%%%%%%%%%%%%%%%%%%%%%%%%%%%%%%%%%%%%%%%%%%%%%%%%
%%%%%%%%%%%%%%%%%%%%%%%%%%%%%%%%%%%%%%%%%%%%%%%%%%
%%%%%%%%%%%%%%%%%%%%%%%%%%%%%%%%%%%%%%%%%%%%%%%%%%
%%%%%%%%%%%%%%%%%%%%%%%%%%%%%%%%%%%%%%%%%%%%%%%%%%
%%%%%%%%%%%%%%%%%%%%%%%%%%%%%%%%%%%%%%%%%%%%%%%%%%
%%%%%%%%%%%%%%%%%%%%%%%%%%%%%%%%%%%%%%%%%%%%%%%%%%
%%%%%%%%%%%%%%%%%%%%%%%%%%%%%%%%%%%%%%%%%%%%%%%%%%
%%%%%%%%%%%%%%%%%%%%%%%%%%%%%%%%%%%%%%%%%%%%%%%%%%

%: MULTIPLIER ESTIMATES

\medskip
\section{Bilinear estimates for the leading order of the NSD}\label{secBE}
%Could follow at the beginning \cite{GHW} to show that the $\partial_k$ weight distributes, and that H\"older (Coifman-Meyer) type estimates hold.

In this section we prove several bilinear estimate for the (singular) multipliers appearing in our problem, 
such as those arising from the asymptotic formulas of Proposition \ref{Propnu+}.
The bilinear operators that we need to look at have the form
\begin{align}\label{bilinearex0}
T_{\mu_j}(g,h)(x) := 
\mathcal{F}^{-1}_{k\rightarrow x} 
	\iint_{\R^3\times\R^3} g(\ell) h(m) \, \mu_j(k,\ell,m) \, \mathrm{d}\ell \mathrm{d}m, \qquad j=1,2,3,
\end{align}
see \eqref{mudec}.
We will often need to consider these operators with additional symbols $b$ with suitable properties to be specified below,
that is, we will look at 
\begin{align}\label{bilinearex1}
T_{\mu_j}[b](g,h)(x) := \mathcal{F}^{-1}_{k\rightarrow x} 
	\iint_{\R^3\times\R^3} g(\ell) h(m) \, b(k,\ell,m) \mu_j(k,\ell,m) \, \mathrm{d}\ell \mathrm{d}m, \qquad j=1,2,3.
\end{align}

Our results will be a series of H\"older type %$L^p \times L^q \mapsto L^r$
estimates with some (small) losses and up to suitable remainders.
These estimates will then be used to establish the nonlinear bounds of Section \ref{secdkL2}. %and \ref{secdk^2L2}.
%cfr. the definitions \eqref{mu1}-\eqref{nu_1} and the formula \eqref{Propnu+1}.

\subsection{Bilinear estimates for $\mu_1$}
The most important operator is the one corresponding to the leading order term $\mu_1$, see \eqref{mudec}-\eqref{mu1}. 
Using the notation \eqref{bilinearex0}, the formula \eqref{mu1}, 
and the symmetry in exchanging $\ell$ and $m$, we see that,
\begin{align}
\label{T_mu}
T_{\mu_1}(g,h) = 2T_1(g,h)(k) + T_2(g,h)(k),
\end{align}
where
\begin{align}
\label{theomu11}
T_1(g,h)(x) & := \mathcal{F}^{-1}_{k\rightarrow x} \iint_{\R^3\times\R^3} g(k-\ell) h(m) \, \nu_1(\ell,m) \, \mathrm{d}\ell \mathrm{d}m,
\\
\label{theomu12}
T_2(g,h)(x) &:= \mathcal{F}^{-1}_{k\rightarrow x} \iint_{\R^3\times\R^3} g(-m-\ell) h(m) \, \bar{\nu_1(\ell,k)} \, \mathrm{d}\ell \mathrm{d}m,
\end{align}
where $\nu_1$ is defined in \eqref{nu_1} and satisfies the formulas of Proposition \ref{Propnu+}.
To allow for additional symbols we define
\begin{align}
\label{theomu11b}
T_1[b](g,h)(x) & := \mathcal{F}^{-1}_{k\rightarrow x} \iint_{\R^3\times\R^3} g(k-\ell) h(m) \,b(k,\ell,m)\, \nu_1(\ell,m) \, \mathrm{d}\ell \mathrm{d}m,
\\
\label{theomu12b}
T_2[b](g,h)(x) &:= \mathcal{F}^{-1}_{k\rightarrow x} \iint_{\R^3\times\R^3} g(-\ell-m) h(m) \,b(k,\ell,m)\, \bar{\nu_1(\ell,k)} \, \mathrm{d}\ell \mathrm{d}m.
\end{align}

\def\AfactorBB1{C_0}
\def\FfactorBB{{\max(L,M)}}
\def\FfactorBB2{{\max(L,K)}}

\medskip
\begin{theorem}[Bilinear bounds 1]\label{theomu1}
Let $T_1[b]$ and $T_2[b]$ be the bilinear operators defined in \eqref{theomu11b} and \eqref{theomu12b}.
Assume that:

\setlength{\leftmargini}{2em}
\begin{itemize}

\medskip
\item The symbol $b$ is such that
\begin{align}\label{theomu1asb1}
\supp(b) \subseteq \big\{ (k,\ell,m) \in \R^9\,:\, |k|+|\ell|+|m| \leq 2^A, \, |\ell| \approx 2^L,
  \, |m| \approx 2^M\big\},
\end{align}
for some $A\geq1$.
%Better to say $b = b \chi_{\leq A}$?

%Add localization in $k$ for $T_1$?
%Check/adjust factors below

\medskip
\item For all $|k| \approx 2^{K}$, $|\ell|\approx 2^L$ and $|m| \approx 2^M$
\begin{align}\label{theomu1asb2}
| \nabla_k^a \nabla^\alpha_\ell \nabla^\beta_m b(k,\ell,m)| 
	\lesssim 2^{-K|a|} 2^{-|\alpha|L}2^{-|\beta|M} \cdot 2^{(|a|+|\alpha|+|\beta|)A}, \qquad |a|,|\a|,|\b|\leq 5.
\end{align}
%number $5$ okay here because standard $4$ plus $1$ for the derivative when taking difference at the $\pv$.

%Put $\mathcal{S}^\infty$ class if want more generality}

\medskip
\item There is $10A \leq D \leq 2^{A/10}$ such that
\begin{align}
\label{theomu1asgh}
\begin{split}
\mathcal{D}(g,h) := {\| g \|}_{L^2} {\| h \|}_{L^2} 
	+ \min\big( {\| \partial_k g \|}_{L^2} {\| h \|}_{L^2}, {\| g \|}_{L^2} {\| \partial_k h \|}_{L^2} \big) 
	\leq 2^{D}.
\end{split}
\end{align}
%Check this condition in the proof

\end{itemize}

Then, the following estimates hold:

\begin{itemize}

\item[(i)] For any $p,q \in (1,\infty)$ and $r> 1$ with %Check endpoint $p=2,\infty$
\begin{align}
\frac{1}{p} + \frac{1}{q} > \frac{1}{r}, \qquad 
\end{align}
we have
\begin{align}
\label{theomu1conc}
\begin{split}
{\big\| T_1[b]\big(g,h\big) \big\|}_{L^r} 
  & \lesssim {\big\| \what{g} \big\|}_{L^p} {\big\| \what{h} \big\|}_{L^q} \cdot 2^{\max(L,M)}
  \cdot 2^{\AfactorBB1 A} + 2^{-D} \mathcal{D}(g,h), 
\end{split}
\end{align}
and
\begin{align}
\label{theomu1concT2}
\begin{split}
{\big\| P_K T_2[b](g,h) \big\|}_{L^r} 
  & \lesssim {\big\| \what{g} \big\|}_{L^p} {\big\| \what{h} \big\|}_{L^q} \cdot 2^{\max(L,K)}
  \cdot 2^{\AfactorBB1 A} + 2^{-D} \mathcal{D}(g,h),
% \min\big( {\| \partial_k g \|}_{L^2} {\| \varphi_M h \|}_{L^\infty}, 
%{\| g \|}_{L^2} {\| \varphi_M \partial_k h \|}_{L^\infty} \big) \cdot 2^{D/2}.
\end{split}
\end{align}
where\footnote{This is a convenient value of the absolute constant $C_0$ 
that we can choose in our proof, but it can certainly be improved. 
In the nonlinear estimates for the evolution equation (Sections \ref{secdkL2} and \ref{ssecmu23Est})
we are going to impose conditions on the smallness of $C_0 \delta_N$ (or similar quantities), 
see the condition \eqref{condition} for example, and recall the definition \eqref{d_N}.
Then, a smaller value of $C_0$ would reduce the total number of derivatives $N$ required for our initial data.
%We need to impose condition $2^{C_0 A} \approx \js^{6C_0/N} \lesssim \js^{\delta'}$ for small $\delta'$.
} $C_0:=65$.
Recall the notation after \eqref{LP2} for the projection $P_K$.

%Re-check in proof
%Adjust multiple $A$ factors in the bilinear estimate, Theorem \ref{theomu1}
%and \ref{theomu1'} according to \eqref{Propnu+3} for example. %Or redefine $A$\dots

\medskip
\item[(ii)] Define the ``good vector-field''
\begin{align}\label{theomu1X}
{\bf X} = \partial_{|\ell|} + \partial_{|m|}
\end{align}
and, for $a \leq 2$, %$M\in\Z$ 
define the operators
\begin{align}\label{theomu1TX}
\begin{split}
T_{{\bf X}^a}[b](g,h)(k) := \mathcal{F}^{-1}_{k\mapsto x} \iint_{\R^3\times\R^3} g(k-\ell) h(m) 
  \, b(k,\ell,m) %\\ \times %\varphi_M(m) 
  \, {\bf X}^a\nu_1(\ell,m) \, \mathrm{d}\ell \mathrm{d}m. %\qquad a=1,2,
\end{split}
\end{align}
Then
\begin{align}
\label{theomu1Xconc}
\begin{split}
{\big\| T_{{\bf X}^a}[b]\big(g,h\big) \big\|}_{L^r} 
  & \lesssim {\big\| \what{g} \big\|}_{L^p} {\big\| \what{h} \big\|}_{L^q} \cdot 2^{-a\min(L,M)}
    \cdot 2^{\max(L,M)} \cdot 2^{(\AfactorBB1 + 12)A} + 2^{-D}\mathcal{D}(g,h).
\end{split}
\end{align}
%where $L- = \min(L,0)$.
\end{itemize}

\end{theorem}

\medskip
Let us make a few comments about the statement of the theorem and its uses:

\setlength{\leftmargini}{2em}
\begin{itemize}

%\smallskip
%\item This Theorem is more general than what is actually needed for the application 
%to the nonlinear problem in this paper.
%In particular, the $L^r$ estimate \eqref{theomu1conc} will only be applied with $r=2$.

%However, it is needed in the proof of Theorem \eqref{theomu1} so that we can use duality to show 
%that the bound for $T_1$ implies the one for $T_2$.

\smallskip
\item Note that our operators are localized according to \eqref{theomu1asb1} and that factors
of $2^{\max(L,M)}$ and $2^A$ enter the final bound \eqref{theomu1conc}.
The power of $2^A$ represents a loss for high frequencies, due to the fact that 
we allow multipliers $b$ which are not standard ones, and satisfies estimates with losses \eqref{theomu1asb2}.
Even when $b=1$, our proof would give similar types of losses coming from the contribution of large frequencies.
The factor of $2^{\max(L,M)}$ is consistent with the homogeneity of $\nu_1$ 
and gives a useful gain for small frequencies.
%, for example, when %$\nu_1$ is singular 
%$L$ and $M$ are comparable and negative.

%By localizing in $|m|$ one could actually get a better bound but we won't need such a refinement here.

\smallskip
\item The choice of $D$ and $A$: in our application of the bilinear estimate \eqref{theomu1conc} to
the nonlinear evolution, a typical choice of the parameters will be, see \eqref{d_N},
\begin{align}\label{theomu1comm}
2^A \approx \jt^{6/N}, \qquad 2^D \approx \jt^{3}.
\end{align}
With this choice of $A$, there are only very small losses in the estimate \eqref{theomu1conc},
The choice of $D$ allows us to: (a) comfortably verify \eqref{theomu1asgh} 
for the various arguments $g$ and $h$ that we will encounter,
(b) treat the $2^{-D}$ factor in \eqref{theomu1conc} as a remainder which decays fast in time 
and can always be disregarded.
Moreover, with the choice \eqref{theomu1comm} the technical restriction $D \leq 2^A$ is clearly satisfied.

\smallskip
\item Note the compatibility of \eqref{theomu1asb2} 
with the properties of the coefficients $b_0$ and $b_{a,J}$ in \eqref{Propnu+1.1} and \eqref{Propnu+2.1}.

\smallskip
\item The estimates for the operators $T_{{\bf X}^a}$ where the good vectorfield \eqref{theomu1X} 
is applied to $\nu_1$ follow from the structural Proposition \ref{Propnu+},
and the proof for the case $a=0$.
The key point is that ${\bf X}f(|\ell|-|m|)=0$ for all $f$.
%The bound \eqref{theomu1Xconc} is consistent with simple localizations in $|\ell|\approx 2^L$ 
%(which is already contained in the assumption \eqref{theomu1asb1}) and $|m| \approx 2^M$.

\smallskip
\item The analogue of the good vectorfield \eqref{theomu1X} for the operator $T_2$, is the derivative $\partial_m$,
so this does not require a separate estimate as \eqref{theomu1TX}-\eqref{theomu1Xconc}.
\end{itemize}

\medskip
Theorem \ref{theomu1} is proved in Subsection \ref{prtheomu1}. 
Its proof will be done in several steps using as key ingredients the decomposition 
and the asymptotic formulas for $\nu_1$ in Proposition \ref{Propnu+}, and Lemma \ref{lemBEan} below.

In certain frequency configuration we will need to differentiate $\nu_1$ in
directions other than ${\bf X}$. The next Theorem establishes bilinear bounds 
for the relevant operators that will appear in the nonlinear analysis in Section \ref{secdkL2}.

\medskip
\begin{theorem}[Bilinear bounds with vectorfields 1]\label{theomu1'}
Under the assumption and notation 
of Theorem \ref{theomu1} the following additional bilinear estimates hold:

\begin{itemize}
\medskip
\item[(i)] %Statement about $T_1$ operator when $||\ell|-|m|| \gtrsim |\ell|$
For $a=(a_1,a_2)$, $1\leq |a| \leq 2$ %and $M\in\Z$, 
define the operators
\begin{align}\label{theomu1'1a}
\begin{split}
T_{\nabla^a}[b](g,h)(k) := \mathcal{F}^{-1}_{k\mapsto x} \iint_{\R^3\times\R^3} g(k-\ell) h(m) 
  \, b(k,\ell,m) \\ \times %\varphi_M(m) 
  \nabla^{a_1}_{\ell} \nabla^{a_2}_{m} \big[ \nu_1(\ell,m) \chi_+(\ell,m) \big]\, 
  \mathrm{d}\ell \mathrm{d}m
\end{split}
\end{align}
where
\begin{align}\label{theomu1'1b}
\chi_+(\ell,m) := \varphi_{\geq \max(L,M)-10}(|\ell|-|m|);
\end{align}
recall the notation \eqref{LP0} for the cutoffs.

\item
Then
\begin{align}\label{theomu1'conc1}
\begin{split}
{\big\| T_{\nabla^a}[b]\big(g,h\big) \big\|}_{L^r} 
  & \lesssim {\big\| \what{g} \big\|}_{L^p} {\big\| \what{h} \big\|}_{L^q}
    %\cdot 2^{-a_1L - a_2M} \cdot 2^{\max(L,M)} 2^{(\AfactorBB1 + 12)A} 
    \cdot %2^{(1-|a|)\min(L,M)} 
    2^{(1-|a|)M}  2^{(\AfactorBB1 + 12)A} + 2^{-D}\mathcal{D}(g,h).
\end{split}
\end{align}

%Note: %$\min(L,M)$ should be okay for the applications as well
% Needed better for the evolution estimates in Subsection \ref{lnotm}\dots
% %Actually could have just $2^{-(|a|-1)M}$ in analogy with estimate just below

\medskip
\item[(ii)] 
%Statement about $\partial_\ell^a \nu_1(\ell,k)$, $a=1,2$ and relative $T_2$ operators
For $1\leq |a| \leq 2$ as above, and $K\in\Z$, define the operators (we omit the $K$ dependence)
\begin{align}\label{theomu1'3a}
\begin{split}
T_{\partial^a}[b](g,h)(k) & := \mathcal{F}^{-1}_{k\mapsto x} \iint_{\R^3\times\R^3} g(-m-\ell) h(m) 
  \, b(k,\ell,m) \\ & \times \big[ \varphi_K(k) \nabla^{a_1}_{k} \nabla^{a_2}_{\ell} %\partial^a 
  \overline{\nu_1(\ell,k)} \chi_+(\ell,k) \big]
  \, \mathrm{d}\ell \mathrm{d}m, \qquad \mbox{with} \quad \partial \in \{ \partial_\ell, \partial_k\},
\end{split}
\end{align}
where (with a slight abuse of notation)
\begin{align}\label{theomu1'3b}
\chi_+(\ell,k) := \varphi_{\geq \max(L,K)-10}(|\ell|-|k|).
\end{align}
Then
\begin{align}
\label{theomu1'conc3}
\begin{split}
{\big\| T_{\partial^a}[b]\big(g,h\big) \big\|}_{L^r} 
  & \lesssim {\big\| \what{g} \big\|}_{L^p} {\big\| \what{h} \big\|}_{L^q} \cdot 
  %2^{-aL} \cdot 2^{\max(L,K)} 
  2^{(1-|a|)K} \cdot 2^{(\AfactorBB1+12)A} + 2^{-D}\mathcal{D}(g,h).
\end{split}
\end{align}

\medskip
\item[(iii)] 
%Statement about $\partial_k^a \nu_1(\ell,k)$, $a=1,2$ and relative $T_2$ operators
Let
\begin{align}\label{theomu1'2b}
{\bf Y} = \partial_k + \frac{k}{|k|} \big( \frac{\ell}{|\ell|} \cdot \partial_{\ell} \big)
\end{align}
and, for $a = 1, 2$, $K\in\Z$, define the operators 
\begin{align}\label{theomu1'2a}
\begin{split}
T_{{\bf Y}^a}[b](g,h)(k) := \mathcal{F}^{-1}_{k\mapsto x} \iint_{\R^3\times\R^3} g(-m-\ell) h(m) 
  \, b(k,\ell,m) \\ \times \big[ \varphi_K(k) {\bf Y}^a \overline{\nu_1(\ell,k)} \big]
  \, \mathrm{d}\ell \mathrm{d}m.
\end{split}
\end{align}
Then %, for $K < L$ we have
\begin{align}
\label{theomu1'conc2}
\begin{split}
{\big\| T_{{\bf Y}^a}[b]\big(g,h\big) \big\|}_{L^r} 
  & \lesssim {\big\| \what{g} \big\|}_{L^p} {\big\| \what{h} \big\|}_{L^q} \cdot 2^{(1-a)K}
    \cdot 2^{(\AfactorBB1 + 12)A} + 2^{-D}\mathcal{D}(g,h).
\end{split}
\end{align}
%and for $K > L$
%\begin{align}
%\label{theomu1'conc2K>L}
%\begin{split}
%{\big\| T_{{\bf Y}^a}[b]\big(g,h\big) \big\|}_{L^r} 
%  & \lesssim {\big\| \what{g} \big\|}_{L^p} {\big\| \what{h} \big\|}_{L^q} \cdot 2^{-K}
%    \cdot 2^{(\AfactorBB1 + 12)A} + 2^{-D}\mathcal{D}(g,h).
%\end{split}
%\end{align}
\end{itemize}
\end{theorem}

Theorem \ref{theomu1'} is proved in Subsection \ref{sectheomu1'}. 
Let us explain how we are going to use these estimates:

\begin{itemize}

 \smallskip
 \item Part (i) is used to prove bilinear bounds for operators of the $T_1$-type
 when the support is restricted away from the singularity of $\nu_1(\ell,m)$; see \ref{lnotm}.
 
 \smallskip
 \item Part (ii) is used similarly to (i) when dealing with operators of $T_2$-type 
 away from the singularity of $\nu_1(\ell,k)$; see \ref{ssecd_kN21}. %\ref{lnotk}.
 
 \smallskip
 \item The bounds in part (iii) are used to estimate $\partial_k^a T_2$; see Subsection \ref{ssecdkN_2}.
 In particular, we are going to use \eqref{theomu1'2b}-\eqref{theomu1'conc2}
 to transform $k$ derivatives of $\nu_1(\ell,k)$
 into $T_{\bf Y}$ operators plus operators involving more manageable $\ell$ derivatives.
 Notice that the operator in \eqref{theomu1'2a} has no restriction on the support in terms of $(|\ell|-|k|)^{-1}$,
 so we are dealing with the full singular kernel.
 The main point of \eqref{theomu1'conc2} is that using the vectorfield ${\bf Y}$
 does not increase the singularity in terms of the size of $|\ell|-|k|$.
 
 \smallskip
 \item Note how the bounds \eqref{theomu1'conc1} and \eqref{theomu1'conc2} 
 have a certain gain in terms of factors of $2^{-K}$:
 the application of ${\bf Y}^a$ only gives a factor $2^{-(a-1)K}$ instead of a more singular $2^{-aK}$.
 This type of gain is consistent %with the estimates for the symbols \eqref{Propnu+1.1} and \eqref{Propnu+2.1},
 %and 
 with the estimate \eqref{Propnu+3imp}, where a $\nabla_q^\beta$-derivative costs $2^{(1-|\beta|)Q_-}$.
 These bounds will be helpful in the nonlinear estimates of Section \ref{secdkL2}.
 %and Lemma \eqref{lemBER} below.

 %Add more explanation?
 
\end{itemize}

\medskip
\subsection{Bilinear operators supported on thin annuli}\label{secBE0}

%We begin by defining a useful norm in which we will measure our symbols

%\begin{definition}[The class of symbols $\mathcal{S}$]
%Have more than $1$ option. Still need to decide what to use exactly\dots
%\end{definition}

Let us first state a Lemma for bilinear operators with ``regular'' symbols.

\begin{lemma}[Bounds for regular symbols]
\label{lemBER}
For $L,M \in \Z$, consider the bilinear operator
\begin{align}
\label{lemBER1}
\begin{split}
B[b](g,h)(x) = \mathcal{F}^{-1}_{k\mapsto x} 
	\iint_{\R^3\times\R^3} g(\ell-k) h(m) \, b(k,\ell,m)
	\, \mathrm{d}\ell \mathrm{d}m,
\end{split}
\end{align}
under the assumptions
\setlength{\leftmargini}{2em}
\begin{itemize}

\medskip
\item For some $A\geq1$
\begin{align}\label{lemBERb1}
\supp(b) \subseteq \big\{ (k,\ell,m) \in \R^9\,:\, |k|+|\ell|+|m| \lesssim 2^A, 
  \, |k|\approx 2^K, \,|\ell| \approx 2^L, \,|m|\approx 2^M \big\};
\end{align}

\medskip 
\item The following estimate holds
\begin{align}\label{lemBERb2}
| \nabla_k^a \nabla^\alpha_\ell \nabla^\beta_m b(k,\ell,m)| 
	\lesssim 2^{-|a|K-|\alpha|L-|\beta|M} \cdot 2^{(|a|+|\alpha|+|\beta|)A},
	\qquad |a|, |\alpha|, |\beta|\leq 4.
\end{align}
\end{itemize}

\medskip
\noindent
Then, for $p,q,r\in[1,\infty]$, we have
\begin{align}
\label{lemBERconc}
{\|B[b](g, h) \|}_{L^r} \lesssim 2^{3\max(L,M)} \cdot 2^{8A} \cdot 
  {\| \what{g} \|}_{L^p} {\| \what{h} \|}_{L^q}, \qquad \frac{1}{p}+\frac{1}{q} = \frac{1}{r}.
\end{align}
%where $\widehat{f} = \widehat{\mathcal{F}}(f)$ denotes the (flat) Fourier transform of $f$.
\end{lemma}

The proof of Lemma \ref{lemBER}, which is more standard than that of Lemma \ref{lemBEan} below,
is given at the end of the section.

\medskip
\begin{lemma}[Bilinear operators restricted to small annuli 1]\label{lemBEan}
Let $j\geq 1$, consider the bilinear operator
\begin{align}
\label{lemBE11}
\begin{split}
B_{j}[b](g,h)(x) = \mathcal{F}^{-1}_{k\mapsto x} 
	\iint_{\R^3\times\R^3} g(\ell-k) h(m) \, b(k,\ell,m) \, \chi(2^j(|\ell| - |m|)) \, \mathrm{d}\ell \mathrm{d}m,
\end{split}
\end{align}
where $\chi$ is a Schwartz function.
Assume:

\setlength{\leftmargini}{2em}
\begin{itemize}

\medskip
\item For some $A\geq1$ and $L \gg -j$ we have
\begin{align}\label{lemBEanb1}
\supp(b) \subseteq \big\{ (k,\ell,m) \in \R^9\,:\, |k|+|\ell|+|m| \lesssim 2^A, \, |k|\approx 2^K, \,|\ell| \approx 2^L \big\};
\end{align}
%in particular we also have $2^{-j} \ll |m| \approx |\ell| \approx 2^L$.

\medskip 
\item The following estimate holds
\begin{align}\label{lemBEanb2}
| \nabla_k^a \nabla^\alpha_\ell \nabla^\beta_m b(k,\ell,m)| 
	\lesssim 2^{-|a|K-|\alpha|L - |\beta|M} \cdot 2^{(|a|+|\alpha|+|\beta|)A},
	\qquad |a|, |\alpha|, |\beta|\leq 4.
\end{align}
%Define norm? %Put $\mathcal{S}^\infty$ class?
\end{itemize}

\medskip
\noindent
Then, for $p,q,r\in[1,\infty]$, we have
\begin{align}
\label{lemBE12}
{\|B_{j}[b](g, h) \|}_{L^r} \lesssim 2^{-j} \cdot 2^{2L} \cdot 2^{8A} \cdot 
  {\| \what{g} \|}_{L^p} {\| \what{h} \|}_{L^q}, \qquad \frac{1}{p}+\frac{1}{q} = \frac{1}{r}.
\end{align}
%where $\widehat{f} = \widehat{\mathcal{F}}(f)$ denotes the (flat) Fourier transform of $f$.
\end{lemma}

The main conclusion of Lemma \ref{lemBEan} in the final bound  \eqref{lemBE12} is the $2^{-j+2L}$ factor
which gives a gain proportional to the volume of the annulus in which the support of the operator lies,
up to some small losses %(for large frequencies) 
due to the presence of the multiplier $b$.
The proof of Lemma \ref{lemBEan} is given in Subsection \ref{proofBEpre}.

\medskip
\subsection{Proof of Theorem \ref{theomu1}}\label{prtheomu1}
Let us begin by estimating the operator $T_1$ in \eqref{theomu11}. 
%The estimate for $T_2$ will be obtained by duality.

\medskip
\noindent
{\it Frequency localized estimate}.
We first claim that we can reduce the proof of the main conclusion \eqref{theomu1conc} to the following slightly 
stronger localized version:
\begin{align}\label{theomu1loc1}
\begin{split}
&{\big\| P_K T_1[b]\big(g, %\varphi_M 
  h\big) \big\|}_{L^r}
  \lesssim {\big\| \what{g} \big\|}_{L^p} {\big\| \what{h} \big\|}_{L^q} \cdot 2^{\max(L,M)}
  \cdot 2^{(C_0-1)A} + 2^{-D'},
\\
& \frac{1}{p}+\frac{1}{q} = \frac{1}{r}, \quad r\geq1, \qquad D':=D + \delta A %\delta(2A-M)
  , \quad \delta \ll 1,
\end{split}
\end{align}
where $A,L,M,D$ are as in the statement of the theorem.
Assume \eqref{theomu1loc1}, let $(p,q,r)$ be such that $1/p+1/q>1/r$ and, for $\delta\ll1$ as above,
let $1<r-\delta<r'<r$ %, $q<q'<q+\delta$ 
be such that %$1/p+1/q' = 1/r'$.
$1/p+1/q = 1/r'$. 
%
%Using the support information on $b$ we may assume that $h=h\varphi_{\sim M}$,
%where is a generic Littlewood-Paley cutoff at scale $2^M$, as defined after \eqref{LP2}.
%
%Littlewood-Paley decompositions, 
Using Bernstein's inequality, and \eqref{theomu1loc1} with exponents $r',p$ and $q$, we have
\begin{align*}
\begin{split}
{\big\| T_1[b]\big(g,h\big) \big\|}_{L^r} 
& \lesssim \sum_{K \leq A} {\big\| P_K T_1[b]\big(g,h\big) \big\|}_{L^r}
\lesssim \sum_{K\leq A} 2^{3(\frac{1}{r'}-\frac{1}{r})K} {\big\| P_K T_1[b]\big(g,h\big) \big\|}_{L^{r'}}
\\
& \lesssim 
	\sum_{K \leq A} 2^{3(\frac{1}{r'}-\frac{1}{r})K}
	\Big(   {\big\| \what{g} \big\|}_{L^p} \cdot
	{\big\| \what{h} \big\|}_{L^{q}}  \cdot  2^{\max(L,M)} \cdot 2^{(C_0-1)A} + 2^{-D'} \Big)
\\
& \lesssim 2^{\delta A} %\cdot \sum_{M \leq A} 
  %\Big( {\| \what{g} \|}_{L^p} \cdot  2^{\max(L,M)} \cdot 2^{(C_0-1)A} 
  %\cdot 2^{3(\frac{1}{q}-\frac{1}{q'})M} {\big\| \what{h} \big\|}_{L^q} 
  %+ 2^{-D + \delta M - 2\delta A} \Big)
%\\
%& \lesssim 2^{2\delta A} 
  {\| \what{g} \|}_{L^p} {\big\| \what{h} \big\|}_{L^q}  \cdot  2^{\max(L,M)} \cdot 2^{(C_0-1)A} + 2^{-D}.
\end{split}
\end{align*}
This implies the main estimate \eqref{theomu1conc}.
In the rest of the proof we then concentrate on the estimate \eqref{theomu1loc1}.

\medskip
\noindent
{\it Decomposition of $T_1$}.
Using the decomposition \eqref{Propnu+1} and defining
\begin{align}
\label{theomu13.0}
\nu_\delta(\ell,m) = \frac{b_0(\ell,m)}{|\ell|}\delta_0(|\ell|-|m|), 
  \qquad \nu_\pv(\ell,m) = \frac{b_0(\ell,m)}{|\ell|} \pv \frac{1}{|\ell|-|m|},
\end{align}
we reduce to proving the desired bound \eqref{theomu1loc1} for the operators
\begin{align}
\label{theomu13}
\begin{split}
T_{\nu}[b](g,h)(x) 
	:= \mathcal{F}^{-1}_{k\mapsto x} 
	\iint_{\R^3\times\R^3} g(k-\ell) h(m) \, b(k,\ell,m) \, \nu(\ell,m) \, \mathrm{d}\ell \mathrm{d}m,
	\\ \nu \in \{ \nu_\delta, \nu_\pv, \nu_L, \nu_R \}.
\end{split}
\end{align}
Recall that, in view of the support restrictions on $b$, we have $|\ell|\approx 2^L$ and $|m|\approx 2^M$,
and that $\nu_\delta$, $\nu_{\pv}$ and\footnote{Please note that this index $L$ has no relation
with the size of $\ell$.} $\nu_L$ are non-zero only when $|L-M|<5$.

\medskip
\noindent
{\it Estimate of $T_{\nu_\delta}$}.
By definition 
\begin{align}
\label{theomu15}
\begin{split}
\what{\mathcal{F}}\big(T_{\nu_\delta}[b](g,h)\big)(k) & = \lim_{\epsilon \rightarrow 0} \frac{1}{\epsilon} T_{\epsilon}(g,h)(k),
\\
T_{\epsilon}(g,h)(k) & :=\iint_{\R^3\times\R^3} g(k-\ell) h(m) 
  \, b(k,\ell,m) \frac{b_0(\ell,m)}{|\ell|}\varphi\Big(\frac{|\ell|-|m|}{\epsilon}\Big) \, \mathrm{d}\ell \mathrm{d}m,
\end{split}
\end{align}
where $\varphi$ is a smooth, even, positive and radially decreasing cutoff which equals $1$ close to the origin,
and whose integral is $1$.
In view of the properties of $b$, see \eqref{theomu11b}-\eqref{theomu12b} and of $b_0$, see \eqref{Propnu+1.1},
we may let $b_0\equiv 1$.
Moreover, we may consider $\epsilon \ll 2^L$ and write
(recall the notation for $\varphi_\sim$ after \eqref{LP2})
%Let us decompose
\begin{align}
T_{\epsilon}(g,h) & = %T_{\epsilon,1}(g,h)  + T_{\epsilon,2}(g,h)
%\\
%\label{theomu16.1}
%T_{\epsilon,1}(g,h)(k) & :=
\iint_{\R^3\times\R^3} g(k-\ell) h(m) 
 	\frac{b(k,\ell,m)}{|\ell|} \varphi_{\sim L}(\ell) \varphi_{\sim L}(m)
  	\varphi\Big(\frac{|\ell|-|m|}{\epsilon}\Big) \, %\varphi_{\geq 10}(|m|\epsilon^{-1}) 
  	\, \mathrm{d}\ell \mathrm{d}m.
%\\
%\label{theomu16.2}
%T_{\epsilon,2}(g,h)(k) & :=\iint_{\R^3\times\R^3} g(k-\ell) h(m) \,
%  \frac{b(k,\ell,m)}{|\ell|}\varphi\Big(\frac{|\ell|-|m|}{\epsilon}\Big) \, 
%  \varphi_{< 10}(|m|\epsilon^{-1}) \mathrm{d}\ell \mathrm{d}m,
\end{align}
%and estimate separately the two contributions.
%\medskip
%{\it Estimate for \eqref{theomu16.1}.}
%Since on the support of $T_{\epsilon,1}$ we have $|m| \geq 2^5\epsilon$ and $||\ell|-|m||\leq 4\epsilon$,
%and in view of the condition \eqref{theomu1asb1} on the support of $b$ , we can write
%\begin{align}
%\label{theomu17}
%\begin{split}
%T_{\epsilon,1}(g,h)(k) =
%  \iint_{\R^3\times\R^3} g(k-\ell) h(m) \varphi_{\sim L}(m) \varphi_{\sim L}(\ell)
%  \frac{b(k,\ell,m)}{|\ell|}
%\\
%  \times \varphi\Big(\frac{|\ell|-|m|}{\epsilon}\Big) \, \varphi_{\geq 10}(|m|\epsilon^{-1}) \, \mathrm{d}\ell \mathrm{d}m,
%\end{split}
%\end{align}
Using the notation of Lemma \ref{lemBEan}, and changing slightly the definition of $\varphi_{\sim L}(\ell)$,
we have, for $K\in\Z$,
\begin{align*}
& \what{\mathcal{F}}^{-1}
	\big(\varphi_K(k) \, T_{\epsilon}(g,h) \Big) = 2^{-L} \, T_{j}[c](g,h), \qquad 2^j = \epsilon^{-1},
\end{align*}
with
\begin{align*}
c(k,\ell,m) := \varphi_K(k) \varphi_{\sim L}(\ell)\varphi_{\sim L}(m) b(k,m,\ell) %\frac{2^L}{|\ell|}.
  %\varphi_{\geq 10}(|m|\epsilon^{-1}).
\end{align*}
The assumptions \eqref{lemBEanb1}-\eqref{lemBEanb2} of Lemma \ref{lemBEan} hold true for such $c$, 
%Show something?
%\begin{align}
%{\| b \|}_\mathcal{S} \lesssim 2^{-M}.
%\end{align}
and applying the conclusion \eqref{lemBE12} we obtain
\begin{align}
\label{theomu19.1}
\frac{1}{\epsilon} {\| \what{\mathcal{F}}^{-1}
	\big(\varphi_K(k) T_{\epsilon}(g,h) \big) \|}_{L^r} 
	\lesssim 2^L \cdot {\| \what{g} \|}_{L^p} {\| \what{h} \|}_{L^q} \cdot 2^{12A}.
\end{align}
This gives the desired bound \eqref{theomu1loc1} for 
${\| P_K T_{\nu_\delta}[b](g,\varphi_M h) \|}_{L^r}$. %, see \eqref{theomu15}.

\def\low{\mathrm{low}}
\def\med{\mathrm{med}}
\def\high{\mathrm{high}}

\medskip
\noindent
{\it Estimate of $T_{\pv}$}.
Recall the definition \eqref{theomu13.0}-\eqref{theomu13}. As above we may disregard the symbol $b_0$.
We define (omitting the dependence on fixed $K,L$ and $M$)
\begin{align}
\label{theomu20.0}
\begin{split}
T_{\epsilon,B}(g,h) := \what{\mathcal{F}}^{-1}_{k\mapsto x} 
  \iint_{| |\ell|-|m| | \geq \epsilon} g(k-\ell) h(m) \, \underline{b}(k,\ell,m) \, 
  \frac{1}{|\ell|} \frac{\varphi_{B}(|\ell|-|m|)}{|\ell|-|m|} \, \mathrm{d}\ell \mathrm{d}m,
\\
\underline{b}(k,\ell,m):= \varphi_K(k)\varphi_{M}(m) \varphi_{\sim L}(\ell)b(k,\ell,m).
\end{split}
\end{align}
Note that these are trivial if $B> L+10$.
We then decompose according to the size of the singularity:
\begin{align}
\label{theomu20}
& P_K T_{\pv}(g,\varphi_M h) = \lim_{\epsilon \rightarrow 0} \Big[ T_{\epsilon,\low}(g,h)
	+ T_{\epsilon,\med}(g,h) + T_{\epsilon,\high}(g,h) \Big], 
\\
\label{theomu20.1}
& T_{\epsilon,\low} := \sum_{B\leq B_0} T_{\epsilon,B}, \qquad B_0 := \min(-50(D'+10A), L-10,0)
	%\iint_{| |\ell|-|m| | \geq \epsilon} g(k-\ell) h(m) 
	%\frac{1}{|\ell|} \frac{1}{|\ell|-|m|} \varphi_{< D}(|\ell|-|m|) \, \mathrm{d}\ell \mathrm{d}m,
\\
\label{theomu20.2}
& T_{\epsilon,\med} := \sum_{B_0< B <L-10} T_{\epsilon,B},
	%:= \iint_{| |\ell|-|m| | \geq \epsilon} g(k-\ell) h(m) \frac{1}{|\ell|} \frac{1}{|\ell|-|m|}
	%\varphi_{[D,\log_2|\ell|-10]}(|\ell|-|m|) \, \mathrm{d}\ell \mathrm{d}m.
\\
\label{theomu20.3}
& T_{\epsilon,\high} := \sum_{B \geq L-10} T_{\epsilon,B}.
\end{align}

\medskip
{\it Estimate of \eqref{theomu20.1}.} 
On the support of \eqref{theomu20.1} we have $||\ell|-|m|| \lesssim 2^{B_0} \ll \min(2^{-D'},2^L)$
and we estimate it using the principal value and the regularity of the inputs. 
In particular, let us assume that the assumption \eqref{theomu1asgh} is fulfilled by
\begin{align}\label{theomu21as}
\mathcal{D}(g,h) = ({\| g \|}_{L^2} + {\| \partial_k g\|}_{L^2}) {\|h\|}_{L^2} \leq 2^{D}.
\end{align}

We split the function $g$ which appears as an argument 
in each term $T_{\epsilon,B}(g,h)$ of the sum \eqref{theomu20.1} as
\begin{align}\label{theomu20g}
g = g_1 + g_2, \qquad g_1 := P_{> X} g, \quad g_2 := P_{\leq X} g, \qquad X:= 2D' + 20A - (1/100)B.
\end{align}
%for $X$ to be determined below.

The pieces corresponding to the high frequency part of $g$ are estimated using Lemma \ref{lemBEan}:
%, and $B \leq \max(M,L)+5$:
\begin{align*}
\begin{split}
\sum_{B \leq B_0} {\| T_{\epsilon,B}(g_1,h) \|}_{L^1}
  & \lesssim \sum_{B\leq B_0} 2^L \cdot {\| \what{g_1} \|}_{L^2} {\| \what{h} \|}_{L^2} \cdot 2^{12A}
\\
 & \lesssim \sum_{B\leq B_0} 2^L \cdot 2^{-X} {\| \nabla g_1 \|}_{L^2} {\| \what{h} \|}_{L^2} \cdot 2^{12A}
 \\
 & \lesssim \sum_{B\leq B_0} 2^{-2D' + (1/100)B} \cdot 2^D
 \\
 & \lesssim 2^{-D'}.
\end{split}
\end{align*}
%having used the assumption \eqref{theomu1asgh} and the definition of $D'$.
The estimate for $L^r$, $r>1$, instead of $L^1$, is obtained by first applying Bernstein and then estimating
as above.

The contribution with $g_2$, see \eqref{theomu20g}, is handled using the principal value. 
We begin by writing
\begin{align}
\label{theomu21.0}
%T_{\epsilon,\low}(g_2,h) & = \sum_{B\leq B_0}
%\Big[ T_{B}^{(1)}(g_2,h) + T_{B}^{(2)}(g_2,h) + T_{B}^{(3)}(g_2,h) \Big],
T_{\epsilon,B}(g_2,h) = T_{B}^{(1)}(g_2,h) + T_{B}^{(2)}(g_2,h) + T_{B}^{(3)}(g_2,h),
\end{align}
where
\begin{align}
\label{theomu21}
\begin{split}
\what{T^{(1)}_B}(g,h)(k) & := \iint_{| |\ell|-|m| | \geq \epsilon} g(k-\ell) 
  h(m) \frac{\underline{b}(k,\ell,m) }{|\ell|^2} \varphi_{B}(|\ell|-|m|) \, \mathrm{d}\ell \mathrm{d}m,
\\
\what{T^{(2)}_B}(g,h)(k) & := \iint_{| |\ell|-|m| | \geq \epsilon} g(k-\ell)
  \big[ \underline{b}(k,\ell,m)- \underline{b}\big(k,|m|\ell/|\ell|,m) \big]
  \\ & \hskip50pt \times h(m)\frac{|m|}{|\ell|^2} \frac{1}{|\ell|-|m|} \varphi_{B}(|\ell|-|m|) 
  \, \mathrm{d}\ell \mathrm{d}m,
\\
\what{T^{(3)}_B}(g,h)(k) & := \iint_{| |\ell|-|m| | \geq \epsilon} 
  \big[g(k-\ell) - g\big(k - |m|\ell/|\ell|\big) \big] \, \underline{b}\big(k,|m|\ell/|\ell|,m) 
  \\ & \hskip50pt \times h(m)\frac{|m|}{|\ell|^2} \frac{1}{|\ell|-|m|} \varphi_{B}(|\ell|-|m|) 
  \, \mathrm{d}\ell \mathrm{d}m,
\end{split}
\end{align}
having used that
\begin{align}
\label{theomu22}
\begin{split}
& \iint_{| |\ell|-|m| | \geq \epsilon} g_2(k-\ell|m|/|\ell|) h(m) \frac{|m|}{|\ell|^2} 
  \frac{\underline{b}(k,|m|\ell/|\ell|,m)}{|\ell|-|m|}
  \varphi_{B}(|\ell|-|m|) \, \mathrm{d}\ell \mathrm{d}m
\\
& = \int_{\R^3_m}  h(m) |m| \int_{\mathbb{S}^2_\theta} g_2(k-\theta|m|) \underline{b}\big(k,|m|\theta,m)
  \\ & \times \Big[ \int_{| \rho-|m| | \geq \epsilon} \frac{1}{\rho-|m|}
  \varphi_{B}(\rho-|m|) \, \mathrm{d}\rho \Big] \, \mathrm{d} \theta \mathrm{d}m
  = 0.
\end{split}
\end{align}

For the first term in \eqref{theomu21} we use Lemma \ref{lemBEan} to estimate
\begin{align*}
{\| T^{(1)}_B(g_2,h) \|}_{L^r} 
%  & \lesssim 
%  {\| g \|}_{L^2} \cdot 2^{-2L} {\Big\| \varphi_L(\ell) \int_{| |\ell|-|m| | \geq \epsilon} |h(m)|
%  \varphi_{\leq B_0}(|\ell|-|m|)  \varphi_{M}(m) \, \mathrm{d}m \Big\|}_{L^1_\ell}
%  \\ 
%  & \lesssim {\| g \|}_{L^2} \cdot 2^{L} \sup_{|\ell|\approx 2^L}
%  \Big| \int_{| |\ell|-|m| | \geq \epsilon} |h(m)|
%  \varphi_{\leq B_0}(|\ell|-|m|)  \varphi_{M}(m) \, \mathrm{d}m \Big|
%  \\
%  & \lesssim {\| g \|}_{L^2} \cdot {\| h\|}_{L^2} \cdot 2^{L} \cdot 2^{(1/2)\min(B_0,M)} 2^{M}.
& \lesssim {\| \what{g} \|}_{L^p} {\| \what{h} \|}_{L^q} \cdot 2^{12A} \cdot 2^B
\end{align*}
Summing this bound over $B\leq B_0\leq L$ suffices.

The  term $ T^{(2)}_B(g_2,h)$ is similar by noticing that
\begin{align*} 
\Big|  \frac{|m|}{|\ell|-|m|} \big[b(k,\ell,m)- b\big(k,|m|\ell/|\ell|,m) \big] \Big| \lesssim 
  \min\Big( \frac{|m|}{|\ell|}, \frac{|m|}{||\ell|-|m||} \Big) \lesssim 1,
\end{align*}
with estimates for the derivatives matching the assumption \eqref{lemBEanb2}
of Lemma \ref{lemBEan}, see \eqref{theomu1asb2}, so that the same argument above applies.

To estimate the contribution from the last term in \eqref{theomu21} 
we take advantage of the restriction of $g_2$ to ``not too large'' frequencies. 
Again, it suffices to prove a (slightly stronger) $L^1$ bound and the $L^r$ bounds follow similarly.
We first rewrite
\begin{align}\label{theomu23.0}
\begin{split}
\what{T^{(3)}_B}(g_2,h)(k) & = \sum_{N\in\Z} \iint_{| |\ell|-|m| | \geq \epsilon} \varphi_N(k-\ell) h(m) 
  \, c(k,\ell,m) \, \frac{1}{|\ell|} \frac{\varphi_{B}(|\ell|-|m|)}{|\ell|-|m|} 
  \, \mathrm{d}\ell \mathrm{d}m,
\end{split}
\end{align}
where the symbol, which now involves $g_2$, is given by
\begin{align}\label{theomu23.1}
\begin{split}
c(k,\ell,m) := \big[g_2(k-\ell) - g_2\big(k - |m|\ell/|\ell|\big) \big] \, \underline{b}\big(k,|m|\ell/|\ell|,m\big).
%2^{M-L}.
% \\ \times \varphi_{\sim K}(k) \varphi_{\sim L}(\ell) \varphi_{\sim M}(m)
\end{split}
\end{align}
The key observation is that $c$ satisfies good symbol bounds, up to the usual small losses,
plus losses in terms of the parameter $X$ defined in \eqref{theomu20g}.
More precisely, using that, with $\theta := \ell/|\ell|$,
%\begin{align*}
%\big| g(k-\ell) - g\big(k - |m|\theta\big) \big| \lesssim  
%  \sqrt{|\ell|-|m|} {\| \partial_k g(k- r'\theta)\|}_{L^2(r\approx|\ell|)}
%\end{align*}
\begin{align*}
g_2(k-\ell) - g_2\big(k - |m|\theta\big) = 
  \int_0^1 \nabla g_2(k-t\ell - (1-t)|m|\theta) \,\mathrm{d}t \cdot \theta \, (|\ell| -|m|),
\end{align*}
we can rewrite
\begin{align}\label{theomu23.5}
\begin{split}
\what{T^{(3)}_B}(g_2,h)(k) & = \sum_{N\in\Z} \iint_{| |\ell|-|m| | \geq \epsilon} \varphi_N(k-\ell) h(m) 
  \, d(k,\ell,m) \, \frac{1}{|\ell|} \varphi_{B}(|\ell|-|m|) 
  \, \mathrm{d}\ell \mathrm{d}m,
\\
d(k,\ell,m) & := \int_0^1 \nabla g_2(k-t\ell + (1-t)|m|\theta) \,\mathrm{d}t \cdot \ell/|\ell| \,
 \, \underline{b}\big(k,|m|\ell/|\ell|,m\big)  % 2^{M-L}.
 %\\ & \qquad \qquad \times \varphi_{\sim K}(k) \varphi_{\sim L}(\ell) \varphi_{\sim M}(m)
\end{split}
\end{align}
In view of the assumptions on $b$, see \eqref{theomu1asb1}-\eqref{theomu1asb2} and $g_2 = P_{\leq X}g$,
we can see that
\begin{align}\label{theomu23.6}
\begin{split}
| \nabla_k^a \nabla^\alpha_\ell \nabla^\beta_m d(k,\ell,m)| 
  \lesssim 2^{-K|a|} 2^{-|\alpha|L}2^{-|\beta|M} \cdot 2^{(|a|+|\alpha|+|\beta|)A} %2^{M-L} \\ \times 
  \cdot 2^{(|a|+|\a|+|\b|+2)X} {\|\nabla g\|}_{L^2}
\end{split}
\end{align}
for $|a|,|\a|,|\b|\leq 4$.
Using \eqref{theomu23.6} to apply Lemma \ref{lemBEan} to \eqref{theomu23.5}, we obtain
\begin{align*} 
{\| T^{(3)}_B(g_2,h) \|}_{L^1} & \lesssim \sum_{N\in\Z} 2^{15X} {\|\nabla g\|}_{L^2} %2^{M-L} 
\cdot 2^B 2^{12A} 2^L \cdot {\| \varphi_N \|}_{L^2} {\| h \|}_{L^2}
\\
& \lesssim 2^{15X} \cdot 2^B \cdot 2^D \cdot 2^{14A} 
\end{align*}
having used the assumption \eqref{theomu21as}, $\| \varphi_N \|_{L^2}\approx2^{3N/2}$ and $N\leq A+5$. 
Recalling the definition of $X$ from \eqref{theomu20g},
and summing over $B\leq B_0$, where $B_0$ is given in \eqref{theomu20.1}, we get
\begin{align*} 
\sum_{B\leq B_0} {\| T^{(3)}_B(g_2,h) \|}_{L^1} & 
	\lesssim 2^D \cdot 2^{14A} \sum_{B\leq B_0} 2^{15X} \cdot 2^B
	\\ & \lesssim 2^{D + 30D' + 320A} \cdot 2^{5B_0/6}
	\\ & \lesssim 2^{-D'}.
\end{align*}
This bound completes the estimate for the term \eqref{theomu20.1} provided \eqref{theomu21as} holds.
If the assumption \eqref{theomu1asgh} is fulfilled
with ${\| g\|}_{L^2} ({\|h\|}_{L^2} + {\|\partial_k h\|}_{L^2}) \leq 2^{-D}$ instead,
we can use a similar argument exchanging the roles of $g$ and $h$ 
and using the following replacement of \eqref{theomu22}:
\begin{align*}%\label{theomu22'}
\begin{split}
& \iint_{| |\ell|-|m| | \geq \epsilon} g(k-\ell) h(m|\ell|/|m|) \frac{|\ell|}{|m|^2} 
  \frac{\underline{b}\big(k,\ell,m|\ell|/|m|)}{|\ell|-|m|}
  \varphi_{\leq B_0}(|\ell|-|m|) \, \mathrm{d}\ell \mathrm{d}m
\\
& = \int_{\R^3_\ell}  g(k-\ell) |\ell| \int_{\mathbb{S}^2_\theta} h(\theta|\ell|) \underline{b}\big(k,\ell,|\ell|\theta)
  \\ & \times \Big[ \int_{||\ell| - \rho| \geq \epsilon} \frac{1}{|\ell|-\rho}
  \varphi_{\leq B_0}(|\ell|-\rho) \, \mathrm{d}\rho \Big] \, \mathrm{d} \theta \mathrm{d}\ell
  = 0.
\end{split}
\end{align*}

\medskip
{\it Estimate of \eqref{theomu20.2}.}
For this term we can use directly Lemma \ref{lemBEan}.
Since by assumption $2^{B_0} \lesssim ||\ell|-|m|| \approx 2^B \ll |\ell| \approx 2^L \approx |m| \lesssim 2^A$
on the support of the integral, there are at most $\sim (A + D)$ possible indexes $B$, see \eqref{theomu20.1}.
Using \eqref{lemBE12} we can estimate
\begin{align*}
\sum_{B_0 \leq B \leq L-10} 
{\| T_{\epsilon,B}(g,h)\|}_{L^r} 
	& \lesssim (|A|+|D|) \cdot 2^L \cdot  2^{12A} \cdot {\| \what{g} \|}_{L^p} {\| \what{h}\|}_{L^q}
	\\
	& \lesssim 2^L \cdot  2^{13A} \cdot {\| \what{g} \|}_{L^p} {\| \what{h}\|}_{L^q}
\end{align*}
having used also $D \lesssim 2^{A/10}$. 
%This bound is consistent with the desired estimate \eqref{theomu1conc}.

%when summed over $K \geq -2(D+A)$.
%For $K \leq -2(D+A)$ we can apply the same argument we used to obtain \eqref{theomu19.2} and estimate 
%\begin{align*} 
%{\| \varphi_K(k) T_{\epsilon,B,M}(g,h)(k) \|}_{L^2}
%	\lesssim 2^{3K/2} {\| g \|}_{L^2} \cdot 2^{L} {\| h \|}_{L^2}.
%\end{align*} 
%Summing over $K \leq -2(D+A)$ 
%we see that this contribution can be absorbed in the remainder term on the right-hand side of \eqref{theomu1conc}.

\medskip
{\it Estimate of \eqref{theomu20.3}.}
To estimate this term we notice that on the support of the integral we have
$2^B \approx ||\ell|-|m|| \gtrsim |\ell| \approx 2^L$ so that the kernel is not singular. 
In particular, we have
%\begin{align}
%& T_{\epsilon,B,M} (g,h) = 2^{-L-B} \, T_{j}[c](g,h), \qquad 2^j = B,
%\end{align}
\begin{align}\label{theomu25}
& T_{\epsilon,B} (g,h) = 2^{-L} 2^{-\max(L,M)} \iint_{| |\ell|-|m| | \geq \epsilon} g(k-\ell) h(m) \,c(k,\ell,m)\, \mathrm{d}\ell \mathrm{d}m,
\end{align}
with
\begin{align}\label{theomu25.1}
c(k,\ell,m) := b(k,\ell,m) \,\varphi_{K}(k) \varphi_{M}(m) \varphi_{\sim L}(\ell) 
  \frac{2^{\max(L,M)}}{|\ell|-|m|} \varphi_{\sim B}(|\ell|-|m|),
  \qquad B \geq L-10.
\end{align}
%As before, we decompose dyadically in $|k| \approx 2^K$,
%and estimate the sum over the terms with $K \leq -2(D+A)$ as we did for \eqref{theomu19.2}.
%We are then left with at most $\sim |A| + |D|$ indexes $K$.
We verify that for all $|k|\approx 2^K$, and $|M-L|<5$ %\leq B+10$,
\begin{align}\label{theomu25.2}
|\nabla_k^a \nabla_\ell^\alpha \nabla_m^\beta c(k,\ell,m)|
  \lesssim 2^{-K|a|} 2^{-L|\alpha|} 2^{-M|\beta|} 2^{(|a|+|\alpha|+|\beta|)A}, \qquad |a|,|\alpha|,|\beta|\leq 4.
\end{align}
%\begin{align*}
%{\Big\| \varphi_{\sim M}(m) 
%  \frac{\varphi_L(\ell)}{|\ell|} \frac{\varphi_{A}(|\ell|-|m|)}{|\ell|-|m|} \Big\|}_{\mathcal{S}} \lesssim 2^{-L} 2^{2M}.
%\end{align*}
Using Lemma \ref{lemBER} we can then estimate, for each fixed $B \in [L-10,L+10]$,
\begin{align}\label{theomu26}
{\| T_{\epsilon,B} (g,h) \|}_{L^r} \lesssim 2^{-2\max(L,M)}
	\cdot {\| \what{g} \|}_{L^p} {\| \what{\varphi_{\sim M} h} \|}_{L^q} 
	\cdot 2^{3\max(L,M)} \cdot 2^{12A}.
\end{align}
%Summing over $B$ 
which gives the desired \eqref{theomu1loc1}.

%If $\min(L,M) \geq -4(D+A)$ we use ${\| \what{\varphi_{\sim M} h} \|}_{L^q} \lesssim {\| \what{h} \|}_{L^q}$ 
%and sum the bound \eqref{theomu26} over the at most $\sim |A|+|D|$ indexes $M$ and $B$ 
%to deduce the desired estimate.
%If instead $\min(L,M) \leq -4(D+A)$, from \eqref{theomu26} we can bound
%\begin{align*}
%{\| \varphi_K(k) T_{\epsilon,B,M} (g,h) \|}_{L^2} 
%  & \lesssim 2^{L} \cdot {\| \what{g} \|}_{L^2} {\| \what{\varphi_{\sim M} h} \|}_{L^\infty} \cdot 2^{8A}
%\\ 
%& \lesssim 2^L {\|g \|}_{L^2}  \cdot 2^{3M/2} {\|h \|}_{L^2}  \cdot 2^{8A}
%\end{align*}
%and sum this over $B$ with $L-10 \leq B \leq A+10$ and $M$ subject to $M \leq A+10$, 
%and include this term into the remainder om the right-hand side of \eqref{theomu1conc}

\medskip
\noindent
{\it Estimate of $T_{\nu_L}$}.
We now want to estimate by the right hand-side of \eqref{theomu1conc} the term $T_{\nu_L}$,
see \eqref{theomu13}. Since $\nu_L$ satisfies \eqref{Propnu+2}--\eqref{Propnu+2.1},
we can write $T_{\nu_L}[b](g,h)$ as a finite sum of terms of the form
\begin{align}\label{theomu30}
\begin{split}
& \sum_{J\in\mathcal{J}} T_{J}(g,h)[b](k),
\\
& \what{\mathcal{F}}(T_{J}(g,h)[b])(k) := \iint_{\R^3\times\R^3} g(k-\ell) h(m) \, b(k,\ell,m) 
  \frac{1}{|\ell|} b_{J}(\ell,m) \cdot 2^J K\big(2^J(|\ell|-|m|)\big)  \, \mathrm{d}\ell \mathrm{d}m
\end{split}
\end{align}
where $K$ is a Schwartz function and, for all $|\ell| \approx 2^{L}$ and $|m| \approx 2^M$, 
the symbols satisfy 
\begin{align}\label{theomu31}
\begin{split}
& \big| \nabla_\ell^\alpha \nabla_m^\beta b_{J}(\ell,m) \big| \lesssim 
  2^{-|\a|L} 2^{-|\b|M} \cdot 2^{(|\a|+|\b|)A} \cdot C_J, \qquad |\alpha|,|\beta| \leq 5,
\\
& \sum_{J\in\mathcal{J}}  C_J \leq 1.
\end{split}
\end{align}
For each term $T_J$ in \eqref{theomu30} we can apply the same arguments used to estimate
$T_\pv$, see \eqref{theomu13.0}-\eqref{theomu13}, based on Lemma \ref{lemBEan}
and, in view of the bounds on the symbols in \eqref{theomu31}, we can obtain
\begin{align*}
{\| T_{J}(g,h)[b] \|}_{L^r} \lesssim 2^L \cdot {\|\what{g}\|}_{L^p} {\|\what{h}\|}_{L^q} \cdot 2^{12A} \cdot C_J.
\end{align*}
Summing over $J$ using \eqref{theomu31} we obtain \eqref{theomu1conc}.

\medskip
\noindent
{\it Estimate of $T_{\nu_R}$}.
Recall the notation \eqref{theomu13.0}-\eqref{theomu13} 
and the definition and properties of $\nu_R$ in Proposition \ref{Propnu+}.
In both cases $|L-M|<5$ and $|L-M|\geq 5$, we use \eqref{Propnu+3} and \eqref{Propnu+3imp} respectively,
to write
\begin{align}\label{theomu40}
\begin{split}
& \what{\mathcal{F}}(T_{\nu_R}[b](g,h))(k) 
  = 2^{-2\max(L,M)} \iint_{\R^3\times\R^3} g(k-\ell) h(m) \, b(k,\ell,m) d(\ell,m) \, \mathrm{d}\ell \mathrm{d}m
\end{split}
\end{align}
with
\begin{align*}
\big| \nabla_\ell^\alpha \nabla_b^\beta d(\ell,m)\big| 
  \lesssim 2^{-|\a|L} 2^{-|\b|M} \cdot 2^{(|\a|+|\b|+2)5A} 2^A, \qquad |\alpha|,|\beta|\leq 4,
\end{align*}
for $|\ell| \approx 2^L$ and $|m|\approx 2^M$.
Then, the bilinear term in \eqref{theomu40} is similar to the one in \eqref{theomu25}-\eqref{theomu25.2},
up to %a slightly different pre-factor, and 
the different power of $2^A$.
From the same argument %given after \eqref{theomu25.2} 
above, using Lemma \ref{lemBER}, it then follows that
\begin{align}
{\| T_{\nu_R}[b](g,h) \|}_{L^r} \lesssim %2^{-2L}
	\cdot {\| \what{g} \|}_{L^p} {\| \what{h} \|}_{L^q} \cdot 2^{\max(L,M)} \cdot 2^{60A} %+ 2^{-D},
\end{align}
which gives \eqref{theomu1conc}.

%In the case $|L-M|\geq5$ we instead resort to \eqref{Propnu+3imp}, and write
%\begin{align}\label{theomu40imp}
%\begin{split}
%& \what{\mathcal{F}}(T_{\nu_R}[b](g,h))(k) 
%  = 2^{-2L} \iint_{\R^3\times\R^3} g(k-\ell) h(m) \, b(k,\ell,m) d(\ell,m) \, \mathrm{d}\ell \mathrm{d}m
%\end{split}
%\end{align}
%with
%\begin{align*}
%\big| \nabla_\ell^\alpha \nabla_b^\beta d(\ell,m)\big| 
%  \lesssim 2^{-|\a|\max(L,M)} \cdot 2^{-(1-|\b|)M} \cdot 2^{(|\a|+|\b|+2)5A}, \qquad |\alpha|,|\beta|\leq 4,
%\end{align*}
%for $|\ell| \approx 2^L$ and $|m|\approx 2^M$.

\medskip
\noindent
{\it Estimate of $T_2$.}
To prove that the operator $T_2$ defined in \eqref{theomu12b} also satisfies the bound \eqref{theomu1conc},
one can use a similar proof to the one above done for $T_1$.
The main ingredient needed is a version of Lemma \ref{lemBEan}
and the main bound \eqref{lemBE12} adapted to the kernel in the operator $T_2$.
This can be seen by duality.
More precisely, in analogy with the notation for the operator $B_{j}[b](g,h)(x)$ in \eqref{lemBE11} %of Lemma \ref{lemBEan} 
we can define
\begin{align}
\label{lemBE11'}
\begin{split}
B_{j,2}[b](g,h)(x) = \mathcal{F}^{-1}_{k\mapsto x} 
	\iint_{\R^3\times\R^3} g(-m-\ell) h(m) \, b(k,\ell,m) \, \chi(2^j(|\ell| - |k|)) \, \mathrm{d}\ell \mathrm{d}m.
\end{split}
\end{align}
Using Plancharel (omitting the irrelevant factors of $\pi$), we have the following identities
\begin{align}
\begin{split}
& \langle B_{j,2}[b](g,h), f\rangle_{L^2(\R^d)} 
  = \int_{\R^3} \iint_{\R^3\times\R^3} g(-m-\ell) h(m) \,b(k,\ell,m)\, 
  \chi(2^j(|\ell| - |k|)) \, \mathrm{d}\ell \mathrm{d}m \, \bar{\what{f}(k)} \, \mathrm{d}k
  %\\
  %& = \int_{\R^3} \Big[ \iint_{\R^3\times\R^3} g(m-\ell) \bar{\what{f}(k)} \,b(k,\ell,-m)\, 
  %\chi(2^j(|\ell| - |k|)) \, \mathrm{d}\ell \mathrm{d}k \Big] \,  h(-m)  \, \mathrm{d}m
  \\
  & = \int_{\R^3} \Big[ \iint_{\R^3\times\R^3} g(k-\ell) \bar{\what{f}(-m)} \,b(m,\ell,-k)\, 
  \chi(2^j(|\ell| - |m|)) \, \mathrm{d}\ell \mathrm{d}m \Big] \,  h(-k)  \, \mathrm{d}k
  \\
  & = \langle \bar{ \what{\mathcal{F}} B_j\big[b'\big]\big(\bar{g},
  \what{\bar{f}} \big) }, \overline{h(-\cdot)} \rangle_{L^2} 
  = \langle \bar{ B_j\big[b'\big]\big(\bar{g},\what{\bar{f}} \, \big) }, \overline{\what{h}} \rangle_{L^2} 
\end{split}
\end{align}
where, according to the notation \eqref{lemBE11}, $b'(k,\ell,m) := \bar{b(m ,\ell, -k)}$.
Then, assumptions analogous to \eqref{lemBEanb1}-\eqref{lemBEanb2} hold for $b'$.
From \eqref{lemBE12} it follows that for any $1/r=1/p+1/q$, $g\in L^p, h\in L^q$ and $f\in L^{r'}$
\begin{align*}
\big| \langle B_{j,2}[b](g,h), f\rangle_{L^2(\R^d)} \big|
  %\lesssim \big| \langle \bar{ B_j\big[b'\big]\big(\bar{g},\what{f}\big) }(-\cdot), \check{h} \rangle_{L^2} \big|
  & \lesssim {\| \check{h} \|}_{L^q} {\| B_j\big[b'\big]\big(\bar{g},\what{\bar{f}}\big) \|}_{L^{q'}}
  \\ & \lesssim {\| \check{h} \|}_{L^q} \cdot 2^{-j} 2^{2L} 2^{8A} {\| \what{\bar{g}}\|}_{L^p} {\|f\|}_{L^{r'}}.
\end{align*}
This give the desired bound by the right-hand side of \eqref{lemBE12} for the operator $B_{j,2}$.

\medskip
{\it Proof of \eqref{theomu1Xconc}.}
To conclude we show how to estimate the operators \eqref{theomu1TX}, where the ``good vectorfield'' \eqref{theomu1X}
acts on the measure $\nu_1$.
By the statement of Proposition \ref{Propnu+}, and since ${\bf X}_{\ell,m} f(|\ell|-|m|) = 0$ for any $f$, 
an application of ${\bf X}^a$ to $\nu_1$ gives:
\begin{align}\label{Propnu+1X}
\begin{split}
{\bf X}^a \nu_1(\ell,m) & = {\bf X}^a \nu_0(\ell,m) + {\bf X}^a \nu_L(\ell,m) + {\bf X}^a \nu_R(\ell,m)
  \\ & = \nu_{0,a}(\ell,m) + \nu_{L,a}(\ell,m) + \nu_{R,a}(\ell,m),
\end{split}
\end{align}
where:

\medskip
\begin{itemize}
\item[(1)] The leading order has the form
\begin{align*}%\label{Propnu+1.0}
\nu_{0,a}(\ell,m) = \frac{b_{0,a}(\ell,m)}{|\ell|} %\Big( {\bf X}^a \frac{b_0(\ell,m)}{|\ell|} \Big) 
  \Big[ \delta(|\ell|-|m|) + \pv \frac{1}{|\ell|-|m|} \Big]
\end{align*}
with
\begin{align}\label{Propnu+1.1X}
\begin{split}
  \big| \nabla_\ell^\alpha \nabla_m^\beta \big( \varphi_L(\ell) \varphi_M(m) b_{0,a}(\ell,m) \big) \big| 
  \lesssim 2^{-a\min(L,M)} \cdot 2^{-|\a|L} 2^{-|\b|M} \\ \cdot 2^{(|\a|+|\b| + a)A} \cdot \mathbf{1}_{\{|L-M|< 5\}}
\end{split}
\end{align}
for all $L,M \leq A$, $|\a|+|\b| + a \leq N_1$.

\medskip
\item[(2)] $\nu_{L,a}(\ell,m)$ can be written as
\begin{align*}%\label{Propnu+2}
\nu_{L,a}(\ell,m) = \frac{1}{|\ell|} 
  \sum_{i=1}^{N_2} \sum_{J \in \mathcal{J}} b_{i,J}^a(\ell,m) \cdot 2^J K_i\big(2^J(|\ell|-|m|)\big) 
\end{align*}
with $K_i \in \mathcal{S}$ and
\begin{align}\label{Propnu+2.1X}
\begin{split}
%b_a \in \mathcal{G}^{N_1-a}
\sum_{J\in\mathcal{J}} 
  \big| \nabla_\ell^\alpha \nabla_m^\beta \big( \varphi_L(\ell) \varphi_M(m) b_{i,J}^a(\ell,m) \big) \big| \lesssim 
  2^{-a\min(L,M)} 2^{-|\a|L} 2^{-|\b|M} \\
  \cdot 2^{(|\a|+|\b| + a)A} \cdot \mathbf{1}_{\{|L-M|< 5\}},
\end{split}
\end{align}
for all $L,M \leq A$, $|\a|+|\b| + a \leq N_1-N_2$.

\medskip
\item[(3)] The remainder term satisfies
\begin{align}\label{Propnu+3X}
\begin{split}
\big| \nabla_\ell^\alpha \nabla_m^\beta \big(\varphi_L(\ell) \varphi_M(m) \nu_{R,a}(\ell,m)\big) \big| 
  \lesssim 2^{-a\min(L,M)} \cdot 2^{-2\max(L,M)} \cdot 2^{-|\a|L} 2^{-|\b|M} \\ \cdot 2^{(|\a|+|\b|+2+a)5A}2^{2A}
  %\max(1,|p|^{-1},|q|^{-1})^{|\alpha|+|\beta|}  {(1+|q|)}^{|\alpha| + N_2+1},
\end{split}
\end{align}
for all $L,M \leq A$ and $|\a|+|\b|+a \leq N_2/2 -3$.
Note how we do not use here the improvement of \eqref{Propnu+3} given by \eqref{Propnu+3imp} here.

\end{itemize}

Therefore the components of ${\bf X}^a \nu_1$ in \eqref{Propnu+1X}
have the same structure of the components of $\nu_1$,
where the bounds \eqref{Propnu+1.1X}-\eqref{Propnu+3X} on the coefficients
are like the bounds in the case of $\nu_1$ times a factor of $2^{-a\min(L,M)}$.
Then, the same proof used to obtain \eqref{theomu1conc} can be applied and \eqref{theomu1Xconc} follows.
$ \hfill \Box$

\medskip
\subsection{Proof of Theorem \ref{theomu1'}}\label{sectheomu1'}

\medskip
{\it Proof of (i)}.
To prove \eqref{theomu1'conc1} it suffices to examine the structure and bounds satisfied by
$\nabla^{a_1}_{\ell} \nabla^{a_2}_{m} [ \nu_1(\ell,m) \chi_-(\ell,m)]$.
First note that when we differentiate the cutoff the behavior is
\begin{align*}
\nabla^{b_1}_{\ell} \nabla^{b_2}_{m} \chi_-(\ell,m) 
  \approx 2^{-(b_1+b_2)\max(L,M)} \varphi'_{\max(L,M)-10}(|\ell|-|m|).
\end{align*}
%see \eqref{theomu1'1a}; 
The bilinear terms corresponding to these symbols are easier to treat than those where the derivatives hit $\nu_1$,
so we concentrate on these latter ones.
According to the decomposition \eqref{Propnu+1} from Proposition \ref{Propnu+} we write
%On the support of $\chi_-$ we have $\nu_0(\ell,m) = b_0(\ell,m)|\ell|^{-1} (|\ell| - |m|)^{-1}$
\begin{align}\label{theomu1'pr1} 
\begin{split}
2^{-2A} 2^{(|a|-1)M} %2^{2\max(L,M)} \, 2^{a_1L} 2^{a_2M} 
  \nabla^{a_1}_{\ell} \nabla^{a_2}_{m} \nu_1(\ell,m) \, \chi_-(\ell,m) 
\\ 
= \big[ \nu_0^a(\ell,m) %\mathbf{1}_{\{|L-M|<5\}} 
  + \nu^a_1(\ell,m) %\mathbf{1}_{\{|L-M|<5\}} 
  + \nu^a_2(\ell,m) \big]  \, \chi_-(\ell,m)
\end{split}
\end{align}
with
\begin{align}\label{theomu1'pr2} 
\begin{split}
\nu_0^a(\ell,m) & := 
  2^{-2A} 2^{(|a|-1)M} %2^{a_1L} 2^{a_2M} 
  \nabla^{a_1}_{\ell} \nabla^{a_2}_{m} 
  \Big( \frac{b_0(\ell,m)}{|\ell|} \frac{1}{|\ell|-|m|} \Big)
  \mathbf{1}_{\{|L-M|<5\}},
\\
\nu_1^a(\ell,m) & := 2^{-2A} 2^{(|a|-1)M} %2^{a_1L} 2^{a_2M} 
  \nabla^{a_1}_{\ell} \nabla^{a_2}_{m} \nu_L(\ell,m)
  \mathbf{1}_{\{|L-M|<5\}},
\\
\nu_2^a(\ell,m) & := 2^{-2A} 2^{(|a|-1)M} %2^{a_1L} 2^{a_2M}
  \nabla^{a_1}_{\ell} \nabla^{a_2}_{m} \nu_R(\ell,m).
\end{split}
\end{align}
%To prove the desired bilinear bound \eqref{theomu1'conc1} 
It will suffices to show that the terms in \eqref{theomu1'pr2} satisfy
\begin{align}
\label{theomu1'pr3} 
|\nabla^{\alpha}_{\ell} \nabla^{\beta}_{m} \nu^a_\ast(\ell,m) |
  \lesssim 2^{-3\max(L,M)} \cdot 2^{-|\alpha|L} 2^{-|\alpha| M} 2^{5A(|\alpha|+|\beta|+4)}, \qquad \ast\in\{0,1,2\},
\end{align}
and then apply Lemma \ref{lemBER} to get the desired bound \eqref{theomu1'conc1}.
%, cfr. \eqref{theomu1'1a} and \eqref{lemBER1} and the assumptions \eqref{lemBERb1}-\eqref{lemBERb2}

To verify \eqref{theomu1'pr3} for $\ast = 0$, 
we recall the estimates \eqref{Propnu+1.1} for the coefficient $b_0$, which in particular imply
\begin{align*}
|\nabla^{\alpha}_{\ell} \nabla^{\beta}_{m} b_0(\ell,m) |
  \lesssim 2^{-|\alpha|L} 2^{-|\beta|M} 2^{A(|\alpha|+|\beta|)}.
\end{align*}
%Since, on the support of $\chi_-$ we have $||\ell|-|m|| \gtrsim \max(|m|,|\ell|)$, then
Moreover, since on the support of $\nu_0^a$ we have $||\ell|-|m|| \approx \max(|m|,|\ell|)$, it follows that
\begin{align*}
|\nabla^{\alpha}_{\ell} \nabla^{\beta}_{m} \frac{1}{|\ell|} \frac{1}{|\ell|-|m|} |
  \lesssim 2^{-L} 2^{-\max(L,M)} \cdot 2^{-|\alpha|L} 2^{-\max(L,M)} 2^{-|\beta|M}. %2^{-\max(L,M)} 2^{-(|\beta|-1)M}
\end{align*}
Combining the last two inequalities, the estimate \eqref{theomu1'pr3} for $\nu_0^a$ follows.

The case of $\nu_L^a$ can be treated similarly by differentiating the formula for $\nu_L$
in \eqref{Propnu+2}, using the symbol bounds \eqref{Propnu+2.1} in the form used just above, 
and the fact that for any Schwartz function $K$ we have, on the support of $\chi_-(\ell,m)$,
\begin{align*}
\big|\nabla^{\alpha}_{\ell} \nabla^{\beta}_{m} 2^J K\big(2^J(|\ell|-|m|)\big) \big| 
  %\lesssim 2^{-\max(L,M)} \cdot 2^{-\max(L,M)} 2^{-(|\alpha|-1)L-} \cdot 2^{-\max(L,M)} 2^{-(|\beta|-1)M-}
  \lesssim 2^{-\max(L,M)} \cdot 2^{-|\alpha|L} \cdot 2^{-|\beta|M}.
\end{align*}

The bound \eqref{theomu1'pr3} for $\nu_2^a$ follows from the estimates \eqref{Propnu+3} 
%in the case $|L-M|<5$.
and \eqref{Propnu+3imp} for $\nu_R$.

%show that the three terms in \eqref{theomu1'pr2} satisfy, respectively, 
%similar or better bounds than $\nu_0,\nu_L$ and $\nu_R$,
%see \eqref{Propnu+1.1}, \eqref{Propnu+2.1} and \eqref{Propnu+3}, and then apply Theorem \ref{theomu1}.
%
%\begin{align*}
%\Big| \nabla^{a_1}_{\ell} \nabla^{a_2}_{m} \Big[ \frac{b_0(\ell,m)}{|\ell|} \frac{1}{|\ell|-|m|}\Big] \Big|
%  \lesssim 2^{-L} 2^{-\max(L,M)} \big( 2^{-a_1L}( 2^{a_1 M} + 2^{(1-a_2)M_-}) + 2^{-a_1L} 2^{-a_2\max(L,M)}\big) 
%  \\ \lesssim 2^{-L} 2^{-\max(L,M)} \big( 2^{-a_1L}( 2^{a_1 M} + 2^{(1-a_2)M_-}) + 2^{-a_1L} 2^{-a_2\max(L,M)}\big) 
%\end{align*}

\medskip
{\it Proof of (ii)}.
%Notice that \eqref{theomu1'conc3} only involves derivatives in the variable $\ell$ of $\nu_1(\ell,k)$.
%Then it is not hard to see that $\partial_\ell^a \nu_1(\ell,k)$ behaves at worst %as a symbol
%like $\nu_1(\ell,k)$ multiplied by a factor of $2^{-aL}$, since we are restricting 
%to $||\ell| -|k|| \gtrsim |\ell|+|k|$.
The proof of \eqref{theomu1'conc3} can be done similarly to part (i) above, so we skip it.

\medskip
{\it Proof of (iii)}.
Our aim is to prove that ${\bf Y}^a \nu_1(\ell,k)$ has a similar structure and enjoys similar 
bilinear estimates as those satisfied by $\nu_1$ multiplied by a factor of $2^{12A} 2^{-(a-1)K} 2^{-\max(K,L)}$.
%$2^{12A} 2^{-(a-1)\min(L,K)} 2^{-\max(L,K)}$.
This will then imply the desired bilinear bound \eqref{theomu1'conc2}
by means of an application of the estimate \ref{theomu1concT2} to $2^{-12A} 2^{(a-1)K} 2^{\max(K,L)} {\bf Y}^a \nu_1$; % 
%$2^{-12A} 2^{(a-1)\min(L,K) + \max(L,K)} {\bf Y}^a \nu_1$;
compare the right-hand sides of \eqref{theomu1conc} and \eqref{theomu1'conc2}.

To prove the desired property we again look at the decomposition of $\nu_1$ in Proposition \ref{Propnu+}.
The point of using the vectorfield  ${\bf Y}$ is that for any function $f$ we have ${\bf Y}f(|\ell|-|k|)=0$.
In particular we have, for $|a|=1,2$,
\begin{align}\label{theomu1'pr10} 
{\bf Y}^a \nu_0(\ell,k) = {\bf Y}^a \Big[ \frac{b_0(\ell,k)}{|\ell|} \Big] 
  \Big[i\pi \delta(|\ell|-|k|) + \pv \frac{1}{|\ell|-|k|}\Big].
\end{align}
We now want to argue that the differentiated coefficient ${\bf Y}^a b_0(\ell,k)$
behave like $b_0$ up to the correct factor.
In view of the estimates \eqref{Propnu+1.1}, and treating ${\bf Y}^a$ 
as a regular $\nabla_\ell^{a_1}\nabla_k^{a_1}$-derivative, we have
\begin{align*}
| {\bf Y}^a b_0(\ell,k) | \lesssim 
  \sup_{a_1+a_2 = |a|} 2^{-a_1L} (2^{a_1K_+} + 2^{-(a_2-1)K_-})
\end{align*}
so that (recall that $|K-L|< 5$)
we have $| {\bf Y} b_0(\ell,k) | \lesssim 2^{-L}2^A$, and for $|a|=2$
\begin{align*}
| {\bf Y}^a b_0(\ell,k) | \lesssim 2^{-2L}2^{2A} + 2^{-K} \lesssim 2^{-2\max(L,K)} 2^{2A}.
\end{align*}
These bounds are consistent with the desired factor of $2^{12A} 2^{-(a-1)K} 2^{-\max(L,K)}$.
%$2^{12A} 2^{-(a-1)\min(L,K)} 2^{-\max(L,K)}$.

The term ${\bf Y}^a \nu_L(\ell,k)$, where $\nu_L$ is given by \eqref{Propnu+2}--\eqref{Propnu+2.1},
can be treated similarly to show that, up to the proper factor, ${\bf Y}^a \nu_L(\ell,k)$
has the same structure of $\nu_L$.

For the remainder term $\nu_R$ we treat the application of ${\bf Y}$
as a general derivative $\nabla_{(\ell,k)}$, and distinguish two cases.
When $|L-K|<5$, \eqref{Propnu+3} gives immediately that the $\nabla_\ell^{a_1} \nabla_k^{a_2}$-derivatives 
of $\nu_R(\ell,k)$ behave like those of $\nu_R$ times a factor of $2^{-a_1L} 2^{-a_2K} 2^{10A}$ which is more 
than sufficient when applying Lemma \ref{lemBER} to bound the associated bilinear operator.
%better than what we claim.

When $|K-L| \geq 5$ we use instead \eqref{Propnu+3imp}.
Writing ${\bf Y}^a$ as a combination of $\nabla_\ell^{a_1} \nabla_k^{a_2}$
gives the bound
%to get a bound on $\nabla_\ell^{a_1} \nabla_k^{a_2} \nu_R$ 
\begin{align*}
| \nabla_\ell^\alpha \nabla_k^\beta {\bf Y}^a \nu_R | \lesssim 2^{-2\max(K,L)} \cdot 2^{-|\alpha|\max(L,K)} 
  2^{-|\beta|K} \cdot 2^{(|\a|+|\b|+2)5A}
  \\ \times \sup_{a_1+a_2 = a} 2^{-a_1\max(L,K)} \max(0,2^{-(a_2-1)K_-}) 2^{5Aa}.
\end{align*}
The first line gives bounds of a regular symbol type times $2^{-2\max(K,L)}$; %like those satisfied by $\nu_R$;
the second line is a factor that, in the worst case $a_1=0,a_2=2$, is bounded by $2^{-K_-} 2^{10A}$.
%Since $K,L\leq A$,
This is consistent with what we want and completes the proof of \eqref{theomu1'conc2}.
$\hfill \Box$

\medskip
\subsection{Proof of Lemma \ref{lemBEan}}\label{proofBEpre}
In view of the assumption \eqref{lemBEanb1} we have
\begin{align}
\label{prlemBE10}
\begin{split}
B_{j}[b](g,h)(x) = \what{\mathcal{F}}^{-1}_{k\mapsto x}
  \iint_{\R^3\times\R^3} g(\ell-k) h(m) \, \varphi_{\sim L}(\ell) \varphi_{\sim L}(m) 
  b(k,\ell,m) \\ \, \times\, \chi(2^j(|\ell| - |m|)) \, \mathrm{d}\ell \mathrm{d}m,
\end{split}
\end{align}
having used $|\ell| \approx 2^L \approx |m| \gg 2^{-j}$ to insert the cutoffs for $\ell$ and $m$.
%We will consider only the case $b=1$ as the case for general $b$ follows identically. Check and argue later\dots
Upon taking the inverse Fourier transform (disregarding factor of $2\pi$)
we may rewrite \eqref{prlemBE10} as
\begin{align}
\label{prlemBE11}
\begin{split}
B(x) & = \iint_{\R^3_k \times \R^3_z} e^{-i(z-x)\cdot k} H(z,k) \what{g}(z) \,\mathrm{d}z\mathrm{d}k
%%= \iint_{\R^3\times\R^3} \varphi_L(\ell)g(\ell) \varphi_M(m) h(m) \,\chi(2^j(|\ell+k| - |m|)) \, \mathrm{d}\ell \mathrm{d}m
%\\ & = \iiint_{\R^3_k \times\R^3_z\times\R^3_y} e^{-i(z-x)\cdot k} A_{j}(z,y,k) \what{g}(z) \what{h}(y)
%  \,\mathrm{d}z\mathrm{d}y\mathrm{d}k
\end{split}
\end{align}
where (we omit the dependence on $L$)
\begin{align}
\label{prlemBE12}
\begin{split}
H(z,k) & := \int_{\R^3} A_j(z,y,k) \what{h}(y)  \, \mathrm{d}y
\\ 
A_j(z,y,k) & := \iint_{\R^3_\ell\times\R^3_m} e^{iz\cdot\ell} e^{iy\cdot m} \, b(k,\ell,m) \,
  \varphi_{\sim L}(\ell) \varphi_{\sim L}(m) \, \chi(2^j(|\ell| - |m|)) \, \mathrm{d}\ell \mathrm{d}m.
\end{split}
\end{align}

As a first step we integrate by parts in $k$ in the first formula of \eqref{prlemBE11}:
\begin{align*}
B(x) & = \iint_{\R^3_k \times \R^3_z} \frac{1}{(1+2^{2K}|x-z|^2)^2}
	 (1-2^{2K} \Delta_k)^2 e^{-i(z-x)\cdot k} \, \varphi_K(k) H(z,k) \, \what{g}(z) \,\mathrm{d}z\mathrm{d}k
\\
& = \iint_{\R^3_k \times \R^3_z} \frac{1}{(1+2^{2K}|x-z|^2)^2}
	e^{-i(z-x)\cdot k} (1-2^{2K} \Delta_k)^2 \big( \varphi_K(k) H(z,k) \big) \what{g}(z) \,\mathrm{d}z\mathrm{d}k
\end{align*}
and deduce, using H\"older's and Haussdorf-Young's inequalities, that
\begin{align*}
{\| B(x) \|}_{L^r} & \lesssim {\Big\| \int_{\R^3_z} \frac{2^{3K}}{(1+2^{2K}|x-z|^2)^2}
	\sup_k \Big| (1-2^{2K} \Delta_k)^2 \big( \varphi_K(k) H(z,k) \big) \Big| \, |\what{g}(z)| \,\mathrm{d}z \Big\|}_{L^r_x}
	\\
& \lesssim  {\Big\| \sup_k \Big| (1-2^{2K} \Delta_k)^2
	\big( \varphi_K(k) H(z,k) \big) \Big| \, |\what{g}(z)|  \Big\|}_{L^r_z}
\\	
& \lesssim  {\Big\| \sup_k \Big| (1-2^{2K} \Delta_k)^2
	\big( \varphi_K(k) H(z,k) \big) \Big| \Big\|}_{L^q_z} {\big\| \what{g} \big\|}_{L^p}.
\end{align*}
From the definition \eqref{prlemBE12} we see that it suffices to prove that
\begin{align}
\label{prlemBEmain}
 {\Big\| \int_{\R^3} \sup_k \big| A_j^\prime (z,y,k) \big| |\what{h}(y)|  \, \mathrm{d}y \Big\|}_{L^q_z}
 	\lesssim 2^{2L} \cdot 2^{-j} \cdot 2^{8A} \cdot {\big\| \what{h} \big\|}_{L^q}
\end{align}	
where $A_j^\prime$ is defined in the same way as $A_j$ in \eqref{prlemBE12} but with the symbol
\begin{align*}
b^\prime(k,\ell,m) := (1-2^{2K}\Delta_k)^2 \big( \varphi_K(k) b(k,\ell,m) \big)
\end{align*}
instead of $b$:
\begin{align*}
A_j^\prime(z,y,k) & := \iint_{\R^3_\ell\times\R^3_m} e^{iz\cdot\ell} e^{iy\cdot m} \, b^\prime(k,\ell,m) \,
  \varphi_{\sim L}(\ell) \varphi_{\sim L}(m) \, \chi(2^j(|\ell| - |m|)) \, \mathrm{d}\ell \mathrm{d}m.
\end{align*}

To estimate this kernel we observe that for fixed $m$, the integral over $\ell$ is supported 
on an annular region $C_{j,m} = \{ \ell \, : \, ||\ell|-|m||\lesssim 2^{-j}\}$,
which has volume $\approx 2^{2L}2^{-j}$.
We then pick points $p_r$, $r=1,\dots,2^{2(j+L)}$ , uniformly distributed on the sphere of radius $|m|$,
of the form $p_r = R_r m$ for suitable rotations $R_r$, with $R_1 = \id$.
We cover the annular region by $2^{2(j+L)}$ balls or radius $\approx 2^{-j}$ 
centered at the points $p_r$, and, with a partition of unity, write
\begin{align}
\label{prlemBE15}
\mathbf{1}_{C_{j,m}} = \sum_{r=1}^{2^{2(j+L)}} \chi_r(2^j(\ell - p_r)), \qquad p_r = R_rm,
\end{align}
for smooth compactly supported radial functions $\chi_r$.
We then write
\begin{align}
\label{prlemBE16}
& A_{j}^\prime(z,y,k) = \sum_{r=1}^{2^{2(j+L)}} A_{j,r}(z,y,k).
\\
\nonumber 
& A_{j,r}(z,y,k) := \iint_{\R^3_\ell\times\R^3_m} e^{iz\cdot\ell} e^{iy\cdot m} 
  \, \varphi_{\sim L}(\ell) \varphi_{\sim L}(m) \, b^\prime(k,\ell,m)
  \\ \nonumber & \qquad \qquad \qquad \qquad \,\times 
  \chi(2^j(|\ell| - |m|)) \, \chi_r\big(2^j(\ell - p_r)\big) \, \mathrm{d}\ell \mathrm{d}m.
\end{align}
We bound the term with $r=1$, that is $p_r = m$.
After changing variables $\ell$ to $\ell +m$, we write
\begin{align*}
(1+ 2^{2L}|z+y|^2)^2 |A_{j,1}(z,y,k)| 
= \Big| \iint_{\R^3_\ell\times\R^3_m} e^{ix\cdot\ell} (1-2^{2L}\Delta_m)^2 e^{i(z+y)\cdot m} 
  \, \varphi_{\sim L}(\ell+m)
  \varphi_{\sim L}(m)  
  \\ \times \, b^\prime(k,\ell+m,m) \, \chi(2^j(|\ell+m| - |m|)) \, \chi_r(2^j \ell) \, \mathrm{d}\ell \mathrm{d}m \Big|
%\\
%= \Big| \iint_{\R^3_\ell\times\R^3_m} e^{ix\cdot\ell} e^{i(z+y)\cdot m}
%	\,  (1 + 2^{2L}\Delta_m)^2 \Big[ \varphi_{\sim L}(\ell+m) \varphi_{\sim L}(m) 
%	\, b^\prime(k,\ell+m,m) \, \\ \chi(2^j(|\ell+m| - |m|)) \Big] \, \chi_r(2^j \ell) \, \mathrm{d}\ell \mathrm{d}m \Big|.
\\
\lesssim \iint_{\R^3_\ell\times\R^3_m} 
	\Big| (1-2^{2L}\Delta_m)^2 \big[ \varphi_{\sim L}(\ell+m) \varphi_{\sim L}(m) 
	\, b^\prime(k,\ell+m,m) \,\\ \times \, \chi(2^j(|\ell+m| - |m|)) \big] \Big| \chi_r(2^j \ell) \,\mathrm{d}\ell \mathrm{d}m.
\end{align*}
Using the hypothesis \eqref{lemBEanb2} we see that, on the support of the integral,
\begin{align*}
& |\partial_m^a [ \varphi_{\sim L}(\ell+m) \varphi_{\sim L}(m) ] | \lesssim 2^{-L |a|},
\\
& |\partial_m^a \chi(2^j(|\ell+m| - |m|)) |\lesssim 2^{-L |a|},
\\
& |\partial_m^a b^\prime(k,\ell+m,\ell) | \lesssim 2^{-L |a|} 2^{|a|  A}, \qquad |a|\leq 4.
\end{align*}
Thus, we obtain
\begin{align}
\label{prlemBE17}
| A_{j,1}(z,y,k) | \lesssim \frac{2^{3L}}{(1 + 2^{2L}|z+y|^2)^2} \cdot 2^{-3j} \cdot 2^{8A}.
\end{align}
Note that by changing variables $\ell \mapsto R_r \ell$ in the integral in \eqref{prlemBE16},
the same argument above shows that, for any $r = 1,\dots,2^{2(j+L)}$,
\begin{align}
\label{prlemBE18}
| A_{j,r}(z,y,k) | \lesssim \frac{2^{3L}}{(1 + 2^{2L}|z+ R_r y|^2)^2} \cdot 2^{-3j} \cdot 2^{8A}.
\end{align}

Going back to the left-hand side of \eqref{prlemBEmain} and using \eqref{prlemBE16} and  \eqref{prlemBE18} 
we can estimate
\begin{align*}
& {\Big\| \int_{\R^3} \sup_k \big| A_j^\prime (z,y,k) \big| |\what{h}(y)|  \, \mathrm{d}y \Big\|}_{L^q_z} 
	\\
	& \lesssim 2^{2(j+L)}
	{\Big\|  \int_{\R^3_y} \sup_k  
	\sup_{r=1,\dots 2^{2(j+L)}}\big|A_{j,r}(z,y, k)\big|  |\what{h}(y)| \,\mathrm{d}y  \Big\|}_{L^q_z}
	\\
	& \lesssim 2^{2(j+L)} \cdot  2^{-3j} \cdot 2^{8A} 
	\sup_{r=1,\dots 2^{2(j+L)}}
	{\Big\|  \int_{\R^3_y}\frac{2^{3L}}{(1 + 2^{2L}|z+y|^2)^2} \, 
		\big|\what{h}(R_r^{-1}y)\big| \,\mathrm{d}y  \Big\|}_{L^q_z}
	\\
	& \lesssim 2^{2L} \cdot  2^{-j} \cdot 2^{8A} \cdot {\big\| \what{h} \big\|}_{L^q}.
\end{align*}
This gives \eqref{prlemBEmain} and concludes the proof.
%
%We then go back to the bilinear operator and estimate from \eqref{prlemBE11}
%\begin{align*}
%{\| B(\varphi_Lg,\varphi_Mh) \|}_{L^2_k} 
%  & \lesssim {\Big\| \int_{\R^3_y} A_{j}(x,y) \mathcal{F}^{-1}(\varphi_L g)(x) \mathcal{F}^{-1}(\varphi_M'h)(y)\,\mathrm{d}y \Big\|}_{L^2_x}
%  \\
%& \lesssim {\| \mathcal{F}^{-1}(\varphi_L g)\|}_{L^p} {\Big\|  \int_{\R^3_y} A_{j}(x,y) \mathcal{F}^{-1}(\varphi_M'h)(y)\,\mathrm{d}y  \Big\|}_{L^q}.
%\end{align*}
%Using \eqref{prlemBE16}, followed by changes of variables (rotations), and applying \eqref{prlemBE17} we can bound
%\begin{align*}
%& {\Big\|  \int_{\R^3_y} A_{j}(x,y) \mathcal{F}^{-1}(\varphi_M'h)(y)\,\mathrm{d}y  \Big\|}_{L^q}
%\\
%& \lesssim 2^{2(j+M)} \sup_{r=1,\dots2j}{\Big\|  \int_{\R^3_y} A_{j,r}(x,R_r y) \mathcal{F}^{-1}(\varphi_M'h)(y)\,\mathrm{d}y  \Big\|}_{L^q}
%\\
%& \lesssim 2^{2(j+M)} \sup_{r=1,\dots2j} {\Big\|  \int_{\R^3_y} \frac{2^{3M}}{(1 + 2^{2M}|x+y|^2)^3} \cdot 2^{-3j}
%	\mathcal{F}^{-1}(\varphi_M'h(R_r^{-1} y))\,\mathrm{d}y  \Big\|}_{L^q}
%\\
%& \lesssim 2^{-j} 2^{2M} {\| \mathcal{F}^{-1}(\varphi_M'h) \|}_{L^q}.
%\end{align*}
%
%\medskip
%\subsubsection{Proof of \ref{lemBE13}}
%\begin{align}
%{\|B_{j,b}(\varphi_L g, \varphi_M h) \|}_{L^2} \lesssim 2^{j/2} 
%  \min \big( {\| \varphi_L g \|}_{L^\infty} {\|\varphi_M \partial_k h  \|}_{L^2}, \dots \big).
%\end{align}
%This implies the desired bound and concludes the proof of the Lemma. 
$\hfill \Box$

\medskip
\begin{proof}[Proof of Lemma \ref{lemBER}]
%The argument is similar to the one used above for the proof of Lemma \ref{lemBEan},
%based on estimates for the kernel using integration by parts.
We insert frequency localization according to \eqref{lemBERb1} and the notation after \eqref{LP2}, and write 
\begin{align}
\label{prBER1}
B[b](g,h)(x) & = %\what{\mathcal{F}}^{-1}_{k\mapsto x}
  %\iint_{\R^3\times\R^3} g(\ell-k) h(m) \, \varphi_{\sim K}(k) \varphi_{\sim L}(\ell) \varphi_{\sim M}(m) 
  %b(k,\ell,m) \, \mathrm{d} \ell \mathrm{d}m,
  %\\
  \iint_{\R^3_y \times \R^3_z} A(x,y,z) \what{g}(y) \what{h}(z) \,\mathrm{d}y\mathrm{d}z
\\
\nonumber
A(x,y,z) & := \iint_{\R^3_k\times\R^3_\ell\times\R^3_m} e^{ik\cdot (x-y)} e^{iz\cdot\ell} e^{iy\cdot m} 
  \, b(k,\ell,m) \,  \varphi_{\sim K}(k) \varphi_{\sim L}(\ell) \varphi_{\sim M}(m) \, \mathrm{d}k \mathrm{d}\ell \mathrm{d}m. 
\end{align}
Assume without loss of generality that $L \geq M$. changing variables $\ell \mapsto \ell+m$
and integrating by parts in $k$ and $m$ we get
\begin{align*}
\begin{split}
& \big(1 + 2^{2K}|x-y|^2\big)^2 \big(1 + 2^{2M}|y+z|^2\big)^2 \big| A(x,y,z) \big| 
\\
& \lesssim \Big| \iint_{\R^3_k\times\R^3_\ell\times\R^3_m}
  \big(1 - 2^{2K}\Delta_k\big)^2 \big(1 - 2^{2M}\Delta_m\big)^2
  \Big[ b(k,\ell+m,m) \, \\  & \qquad \qquad \qquad \qquad
  \times \varphi_{\sim K}(k) \varphi_{\sim L}(\ell+m) \varphi_{\sim M}(m) \Big] 
  \, \mathrm{d}k \mathrm{d}\ell \mathrm{d}m \Big|
  \\
  & \lesssim 2^{8A} \cdot 2^{3(K+L+M)},
\end{split}
\end{align*}
having used the estimates \eqref{lemBERb2} for the symbol.
The conclusion follows from \eqref{prBER1} and the Hausdorff-Young inequality.
\end{proof}

\medskip
\section{Weighted estimates for leading order terms}\label{secdkL2}

Recall Duhamel's formula \eqref{Duhameldec} 
according to the decomposition \eqref{mudec}-\eqref{mu3}.
Our main aim in this section is to estimate the weighted norms of the leading order term $\mathcal{D}_1$
as written in \eqref{D12}. 

For convenience we first localize the integrals dyadically in time $s\approx2^S$ 
by introducing a partition of $[0,t]$ associated to cutoff functions $\tau_S$, $S=0,\dots \log_2t$,
with $\int_0^t |\tau^\prime_S(s)| \, \mathrm{d}s\lesssim 1$.
More precisely we write
\begin{align}\label{dkL20}
\begin{split}
\mathcal{D}_1(t)(f,f) & = \sum_{S=0}^{\log_2t} \mathcal{N}_{1,S}(t)(f,f) + 2\mathcal{N}_{2,S}(t)(f,f),
\\
\mathcal{N}_{1,S}(t)(f,f) & := \int_0^t \iint e^{is (|\ell|^2 + 2 k\cdot \ell + |m|^2 )} 
  \widetilde{f}(s,\ell+k) \widetilde{f}(s,m)
  \, \nu_1(\ell,m) \, \mathrm{d}\ell \mathrm{d}m \, \tau_S(s) \mathrm{d}s,
\\
\mathcal{N}_{2,S}(t)(f,f) & = \int_0^t \iint e^{is (-|k|^2 + |\ell|^2 - 2\ell \cdot m + 2|m|^2 )} 
  \widetilde{f}(s,m-\ell) \widetilde{f}(s,m)
  \, \bar{\nu_1(\ell,k)}  \, \mathrm{d}\ell \mathrm{d}m \, \tau_S(s) \mathrm{d}s.
\end{split}
\end{align}
For lighter notation, we will often omit the dependence on $S$ in what follows; see for example 
\eqref{N_1L} where we omit the dependence on $S$ of the terms $\mathcal{N}_{L,M}$.

In view of Proposition \ref{proHF}, see \eqref{HF1}--\eqref{d_N},
we assume that on the support of the integral in \eqref{dkL20} we have
\begin{align}\label{dkL2HFass}
|\ell| + |m| + |k| \lesssim 2^A := 2^{\delta_NS}, \qquad \delta_N = \frac{3}{N-5}.
\end{align}
%consistently with the restriction \eqref{HFrestr} already adopted in Section \ref{secNSM}.
%In this section we prove the bootstrap estimates for the weighted norms of $\mathcal{D}_1$.
This is our main Proposition:

\begin{proposition}\label{prodkL2}
With the definitions \eqref{dkL20}-\eqref{dkL2HFass},
under the apriori assumptions \eqref{apriori}, we have
\begin{align}\label{dkest1}
{\| \partial_k \mathcal{N}_{1,S}(t)(f,f) \|}_{L^2} 
   %+ {\| \partial_k \mathcal{N}_{2,S}(t)(f,f) \|}_{L^2} 
   \lesssim \e^2 2^{-\delta' S},
\end{align}
for some $\delta'>0$, and
\begin{align}\label{dkest2}
{\| \partial_k^2 \mathcal{N}_{1,S}(t)(f,f) \|}_{L^2} 
  %+ {\| \partial_k^2 \mathcal{N}_{2,S}(t)(f,f) \|}_{L^2}
  \lesssim 2^{(1/2+\delta)S} \e^2.
\end{align}
\end{proposition}

The estimate \eqref{dkest1} is proven in Subsection \ref{ssecdkest1}
while \eqref{dkest2} is proven in Subsection \ref{ssecdkest2}.
In Subsection \ref{ssecdkN_2} we discuss how to obtain the same bounds for $\mathcal{N}_{2,S}$, and therefore
conclude the desired bootstrap estimate for $\mathcal{D}_1$ upon summing over $S$.

\medskip
\subsection{Proof of \eqref{dkest1}}\label{ssecdkest1}
In order to apply Theorem \ref{theomu1} we begin by localizing to $|\ell| \approx 2^L$ and $|m|\approx 2^M$,
%and the localization in $s\approx 2^S$ is done so we can compare time and various other quantities.
and write
\begin{align}
\label{N_1L}
\begin{split}
\mathcal{N}_{1,S}(f,f) & = \sum_{L,M\in(-\infty,A]\cap\Z} \mathcal{N}_{L,M}(f,f),
\\
\mathcal{N}_{L,M}(g,h) & := \int_0^t \iint e^{is \Phi(k,\ell,m)} \widetilde{g}(s,\ell+k) \widetilde{h}(s,m)
  \, \varphi_L(\ell) \varphi_M(m) \nu_1(\ell,m)\, \mathrm{d}\ell \mathrm{d}m \,\tau_S(s) \mathrm{d}s,
\\
\Phi(k,\ell,m) & := |\ell|^2 + 2 k\cdot \ell + |m|^2. 
\end{split}
\end{align}
Taking $\partial_k$ we get
\begin{align}
\label{dk_N1}
\partial_k \mathcal{N}_{L,M}(f,f) & = \mathcal{M}_{1,L,M}(f,f) + 2\mathcal{M}_{2,L,M}(f,f),
\end{align}
where %(omitting the implicit dependence on $L$ for easier notation)
\begin{align}
\label{dkN_11}
\mathcal{M}_{1,L,M}(f,f) & = \int_0^t \iint e^{is \Phi(k,\ell,m)} \, \partial_k \widetilde{f}(s,\ell+k) \widetilde{f}(s,m)
  \,\varphi_L(\ell)\varphi_M(m) \nu_1(\ell,m)\, \mathrm{d}\ell \mathrm{d}m \,\tau_S(s) \mathrm{d}s,
\\
\label{dkN_12}
\mathcal{M}_{2,L,M}(f,f) & = \int_0^t \iint i s \ell \, e^{is \Phi(k,\ell,m)} \widetilde{f}(s,\ell+k) \widetilde{f}(s,m)
  \, \varphi_L(\ell)\varphi_M(m) \nu_1(\ell,m) \, \mathrm{d}\ell \mathrm{d}m \,\tau_S(s) \mathrm{d}s.
\end{align}
We proceed to estimate the terms above with the understanding that all the bounds will need to be summed over 
$L,M\in(-\infty,A] \cap \Z$.

\medskip
\subsubsection{Estimate of \eqref{dkN_11}}\label{dkN_11sec}
This term can be treated almost directly by applying the bilinear estimate \eqref{theomu1conc} from Theorem \ref{theomu1}.
%We first localize in $|m| \approx 2^M$ by 
Define the distorted Littlewood-Paley projection $g_M := \Ftil^{-1} \varphi_M(m) \wt{g}$,
%We provide full details below but will skip some of the details of similar estimates later on.
%$T_1[b] = T_1$ 
and let \[b=\varphi_{\sim L}(\ell)\varphi_{\sim M}(m),\]
(recall the notation for $\varphi_\sim$ after \eqref{LP2}).
Using the notation in \eqref{theomu11b}, we can write
\begin{align}\label{dkest11}
\mathcal{M}_{1,L,M}(f,f)(t) & = 
  \int_0^t e^{-is|k|^2} \Fhat_{x\mapsto k} T_1[b]\big(e^{is|k|^2}\partial_k\wt{f}(s), \Ftil u_M(s) \big) 
  \, \tau_S(s) \mathrm{d}s.
\end{align}
In view of the apriori assumptions \eqref{apriori} we have 
${\| \partial_k \wt{f} \|}_{L^2} {\| \wt{f} \|}_{L^2} \lesssim \e^2$,
so that, choosing $2^D = \js^{2} \approx 2^{2S}$, the hypothesis \eqref{theomu1asgh} is verified.
Let us assume that $M\leq L$.
An application of \eqref{theomu1conc} then gives, for arbitrary large $q$,
\begin{align}
\nonumber
& {\| \mathcal{M}_{1,L,M}(f,f)(t) \|}_{L^2} \lesssim 
\int_0^t {\big\| T_1[b]\big(e^{is|k|^2}\partial_k\wt{f}(s), \Ftil u_M(s) \big) \big\|}_{L^2} \,\tau_S(s)  \mathrm{d}s
%\\ 
%& \lesssim \int_0^t \big[ {\| \Fhat \partial_k \wt{f}(s) \|}_{L^2} 
%  \cdot {\| \Fhat \Ftil u_M(s) \|}_{L^q} 
%  \cdot 2^{\max(L,M)} \cdot 2^{C_0A} + \js^{-2} \e^2\, \big] \tau_S(s) \mathrm{d}s \\ 
\\ 
\label{dkest11M}
& \lesssim \int_0^t \big[ {\| \Fhat \partial_k \wt{f}(s) \|}_{L^2} 
  \cdot 2^{(0+)M} {\| u(s) \|}_{L^{q-}} 
  \cdot 2^{\max(L,M)} \cdot 2^{C_0A} + \js^{-2} \e^2\, \big] \tau_S(s) \mathrm{d}s.
\end{align}
For the second inequality we have used
\begin{align*}
{\| \Fhat \Ftil u_M(s) \|}_{L^{q}} \lesssim 2^{(0+)M} \cdot {\| \Fhat \wt{u}(s) \|}_{L^{q-}},
\end{align*}
which follows from Bernstein's inequality (for flat Littlewood-Paley projections) 
and the boundedness of wave operators.
If we have $L\leq M$, a similar $L^{2+}\times L^q$ application of \eqref{theomu1conc}
would give the same conclusion \eqref{dkest11M} with a factor of $2^{(0+)L}$ instead of $2^{(0+)M}$.

Using that $2^A \approx 2^{\delta_NS}$ %the boundedness of wave operators, 
and the a priori dispersive estimate \eqref{aprioriL>6}, we obtain
\begin{align*}
\begin{split}
\sum_{L,M\leq A} {\| \mathcal{M}_{1,L,M}(t) \|}_{L^2} 
  \lesssim \sum_{L,M\leq A} \int_0^t \e \cdot \e \js^{-5/4+\delta} \cdot 2^{(0+)\min(L,M)} \cdot \js^{(C_0+1)\delta_N} 
  \, \tau_S(s)\mathrm{d}s + \e^2 2^{-S}
  \\ 
  \lesssim \e^2 2^{-S\delta'}
\end{split}
\end{align*}
provided 
\begin{align}\label{condition}
\delta + (C_0+1)\delta_N \leq 1/4 - \delta';
\end{align}
this last inequality holds with our choice of $\delta_N$ in \eqref{dkL2HFass}, $N$ in \eqref{Nparam}, and $C_0=65$.

\medskip
\subsubsection{Estimate of  \eqref{dkN_12}: Set-up}\label{ssecM_2}
This is the main term in the estimate of the first weight \eqref{dkest1}. 
Here we need to exploit the oscillations given by $e^{is\Phi}$,
by integrating by parts in the distorted frequency space and in time.
Due to the singularity of $\nu_1$ we can only integrate by parts in the direction
of the ``good vectorfield''
\def\X{{\bf X}}
\begin{align}
\label{Xlm}
%{\bf X} 
\X := \partial_{|\ell|} + \partial_{|m|}.
\end{align}
%At the same time we will also need to restrict to a configuration where ${\bf X}\Phi$ is not too small.

In order to be able to exploit the oscillations efficiently, we first split \eqref{dkN_12}
into a region where $||\ell|-|m| |\ll |\ell|$ and the complement one, by defining:
\begin{align}\label{ssecM_2cut}
\chi_+(\ell,m) := \varphi_{> -10}\Big(\frac{|m|-|\ell|}{|m|+|\ell|}\Big), 
  \qquad  \chi_-(\ell,m) := \varphi_{\leq-10}\Big(\frac{|m|-|\ell|}{|m|+|\ell|}\Big) = 1 - \chi_+(\ell,m),
\end{align}
and the corresponding terms
\begin{align}
\label{dkest12}
\mathcal{M}_{2}^- & := \int_0^t \iint i s \ell \, e^{is \Phi} \widetilde{f}(s,\ell+k) \widetilde{f}(s,m)
  \varphi_L(\ell)\varphi_{\sim L}(m) \, \chi_-(\ell,m) \, \nu_1(\ell,m) \, \mathrm{d}\ell \mathrm{d}m \,\tau_S(s) \mathrm{d}s,
\\
\label{dkest12b}
\mathcal{M}_{2}^+ & := \int_0^t \iint i s \ell \, e^{is \Phi} \widetilde{f}(s,\ell+k) \widetilde{f}(s,m)
  \varphi_L(\ell)\varphi_M(m) \, \chi_+(\ell,m) \, \nu_1(\ell,m) \, \mathrm{d}\ell \mathrm{d}m \,\tau_S(s) \mathrm{d}s.
\end{align}
Notice how we have dropped the dependence on the indexes $L,M$ for brevity,
and replaced $\varphi_M(m)$ with $\varphi_{\sim L}(m)$ since on the support of $\chi_-$,
$|m|$ and $|\ell|$ are comparable.

%For brevity we have omitted the dependence on some of the variables.
%We localize in $|m| \approx 2^M$ by defining
For \eqref{dkest1} it will suffice to prove the following bounds:
\begin{align}
\label{dkest1a}
& \sum_{L\leq A} {\| \mathcal{M}_{2}^-(t) \|}_{L^2} \lesssim \e^2 2^{-S\delta'},
\\
\label{dkest1b}
& \sum_{L,M\leq A} {\| \mathcal{M}_{2}^+(t) \|}_{L^2} \lesssim \e^2 2^{-S\delta'}.
\end{align}
%These are proven below in \ref{l=m} and \ref{lnotm} respectively.

\medskip
\subsubsection{Proof of \eqref{dkest1a}}\label{l=m}
%On the support of \eqref{dkest12} we have $||\ell|-|m|| \ll |\ell|\approx |m|$.
We calculate
\begin{align}
\label{XPhi}
\X \Phi := 2k \cdot \frac{\ell}{|\ell|} + 2 |\ell| + 2|m|, \qquad 
  \Phi(k,\ell,m) = |\ell|^2 + 2 k\cdot \ell + |m|^2 %= (|m|-|\ell|)^2 + 2|m||\ell| + 2k \cdot \ell,
\end{align}
%$X$ has the advantage of being tangential to the singularity of the kernel so
%that integrating by parts in its direction does not increase the singularity, 
%see \eqref{theomu1X}-\eqref{theomu1Xconc} in Theorem \ref{theomu1}.
and see that the following identity holds:
\begin{align}
\label{PhiXPhi}
\begin{split}
|\ell|\X \Phi(k,\ell,m) - \Phi(k,\ell,m) & = |\ell|^2 + 2|m||\ell| -|m|^2 
\\ & = |m|^2 + |\ell|^2 - 2|m|(|m|-|\ell|) =: c(\ell,m).
\end{split}
\end{align}
This implies
%\begin{align}
%\label{PhiXPhi2}
%|X\Phi(k,\ell,m)| \not\approx |\ell| \quad \Longrightarrow  \quad |\Phi| \gtrsim |\ell| \max(|\ell|, |X\Phi|). 
%\end{align}
\begin{align}
\label{PhiXPhiIBP}
\frac{1}{c(\ell,m)} \Big( \frac{1}{is} |\ell|\X + i \partial_s \Big) e^{is\Phi(k,\ell,m)} = e^{is\Phi(k,\ell,m)}, 
\end{align}
which we can use to integrate by parts in \eqref{dkest12}.
Note that, since $||\ell|-|m|| \ll |\ell| \approx |m|$, 
then $1/c(\ell,m)$ behaves like a multiplier of the form $1/(|m|^2 + |\ell|^2)$: for all $\alpha,\beta\in\Z^3$
\begin{align}
\label{PhiXPhisym}
\big| \nabla_\ell^\alpha \nabla_m^\beta \frac{1}{c(\ell,m)} \big| \lesssim \frac{1}{|\ell|^2+|m|^2} 
  |\ell|^{-|\alpha|} |m|^{-|\beta|}.
\end{align}
Also, we can freely insert a cutoff in the size of $|m|\approx 2^L$ in \eqref{dkest12}.
%Without loss of generality, in what follows we will replace the cutoffs $\varphi_M$ by $\varphi_{\sim L}$ 
%in the expression \eqref{dkest12}, and dispense of the irrelevant sum over $M$ in \eqref{dkest1a}.
We then have to estimate the $L^2$-norm of the term
\begin{align}
\label{dkest12'}
\begin{split}
\mathcal{M}_{L}(f,g)(t) & := \int_0^t \iint s \ell \, e^{is \Phi} \widetilde{f}(s,\ell+k) \widetilde{f}(s,m)
  \, a_L(\ell,m) \, \nu_1(\ell,m) \, \mathrm{d}\ell \mathrm{d}m \,\tau_S(s) \mathrm{d}s,
\\
a_L(\ell,m) & := \varphi_{\sim L}(m)\varphi_L(\ell) \chi_-(\ell,m). %\varphi_{\leq L-10}(|m|-|\ell|)
\end{split}
\end{align}

\medskip
\subsubsection*{$-$ The case $L \leq L_0:=-S/3$.}
First, we show that when $L$ is sufficiently small we can obtain \eqref{dkest1a} 
directly using the bilinear estimates \eqref{theomu1conc} 
of Theorem \ref{theomu1}: with
\[b = 2^{-L} \ell \varphi_{\sim L}(m)\varphi_{\sim L}(\ell) \chi_-(\ell,m) %\varphi_{\leq L-10}(|m|-|\ell|)
\] 
%(recall the notation after \eqref{LP2}), 
and using the a priori decay estimates \eqref{aprioriL<6}, we have
\begin{align*}
\begin{split}
{\| \mathcal{M}_{L}(t) \|}_{L^2} & \lesssim 
\int_0^t s \, 2^L {\big\| T_1[b]\big(\Ftil u(s), \Ftil u(s) \big) \big\|}_{L^2} \,\tau_S(s)  \mathrm{d}s
\\ 
& \lesssim \int_0^t \big[ {\| \Fhat \wt{u}(s) \|}_{L^3} {\| \Fhat \wt{u}(s) \|}_{L^{6-}}
  \cdot 2^{L} \cdot 2^{C_0A} + \js^{-3} \e^2\, \big] 2^L 2^S \, \tau_S(s) \mathrm{d}s.
\\ 
& \lesssim \int_0^t \e 2^{-S/2} \cdot \e 2^{-(1-)S}
  \cdot 2^{2L} \cdot 2^{C_0A} \cdot 2^S \tau_S(s) \mathrm{d}s + \e^2 2^{-S}
\\
& \lesssim \e^2 2^{(1/2 + (C_0+1)\delta_N)S} \cdot 2^{2L} + \e^2 2^{-S}
\end{split}
\end{align*}
Since $L \leq L_0 := -S/3$, and our choice of parameters
guarantees $(C_0+1)\delta_N + \delta' \leq 1/6$ (see \eqref{d_N} and \eqref{Nparam}),
we have
\begin{align}
\label{dkest1alow}
& \sum_{\substack{L \leq L_0}} {\| \mathcal{M}_{L}(t) \|}_{L^2} \lesssim \e^2 2^{-S\delta'}
\end{align}
consistently with \eqref{dkest1a}.
Note that, similarly to the estimate of $\mathcal{M}_1$ above,
we have used that the assumption \eqref{theomu1asgh} is fulfilled by choosing $2^D = \js^3 \approx 2^{3S}$.

\medskip
{\it Notation}.
For simplicity, in what follows we will often disregard the extra lower order terms 
coming from the $2^{-D}$ factor on the right-hand side of \eqref{theomu1conc}, 
since we will always be able to guarantee \eqref{theomu1asgh} with large enough $2^D$ of the form $\js^n$, 
under the apriori assumptions \eqref{apriori}.
Also, often times we will not fully detail how to write our bilinear terms as operators of the form $T_1[b]$
when this will be clear from the context.

\medskip
\subsubsection*{$-$ Integration by parts.}
For $L > L_0$, using \eqref{PhiXPhiIBP}, we write $\mathcal{M}_{L}$ in \eqref{dkest12'} as:
\begin{align}
\label{M_2split}
& \mathcal{M}_{L} = A_L + B_L, %\qquad L_0 = -\frac{1}{3}S
\\
\label{M_11}
& A_L = \int_0^t \iint \frac{\ell |\ell|}{c(\ell,m)} \big(\X e^{is \Phi}\big)\,
  \widetilde{f}(s,\ell+k) \widetilde{f}(s,m)
  \, a_L(\ell,m) \, \nu_1(\ell,m)
  \, \mathrm{d}\ell \mathrm{d}m \,\tau_S(s)\mathrm{d}s,
\\
\label{M_12}
& B_L = \int_0^t \iint \frac{-\ell}{c(\ell,m)} \, s \, \big(\partial_s e^{is \Phi}\big)\, 
  \widetilde{f}(s,\ell+k) \widetilde{f}(s,m)
  \, a_L(\ell,m) \, \nu_1(\ell,m) 
  \, \mathrm{d}\ell \mathrm{d}m \,\tau_S(s)\mathrm{d}s.
\end{align}
%On the support of $A_L$ we can integrate by parts in the good direction $X$ using the identity
%\begin{align}
%i s \,e^{is\Phi}  = \frac{1}{X\Phi} X \big[ \, e^{is\Phi} \big],
%\end{align}
%and the fact that $|\ell| \approx 2^L \gtrsim 2^{-S/3}$.
Integrating by parts in $\X$ we get %(we omit the cutoffs $\tau_S(s)$ here)
\begin{align}
\label{M_110}
\begin{split}
-A_L & = \int_0^t \iint e^{is \Phi} \X \Big[ \frac{\ell |\ell|}{c(\ell,m)} a_L(\ell,m)
  \, \widetilde{f}(s,\ell+k) \widetilde{f}(s,m) \, \nu_1(\ell,m) \Big] \, \mathrm{d}\ell \mathrm{d}m \,\tau_S(s)\mathrm{d}s
\\
& = A_L^1 + A_L^2 + A_L^3,
\end{split}
\end{align}
where
\begin{align}
\label{M_111}
A_L^1 & = \int_0^t \iint e^{is \Phi} X \Big[ \frac{\ell |\ell|}{c(\ell,m)} a_L(\ell,m) \Big]
  \, \widetilde{f}(s,\ell+k) \widetilde{f}(s,m) \, \nu_1(\ell,m) \, \mathrm{d}\ell \mathrm{d}m \,\tau_S(s)\mathrm{d}s,
\\
\label{M_112}
A_L^2 & = \int_0^t \iint e^{is \Phi} \frac{\ell |\ell|}{c(\ell,m)} a_L(\ell,m)
  \, \X \Big[ \widetilde{f}(s,\ell+k) \widetilde{f}(s,m) \Big]\, \nu_1(\ell,m) 
  \,\mathrm{d}\ell \mathrm{d}m \,\tau_S(s)\mathrm{d}s,
\\
\label{M_113}
A_L^3 & = \int_0^t \iint e^{is \Phi} \frac{\ell |\ell|}{c(\ell,m)} a_L(\ell,m)
  \, \widetilde{f}(s,\ell+k) \widetilde{f}(s,m) \,\X \Big[\nu_1(\ell,m)\Big] 
  \,\mathrm{d}\ell \mathrm{d}m \,\tau_S(s)\mathrm{d}s.
\end{align}
Integrating by parts in $s$ instead we can write
\begin{align}
\label{M_120}
\begin{split}
B_L(t) %& = \int_0^t \iint e^{is \Phi} \frac{\ell}{c(\ell,m)}a(\ell,m)
  %\, \partial_s \Big[ \widetilde{f}(s,\ell+k) \widetilde{f}(s,m) \,s \tau_S(s) \Big] \nu_1(\ell,m) 
  %\, \mathrm{d}\ell \mathrm{d}m \, \mathrm{d}s \\
& = - B_L^1 + B_L^2 + B_L^3 + B_L^4,
\end{split}
\end{align}
where
\begin{align}
\label{M_121}
B_L^1 & = t \iint e^{it \Phi} \frac{\ell}{c(\ell,m)} a_L(\ell,m)
  \, \widetilde{f}(t,\ell+k) \widetilde{f}(t,m) \, \tau_S(t) \nu_1(\ell,m) 
  \, \mathrm{d}\ell \mathrm{d}m \, \tau_S(t),
\\
\label{M_122}
B_L^2 & = \int_0^t \iint e^{is \Phi} \frac{\ell}{c(\ell,m)} a_L(\ell,m)
  \, \big( \partial_s\widetilde{f}(s,\ell+k) \big) \widetilde{f}(s,m) \nu_1(\ell,m) 
  \, \mathrm{d}\ell \mathrm{d}m \, s \tau_S(s)  \mathrm{d}s,
\\
\label{M_123}
B_L^3 & = \int_0^t \iint e^{is \Phi} \frac{\ell}{c(\ell,m)} a_L(\ell,m)
  \, \widetilde{f}(s,\ell+k) \big(\partial_s\widetilde{f}(s,m)\big) \nu_1(\ell,m) 
  \, \mathrm{d}\ell \mathrm{d}m \, s \tau_S(s)  \mathrm{d}s,
\\
\label{M_124}
B_L^4 & = \int_0^t \iint e^{is \Phi} \frac{\ell}{c(\ell,m)} a_L(\ell,m)
  \, \widetilde{f}(s,\ell+k)\widetilde{f}(s,m) \nu_1(\ell,m) 
  \, \mathrm{d}\ell \mathrm{d}m \, \partial_s \big(s \tau_S(s)\big) \mathrm{d}s.
\end{align}

\smallskip
\paragraph{{\it Estimate of \eqref{M_111}}}
With the notation of Theorem \ref{theomu1} we have
\begin{align}
\begin{split}
{\| A_L^1 \|}_{L^2_k} \lesssim \int_0^t  2^{-L} {\| T_1[b]\big(\Ftil u(s), \Ftil u(s) \big) \|}_{L^2} 
\, \tau_S(s) \mathrm{d}s
\\
\mbox{with} \qquad b(\ell,m) := 2^L \X \Big[ \frac{\ell |\ell|}{c(\ell,m)} a_L(\ell,m) \Big]
\end{split}
\end{align}
Since $||m|-|\ell||\ll |\ell|$ on the support of the integral \eqref{M_111},
it is easy to check that $b$ satisfies %even better 
the symbol bounds \eqref{theomu1asb1}-\eqref{theomu1asb2}.
Using the bilinear estimate \eqref{theomu1conc}, and the apriori decay estimates \eqref{aprioriL<6}, we have
\begin{align*}
\begin{split}
{\| A_L^1 \|}_{L^2_k} & \lesssim 
\int_0^t  2^{-L}  \big[ {\| \Fhat \wt{u}(s) \|}_{L^3} {\| \Fhat \wt{u}(s) \|}_{L^{6-}}
  \cdot 2^{L} \cdot 2^{C_0A} + \js^{-3} \e^2\, \big] \, \tau_S(s) \mathrm{d}s
\\ 
& \lesssim \int_0^t \e 2^{-S/2} \cdot \e 2^{-(1-)S}
	\cdot 2^{C_0A} \tau_S(s) \mathrm{d}s + \e^2 2^{-S}
\\
& \lesssim \e^2 2^{-S/4}.
\end{split}
\end{align*}
%provided, say, $16 \delta_N \leq 1/4$.
Summing over $L \in [-L_0,A] \cap \Z$ at the expense of an $O(S)$
factor gives the desired bound on the right-hand side of \eqref{dkest1a}.

\smallskip
\paragraph{{\it Estimate of \eqref{M_112}}}
This term consists of two similar terms (corresponding to one of the two profiles being hit by $\X$),
so we only estimate one of them. We denote
\begin{align*}
A_L^2 = \int_0^t \iint e^{is \Phi} b(\ell,m) \varphi_L(\ell)
  \, \partial_{|\ell|} \widetilde{f}(s,\ell+k) \widetilde{f}(s,m) \, \nu_1(\ell,m) 
  \,\mathrm{d}\ell \mathrm{d}m \,\tau_S(s)\mathrm{d}s
\\
b = (\ell |\ell| / c(\ell,m)) a_L(\ell,m),
\end{align*}
and estimate
\begin{align*}
{\| A_L^2 \|}_{L^2_k} & \lesssim 
	\int_0^t {\| T_1[b]\big(e^{is|k|^2}\partial_k \wt{f}(s), \Ftil u(s) \big) \|}_{L^2} \, \tau_S(s) \mathrm{d}s.
\end{align*}
Note that we have made a little abuse of notation converting $\partial_{|\ell|}$ into $\partial_\ell$;
this can be done without loss of generality by slightly redefining the symbol $b$.
Using that $b$ satisfies the hypotheses \eqref{theomu1asb1}-\eqref{theomu1asb2},
and applying \eqref{theomu1conc} together with the apriori bound \eqref{aprioriL>6}, 
we obtain
\begin{align*}
\begin{split}
{\| A_L^2 \|}_{L^2_k} & \lesssim 
\int_0^t {\| \Fhat e^{is|k|^2}\partial_k \wt{f}(s) \|}_{L^2} {\| \Fhat \wt{u}(s) \|}_{L^{\infty-}}
  \cdot 2^{L} \cdot 2^{C_0A} %+ \js^{-3} \e^2\, \big] 
  \, \tau_S(s) \mathrm{d}s
\\ 
& \lesssim \int_0^t \e \cdot \e 2^{-(5/4 - \delta)S} \cdot 2^{(C_0+1)A} \tau_S(s) \mathrm{d}s 
\lesssim \e^2 2^{-2\delta'S},
\end{split}
\end{align*}
see \eqref{condition}, and note that we have now disregarded the fast decaying remainder term from $2^{-D}\mathcal{D}$.
%provided $\delta_N$ is small enough. %$16\delta_N + \delta \leq 1/4 -2\delta'$.

\smallskip
\paragraph{{\it Estimate of \eqref{M_113}}}
Using the notation \eqref{theomu1TX} we see that
\begin{align*}
A_L^3 & = \int_0^t e^{-is|k|^2} \mathcal{F}_{x\mapsto k} \big(T_{\mathbf{X}} [b](\wt{u},\wt{u})\big)
	\,\tau_S(s)\mathrm{d}s
\end{align*}
where $b = (\ell |\ell| / c(\ell,m)) a_L(\ell,m)$.
Using the bounds \eqref{theomu1Xconc} from Theorem \ref{theomu1} 
we see that \eqref{M_113} enjoys a similar bound (up a different power of $2^{A}$) 
to that of \eqref{M_111}. We skip the details.

\smallskip
\paragraph{{\it Estimate of \eqref{M_121}}}
%Again similar to \eqref{M111}
With the notation of Theorem \ref{theomu1} we have
\begin{align}\label{M_121est}
\begin{split}
{\| B_L^1 \|}_{L^2_k} \lesssim t \, 2^{-L} {\| T_1[b]\big(\Ftil u(s), \Ftil u(s) \big) \|}_{L^2},
	\qquad b(\ell,m) = \frac{\ell 2^L a_L(\ell,m)}{c(\ell,m)}
\end{split}
\end{align}
Similarly to before we can apply \eqref{theomu1conc} (again disregarding the remainder term on the
right-hand side of the inequality)
\begin{align*}
{\| B_L^1 \|}_{L^2_k} & \lesssim 
t \, 2^{-L}  {\| \Fhat \wt{u}(t) \|}_{L^3} {\| \Fhat \wt{u}(t) \|}_{L^{6-}}
  \cdot 2^{L} \cdot 2^{C_0A} \tau_S(t)
\\ 
& \lesssim 2^S \cdot \e 2^{-S/2} \cdot \e 2^{-(1-)S} \cdot 2^{C_0A}
\\
& \lesssim \e^2 2^{-S/2} 2^{(C_0+1)\delta_N}
\end{align*}
%provided, say, $16 \delta_N \leq 1/4$.
%Summing over $L \in [-L_0,A] \cap \Z$ easily gives the desired bound on the right-hand side of \eqref{dkest1a}.
which is more than sufficient. %for $\delta_N$ small enough.

\smallskip
\paragraph{{\it Estimate of \eqref{M_122} and \eqref{M_123}}}
The terms \eqref{M_122} and \eqref{M_123} are similar to each other, 
so it suffices to show how to estimate the first one.
With the usual notation we have
\begin{align*}
& {\| B_L^2 \|}_{L^2_k} \lesssim \int_0^t 2^{-L} 
	{\| T_1[b]\big(e^{is|k|^2}\partial_s \wt{f}(s), \Ftil u(s) \big) \|}_{L^2} \, \tau_S(s) \mathrm{d}s,
\\
& b = 2^L \ell a_L(\ell,m) / c(\ell,m).
\end{align*}
Recall that $e^{is|k|^2} \partial_s \widetilde{f} = \Ftil (u^2)$, see \eqref{prHFdsf}, and therefore
\begin{align}\label{estdsf}
{\| \Fhat^{-1}e^{is|k|^2} \partial_s \wt{f} \|}_{L^p} 
	\lesssim {\| u \|}_{L^{2p}}^2. %\lesssim \e^2 \js^{-3/2}.
\end{align}
Applying \eqref{theomu1conc}, \eqref{estdsf} with $p=3$, the apriori decay estimates \eqref{aprioriL<6},
we obtain
\begin{align*}
\begin{split}
{\| B_L^2 \|}_{L^2_k} & \lesssim 
\int_0^t 2^{-L} {\| \Fhat e^{is|k|^2} \partial_s \wt{f} \|}_{L^3} {\| \Fhat \wt{u}(s) \|}_{L^{6-}}
  \cdot 2^{L} \cdot 2^{C_0A} %+ \js^{-3} \e^2\, \big] 
  \, 2^S \tau_S(s) \mathrm{d}s
\\ 
& \lesssim \int_0^t (\e  2^{-S})^2 \cdot \e 2^{-(1-)S} \cdot 2^{C_0A} \cdot 2^S \tau_S(s) \mathrm{d}s 
\\
& \lesssim \e^2 2^{-2S/3}
\end{split}
\end{align*}
provided $\delta_N$ is small enough as usual. %$16\delta_N + \delta \leq 1/4 -2\delta'$.

\smallskip
\paragraph{{\it Estimate of \eqref{M_124}}}
This term is almost identical to \eqref{M_121} by noticing that $\partial_s (s\tau_S(s)) \approx \tau_S(s)$,
and that the time integration is equivalent to the factor of $t$ in front of the expression in \eqref{M_121}.
We can then skip the details.
This concludes the estimate of $B_L$, hence of $\mathcal{M}_L$, see \eqref{M_2split},
and gives us \eqref{dkest1a}.

\medskip
\subsubsection{Proof of \eqref{dkest1b}}\label{lnotm}
We now look at the term \eqref{dkest12b}. 
In this case, the integrals are supported on a region where $||\ell|-|m||\gtrsim |m|+|\ell|$ 
so that $\nu_1$ is not really singular.
In particular, we can differentiate it in other directions besides the direction of $X$
and use the bilinear bounds from Theorem \ref{theomu1'}(i).
This makes the estimates more straightforward.

%This might need adjustments\dots
%
%%Either identity for IBP or the statement of bilinear bounds \eqref{theomu1'conc1} with an extra $\max$ term
%%or with $\partial_\ell^2$ giving $2^{-L}$ and $\partial_m^2$ giving $2^{-M}$
%
%or better, details when $\partial_m^2$ hits $\nu_1\chi_+$\dots

More precisely, we have the identity
\begin{align}\label{lnotm1}
%\partial_m \Phi = m, \qquad \ell \cdot \partial_\ell \Phi = \Phi + |\ell|^2 - |m|^2. 
\sqrt{|\ell|^2+|m|^2} X_+ \Phi = \Phi + |\ell|^2 + |m|^2,
\qquad X_+ := \frac{(\ell,m)}{\sqrt{|\ell|^2+|m|^2}}\cdot \nabla_{(\ell,m)}, 
\end{align}
and, consequently,
\begin{align}\label{lnotmIBP}
\frac{1}{|\ell|^2+|m|^2} \Big( \frac{1}{is} \sqrt{|\ell|^2+|m|^2} X_+  
  + i \partial_s \Big) e^{is\Phi(k,\ell,m)} = e^{is\Phi(k,\ell,m)}.
\end{align}
Note how (a) $\sqrt{|\ell|^2+|m|^2} X_+$ plays the role of $|\ell|\X$ in \eqref{PhiXPhi},
(b) the identity \eqref{lnotmIBP} is analogous to the identity \eqref{PhiXPhiIBP},
and (c) integration by parts in the $X_+$ direction is possible thanks to the bilinear estimates %\eqref{theomu1'1a}-
\eqref{theomu1'conc1}, which are analogous to the estimate \eqref{theomu1Xconc} 
that we have used before when we integrated by parts in the $\X$ direction.
%In particular, each intergration by parts in $(\ell,m)$ gains a power of $1/s$ and loses a power of $2^{-2\max(L,M)}$.
Then, the terms $\mathcal{M}_{2}^+$ in \eqref{dkest12b}
have the same structure of the terms $\mathcal{M}_{2}^-$ in \eqref{dkest1}. %, that we treated in the previous Subsection.
We can then skip the details for the proof of \eqref{dkest1b}.
%This concludes the validity of \eqref{dkest1}. %for the term $\mathcal{N}_1$ in \eqref{dkL20}.

%\begin{align}\label{}
%\begin{split}
%T_{\nabla^a}[b](g,h)(k) := \mathcal{F}^{-1}_{k\mapsto x} \iint_{\R^3\times\R^3} g(k-\ell) h(m) 
%  \, b(k,\ell,m) \\ \times \varphi_M(m) \nabla^{a_1}_{\ell} \nabla^{a_2}_{m}
%  \big[ \nu_1(\ell,m) \chi_-(\ell,m) \big]\, \mathrm{d}\ell \mathrm{d}m
%\end{split}
%\end{align}

\medskip
\subsection{Proof of \eqref{dkest2}}\label{ssecdkest2}
We now estimate the highest weighted norm.
Recall the notation \eqref{N_1L} and the formulas \eqref{dk_N1}-\eqref{dkN_12} for the first weight.
%Applying $\partial_k$ to $\mathcal{N}_1$ gave us
%\begin{align*}
%\partial_k \mathcal{N}_1(f,f) & = \mathcal{M}_1(f,f) + 2\mathcal{M}_2(f,f),
%\end{align*}
%where the terms on the right-hand side are defined in \eqref{dkN_11}-\eqref{dkN_12}.
Taking an extra derivative gives: %Could apply $\Delta_k$ to get cleaner formulas...
\begin{align}
\label{dk2N_10}
{\| \partial_k^2 \mathcal{N}_L(f,f) \|} & \lesssim \sum_{j=1}^3 {\| \mathcal{A}_j(f,f) \|}_{L^2}
\end{align}
where 
\begin{align}
\label{dk^2N_11}
\mathcal{A}_{1}(f,f) & = 
	\int_0^t \iint e^{is \Phi(k,\ell,m)} \, \partial_k^2 \widetilde{f}(s,\ell+k) \widetilde{f}(s,m)
 	\,\varphi_L(\ell)  \nu_1(\ell,m)\, \mathrm{d}\ell \mathrm{d}m \,\tau_S(s) \mathrm{d}s,
\\
\label{dk^2N_12}
\mathcal{A}_{2}(f,f) & = \int_0^t \iint s^2 |\ell|^2 \, e^{is \Phi(k,\ell,m)} 
	\, \widetilde{f}(s,\ell+k) \widetilde{f}(s,m)
	\, \varphi_{\sim L}(\ell) \nu_1(\ell,m) \, \mathrm{d}\ell \mathrm{d}m \,\tau_S(s) \mathrm{d}s,
\\
\label{dk^2N_13}
\mathcal{A}_{3}(f,f) & = \int_0^t \iint i s \ell \, e^{is \Phi(k,\ell,m)} 
	\, \partial_k \widetilde{f}(s,\ell+k) \widetilde{f}(s,m)
	\, \varphi_L(\ell) \nu_1(\ell,m) \, \mathrm{d}\ell \mathrm{d}m \,\tau_S(s) \mathrm{d}s.
\end{align}
Note that in \eqref{dk^2N_12} we have used the shorthand notation $|\ell|^2$ in place of the expressions 
$\ell_i\ell_j=\partial_{k_i} (k \cdot \ell) \partial_{k_j} (k \cdot \ell)$.

\medskip
\subsubsection{Estimate of \eqref{dk^2N_11}}
This term can be treated with a direct application of Theorem \ref{theomu1}
with an $L^2 \times L^{\infty-}$ estimate, and using the decay estimate from \eqref{aprioriL>6},
in the same way that we estimated \eqref{dkN_11} in \ref{dkN_11sec}.

\medskip
\subsubsection{Estimate of \eqref{dk^2N_12}}\label{ssecdk^2N_12}
This is the most complicated term due to the presence of an $s^2$ factor.
%To treat it we are going to use the same structure of proof used in Subsection \ref{ssecdkest1} above.
Similarly to \eqref{dkest12}-\eqref{dkest12b}, and using the 
notation  $$a_L(\ell,m):=a(\ell,m)=\varphi_{\sim L}(m)\varphi_L(\ell) \chi_-(\ell,m)$$ 
from \eqref{dkest12'} and \eqref{ssecM_2cut}, 
we split the integral in two pieces by defining
\begin{align}
\label{dk^2N_20}
\mathcal{A}_{2}^- & := \int_0^t \iint s^2 |\ell|^2 \, e^{is \Phi} 
	\, \widetilde{f}(s,\ell+k) \widetilde{f}(s,m)
	\, a(\ell,m)\nu_1(\ell,m) \, \mathrm{d}\ell \mathrm{d}m \,\tau_S(s) \mathrm{d}s.
\end{align} 
and
\begin{align}
\label{dk^2N_20+}
\mathcal{A}_{2}^+ & = \int_0^t \iint s^2 |\ell|^2 \, e^{is \Phi} 
	\, \widetilde{f}(s,\ell+k) \widetilde{f}(s,m)
	\, \varphi_{L}(\ell) \chi_+(\ell,m)  %\varphi_{>L-10}(|\ell|-|m|) 
	\nu_1(\ell,m) \, \mathrm{d}\ell \mathrm{d}m \,\tau_S(s) \mathrm{d}s.
\end{align} 
%We first look at the case $||\ell|-|m||\ll |\ell|$, and thus insert a cutoff $\varphi_{\sim L}(m)$ into \eqref{dk^2N_12}.
As before, the most difficult term will be \eqref{dk^2N_20}.

To start, note that we may again assume $L\geq L_0:= -S/3$ by an $L^6\times L^{3-}$ estimate similar
to the one that led to \eqref{dkest1alow}.
Next, recall the identity \eqref{PhiXPhiIBP}
\begin{align}
\label{PhiXPhiIBP'}
\frac{1}{c(\ell,m)} \Big( \frac{1}{is} |\ell|\X + i \partial_s \Big) e^{is\Phi} = e^{is\Phi},
%\label{PhiXPhic'}
\qquad |\nabla^\alpha_\ell \nabla^\beta_m c(\ell,m) |\lesssim |\ell|^{-2-|\alpha|-|\beta|},
%= |m|^2 + |\ell|^2 - 2|m|(|m|-|\ell|)
\end{align}
where the estimates hold on the support of $\mathcal{A}_{2}^-$ where $2^L \approx |\ell| \approx |m| \gg ||\ell|-|m||$.
The general idea is similar to the one used before, based on integration by parts using \eqref{PhiXPhiIBP'}.
This procedure will generate many terms. 
We will give full details for the ones that need to be analyzed more carefully, and skip some details for the easier ones.

Using \eqref{PhiXPhiIBP'} in \eqref{dk^2N_20} and doing some algebra one sees that
\begin{align}
\label{dk^2N_20'}
{\| \mathcal{A}_{2}^- \|}_{L^2_k} \lesssim {\| B_1 \|}_{L^2} + {\| B_2 \|}_{L^2} 
  + {\| C_1 \|}_{L^2} + {\| C_2 \|}_{L^2} + {\| D\|}_{L^2} + \cdots
\end{align}
where the first terms are those obtained integrating by parts in $s$:
\begin{align}
\label{dk^2N_21}
B_1 & = t^2 \iint e^{it \Phi} \frac{|\ell|^2}{c(\ell,m)}
	\, \widetilde{f}(t,\ell+k) \widetilde{f}(t,m)
	\, a(\ell,m) \nu_1(\ell,m) \, \mathrm{d}\ell \mathrm{d}m \,\tau_S(t),
\\
\label{dk^2N_22}
B_2 & = \int_0^t s^2 \iint \frac{|\ell|^2}{c(\ell,m)} \, e^{is \Phi} 
	\, \big( \partial_s \widetilde{f}(s,\ell+k) \big) \, \widetilde{f}(s,m)
	\, a(\ell,m)
	\nu_1(\ell,m) \, \mathrm{d}\ell \mathrm{d}m \,\tau_S(s) \mathrm{d}s;
\end{align}
other terms arise when integrating by parts in the $X$ direction without hitting one of the profiles
\begin{align}
\label{dk^2N_23}
C_1 & = \int_0^t s \iint  \, e^{is \Phi} \X\Big(\frac{|\ell|^3}{c(\ell,m)} a(\ell,m) \Big)
	\, \widetilde{f}(s,\ell+k) \, \widetilde{f}(s,m)
	\, \nu_1(\ell,m) \, \mathrm{d}\ell \mathrm{d}m \,\tau_S(s) \mathrm{d}s,
\\
\label{dk^2N_24}
C_2 & = \int_0^t s \iint  \, e^{is \Phi} \frac{|\ell|^3}{c(\ell,m)} a(\ell,m)
	\, \widetilde{f}(s,\ell+k) \, \widetilde{f}(s,m)
	\, \big( \X\nu_1(\ell,m) \big) \, \mathrm{d}\ell \mathrm{d}m \,\tau_S(s) \mathrm{d}s;
\end{align}
and
\begin{align}
\label{dk^2N_25.0}
D & := \int_0^t s \iint  \, e^{is \Phi} \frac{|\ell|^3}{c(\ell,m)} a(\ell,m)
	\, \big( \partial_{|\ell|}\widetilde{f}(s,\ell+k) \big) \, \widetilde{f}(s,m)
	\, \nu_1(\ell,m) \, \mathrm{d}\ell \mathrm{d}m \,\tau_S(s) \mathrm{d}s;
\end{align} 
in \eqref{dk^2N_20'} the ``$\cdots$'' denote all the other 
terms that are similar or easier to treat than those in \eqref{dk^2N_21}--\eqref{dk^2N_25.0}.
%
%Comment: Could write out the term with $X^2 \nu_1$ and just mention the use of \eqref{theomu1Xconc}}
%
The term \eqref{dk^2N_25.0} requires a further use of \eqref{PhiXPhiIBP'}, which gives
\begin{align}
& {\| D \|}_{L^2} \lesssim {\| D_1 \|}_{L^2} + {\| D_2 \|}_{L^2} + {\| D_3 \|}_{L^2} + {\| D_4 \|}_{L^2} + \cdots
\end{align}
where
\begin{align}
\label{dk^2N_25}
D_1 &:= \int_0^t \iint  \, e^{is \Phi} \frac{|\ell|^4}{c^2(\ell,m)} a(\ell,m)
	\, \big( \partial_{|\ell|}^2\widetilde{f}(s,\ell+k) \big) \, \widetilde{f}(s,m)
	\, \nu_1(\ell,m) \, \mathrm{d}\ell \mathrm{d}m \,\tau_S(s) \mathrm{d}s,
\\
\label{dk^2N_26}
D_2 &:= \int_0^t \iint  \, e^{is \Phi} \frac{|\ell|^4}{c^2(\ell,m)} a(\ell,m)
	\, \big( \partial_{|\ell|}\widetilde{f}(s,\ell+k) \big) \, \big( \partial_{|m|}\widetilde{f}(s,m) \big)
	\, \nu_1(\ell,m) \, \mathrm{d}\ell \mathrm{d}m \,\tau_S(s) \mathrm{d}s,
\\
\label{dk^2N_27}
D_3 &:= \int_0^t s \iint  \, e^{is \Phi} \frac{|\ell|^3}{c^2(\ell,m)} a(\ell,m)
	\, \big( \partial_{|\ell|}\widetilde{f}(s,\ell+k) \big) \, \big( \partial_s\widetilde{f}(s,m) \big)
	\, \nu_1(\ell,m) \, \mathrm{d}\ell \mathrm{d}m \,\tau_S(s) \mathrm{d}s,
\\
\label{dk^2N_28}
D_4 &:= \int_0^t s \iint  \, e^{is \Phi} \frac{|\ell|^3}{c^2(\ell,m)} a(\ell,m)
	\, \big( \partial_s \partial_{|\ell|} \widetilde{f}(s,\ell+k) \big) \, \widetilde{f}(s,m)
	\, \nu_1(\ell,m) \, \mathrm{d}\ell \mathrm{d}m \,\tau_S(s) \mathrm{d}s.
\end{align} 
and ``$\cdots$'' denote terms that are easier to estimate.

\smallskip
\paragraph{{\it Estimate of \eqref{dk^2N_21}--\eqref{dk^2N_22}}}
The boundary integral $B_1$ is the term that causes the growth in time of the weighted norm \eqref{dkest2}.
With the notation of Theorem \ref{theomu1} we have
\begin{align*}
\begin{split}
{\| B_1 \|}_{L^2_k} \lesssim t^2 \cdot {\| T_1[b]\big(\Ftil u(t), \Ftil u(t) \big) \|}_{L^2},
	\qquad b(\ell,m) = \frac{|\ell|^2 a(\ell,m)}{c(\ell,m)}.
\end{split}
\end{align*}
Applying \eqref{theomu1conc} (disregarding the remainder term on the right-hand side as usual),
and the decay estimates \eqref{aprioriL<6}, we get, for $t \approx 2^S$, 
\begin{align*}
{\| B_1 \|}_{L^2_k} & \lesssim 
t^2 \cdot \big[ {\| \Fhat \wt{u}(t) \|}_{L^3} {\| \Fhat \wt{u}(t) \|}_{L^{6-}}
  \cdot 2^{L} \cdot 2^{C_0A} %\tau_S(t)
\\ 
& \lesssim 2^{2S} \cdot \e 2^{-S/2} \cdot \e 2^{-(1-)S} \cdot 2^{(C_0+1)A}
\\
& \lesssim \e^2 2^{S/2-} 2^{(C_0+1)\delta_N}.
\end{align*}
This is bounded as in \eqref{dkest2} provided we choose $\delta$ depending on $N$ 
so that $(C_0+1)\delta_N < \delta$, see \eqref{d_N}.

The term \eqref{dk^2N_22} can be seen to satisfy an even stronger bound 
with a similar application of \eqref{theomu1conc} and using \eqref{estdsf}.

\smallskip
\paragraph{{\it Estimate of \eqref{dk^2N_23}--\eqref{dk^2N_24}}}
We have
\begin{align*}
\begin{split}
{\| C_1 \|}_{L^2_k} \lesssim \int_0^t s \, {\| T_1[b]\big(\Ftil u(s), \Ftil u(s) \big) \|}_{L^2} 
\, \tau_S(s) \mathrm{d}s, \qquad b(\ell,m) = X\Big(\frac{|\ell|^3}{c(\ell,m)} a(\ell,m) \Big)
\end{split}
\end{align*}
so that \eqref{theomu1conc} and \eqref{aprioriL<6} give us
\begin{align*}
{\| C_1 \|}_{L^2_k} 
& \lesssim \int_0^t  2^S \cdot {\| \Fhat \wt{u}(s) \|}_{L^3} {\| \Fhat \wt{u}(s) \|}_{L^{6-}}
  \cdot 2^{L} \cdot 2^{C_0A} \, \tau_S(s) \mathrm{d}s
\\ 
& \lesssim \int_0^t  2^S \cdot \e 2^{-S/2} \cdot \e 2^{-(1-)S}
	\cdot 2^{(C_0+1)A} \tau_S(s) \mathrm{d}s
\\
& \lesssim \e^2 2^{S/2} 2^{(C_0+2)\delta_N S}
\end{align*}
which suffices. %for $\delta$ as before.
The term \eqref{dk^2N_24} can be estimated similarly using the bilinear estimate \eqref{theomu1Xconc}.
%so we skip the details.

\smallskip
\paragraph{{\it Estimate of \eqref{dk^2N_25}--\eqref{dk^2N_27}}}
The first term \eqref{dk^2N_25} can be handled directly with an $L^2\times L^{\infty-}$ estimate,
using that the symbol satisfies the hypothesis \eqref{theomu1asb2} of Theorem \ref{theomu1},
and the $L^{\infty-}$ decay from \eqref{aprioriL>6}.

For the second term we have 
\begin{align*}
\begin{split}
{\| D_2 \|}_{L^2_k} \lesssim \int_0^t {\| T_1[b]\big(e^{is|k|^2}\partial_{|k|} \wt{f}(s), 
   e^{is|k|^2}\partial_{|k|} \wt{f}(s) \big) \|}_{L^2} 
\, \tau_S(s) \mathrm{d}s, \qquad b(\ell,m) = \frac{|\ell|^4a(\ell,m)}{c^2(\ell,m)}.
\end{split}
\end{align*}
We then use \eqref{theomu1conc} with an $L^{6-}\times L^3$ estimate, and interpolation, to get
\begin{align*}
\begin{split}
{\| D_2 \|}_{L^2_k} & \lesssim \int_0^t {\|\Fhat e^{is|k|^2}\partial_{|k|} \wt{f}(s)\|}_{L^6}
   {\| \Fhat e^{is|k|^2}\partial_{|k|} \wt{f}(s) \|}_{L^{3-}} \,
   \cdot 2^{L} \cdot 2^{C_0A} \tau_S(s) \mathrm{d}s
   \\
   & \lesssim \int_0^t 2^{-(1-(C_0+1)\delta_N)S} {\|\partial_{k}^2 \wt{f}(s)\|}_{L^2}
   \cdot {\| \Fhat e^{is|k|^2}\partial_{|k|}  \wt{f}(s)  \|}_{L^6}^{1/2-} {\| \partial_{|k|} \wt{f}(s) \|}_{L^2}^{1/2+}
   \, \tau_S(s) \mathrm{d}s
   \\
   & \lesssim \int_0^t 2^{-(1-(C_0+1)\delta_N)S} {\|\partial_{k}^2 \wt{f}(s)\|}_{L^2}
   \cdot 2^{-(1/2-)S} {\| \partial_{k}^2 \wt{f}(s) \|}_{L^2}^{1/2-} {\| \partial_{k} \wt{f}(s) \|}_{L^2}^{1/2+}
   \, \tau_S(s) \mathrm{d}s
   \\
   & \lesssim \int_0^t 2^{-(1-(C_0+1)\delta_N)S} \e 2^{(1/2+\delta)S} \cdot 2^{-(1/2-)S} \e 2^{(1/2-)(1/2+\delta)S}
   \, \tau_S(s) \mathrm{d}s
   \\
   & \lesssim \e^2 2^{S/2},
\end{split}
\end{align*}
having used, see \eqref{LemLinearL6}, that
%\begin{align*}
${\|\Fhat e^{is|k|^2}\wt{g}(s)\|}_{L^6} \lesssim 2^{-S} {\| \partial_{k} \wt{g}(s)\|}_{L^2}$ in the second and 
third inequalities. 
%\end{align*}

The term \eqref{dk^2N_27} can be treated similarly with an $L^6 \times L^{3-}$ estimate and using \eqref{estdsf}.
We skip the details.

\smallskip
\paragraph{{\it Estimate of \eqref{dk^2N_28}}}
This is the last term arising from the double integration by parts argument in the main term \eqref{dk^2N_20}.
\iffalse
To bound this we perform some additional algebraic manipulations first,
in order to avoid a separate estimate for $\partial_{|k|}\partial_s\wt{f}$.
More precisely, we integrate by parts $\partial_{|\ell|}$ to remove it from $\partial_s \wt{f}$;
this $\partial_{|\ell|}$ can then hit the exponential (we leave this term as is and estimate it below),
or the symbol (using that $L\geq -S/3$ this term is not hard to estimate, so we will disregard it)
or the distribution $\nu_1$;  
in this last case we add and subtract a term with $\partial_{|m|}\nu_1$ to reconstruct $X\nu_1$,
and then integrate away from $\nu_1$ the extra $\partial_{|m|}$.
Then, up to similar or easier terms, we have
\begin{align}
\begin{split}
{\| D_4 \|}_{L^2} \lesssim {\| E_1 \|}_{L^2} + {\| E_2 \|}_{L^2} + {\| E_3 \|}_{L^2} + \cdots 
\end{split}
\end{align}
with
\begin{align}
E_1 &:= \int_0^t s^2 \iint  \, e^{is \Phi} \frac{X\Phi \, |\ell|^3}{c^2(\ell,m)} a(\ell,m)
  \, \partial_s \widetilde{f}(s,\ell+k) \, \widetilde{f}(s,m)
  \, \nu_1(\ell,m) \, \mathrm{d}\ell \mathrm{d}m \,\tau_S(s) \mathrm{d}s,
\\
E_2 &:= \int_0^t s \iint  \, e^{is \Phi} \frac{|\ell|^3}{c^2(\ell,m)} a(\ell,m)
  \, \partial_s \widetilde{f}(s,\ell+k) \, \partial_{|m|} \widetilde{f}(s,m)
  \, \nu_1(\ell,m) \, \mathrm{d}\ell \mathrm{d}m \,\tau_S(s) \mathrm{d}s,
\\
E_3 &:= \int_0^t s \iint  \, e^{is \Phi} \frac{|\ell|^3}{c^2(\ell,m)} a(\ell,m)
  \, \partial_s \widetilde{f}(s,\ell+k) \, \widetilde{f}(s,m)
  \, X \nu_1(\ell,m) \, \mathrm{d}\ell \mathrm{d}m \,\tau_S(s) \mathrm{d}s.
\end{align} 
\fi
To bound it we need an additional bound for $\partial_{|k|}\partial_s\wt{f}$.
In particular, let us assume for the moment that 
\begin{align}\label{estdsXf}
{\| \varphi_{\leq A}(k) \partial_k \partial_t\wt{f}(t) \|}_{L^2_k} \lesssim 2^{-S/2} 2^{2\delta_NS}, \qquad t \approx 2^S.
\end{align}
Then we can bound
\begin{align*}
\begin{split}
{\| D_4 \|}_{L^2_k} \lesssim \int_0^t s \, 2^{-L} {\| T_1[b]\big(e^{is|k|^2}\partial_{|k|} \partial_s \wt{f}(s), 
   \wt{u}(s) \big) \|}_{L^2} 
\, \tau_S(s) \mathrm{d}s, \qquad b(\ell,m) = 2^L \frac{|\ell|^3 a(\ell,m)}{c^2(\ell,m)},
\end{split}
\end{align*}
and using  \eqref{theomu1conc} for an $L^2 \times L^{\infty-}$ estimate, with the decay bound \eqref{aprioriL>6}
and \eqref{estdsXf}, we get:
\begin{align*}
{\| D_4 \|}_{L^2_k}
   & \lesssim \int_0^t 2^S \cdot {\|\Fhat e^{is|k|^2}\partial_{|k|} \partial_s \wt{f}(s)\|}_{L^2}
   {\| \Fhat \wt{u}(s) \|}_{L^{\infty-}} \cdot 2^{C_0A} \tau_S(s) \mathrm{d}s
   \\
   & \lesssim \int_0^t 2^S \cdot  2^{-S/2 + 2\delta_N S}
   \cdot \e 2^{-(5/4-\delta)S} \cdot 2^{C_0\delta_N S} \tau_S(s) \mathrm{d}s
   \\
   & \lesssim \e^2 2^{S/2}
\end{align*}
since we have $(C_0+2)\delta_N + \delta \leq 1/4$.
The desired estimate \eqref{dkest1b} for \eqref{dk^2N_20} is concluded once we verify \eqref{estdsXf}.

%Quick argument:
%$\partial_t f$ is quadratic, and taking $\partial_k$ of it at most loses a factor of $t$.
%Then with an $L^6\times L^{3-}$ decay estimate we get (almost) $t^{-1/2}$. See the similar term \eqref{M_121}

To prove \eqref{estdsXf} we start from \eqref{Duhamel0} in the form
\begin{align*}
\partial_t \wt{f}(t,k) = \iint e^{it (-|k|^2 + |\ell|^2 + |m|^2 )} \widetilde{f}(t,\ell) \widetilde{f}(t,m)
  \, \mu(k,\ell,m) \, \mathrm{d}\ell \mathrm{d}m,
\end{align*}
and look back at the ``commutation formula'' \eqref{proGHW2} 
for $\partial_k$ and the operator \eqref{proGHW0} in Proposition \ref{proGHW}.
Using this and the same notation in that proposition, 
we obtain (compare with the similar terms obtained in \eqref{prHF1}):
%The only difference is that now there is no
%cutoff $\varphi_{\geq 0} = \varphi_{\geq 0}((|k|^2+|\ell|^2+|m|^2)s^{-2\delta_N})$, and no time integration.
\begin{align}
%\label{dkdsf0}
\nonumber
\partial_k \partial_t \wt{f} & = E_1 + E_2 + E_3,
\\
\label{dkdsf1}
E_1(f,f)(t,k) & = \iint e^{it \Phi(k,\ell,m)} \widetilde{f}(t,\ell) \widetilde{f}(t,m)
  \, \big( -it k \mu + \mu'  \big) \, \mathrm{d}\ell\mathrm{d}m,
\\
\label{dkdsf2}
E_2(f,f)(t,k) & = \iint e^{it \Phi(k,\ell,m)} \widetilde{f}(t,\ell) \widetilde{f}(t,m)
  \, i t \ell \, \mu'(k,\ell,m) \, \mathrm{d}\ell\mathrm{d}m,
\\
\label{dkdsf3}
E_3(f,f)(t,k) &= \iint e^{it \Phi(k,\ell,m)} 
  \partial_\ell \widetilde{f}(t,\ell) \widetilde{f}(t,m) \, \mu'(k,\ell,m) \, \mathrm{d}\ell\mathrm{d}m.
%E_3(f,f)(t,k) &= \iint e^{it \Phi(k,\ell,m)} 
%  \widetilde{f}(t,\ell) \widetilde{f}(t,m) \, \mu'(k,\ell,m) \, \mathrm{d}\ell \mathrm{d}m.
\end{align}
Notice that since $E_3(f,f)=\mathcal{B}^\prime(\partial_kf,f)$ this term
can be easily bounded using an $L^2 \times L^{\infty-}$ H\"older estimate
(which holds for the bilinear operator $\mathcal{B}^\prime$).

For the term \eqref{dkdsf1}
%using again the H\"older estimates from Proposition \eqref{}, 
we have, for $t \approx 2^S$,
\begin{align*}
{\| \varphi_{\leq A}(\cdot) E_1(t,\cdot) \|}_{L^2_k}
   & \lesssim 2^A \cdot 2^S \cdot {\|\Fhat \wt{u}(t)\|}_{L^3}
   {\| \Fhat \wt{u}(s) \|}_{L^{6-}} \lesssim \e^2 2^{2\delta_N S} 2^{-S/2}
\end{align*}
which is more than sufficient.
For the remaining term we can use a similar estimate when $|\ell| \leq 2^A$.
If instead $|\ell| \geq 2^A$ we can use an $L^2 \times L^{\infty-}$ bound and the 
control on the high Sobolev norm:
\begin{align*}
{\| E_2(t) \|}_{L^2}
   & \lesssim 2^S \cdot {\|\Fhat \varphi_{\geq A} \wt{u}(t)\|}_{H^1}
   {\| \Fhat \wt{u}(s) \|}_{L^{\infty-}}
   \\
   & \lesssim 2^S \cdot 2^{-A(N-1)} {\|\Fhat \wt{u}(t)\|}_{H^N} \cdot \e 2^{-S}
   \lesssim \e^2 2^{-S}.
\end{align*}
%since $A=\delta_N=6/N$.

\smallskip
\paragraph{{\it Estimate of \eqref{dk^2N_20+}}}
This term is easier to treat than the previous one \eqref{dk^2N_20}.
We can proceed similarly to the case of the corresponding term \eqref{dkest12b} 
in the estimate of the first weight, see \ref{lnotm}. %and the identities \eqref{lnotm1}-\eqref{lnotmIBP}.
In particular, we can use the identities \eqref{lnotm1}-\eqref{lnotmIBP} as explained before, 
to integrate by parts (once or twice) as done for the term \eqref{dk^2N_20} just above.
%We leave the details to the reader.

\medskip
\subsubsection{Estimate of \eqref{dk^2N_13}}
By using the same splitting depending on the size of $|\ell|-|m|$ relative to $|\ell|$ as before,
and integrating by parts once using the same formula \eqref{PhiXPhiIBP'} (and \eqref{lnotm1}) above, 
we can see that \eqref{dk^2N_13} gives rise to terms of the same form as those treated before. 
We can therefore skip the details.
The proof of \eqref{dkest2} and Proposition \ref{prodkL2} is concluded.

\medskip
\subsection{Estimates for $\mathcal{N}_{2,S}$}\label{ssecdkN_2}
We now prove the weighted bounds \eqref{dkest1}-\eqref{dkest2} for 
the term $\mathcal{N}_{2,S}$, see \eqref{dkL20}. 
We can use a similar strategy to the one used for $\mathcal{N}_{1,S}$.
To implement such a strategy we will need:
(1) a similar structure including
a ``good vectorfield'' and algebraic relations like \eqref{PhiXPhiIBP},
and (2) proper multiplier estimates, like those of Theorem \ref{theomu1}(i)
and Theorem \ref{theomu1'}(ii)-(iii).
%Some additional care needs to be used since here we are differentiating $\nu_1(\ell,k)$ in $k$.

As before we will sometimes drop some of the indexes (for example the index $S$ in $\mathcal{N}_{2,S}$)
when this causes no confusion.
For convenience let us recall the definition from \eqref{dkL20} here:
\begin{align}
\label{N_2}
\begin{split}
\mathcal{N}_{2,S}(t)(f,f) & = \int_0^t \iint e^{is \Psi(k,\ell,m)} 
  \widetilde{f}(s,-\ell-m) \widetilde{f}(s,m)
  \, \bar{\nu_1(\ell,k)}  \, \mathrm{d}\ell \mathrm{d}m \, \tau_S(s) \mathrm{d}s,
\\
& \Psi(k,\ell,m) := -|k|^2 + |\ell|^2 + 2\ell \cdot m + 2|m|^2.
\end{split}
\end{align}
We want to estimate $\partial_k \mathcal{N}_2$ and $\partial_k^2 \mathcal{N}_2$ as in Proposition \ref{prodkL2}
and show that, under the apriori assumptions \eqref{apriori}, we have
\begin{align}
\label{dkN_21}
& {\| \partial_k \mathcal{N}_{2,S}(t)(f,f) \|}_{L^2} \lesssim \e^2 2^{-\delta' S}, \quad \delta'>0
\\
\label{dkN_22}
& {\| \partial_k^2 \mathcal{N}_{2,S}(t)(f,f) \|}_{L^2} \lesssim 2^{(1/2+\delta)S} \e^2.
\end{align}

\subsubsection{Proof of \eqref{dkN_21}}\label{ssecd_kN21}
We concentrate only on the singular region where $||\ell|-|k||\ll |\ell|\approx |k|$.
The non-singular region can be treated more easily as in the case of $\nu_1(\ell,m)$ in Subsection \ref{lnotm},
and using part (ii) of Theorem \ref{theomu1'} to absorb derivatives in $k$ and integrate by parts in $\ell$;
also notice that $\partial_m$ is also always a ``good direction'' 
for integration since $\nu_1(\ell,k)$ is independent of $m$.

We restrict close to the singularity of $\nu_1(\ell,k)$
by inserting localization in $|\ell| \approx 2^L$, $|k| \approx 2^K$ 
and a cutoff $\chi_+(\ell,k)$ as in \eqref{theomu1'3b}.
For lighter notation, we do not display the cutoff and assume that $\nu_1 = \nu_1\chi_+$
(plus we disregard terms where derivatives hit the cutoff).

We recall the definition of ${\bf Y}$ from \eqref{theomu1'2b},
\begin{align}\label{dkN_2Y}
{\bf Y} = \partial_k + \frac{k}{|k|} \big( \frac{\ell}{|\ell|} \cdot \partial_{\ell} \big), 
\end{align}
and use this to compute 
\begin{align}\label{dkN_2}
\begin{split}
\partial_k \mathcal{N}_2(f,f) & = \int_0^t \iint (is \partial_k \Psi) \, e^{is\Psi(k,\ell,m)}
  \widetilde{f}(s,-\ell-m) \widetilde{f}(s,m)
  \, \bar{\nu_1(\ell,k)}  \, \mathrm{d}\ell \mathrm{d}m \, \tau_S(s) \mathrm{d}s
\\
& + \int_0^t \iint e^{is\Psi(k,\ell,m)}
  \widetilde{f}(s,-\ell-m) \widetilde{f}(s,m)
  \, {\bf Y} \bar{\nu_1(\ell,k)}  \, \mathrm{d}\ell \mathrm{d}m \, \tau_S(s) \mathrm{d}s
\\
& + \int_0^t \iint  \frac{k}{|k|} \, \mathrm{div}_\ell \Big( \frac{\ell}{|\ell|}  e^{is\Psi(k,\ell,m)}
  \widetilde{f}(s,-\ell-m) \widetilde{f}(s,m) \Big) 
  \, \bar{\nu_1(\ell,k)}\, \mathrm{d}\ell \mathrm{d}m \, \tau_S(s) \mathrm{d}s
\\
& = \mathcal{M}_3(f,f) + \mathcal{M}_4(f,f) + \mathcal{M}_5(f,f)
\end{split}
\end{align}
where
\begin{align}\label{dkN_23}
\mathcal{M}_3(f,f) & := \int_0^t \iint \frac{k}{|k|} e^{is \Psi(k,\ell,m)} 
  \partial_{|\ell|}\widetilde{f}(s,-\ell-m) \widetilde{f}(s,m)
  \, \bar{\nu_1(\ell,k)}\, \mathrm{d}\ell \mathrm{d}m \, \tau_S(s) \,\mathrm{d}s,
\end{align}
\begin{align}\label{dkN_24}
\begin{split}
\mathcal{M}_4(f,f) & := \int_0^t \iint is \, Z\Psi(k,\ell,m) \, %m_3(k,\ell,m) \, 
  e^{is \Psi(k,\ell,m)} 
  \widetilde{f}(s,-\ell-m) \widetilde{f}(s,m)
  \,\bar{\nu_1(\ell,k)}\, \mathrm{d}\ell \mathrm{d}m \, \tau_S(s) \,\mathrm{d}s,
\\
& %m_3(k,\ell,m) := 
Z\Psi := \big( \partial_k + \frac{k}{|k|} \frac{\ell}{|\ell|} \cdot \partial_\ell \Big) \Psi 
  = 2\Big[-k + \frac{k}{|k|}(|\ell| + \frac{\ell}{|\ell|} \cdot m)\Big],
\end{split}
\end{align}
and
\begin{align}\label{dkN_25}
\begin{split}
\mathcal{M}_5(f,f) & := \int_0^t \iint e^{is \Psi(k,\ell,m)} 
  \widetilde{f}(s,-\ell-m) \widetilde{f}(s,m)
  \, {\bf Y}' \bar{\nu_1(\ell,k)}
  \, \mathrm{d}\ell \mathrm{d}m \, \tau_S(s) \,\mathrm{d}s.
\\
& {\bf Y}' := \partial_k + \frac{k}{|k|} \mathrm{div}_\ell \frac{\ell}{|\ell|}.
\end{split}
\end{align}
We discuss briefly the \eqref{dkN_23} and \eqref{dkN_25} and 
give full details for the treatment of the harder term \eqref{dkN_24}.

Notice that \eqref{dkN_23} is essentially the same as \eqref{dkN_11},
and therefore can be estimated in the same way as we did in \ref{dkN_11sec}, 
with a direct application of the bilinear bound for the $T_2$-type operator in Theorem \ref{theomu1}.

In \eqref{dkN_25} %we view ${\bf Y}'$ as an operator acting on $\nu_1$.
note that ${\bf Y}' = {\bf Y} + (k/|k|)\mathrm{div}_\ell (\ell/|\ell|)$.
The piece with ${\bf Y} \bar{\nu_1(\ell,k)}$ is an operator of the form \eqref{theomu1'2a} with $a=1$,
and applying \eqref{theomu1'conc2} (after localization) 
for an $L^{6-}\times L^3$ estimate is more than sufficient.
Let us call $\mathcal{M}_5'$ the piece in \eqref{dkN_25} 
with symbol $(k/|k|)\mathrm{div}_\ell (\ell/|\ell|) \bar{\nu_1(\ell,k)}$.
After being localized to $|\ell| \approx 2^L$ and $|k|\approx 2^K$, 
we see that, as a  bilinear operator, it satisfies 
the same estimates of the bilinear operator with symbol $2^{-L}b(k,\ell,m)\bar{\nu_1(\ell,k)}$,
for a standard $b$ as in Theorem \ref{theomu1}. %since $|\ell|\approx |k|$.
%Therefore, when $L\geq K-5$ 
Then we can bound it using \eqref{theomu1concT2}:
\begin{align*}
{\| P_K \Fhat^{-1} \mathcal{M}_5'(f,f) \|}_{L^2}
  %\lesssim 2^K {\| P_K \Fhat^{-1} \mathcal{M}_5'(f,f) \|}_{L^{6/5}}
  & \lesssim 2^S \cdot 2^{-L} \cdot {\| \varphi_K T_2[b](f,f) \|}_{L^2}
  \\
  & \lesssim 2^S \cdot 2^{C_0A} \cdot {\|\Fhat \wt{u}(t)\|}_{L^{6-}} {\| \Fhat \wt{u}(s) \|}_{L^3}
  \lesssim \e^2 \cdot 2^{-S/4}.
\end{align*}

%so that the bilinear bound \eqref{theomu1'conc2}, and \eqref{theomu1concT2},
%imply that this term  will satisfy bounds 
%like those satisfied by $\mathcal{N}_2$ before being differentiated in $k$,
%times a factor of $2^{-L}$. Since we may assume that $L \geq -S/3$
%such a term can be bounded by the right-hand side of \eqref{dkN_21} via an $L^6 \times L^{3-}$ estimate.

%When instead $L<K-5$ one needs to deal with the slightly singular extra factor of $2^{-L}$.

\medskip
\subsubsection*{Estimate of \eqref{dkN_24}} 
We now show
\begin{align}\label{dkN_24.1}
{\| \mathcal{M}_4(t)(f,f) \|}_{L^2} \lesssim \e^2 2^{-\delta' S}, \quad \delta'>0.
\end{align}
This term is similar to \eqref{dkN_12} (it has a growing factor of $s$
and no differentiation of the profiles $\wt{f}$).
%As for \eqref{dkN_12} 
We will need to distinguish different cases depending on whether 
we are close to the singularity of $\nu_1(\ell,k)$ or not 
(see \ref{ssecM_2} for the similar splitting
in the case of \eqref{dkN_12}) and use a good vectorfield to integrate by parts 
(see the relations at the beginning of \ref{l=m}).
%Since the arguments are fairly similar to the previous ones, we only discuss the main aspects of the proof.

%All the terms in \eqref{dkN_2} can be estimated similarly to the ones in \eqref{dk_N1}.
%In particular, the terms \eqref{dkN_24}-\eqref{dkN_25} are similar to \eqref{dkN_11}-\eqref{dkN_12},
%and we can estimate them in the same way, provided a structure as in \eqref{XPhi2} is present for the phase 
%$\Psi$ in \eqref{N_2}. 

First of all, notice that the singular kernel $\nu_1(\ell,k)$ does not depend on $m$, 
so that one can use $X:=\partial_m$ as a ``good direction'' here.
We then calculate
\begin{align}
\label{PsiXPsi} 
\begin{split}
& X \Psi := 2 \ell + 4 m, \qquad X := \partial_{m},
\\
& \Psi(k,\ell,m) -  m \cdot X\Psi %:= -|k|^2 + |\ell|^2 - 2\ell \cdot m + 2|m|^2 
  %= (|\ell|-|k|)(|\ell| + |k|) + m \cdot X\Psi - 2|m|^2
  = (|\ell|-|k|)(|\ell| + |k|) - 2|m|^2
\end{split}
\end{align}
This last identity is the analogue of \eqref{PhiXPhi}, and an identity similar to \eqref{PhiXPhiIBP},
see \eqref{PsiXPsiIBP} below, follows from it.

To obtain \eqref{dkN_24.1} it suffices to show that
\begin{align}\label{dkN_24.1'}
{\| \mathcal{M}_{K,L,M}(t)(f,f) \|}_{L^2} \lesssim \e^2 2^{-2\delta' S}, \quad \delta'>0
\end{align}
where
\begin{align}\label{dkN_24'}
\begin{split}
\mathcal{M}_{K,L,M}(f,f) := \varphi_{K}(k) & \int_0^t \iint s \, Z\Psi \,
  e^{is \Psi} \varphi_{L}(\ell) \varphi_{M}(m) \,
  \\ & \times \widetilde{f}(s,-\ell-m) \widetilde{f}(s,m)
  \,\bar{\nu_1(\ell,k)}\, \mathrm{d}\ell \mathrm{d}m \, \tau_S(s) \,\mathrm{d}s,
\end{split}
\end{align}
is a localized version of \eqref{dkN_24} at scales $|\ell|\approx 2^L$, $|m|\approx 2^M$ and $|k|\approx 2^K$.
The reduction to \eqref{dkN_24.1'} can be done without loss of generality by estimating
the contribution of very small frequencies first, say for example $\min(K,L,M) \leq -10S$,
using Bernstein and Theorem \ref{theomu1}(i);
summing over the remaining dyadic scales can be done at the expense of a small loss of $O(S^3)+(A^3)$.
%see also a similar argument at the beginning of \ref{l=m}.

In order to use efficiently \eqref{PsiXPsi} we need to further split (we omit the dependence on the indexes $K,L,M$)
%(instead of just two like for \eqref{dkN_12}).
\begin{align*}
& \mathcal{M}_{K,L,M} = I_1+I_2+I_3
%\\
%& I_1 := \varphi_{K}(k)\int_0^t \iint s \, Z\Psi
%  e^{is \Psi} \varphi_{L}(\ell) \varphi_{M}(m) \,
%  \widetilde{f}(s,\ell-m) \widetilde{f}(s,m)
%  \,\bar{\nu_1(\ell,k)}\, \mathrm{d}\ell \mathrm{d}m \, \tau_S(s) \,\mathrm{d}s,
%\\
%& I_2 := \varphi_{K}(k)\int_0^t \iint s \, Z\Psi
%  e^{is \Psi} \varphi_{L}(\ell) \varphi_{M}(m) \,
%  \widetilde{f}(s,\ell-m) \widetilde{f}(s,m)
%  \,\bar{\nu_1(\ell,k)}\, \mathrm{d}\ell \mathrm{d}m \, \tau_S(s) \,\mathrm{d}s,
%\\
%& I_3 := \varphi_{K}(k)\int_0^t \iint s \, Z\Psi
%  e^{is \Psi} \varphi_{L}(\ell) \varphi_{M}(m) \,
%  \widetilde{f}(s,\ell-m) \widetilde{f}(s,m)
%  \,\bar{\nu_1(\ell,k)}\, \mathrm{d}\ell \mathrm{d}m \, \tau_S(s) \,\mathrm{d}s.
\end{align*}
where the following conditions on the support are imposed by inserting smooth cutoffs:
\begin{align}\label{dkN_24split}
\begin{split}
& I_1 \quad \mbox{is supported on} \quad  |\ell+2m| \geq 2^{\max(L,M)-10}
\\
& I_2 \quad \mbox{is supported on} \quad  |\ell+2m| \leq 2^{\max(L,M)-9} \quad 
  \mbox{and} \quad ||\ell|-|k|| \leq 2^{\max(L,K)-9}
\\
& I_3 \quad \mbox{is supported on} \quad  |\ell+2m| \leq 2^{\max(L,M)-9} \quad 
  \mbox{and} \quad ||\ell|-|k|| \geq 2^{\max(L,K)-10}
\end{split}
\end{align}

\subsubsection*{Estimate of $I_1$} 
Here we have $|\ell+2m| \geq 2^{\max(L,M)-10}$ and therefore $|X\Psi| \gtrsim 2^{\max(L,M)-10}$. 
In particular we can use the identity
\begin{align}\label{XPsi}
\frac{1}{is} \frac{X\Psi}{|X\Psi|^2} \cdot X e^{is\Psi} = e^{is\Psi}
\end{align}
to recover the factor of $s$ in \eqref{dkN_24'}.
Also notice that $X\Psi/|X\Psi|^2$ behaves like the symbol $1/|\ell+2m|$ which, on the support of $I_1$, gives
\begin{align}\label{XPsisym}
\big| \nabla^\alpha_\ell \nabla^\beta_m \frac{X\Psi}{|X\Psi|^2} \big|
  \lesssim \frac{1}{\sqrt{|\ell|^2+|m|^2}} |\ell|^{-|\alpha|} |m|^{-|\beta|}.
\end{align}
Then, when multiplied by $2^{\max(M,L)}$ this is a symbol like those allowed 
by \eqref{theomu1asb2} in the assumptions of Theorem \ref{theomu1}.
An $L^{6-}\times L^3$ application of \eqref{theomu1concT2} to \eqref{dkN_24'} 
gives a bound of $\e^2 2^{-S/2}$ up to powers of $2^A$. %stronger bound than the right-hand side of \eqref{dkN_24.1'}.

%Lack of cancellation from $Z\Psi$ plays a role?

\subsubsection*{Estimate of $I_2$}%$||\ell|-|k|| \leq 2^{\max(L,K)-10}$
In this case we have 
\begin{align*}
(19/20)|m| \leq |\ell| \leq (21/20)|m|, \qquad (19/20)|k| \leq |\ell| \leq (21/20)|k|
\end{align*}
and we can use efficiently the identity \eqref{PsiXPsi}.
We write
\begin{align}
\label{PsiXPsiIBP}
\begin{split}
& \frac{1}{d(k,\ell,m)} \Big( \frac{1}{is} m\cdot X + i \partial_s \Big) e^{is\Psi(k,\ell,m)} = e^{is\Psi(k,\ell,m)}, 
\\
& d(k,\ell,m) := -(|\ell|-|k|)(|\ell| + |k|) + 2|m|^2,
\end{split}
\end{align}
%compare with the analogous identity \eqref{PhiXPhiIBP}.
%In the present frequency scenario, that is, $|\ell-2m| \leq 2^{\max(L,M)-10}$ and 
%$||\ell|-|k|| \leq 2^{\max(L,K)-10}$, we have in particular that 
%$||\ell| - |k|| \ll |\ell|\approx |k| \approx |m|$.
and notice that, thanks to the current frequency restrictions, 
we have %the symbol $1/d$ behaves nicely:
\begin{align}
\label{PsiXPsisym}
\big| \nabla_k^a \nabla_\ell^\alpha \nabla_m^\beta \frac{1}{d(k,\ell,m)} \big| \lesssim \frac{1}{|\ell|^2+|m|^2+|k|^2} 
  |k|^{-|a|} |\ell|^{-|\alpha|} |m|^{-|\beta|},
\end{align}
consistently with the assumption \eqref{theomu1asb2}.
One can then use \eqref{PsiXPsiIBP} to integrate by parts in \eqref{dkN_24},
and arrive at the estimate \eqref{dkN_24.1'} through applications of the bilinear estimate \eqref{theomu1concT2}.
Since this arugment is similar to what was done before for the term \eqref{dkN_12} in \ref{ssecM_2}-\ref{l=m},
we can skip the details.

\subsubsection*{Estimate of $I_3$}%{Subcase 2.2: $||\ell|-|k|| \geq 2^{\max(L,K)-10}$}
In this last case $\nu_1(\ell,k)$ is not singular and we can integrate by parts also in $\ell$
through the identity
\begin{align}\label{dlPsi}
\frac{1}{is} \frac{\nabla_\ell \Psi}{|\nabla_\ell\Psi|^2} \cdot \nabla_\ell e^{is\Psi} = e^{is\Psi} 
\end{align}
and then make use of Theorem \ref{theomu1'}(ii). 
Notice in particular that since $\partial_\ell \Psi = 2(\ell+m)$,
then $\nabla_\ell \Psi/|\nabla_\ell\Psi|^2$ has the behavior of the symbol $1/|m+\ell|$
which is the same as $1/\sqrt{|\ell|^2+|m|^2}$ in the region under consideration: %where $|\ell-2m| \ll |\ell|+|m|$:
\begin{align}
\label{dlPsisym}
\big| \nabla_\ell^\alpha \nabla_m^\beta \frac{\nabla_\ell \Psi}{|\nabla_\ell\Psi|^2} \big| 
  \lesssim \frac{1}{\sqrt{|\ell|^2+|m|^2}} 
  |\ell|^{-|\alpha|} |m|^{-|\beta|}.
\end{align}
%We can then apply integration by parts arguments as above and use \eqref{theomu1'}(ii)
%to obtain the desired bound \eqref{dkN_24.1}.

\iffalse
In particular, we see that, in the singular scenario $||k|-|\ell|| \ll |\ell| \approx |k|$, 
we have that either $|Y\Psi| \gtrsim |\ell|\approx|k|$,
in which case we would integrate by parts in the $Y$ direction, or alternatively we have
\begin{align}
\label{YPsi2}
|Y\Psi| \ll |\ell| \qquad \Longrightarrow \qquad |\Psi| \gtrsim |\ell|^2,
\end{align}
in which case we integrate by parts in $s$.
To verify \eqref{YPsi2}, notice that if $||\ell| - |k||, |\ell - 2 m| \leq |\ell|/20$, then all frequencies have comparable sizes,
\begin{align}
|\ell| \leq (20/19)|k| \leq (21/19)|\ell|, \qquad  |\ell|\leq (40/19)|m| \leq (21/19)|\ell|,
\end{align}
and therefore 
\begin{align}
\begin{split}
|\Psi| \geq 2|m|^2 - (|\ell|+|k|)\big||\ell|-|k|\big| - 2|m||\ell-2m| \geq 2|m|^2 - \big[(|\ell|+|k|) + 2|m|\big]|\ell|/20
  \\ \geq 2|m|^2 - 6|m| \cdot (3/20)|m| \gtrsim |m|^2
\end{split}
\end{align}
as claimed.
One can then use the same arguments as in the previous subsection replacing $X$ and $\Phi$ by $Y$ and $\Psi$.
\fi

\subsubsection{Proof of \eqref{dkN_22}}\label{ssecdkN_22}
To complete the bounds on $\mathcal{D}_1$, see \eqref{dkL20}, we are left with showing
that the second derivative of $\mathcal{N}_{2,S}$, see \eqref{N_2}, can be estimated as in \eqref{dkN_22}.
Once again we restrict our analysis to the singular region $||k|-|\ell||\ll |\ell|\approx |k|$.
Applying $\partial_k$ to \eqref{dkN_2} we obtain
\begin{align*}
\partial_k^2 \mathcal{N}_2(f,f) = \partial_k \mathcal{M}_3(f,f) + 
  \partial_k \mathcal{M}_4(f,f) + \partial_k \mathcal{M}_5(f,f)
\end{align*}
see \eqref{dkN_23}--\eqref{dkN_25}.
As in the analysis of $\partial_k^2 \mathcal{N}_{1}$ in Subsection \ref{ssecdkest2},
this generates a large number of terms. 
It is not hard to see that many of them will be similar 
to those treated previously and can be handled by similar arguments, 
through the apposite bilinear estimates for $\nu_1(\ell,k)$ in \eqref{theomu1concT2}
and in Theorem \ref{theomu1'}(ii)-(iii), and the identities \eqref{XPsi} with \eqref{XPsisym}, 
and \eqref{PsiXPsiIBP} with \eqref{PsiXPsisym},
and \eqref{dlPsi} with \eqref{dlPsisym}.
Notice that the factors of $k/|k|$ in \eqref{dkN_23} and \eqref{dkN_25} do not constitute any additional difficulty 
since $\| \partial_k (k/|k|) F\|_{L^2} \lesssim \| \partial_k F\|_{L^2}$ by Hardy's inequality.

For completeness we discuss the details for the hardest terms which are coming from
$\partial_k \mathcal{M}_4(f,f)$.
Applying $\partial_k$ to the formula \eqref{dkN_24},
and the same algebra used to obtain \eqref{dkN_2}, we get
\begin{align}\label{dkN_22.0}
{\| \partial_k \mathcal{M}_4(f,f) \|} & \lesssim \sum_{j=1}^4 {\| \mathcal{M}_{4,j}(f,f) \|}_{L^2}
\end{align}
where %(we omit some of the arguments)
\begin{align}
\label{dkN_22.1}
\mathcal{M}_{4,1} & := \int_0^t \iint s \, Z\Psi \,
  e^{is \Psi} 
  \widetilde{f}(s,\ell-m) \widetilde{f}(s,m)
  \, {\bf Y}' \bar{\nu_1(\ell,k)}\, \mathrm{d}\ell \mathrm{d}m \, \tau_S(s) \,\mathrm{d}s,
\\
\label{dkN_22.2}
\mathcal{M}_{4,2} & := \int_0^t \iint s \, Z^2\Psi \,
  e^{is \Psi} 
  \widetilde{f}(s,\ell-m) \widetilde{f}(s,m)
  \,\bar{\nu_1(\ell,k)}\, \mathrm{d}\ell \mathrm{d}m \, \tau_S(s) \,\mathrm{d}s,
\\
\label{dkN_22.3}
\mathcal{M}_{4,3} & := \int_0^t \iint s \, Z\Psi \, 
  e^{is \Psi} 
  \, \partial_{|\ell|} \widetilde{f}(s,\ell-m) \widetilde{f}(s,m)
  \,\bar{\nu_1(\ell,k)}\, \mathrm{d}\ell \mathrm{d}m \, \tau_S(s) \,\mathrm{d}s,
\\
\label{dkN_22.4}
\mathcal{M}_{4,4} & := \int_0^t \iint s^2 \, (Z\Psi)^2 \, 
  e^{is \Psi} 
  \widetilde{f}(s,\ell-m) \widetilde{f}(s,m)
  \,\bar{\nu_1(\ell,k)}\, \mathrm{d}\ell \mathrm{d}m \, \tau_S(s) \,\mathrm{d}s.
\end{align}
Recall that the goal is to bound the $L^2$-norms of these terms by $\e^2 s^{(1/2+\delta)S}$.
%the right-hand side of \eqref{dkN_22}.

The first two terms %\eqref{dkN_22.1} and \eqref{dkN_22.2} 
are relatively easy to handle. \eqref{dkN_22.1} can be bounded with an $L^{6-} \times L^3$ estimates using
\eqref{theomu1concT2} and \eqref{theomu1'conc2} with $a=1$.
%: HERE NOW
For \eqref{dkN_22.2} some attention needs to be paid to the fact that $Z^2\Psi$ has a $1/|k|$ type singularity.
More precisely, from  \eqref{dkN_24} we see that $Z^2\Psi$ is made of harmless terms plus the
symbol
\begin{align}
z(k,\ell,m) := \partial_k (k/|k|) (\ell/|\ell|) \cdot (\ell+m).
\end{align}
The $1/|k|$ factor is only problematic away from the singularity of $\nu_1(\ell,k)$ when $|k|\ll |\ell|$. 
But in this case we can integrate by parts in $\ell$ using that
$z(k,\ell,m) = (1/2)\partial_k (k/|k|) (\ell/|\ell|) \cdot \partial_\ell \Psi$.
This gains back a factor of $s^{-1}$.
With Hardy's inequality $\||k|^{-1} F\|_{L^2} \lesssim \|F\|_{L^{6/5}}$
and an  $L^2 \times L^{3-}$ application of \eqref{theomu1concT2} (when $\partial_\ell$ hits the profile)
or \eqref{theomu1'conc3}  (when $\partial_\ell$ hits $\bar{\nu_1(\ell,k)}$), we get the desired bound.

The remaining two terms \eqref{dkN_22.3}--\eqref{dkN_22.4} are similar to \eqref{dk^2N_12}--\eqref{dk^2N_13}
and require integration by parts arguments using a splitting similar to \eqref{dkN_24split} 
and the formulas \eqref{XPsi}-\eqref{dlPsisym}.
We concentrate on the hardest contribution, which is \eqref{dkN_22.4}.

\subsubsection*{Estimate of \eqref{dkN_22.4}}\label{secdkN_22.4}
As before, without loss of generality we may assume that the integral \eqref{dkN_22.4}
is localized to $|\ell|\approx 2^L$, $|m|\approx 2^M$, and $|k|\approx 2^K$.
Moreover, we may reduce to the case $K \leq \max(L,M)+10$. Indeed, if $|k| \gg |\ell|,|m|$ 
we have that $|\Psi| \gtrsim |k|^2 \approx (|k|+|\ell|+|m|)^2$ and an integration by parts in $s$
will give us the desired bound; a similar argument was already used in the proof of Proposition \ref{proHF}, 
see \eqref{prHF2} and the arguments that follow.

We then split
\begin{align*}
& \mathcal{M}_{4,4} = J_1+J_2+J_3
\end{align*}
where we may assume that
\begin{align}\label{dkN_22.4split}
\begin{split}
& J_1 \quad \mbox{is supported on} \quad  |\ell+2m| \geq 2^{\max(L,M)-10}
\\
& J_2 \quad \mbox{is supported on} \quad  |\ell+2m| \leq 2^{\max(L,M)-9} \quad 
  \mbox{and} \quad ||\ell|-|k|| \leq 2^{\max(L,K)-9}
\\
& J_3 \quad \mbox{is supported on} \quad  |\ell+2m| \leq 2^{\max(L,M)-9} \quad 
  \mbox{and} \quad ||\ell|-|k|| \geq 2^{\max(L,K)-10}
\end{split}
\end{align}

To estimate $J_1$ we integrates by parts twice in $m$ using \eqref{XPsi}-\eqref{XPsisym}. 
Notice how the factor $Z\Psi$ helps canceling the mild singularity introduced by the symbol $X\Psi/|X\Psi|^2$.

The term $J_2$  in \eqref{dkN_22.4split} is the one which requires the most 
algebraic manipulations to integrate by parts using \eqref{PsiXPsiIBP}-\eqref{PsiXPsisym}.
However this term is almost identical to \eqref{dk^2N_20}, so the estimates can be done as in \ref{ssecdk^2N_12}.
Notice that since $|\ell|\approx|m|\approx|k|$ on the support of $J_2$, then the term $(Z\Psi)^2$
plays a role analogous to the factor $|\ell|^2$ in \eqref{dk^2N_20}, and helps to cancel 
the mild singularities introduced by the division by $d$, see \eqref{PsiXPsisym}.

Finally, for $J_3$ we can use integration in $\ell$  through \eqref{dlPsi}-\eqref{dlPsisym},
and applications of Theorem \ref{theomu1'}(ii).
This concludes the proof of \eqref{dkN_22} and the weighted estimates of $\mathcal{D}_1$.

%\comment{This last part is fairly quick\dots More details?
%
%Could add a few lines about cases of small frequencies, e.g. $M,L \leq -S/2+$
%
%Also about $\partial_k \mathcal{M}_5(f,f)$ which gives a term like
%\begin{align}\label{dkN_25Y}
%\begin{split}
%\mathcal{M}_{5,1}(f,f) & := \int_0^t \iint e^{is \Psi(k,\ell,m)} 
%  \widetilde{f}(s,-\ell-m) \widetilde{f}(s,m)
%  \, {\bf Y}^2 \bar{\nu_1(\ell,k)}
%  \, \mathrm{d}\ell \mathrm{d}m \, \tau_S(s) \,\mathrm{d}s.
%\end{split}
%\end{align}
%using Theorem \ref{theomu1'}, \eqref{theomu1'conc2}.}

%%%%%%%%%%%%%%%%%%%%%%%%%%%%%%%%%%%%%%%%%%%%%%%%%%
%%%%%%%%%%%%%%%%%%%%%%%%%%%%%%%%%%%%%%%%%%%%%%%%%%
%%%%%%%%%%%%%%%%%%%%%%%%%%%%%%%%%%%%%%%%%%%%%%%%%%
%%%%%%%%%%%%%%%%%%%%%%%%%%%%%%%%%%%%%%%%%%%%%%%%%%
%%%%%%%%%%%%%%%%%%%%%%%%%%%%%%%%%%%%%%%%%%%%%%%%%%
%%%%%%%%%%%%%%%%%%%%%%%%%%%%%%%%%%%%%%%%%%%%%%%%%%
%%%%%%%%%%%%%%%%%%%%%%%%%%%%%%%%%%%%%%%%%%%%%%%%%%
%%%%%%%%%%%%%%%%%%%%%%%%%%%%%%%%%%%%%%%%%%%%%%%%%%

\medskip
\section{Analysis of the NSD II: Lower order terms} %Lower order terms and remainders}
\label{SecOther}
In this section we analyze all remaining terms in the nonlinear equations.
We begin by studying the lower order terms $\mu_2$ and $\mu_3$ 
from the expansion of the distribution $\mu$, defined in \eqref{mu2}-\eqref{mu3}.
We will make use of the ``building block'' Lemma \ref{Lemmanu0}, and establish:

(1) structural propositions for $\mu_2$ and $\mu_3$ analogous to Proposition \ref{Propnu+} for $\mu_1$
(in fact $\nu_1$, see \eqref{mu1}), and 

(2) bilinear multiplier estimates for $\mu_2$ and $\mu_3$ which are the analogues of Theorem \ref{theomu1}.
We will then show how to use these in Subsection \ref{ssecmu23Est} to establish
the desired a priori bounds on the nonlinear terms from \eqref{Duhameldec} corresponding to $\mu_2$ and $\mu_3$.
We conclude with a discussion of the nonlinear estimates for the ``flat'' nonlinear part $\mathcal{D}_0$, 
see \eqref{Duhameldec}.

\medskip
%\subsection{Lower orders in the nonlinear spectral measure}
Recall from \eqref{mu2}--\eqref{Duhameldec} the nonlinear contribution to Duhamel's formula \eqref{Duhamel0}:
\begin{align}\label{SODuhamel}
\begin{split}
& \mathcal{D}(t)(f,f) = \mathcal{D}_0(t)(f,f) + \mathcal{D}_1(t)(f,f) + \mathcal{D}_{2}(t)(f,f) + \mathcal{D}_{3}(t)(f,f),
\\
& \mathcal{D}_0(t)(f,f) := \int_0^t \iint e^{is (-|k|^2 + |\ell|^2 + |k-\ell|^2 )} \widetilde{f}(s,\ell) 
  \widetilde{f}(s,k-\ell) \, \mathrm{d}\ell \,\mathrm{d}s,
\\
& \mathcal{D}_\ast(t)(f,f) := \int_0^t \iint e^{is (-|k|^2 + |\ell|^2 + |m|^2 )} \widetilde{f}(s,\ell) \widetilde{f}(s,m)
  \, \mu_\ast(k,\ell,m) \, \mathrm{d}\ell \mathrm{d}m \,\mathrm{d}s,
\end{split}
\end{align}
where $\mu_1$ is given in \eqref{mu1}-\eqref{nu_1},
and, as in \eqref{mu2}-\eqref{mu3},
\begin{align}\label{SOmu2}
\begin{split}
\mu_2(k,\ell,m) & := \nu_2^1(k,\ell,m) + \nu_2^2(k,\ell,m) + \nu_2^2(k,m,\ell)
\\
\nu_2^1(k,\ell,m) & := \int e^{-ix \cdot k} \frac{e^{i(|\ell|+|m|)|x|}}{|x|^2} \psi_1(x,\ell) \psi_1(x,m)\, \mathrm{d}x,
\\
\nu_2^2(k,a,b) & := \int e^{ix\cdot a} \frac{e^{i|x|(-|k|+|b|)}}{|x|^2} \bar{\psi_1(x,k)} \psi_1(x,b) \, \mathrm{d}x,
\end{split}
\end{align}
and
\begin{align}\label{SOmu3}
\begin{split}
\mu_3(k,\ell,m) := \int \frac{e^{i(-|k|+|\ell|+|m|)|x|}}{|x|^3} \bar{\psi_1(x,k)} \psi_1(x,\ell) \psi_1(x,m) \, \mathrm{d}x.
\end{split}
\end{align}

\medskip
\subsection{Analysis of $\mu_2$: Structure}\label{ssecmu2}
Let us begin by looking at $\nu_2^1(k,\ell,m)$. 
%& = \int e^{-ix \cdot k} \frac{e^{i(|\ell|+|m|)|x|}}{|x|^2} \psi_1(x,\ell) \psi_1(x,m)\, \mathrm{d}x,
We have the following analogue of Proposition \ref{Propnu+}:

\medskip
\begin{proposition}[Structure of $\nu_2^1$]\label{Propnu21}
Let $\nu_2^1$ be the measure defined in \eqref{SOmu2}, with %$\psi_1 \in \mathcal{G}^{N_1}$, see Definition \ref{Gclass}. 
$\psi_1$ defined by \eqref{psipsi1}. 
Fix $N_2 \in [5,N_1/4] \cap \Z$.
Let $k,\ell,m \in\R^3$ with $|k|\approx 2^K, |\ell|\approx 2^L$ and $|m|\approx 2^M$, 
and assume that $K,L,M \leq A$ for some $A>0$.
Then we can write
\begin{align}\label{Propnu210}
\nu_2^1(k,\ell,m) = %\nu_{2,0}^1(k,\ell,m) + 
  \nu_{2,0}^1(k,\ell,m) + \nu_{2,R}^1(k,\ell,m),
\end{align}
where:

%\medskip
\begin{itemize}

%\item[(1)] The leading order term is
%\begin{align}\label{Propnu211.1}
%\nu_{2,0}^1(k,\ell,m) := \frac{b_0(k,\ell,m)}{|k|} \chi\big(2^J(|k|-|\ell|-|m|)\big)
%\end{align}
%for some Schwartz function $\chi$, and with $b_0$ satisfying the bounds
%\begin{align}\label{Propnu211.2}
%\end{align}
%for all $K,L,M \leq A$, $|a|+|\a|+|\b| \leq N_1$.

\medskip
\item[(1)] 
%With 
%\begin{align}\label{NSM2.0}
%\mathcal{J}(A,P,Q) := \mathcal{J}:=\{J \geq 4A, \, J\geq -\min(P,Q)+ 4A\},
%\end{align}
$\nu_{2,0}^1(k,\ell,m)$ can be written as
\begin{align}\label{Propnu211}
\nu_{2,0}^1(k,\ell,m) = \frac{1}{|k|} 
  \sum_{i=1}^{N_2} \sum_{J\in\Z} b_{i,J}(k,\ell,m) \cdot K_i\big(2^J(|k|-|\ell|-|m|)\big) 
\end{align}
with %$b_a \in \mathcal{G}^{N_1-a}$ 
$K_i \in \mathcal{S}$ and the symbols $b_{i,J}$ satisfy
\begin{align}\label{Propnu212}
\begin{split}
  & \big| \varphi_K(k) \varphi_L(\ell) \varphi_M(m) 
  \nabla_k^a \nabla_\ell^\alpha \nabla_m^\beta  b_{i,J}(k,\ell,m) \big| 
  \\ 
  & \lesssim 2^{-|a|K} \cdot \big( 2^{|a|\max(L,M)} + 2^{(1-|\alpha|)L} 2^{(1-|\beta|)M} \big)
  \mathbf{1}_{\{ |K-\max(L,M)|<5 \}},
\end{split}
\end{align}
for all $K,L,M \leq A$, and $|a|+|\a|+|\b| \leq N_2$. %\leq N_1-N_2$.

\medskip
\item[(2)] For $M\leq L$, the remainder term $\nu_{2,R}^1$ satisfies %(note the cutoff before the differentiation)
\begin{align}\label{Propnu213}
\begin{split}
\big| %\varphi_K(k) \varphi_L(\ell)  \varphi_M(m) 
  \nabla_k^a \nabla_\ell^\alpha \nabla_m^\beta \nu_{2,R}^1(k,\ell,m) \big| 
  & \lesssim 2^{-2\max(K,L)} \cdot 2^{-|a|\max(K,L)}
  \\ & \times 2^{-|\a|\max(K,L)} \max(1,2^{-(|\b|-1)M}) \cdot 2^{(|a|+|\a|+|\b|+2)5A}
\end{split}
\end{align}
for all $K,L,M \leq A$ and $|a|+|\a|+|\b| \leq N_2/2 -2$.
A similar statement holds when $L \leq M$ exchanging the roles of $L$ and $M$ (and $\alpha$ and $\beta$).

%Previous version\dots
%%The proof does better in most cases\dots check if allow
%\begin{align}
%\begin{split}
%\big| %\varphi_K(k) \varphi_L(\ell)  \varphi_M(m) 
%  \nabla_k^a \nabla_\ell^\alpha \nabla_m^\beta \nu_{2,R}^1(k,\ell,m) \big| 
%  & \lesssim 2^{-\max(K,L,M)} \cdot 2^{-|a|\max(K,L,M)}
%  \\ & \times \max(1,2^{-(|\a|-1)L}) \max(1,2^{-(|\b|-1)M}) \cdot 2^{(|a|+|\a|+|\b|+2)5A}
%\end{split}
%\end{align}

\end{itemize}

\end{proposition}

\medskip
\begin{proof}
We proceed in two main steps. In the first step we analyze a building block analogous to the one in
\eqref{Lemmanu0def} of Lemma \ref{Lemmanu0};
the second step uses the first step, the expansion of $\psi$ from Lemma \ref{lemmapsi1}, 
and arguments similar to those in the proof of Proposition \ref{Propnu+}
to obtain the final statement.

With $|k|\approx 2^K, |\ell|\approx 2^L$ and $|m|\approx 2^M$, 
we assume without loss of generality $L\geq M$ and write
\begin{align}\label{nu21pr1}
\begin{split}
\nu_2^1(k,\ell,m) & = \big( \nu^+(k,\ell,m) + \nu^-(k,\ell,m) \big) \mathbf{1}_{\{|K-L|<5\}},
  + \nu_2^1(k,\ell,m) \mathbf{1}_{\{|K-L|\geq 5\}}
\\
\nu^+(k,\ell,m) & := \sum_{J \in \mathcal{J}} \nu^J(k,\ell,m)
\\
\nu^J(k,\ell,m) & := \int_{\R^3} e^{-ix \cdot k} 
  \frac{e^{i(|\ell|+|m|)|x|}}{|x|^2} \psi_1(x,\ell) \psi_1(x,m) \, \varphi_J(x) \, \mathrm{d}x,
\end{split}
\end{align}
%where $\varphi_J^{(0)}$ is defined in \eqref{LP0}-\eqref{LP2}.
where
\begin{align}\label{nu21prJ}
\mathcal{J} & := %\{ J \geq 4A, \, J \geq -K + 4A\}.
  \{ J + \min(K,0) \geq 4A\}.
  %\{ J + \min(K,L,0) \geq  4A\}.
\end{align}

\medskip
{\it The ``building block''}.
%The starting point are the arguments in Subsection \ref{ssecprnu0}
%where we proved the ``building block'' Lemma \ref{Lemmanu0}.
Similarly to \eqref{Lemmanu0def} we can define
\begin{align}\label{Lemmanu0def'}
\begin{split}
\K_J(k,\ell,m) := \int_{\R^3} e^{-ix \cdot k} \frac{e^{i|x|(|\ell|+|m|)}}{|x|^2} g_1(\omega,\ell)g_2(\omega,m) 
  \, \varphi(x 2^{-J}) \, \mathrm{d}x, \\ J\in \mathcal{J}, \quad g_i \in \mathcal{G}^{s_i}
\end{split}
\end{align}
as our building block, for some $N_2 \leq s_i \leq N_1$, and a generic compactly supported function $\varphi$.
%Compare with the definition of $\K_J$ from \eqref{Lemmanu0def}.
%\begin{align}\label{Lemmanu0def}
%\K_J(p,q) := \int_{\R^3} e^{ix \cdot p} \frac{e^{i|x||q|}}{|x|} g(\omega,q) \, \varphi(x 2^{-J}) 
%  \, \mathrm{d}x, \qquad J\in \mathcal{J};
%\end{align}
%With a slight abuse of notation we can regard $\nu^J(k,\ell,m)$ in \eqref{nu21pr1} as having the following form:
%\begin{align}\label{nu21pr2}
%\nu^J(k,\ell,m) \approx \K_J(k,|\ell|+|m|).
%\end{align}
We then write the integrand in polar coordinates $x=r\omega$
\begin{align}\label{nu21pr5}
\begin{split}
\K_J(k,\ell,m) & = \int_0^\infty e^{ir(|\ell|+|m|)} \Big( \int_{\mathbb{S}^2}
  e^{-ir \omega \cdot k}  g_1(\omega,\ell)g_2(\omega,m) \,\mathrm{d}\omega \Big) 
  \, \varphi(r2^{-J}) \, \mathrm{d}r,
\\
& = \int_0^\infty e^{ir(|\ell|+|m|)} \big(I_0 + II%(r,k,\ell,m)
  \big) \, \varphi(r2^{-J}) \, \mathrm{d}r,
\end{split}
\end{align}
where,
%Without loss of generality we assume $k=|k|e_3$
%\begin{align}\label{nu21pr6}
%\end{align}
with $X := r|k|$, we define
\begin{align}\label{nu21pr7}
\begin{split}
%I & = \int_0^{2\pi} \int_0^\pi e^{iX \cos\phi} g(\omega,q)\, \sin\phi \, \mathrm{d}\phi \mathrm{d}\theta = I_0 + II
%\\ 
I_0 & := \frac{2\pi}{iX} \big[ e^{-iX} g_1(k/|k|,\ell) g_2(k/|k|,m)- 
  e^{iX} g_1(-k/|k|,\ell) g_2(-k/|k|,m) \big]
\end{split}
\end{align}
and
\begin{align}\label{nu21pr8}
\begin{split}
II & := \frac{1}{iX} \int_0^{2\pi} \int_0^\pi e^{-iX \cos\phi} 
  \partial_\phi \big( g_1(\omega,\ell) g_2(\omega,m) \big) \,\mathrm{d}\phi \mathrm{d}\theta.
\end{split}
\end{align}
Compare with \eqref{prLemma1pol}--\eqref{prLemma1Isplit}.
The contribution to \eqref{nu21pr5} from $I_0$ is, up to omitting irrelevant constants, a sum of the terms
\begin{align}\label{nu21pr9}
\begin{split}
I_0^\pm := & \int_0^\infty  \frac{1}{r|k|} e^{ir(|\ell|+|m| \pm |k|))} g_1(\mp k/|k|,\ell) g_2(\mp k/|k|,m)
  \, \varphi(r 2^{-J}) \, \mathrm{d}r
  \\
  = & \frac{1}{|k|}\, g_1(\mp k/|k|,\ell) g_2(\mp k/|k|,m) \chi\big(2^J(|\ell|+|m| \pm |k|)\big)
\end{split}
\end{align}
where $\chi := \what{\varphi/r}$.

For the term \eqref{nu21pr8}
we want apply the stationary phase estimates \eqref{ST1}-\eqref{ST2} 
to the integral in $\mathrm{d}\phi$ (recall $2^J|k| \gg 1$)
as we did for \eqref{prL15} in the proof of Lemma \ref{Lemmanu0}.
We can use the same arguments there and obtain analogous formulas by replacing the quantities in \eqref{prL15}
as follows:
\begin{align}\label{nu21pr10}
g(\omega,q) \mapsto g_1(\omega,\ell) g_2(\omega,m), \qquad
e^{ir|q|} \mapsto e^{ir(|\ell|+|m|)}, \qquad |p| \mapsto |k|, \qquad 
\varphi \mapsto 2^{-J} (\varphi/r).
\end{align}
This gives an expansion with properties similar to \eqref{prL17.1}-\eqref{prL18.2} as follows:
%We then plug the asymptotics \eqref{prL1asymain1}-\eqref{prL1asymain2}
\begin{align}
\nonumber
& \int_0^\infty e^{ir(|\ell|+|m|)} II(X;\ell,m) \, \varphi(r2^{-J}) \, \mathrm{d}r 
\\
\label{nu21pr10.1}
& =\sum_{j=0}^{M'-1} \frac{%2\pi 
  b_j^-(k,\ell,m)}{%i
  |k|^{3/2 + j/2}} \,
  \int_0^\infty e^{ir(-|k|+|\ell|+|m|)} r^{-(j+1)/2} \, 2^{-J} (\varphi/r)(r2^{-J}) \,\mathrm{d}r
\\
\label{nu21pr10.2}
& + \sum_{j=0}^{M'-1} \frac{%2\pi 
  b_j^+(k,\ell,m)}{%i
  |k|^{3/2 + j/2}} \,
  \int_0^\infty e^{ir(|k| + |\ell| + |m|)} \, r^{-(j+1)/2} 2^{-J} (\varphi/r)(r2^{-J}) \,\mathrm{d}r
\\
\label{nu21pr10.3}
& + \frac{1}{|k|}R_{J,M'}^-(k,\ell,m) + \frac{1}{|k|}R_{J,M'}^+(k,\ell,m),
\end{align}
where
\begin{align}\label{nu21pr10.4}
b_j^\pm(k,\ell,m) := c_j \partial_\phi^{j+1} \big( g_1(\omega,\ell) g_2(\omega,m) 
  \big)\big|_{\omega=\pm k/|k|},
\end{align}
for some constants $c_\ell$
(we are using here the same convention adopted in the proof of Lemma \ref{Lemmanu0} as per the discussion
following \eqref{ST2})
and
\begin{align}\label{nu21pr11.1}
& R_{J,M'}^\pm(k,\ell,m) := \int_0^\infty e^{\pm ir|k|} \, e^{ir(|\ell|+|m|)} 
  \, 2^{-J}(\varphi/r)(r2^{-J}) B^\pm(r|k|;\ell,m) \,\mathrm{d}r,
\end{align}
where $B^\pm(X;\ell,m)$ satisfy
%\begin{align}\label{prL18.B}
%B^\pm(X;q) := J_\pm(X) - e^{\pm iX} X^{-1/2} \sum_{\ell=0}^{M-1} b_\ell^\pm X^{-\ell/2} 
%\end{align}
%satisfies
\begin{align}\label{nu21pr11.2}
\begin{split}
\Big| \Big(\frac{d}{dX}\Big)^a \partial_\ell^\alpha  \partial_m^\beta B^\pm(X;\ell,m) \Big| 
  & \lesssim X^{-M'/2-a} \sup_{\omega\in\mathbb{S}^2} 
  \big|\partial_\phi^{M'+1} \big( \partial_\ell^\alpha g_1(\omega,\ell) \partial_m^\beta g_2(\omega,m) \big) \big|.
\end{split}
\end{align}
%Moreover, we can see that
%\begin{align}\label{nu21pr11.3}
%\begin{split}
%\Big| \partial_p^\alpha \partial_q^\beta B(r|p|;q) \Big| & \lesssim 
%  \sup_{\alpha_1 + \alpha_2 = \alpha} 
%  X^{-M/2-|\alpha_1|} r^{|\alpha_1|}
%  \sup_{|\alpha'| \leq |\alpha_2|, \, \omega\in\mathbb{S}^2}|\partial_\phi^{M+1} \partial_\omega^{\alpha'} \partial_q^\beta g(\omega,q)|
%  \cdot |p|^{-|\alpha_2|}
%  \\ 
%  & \lesssim %\sup_{\alpha_1 + \alpha_2 = \alpha} 
%  X^{-M/2} 2^{-P|\alpha|} \cdot (2^{Q(M+1+|\alpha|)} + 2^{-Q|\beta|}) 
%\end{split}
%\end{align}

\medskip
{\it Expansion of $\nu_2^1$}.
From the definition \eqref{nu21pr1} and using the expansion of $\psi_1$ from Lemma \ref{lemmapsi1}, we can write
(omitting the dependence on some of the indexes)
\begin{align}\label{nu21pr20}
\nu^J(k,\ell,m) & = \sum_{j_1,j_2=0}^{N_2-1} \nu^J_{j_1j_2}(k,\ell,m) + R_1(k,\ell,m) + R_2(k,\ell,m), 
\end{align}
where
\begin{align}\label{nu21pr20.1}
\begin{split}
\nu^J_{j_1j_2}(k,\ell,m) := \int_{\R^3} e^{-ix \cdot k} 
  \frac{e^{i(|\ell|+|m|)|x|}}{|x|^2} g_{j_1}(x,\ell) g_{j_2}(x,m) \, |x|^{-j_1-j_2} \langle \ell \rangle^{j_1} 
  \langle m \rangle^{j_2} \, \varphi_J(x) \, \mathrm{d}x,
  \\  %g_{j_1}, g_{j_2} \in \mathcal{G},
  g_{j_i} \in \mathcal{G}^{N_1-j_i}, %, \quad g_{j_2} \in \mathcal{G}^{N-j_2},
\end{split}
\end{align}
and
\begin{align}
\label{nu21pr20.2}
R_1(k,\ell,m) & := \int_{\R^3} e^{-ix \cdot k} 
  \frac{e^{i(|\ell|+|m|)|x|}}{|x|^2} \big(\psi_1(x,\ell) - R_{N_2}(x,\ell) \big) R_{N_2}(x,m) \,
  \varphi_J(x) \, \mathrm{d}x,
\\
\label{nu21pr20.3}
R_2(k,\ell,m) & := \int_{\R^3} e^{-ix \cdot k} 
  \frac{e^{i(|\ell|+|m|)|x|}}{|x|^2} R_{N_2}(x,\ell) \psi_1(x,m) \,
  \varphi_J(x) \, \mathrm{d}x.
\end{align}

For each of the terms in \eqref{nu21pr20.1} we use the result in the first step above.
Indeed we can write
\begin{align}\label{nu21pr21}
\begin{split}
\nu^J_{j_1j_2}(k,\ell,m) = \frac{\langle \ell \rangle^{j_1}\langle m \rangle^{j_2}}{2^{J(j_1+j_2)}} 
  \mathcal{K}_J(k,\ell,m)
\end{split}
\end{align}
where $\mathcal{K}_J(k,\ell,m)$ is a building block as in \eqref{Lemmanu0def'} 
with $g_{j_i} \in \mathcal{G}^{N_1-j_i}$.
Then, each one of the terms $\nu^J_{j_1j_2}$ admits an expansion as in \eqref{nu21pr5}--\eqref{nu21pr11.2},
up to a multiplication by the factor $\langle \ell \rangle^{j_1}\langle m \rangle^{j_2} 2^{-J(j_1+j_2)}$.
We now analyze the terms in such expansions.

\medskip
{\it The leading order terms \eqref{Propnu211}-\eqref{Propnu212}}.
According to \eqref{nu21pr9}, the first term in the expansion of $\nu^J_{j_1j_2}(k,\ell,m)$ has the form 
\begin{align}
\frac{\langle \ell \rangle^{j_1}\langle m \rangle^{j_2}}{2^{J(j_1+j_2)}}
  \frac{1}{|k|} g_{j_1}(\mp k/|k|,\ell) g_{j_2}(\mp k/|k|,m) \chi\big(2^J(|\ell|+|m| \pm |k|)\big) 
\end{align}
With the choice of the $-$ in the argument of $\chi$ this is a term as in \eqref{Propnu211}-\eqref{Propnu212}.
When the argument of $\chi$ is $2^J(|k|+|\ell|+|m|)$ we obtain a remainder term that can be absorbed in 
$\nu_{2,R}^1$ consistently with \eqref{Propnu213}.

The next terms in the expansion, corresponding to \eqref{nu21pr10.1}-\eqref{nu21pr10.2}
have the form
\begin{align}
\sum_{j=0}^{M'-1} \frac{b_j^\pm(k,\ell,m)}{
  |k|^{3/2 + j/2}} \, 2^{-J(j+3/2)} \what{\chi_j} \big(2^J(\pm |k|+|\ell|+|m|)\big)
\end{align}
for some Schwartz functions $\chi_j$, and with coefficients of the form, see \eqref{nu21pr10.4},
\begin{align*}
b_j^\pm(k,\ell,m) = \frac{\langle \ell \rangle^{j_1}\langle m \rangle^{j_2}}{2^{J(j_1+j_2)}}
 c_j \partial_\phi^{j+1} \big( g_{j_1}(\omega,\ell) g_{j_2}(\omega,m) \big)\big|_{\omega=\pm k/|k|}.
\end{align*}
When the choice of the sign is `$-$' in the argument of $\what{\chi_j}$, these terms belong to 
the sum in \eqref{Propnu211}-\eqref{Propnu212},
provided we impose $|a|,|\alpha|,|\beta| \leq N_1-N_2-M'$.
With the opposite choice of the sign we obtain smooth remainder terms satisfying the bounds \eqref{Propnu213}.

%\medskip
%{\it  \eqref{Propnu211}-\eqref{Propnu212}}.

\medskip
{\it Remainder terms}.
We have several types of remainders: those coming from the ``building block'' expansion of \eqref{nu21pr20.1},
the remainders \eqref{nu21pr20.2}--\eqref{nu21pr20.3}, the measure $\nu^-$ in \eqref{nu21pr1},
and the full $\nu_2^1$ when $|K-L|\geq 5$.

\medskip
According to \eqref{nu21pr10.3} and \eqref{nu21pr11.1}-\eqref{nu21pr11.2} 
in the expansion of the building block,
each of the terms $\nu^J_{j_1j_2}(k,\ell,m)$ in \eqref{nu21pr20.1} has a remainder of the form
\begin{align}
\label{nu21pr30.1}
& R^\pm_{J,M',j_1,j_2}(k,\ell,m) := \frac{1}{|k|} \int_0^\infty e^{\pm ir|k|} \, e^{ir(|\ell|+|m|)} 
  \, 2^{-J}(\varphi/r)(r2^{-J}) B^\pm(r|k|;\ell,m) \,\mathrm{d}r,
\\
\label{nu21pr30.2}
& \Big| \Big( \frac{d}{dX}\Big)^a \partial_\ell^\alpha \partial_m^\beta B^\pm(X;\ell,m) \Big| 
  \lesssim X^{-M'/2-a} \sup_{\omega\in\mathbb{S}^2} 
  \big|\partial_\phi^{M'+1} \big( \partial_\ell^\alpha g_{j_1}(\omega,\ell) 
  \partial_m^\beta g_{j_2}(\omega,m) \big) \big|.
\end{align}
To see how these terms satisfy the desired estimate \eqref{Propnu213},
we can use arguments similar to those given at the end of the proof of Lemma \ref{Lemmanu0}, see \eqref{prL22},
so we will only give some of the details.

%We first look at the case $||\ell|+|m|-|k|| \ll |k| \approx |\ell|+|m|$ and see that %\approx |\ell|$.
Using \eqref{nu21pr30.1}-\eqref{nu21pr30.2} we see that
\begin{align}\label{nu21pr30.3}
\Big| R^\pm_{J,M',j_1,j_2}(k,\ell,m) \Big| \lesssim 2^{-K} \cdot 2^{(J+K)(-M'/2)} \cdot 2^{(M'+1)A}
  \lesssim 2^{-K} \cdot 2^{(J+K)(-M'/4)},
\end{align}
which gives \eqref{Propnu213} with $a=\alpha=\beta=0$ since $K \geq \max(K,L,M)-10$ and $J+K \geq 4A$.
For the general estimates we apply derivatives and notice that,
compared to the right-hand side of \eqref{nu21pr30.1},
each $\nabla_{(k,\ell)}$-derivative is going to cost an additional factor of $\approx 2^{-K} + 2^J \lesssim 2^J$,
and each $\nabla_{m}$-derivative is going to cost $2^J + 2^{-M}$. %\lesssim 2^J$. %(recall $M\leq L$).
Then, for $|\beta|\geq 1$, we get
\begin{align}\label{nu21pr30.4}
\begin{split}
& \Big| \nabla_k^a \nabla^\alpha_\ell \nabla^\beta_m R^\pm_{J,M',j_1,j_2}(k,\ell,m) \Big| 
  \\
  & \lesssim 2^{-K} \cdot 2^{(J+K)(-M'/2)} \cdot 2^{(M'+1)A} \cdot 2^{J(|a|+|\alpha|)}
  \cdot 2^J(2^{(|\beta|-1)J} + 2^{(1-|\beta|)M}).
%  \\
%  & \lesssim 2^{-K} \cdot 2^{(J+K)(-M'/4)},
\end{split}
\end{align}
When $J\leq -M$ we get \eqref{Propnu213} since $J+K \geq 4A$, $|K - L| < 5$,
%and we can impose $|a|+|\alpha|\leq M'/4-2$ 
and we can choose $M'$ large enough.
For $J\geq -M$, and $|a|+|\alpha|+|\beta|\leq M'/4-2$ we get an even better bound
\begin{align*}
& \Big| \nabla_k^a \nabla^\alpha_\ell \nabla^\beta_m R^\pm_{J,M',j_1,j_2}(k,\ell,m) \Big| 
  %\\ & \lesssim 2^{-K} \cdot 2^{(J+K)(-M'/2)} \cdot 2^{(M'+1)A} \cdot 2^{J(|a|+|\alpha|+|\beta|)} 
  \lesssim 2^{-K} \cdot 2^{-K(|a|+|\alpha|+|\beta|)}.
\end{align*}
%All these factors can be absorbed by the factor $2^{(J+K)(-M'/2)}$ on the right-hand side of \eqref{nu21pr30.1}:
%using $|a|+|\alpha|+|\beta| \leq M'/4$, we have
%\[ 2^{(|a|+|\alpha|+|\beta|)J} \cdot 2^{(J+K)(-M'/2)} \lesssim  2^{-(|a|+|\alpha|+\beta|)K} 2^{-2(J+K)},\]
%consistently with \eqref{Propnu213} since we are considering $K \geq \max(K,L,M)-10$.

\iffalse
For the remaining case $||\ell|+|m|-|k|| \gtrsim \max(|k|,|\ell|)$,
after applying $\nabla^a_k \nabla^\alpha_\ell \nabla^\beta_m$,
we can recover the losses of $2^{(|a|+|\alpha|+\beta|)J}$ integrating by parts in $r$ in \eqref{nu21pr30.1}, 
using
\[ e^{ir (\pm |k| + |\ell|+|m|)}
   = \frac{-i}{\pm|k| + |\ell|+|m|} \frac{d}{dr} \big[ e^{ir(\pm |k|+|\ell|+|m|)} \big], \]
and
\[ \Big| \Big(\frac{d}{dr}\Big)^D \big[ (\varphi/r)(r2^{-J}) B^{\pm}(r|k|;\ell,m) \big] \Big| 
  \lesssim 2^{-JD} \cdot 2^{(J+K)(-M'/4)}. \] %2^{(M+1)A}. \]
\fi

\medskip
Next, we analyze the remainders \eqref{nu21pr20.2}-\eqref{nu21pr20.3}.
These are similar to the remainder \eqref{prnu+R} in the proof of Proposition \ref{Propnu+}.
The argument are similar to those given after \eqref{prnu+30}.
We only analyze \eqref{nu21pr20.3}, because \eqref{nu21pr20.2} can be handled in the same way since each
term in the expansion of $\psi_1-R_{N_2}$ has properties similar to $\psi_1$.
We calculate
\begin{align}\label{nu21pr40}
\begin{split}
\nabla_k^a \nabla_\ell^\alpha \nabla_m^\beta R_2(k,\ell,m)
  = \int_{\R^3} \frac{(-ix)^a}{|x|^2} e^{-ix\cdot k} 
   & \nabla_\ell^\alpha \big( e^{i|x||\ell|} R_{N_2}(x,\ell) \big) 
   \\ \times  & \nabla_m^\beta \big( e^{i|x||m|} \psi_1(x,m)\big) \, \varphi_J(x) \, \mathrm{d}x,
\end{split}
\end{align}
and use the estimate \eqref{psi10} for $\psi_1$ and \eqref{lemmapsi1R} for $R_{N_2}$ to see that
%for $J + \min(K,L,M) \geq 4A$,
\begin{align}\label{nu21pr41}
\begin{split}
|\nabla_k^a \nabla_\ell^\alpha \nabla_m^\beta R_2(k,\ell,m)|
  \lesssim 2^{J} \cdot 2^{J|a|} \cdot 2^{|\alpha|J} %2^J \big(2^{(|\alpha|-1)J} + 2^{(1-|\alpha|)L} \big) 
  \cdot 2^{-N_2J} 2^{N_2L_+} 
  \cdot 2^J \big( 2^{(|\beta|-1)J} + 2^{(1-|\beta|)M} \big).
  %\\
  %& \lesssim 2^{-J} 2^{-(|a|+|\alpha|+|\beta|)J}
\end{split}
\end{align}
In the case $J\leq -M$, $|\beta|\geq 1$, this gives
\begin{align*}
\begin{split}
|\nabla_k^a \nabla_\ell^\alpha \nabla_m^\beta R_2(k,\ell,m)|
  \lesssim 2^{J} \cdot 2^{J(|a|+|\alpha|+1-N_2)} \cdot 2^{N_2 A} \cdot  2^{(1-|\beta|)M}
  %\\
  %& \lesssim 2^{-J} 2^{-(|a|+|\alpha|+|\beta|)J}
\end{split}
\end{align*}
which is better than \eqref{Propnu213} since we are considering $J \geq 4A$, $\max(K,L,M)\leq A$,
and $|a|+|\alpha|+|\beta| \leq N_2/2 - 2$.

\medskip
Let us now look at $\nu^-$. From \eqref{nu21pr1}-\eqref{nu21prJ},
and looking at $K\leq 0$ (the other case is similar), we write
\begin{align*}
\begin{split}
& \nu^-(k,\ell,m) = \sum_{J\in\mathcal{J}^c}\nu^J(k,\ell,m), \qquad
%& := \int_{\R^3} e^{-ix \cdot k} 
%  \frac{e^{i(|\ell|+|m|)|x|}}{|x|^2} \psi_1(x,\ell) \psi_1(x,m) \, \varphi_J(x) \, \mathrm{d}x,
\mathcal{J}^c = \{ J < -K + 4A\} %\mathcal{J}_1 \cup \mathcal{J}_2 := \{ J \leq 4A \} \cup \{ J \leq -K + 4A\}.
\end{split}
\end{align*}
%Let us look at $J \in \mathcal{J}_2$ ($K \leq 0$), the case $J\in\mathcal{J}_1$ being similar.
The arguments here are similar to the ones that follow \eqref{prnu+41}.
%Let us consider first the case $|K-L| \leq 5$ (so that $M \leq L \leq 5$).
Using the estimates \eqref{psi10} we can see that, for $J\in\mathcal{J}^c$
\begin{align}
\nonumber
\Big| \nabla_k^a \nabla_\ell^\alpha \nabla_m^\beta \nu^J(k,\ell,m) \Big| 
  & = \Big| \int_{\R^3} x^a e^{-ix \cdot k} \frac{1}{|x|^2} 
  \nabla_\ell^\alpha\big(e^{i|\ell||x|}\psi_1(x,\ell)\big)
  \nabla_m^\beta \big(e^{i|m||x|}\psi_1(x,m)\big) \, \varphi_J(x) \, \mathrm{d}x \Big|
  \\
\label{nu21pr45}
  & \lesssim  2^J \cdot 2^{|a|J} \cdot (2^{|\alpha|J} + 2^J2^{(1-|\alpha|)L}) 
  \cdot (2^{|\beta|J} + 2^J2^{(1-|\beta|)M})
  \\ 
  & \lesssim 2^{-K} 2^{4A} \cdot 2^{-|a|K} \cdot 2^{-|\alpha|L}
  \cdot 2^{-K} \cdot (1 + 2^{(1-|\beta|)M}) \cdot 2^{4A(|a|+|\alpha|+|\beta|+2)},
\nonumber
\end{align}
having used $J \leq -K + 4A \leq -M + 4A + 5$; 
this is consistent with \eqref{Propnu213} since $|K-L|\leq 5$.

\medskip
Finally, we analyze the case $|K-L| \geq 5$. %instead we need to avoid losing powers of $2^J$. %\ll 2^{-K}$.
Similarly to \eqref{prnu+50}, we can use integration by parts through the identity
\begin{align*} 
e^{i(-x \cdot k + |x|(|\ell|+|m|))} 
  = T^\rho e^{i(-x \cdot k + |x|(|\ell|+|m|))}, \qquad
  T := \frac{k-(x/|x|)(|\ell|+|m|)}{|k-(x/|x|)(|\ell|+|m|)|^2} \cdot i\nabla_x,
\end{align*}
since $|k+(x/|x|)(|\ell|+|m|)| \gtrsim 2^{\max(K,L)}$.
Using also the estimates
$|\nabla_x^\gamma (k+(x/|x|)(|\ell|+|m|))| \lesssim 2^{-|\gamma| J} 2^{L}$ for $|\gamma|\geq 1$,
and, see \eqref{psi10},
\begin{align*}%\label{prnu+51}
& \big| \nabla_x^\gamma (\psi_1(x,\ell)\psi_1(x,m)) \big| \lesssim 2^{-|\gamma|J} 2^{|\gamma|L_+},
\end{align*}
(with the proper version of the estimates for the derivatives in $\ell$ and $m$),
we see that each integration by parts through $T$ gives a gain of $2^{-J} \cdot 2^{-\max(K,L)} 2^A$.
Applying derivatives as in \eqref{nu21pr45} and then integrating by parts leads to the bound
\begin{align}\label{nu21pr46}
\begin{split}
\Big| \nabla_k^a \nabla_\ell^\alpha \nabla_m^\beta \nu^J(k,\ell,m) \Big| 
  & \lesssim 2^J \cdot 2^{|a|J} \cdot (2^{|\alpha|J} + 2^J2^{(1-|\alpha|)L}) 
  \\
  & \times (2^{|\beta|J} + 2^J2^{(1-|\beta|)M}) \cdot 2^{-\rho J} 2^{-\rho \max(K,L)} 2^{A\rho}.
  %\\ 
  %& \lesssim 2^{-2\max(K,L)} 2^{4A} \cdot 2^{-|a|K} \cdot 2^{-|\alpha|L}
  %\cdot 2^{(1-|\beta|)M_-} \cdot 2^{4A(|a|+|\alpha|+|\beta|)}
\end{split}
\end{align}
In the case $J\leq -L$ we can use \eqref{nu21pr46} with $\rho=|a|+3$ to obtain \eqref{Propnu213}.
If $-L< J \leq - M$ we take instead $\rho = |a|+|\alpha|+2$.
When $J > -M$ is suffices to let $\rho = |a|+|\alpha|+|\beta|+1$. This concludes the proof of the Proposition.
\end{proof}

\medskip
We also have the following analogue of Proposition \ref{Propnu21} for the measure $\nu_2^2$.

\begin{proposition}[Structure of $\nu_2^2$]\label{Propnu22}
Let $\nu_2^2$ be the measure defined in \eqref{SOmu2}, with
$\psi_1$ defined by \eqref{psipsi1}. 
Fix $N_2 \in (5,N_1/4]\cap \Z$
and let $k,\ell,m \in\R^3$ with $|k|\approx 2^K, |\ell|\approx 2^L$ and $|m|\approx 2^M$, 
and $K,L,M \leq A$ for some $A>0$.
Then we can write
\begin{align}\label{Propnu220}
\nu_2^2(k,\ell,m) = %\nu_{2,0}^2(k,\ell,m) + 
  \nu_{2,+}^2(k,\ell,m) + \nu_{2,-}^2(k,\ell,m) + \nu_{2,R}^2(k,\ell,m),
\end{align}
where:

%\begin{align}
%\nu_2^2(k,\ell,m) & := \int e^{ix\cdot \ell} \frac{e^{i|x|(-|k|+|m|)}}{|x|^2} \bar{\psi_1(x,k)} \psi_1(x,m) \, \mathrm{d}x; 
%\end{align}

\begin{itemize}
 
\item The leading order is
\begin{align}\label{Propnu221}
\nu_{2,\pm}^2(k,\ell,m) = \frac{1}{|\ell|} 
  \sum_{i=1}^{N_2} \sum_{J\in\Z} b_{i,J}^\pm(k,\ell,m) \cdot K_i\big(2^J(|k|\pm|\ell|-|m|)\big) 
\end{align}
with
\begin{align}\label{Propnu222}  
\begin{split}
%\sum_{J\in\Z}  %% Don't need sum
  & \big| \varphi_K(k) \varphi_L(\ell) \varphi_M(m) 
  \nabla_k^a \nabla_\ell^\alpha \nabla_m^\beta  b_{i,J}^\pm(k,\ell,m) \big| 
  \\ 
  & \lesssim 2^{-|\alpha|L} \cdot \big( 2^{|\a|\max(K,M)} + 2^{(1-|a|)K} 2^{(1-|\beta|)M} \big)
  \mathbf{1}_{\{\max(K,L,M)-\med(K,L,M)<5\}},
\end{split}
\end{align}
for all $K,L,M \leq A$, and $|a| + |\a|+|\b| \leq N_2$.
%\leq N_1-N_2$.

\medskip
\item The remainder term satisfies, for $K \leq M$,
\begin{align}\label{Propnu223}
\begin{split}
\big| \nabla_k^a \nabla_\ell^\alpha \nabla_m^\beta \nu_{2,R}^2(k,\ell,m) \big| 
  & \lesssim 2^{-2\max(L,M,K)} \cdot 2^{-(|\a|+|\b|)\max(L,M,K)}
  \\ & \times \max(1,2^{-(|a|-1)K}) \cdot 2^{(|a|+|\a|+|\b|+2)5A}
\end{split}
\end{align}
for all $K,L,M \leq A$ and $|a|+|\a|+|\b| \leq N_2/2 -2$. 
A similar estimate holds when $M \leq K$ by exchanging the roles of $M$ and $K$ (and $a,\beta$).
\end{itemize}

%Previous version, which is better than \eqref{Propnu213}, was
%\begin{align*}
%\begin{split}
%\big| %\varphi_K(k) \varphi_L(\ell)  \varphi_M(m) 
%  \nabla_k^a \nabla_\ell^\alpha \nabla_m^\beta \nu_{2,R}^2(k,\ell,m) \big| 
%  & \lesssim 2^{-\max(L,M,K)} \cdot 2^{-|\alpha|\max(L,M,K)}
%  \\ & \times \max(1,2^{-(|a|-1)K}) \max(1,2^{-(|\b|-1)M}) \cdot 2^{(|a|+|\a|+|\b|+2)5A}
%\end{split}
%\end{align*}

\medskip
A similar statement holds for $\nu_2^2(k,m,\ell)$ by exchanging the roles of $\ell$ and $m$.
Notice how this leaves unchanged the structure of the singularities in \eqref{Propnu221}.
\end{proposition}

Since the proof is similar to the ones above we can skip the details.

%\medskip
%\begin{proof}
 
%\comment{Give some elements ?}

%\end{proof}

\medskip
\subsection{Analysis of $\mu_2$: Bilinear estimates}\label{ssecmu2BE}
Let us define, for a general measure $\nu$ and symbol $b$,
\begin{align}\label{theomu2T}
\begin{split}
T_{\nu}[b](g,h)(x) = & \, T[\nu;b](g,h)(x) := \, \mathcal{F}^{-1}_{k\rightarrow x} 
  \iint_{\R^3\times\R^3} g(\ell) h(m) \,b(k,\ell,m)\, \nu(k,\ell,m) \, \mathrm{d}\ell \mathrm{d}m.
\end{split}
\end{align}
Using Proposition \ref{Propnu22} and Lemma \ref{lemBEan2}
we can establish H\"oder-type bilinear bounds for pseudo-products involving the measure $\mu_2$ in \eqref{SOmu2}
(Theorem \ref{theomu2})
and vectorfields acting on its components (Theorem \ref{theomu2vf}).
%which are the analogues of the Theorem \ref{theomu1} for $\mu_1$.

\medskip
\begin{theorem}[Bilinear bounds 2]\label{theomu2}
Consider the operator $T_{\mu_2}[b]$, defined according to \eqref{theomu2T} and \eqref{SOmu2}, and assume that:

\setlength{\leftmargini}{2em}
\begin{itemize}

\medskip
\item The symbol $b$ is such that
\begin{align}\label{theomu2asb1}
\supp(b) \subseteq \big\{ (k,\ell,m) \in \R^9\,:\, |k|+|\ell|+|m| \leq 2^A, \, |\ell| \approx 2^L, \, |m|\approx 2^M \big\},
\end{align}
for some $A\geq1$.
%Better to say $b = b \chi_{\leq A}$?

\medskip
\item For all $|k| \approx 2^{K}$, $|\ell|\approx 2^L$ and $|m| \approx 2^M$
\begin{align}\label{theomu2asb2}
| \nabla_k^a \nabla^\alpha_\ell \nabla^\beta_m b(k,\ell,m)| 
	\lesssim 2^{-K|a|} 2^{-|\alpha|L}2^{-|\beta|M} \cdot 2^{(|a|+|\alpha|+|\beta|)A}, \qquad |a|,|\a|,|\b|\leq 5.
\end{align}

\end{itemize}

\medskip
\noindent
Then, for any $p,q \in [2,\infty)$ %and $r\geq 2$ with %Check endpoint $p=2,\infty$
\begin{align}
\frac{1}{p} + \frac{1}{q} > \frac{1}{2}, \qquad 
\end{align}
the following estimate holds:
\begin{align}
\label{theomu2conc}
\begin{split}
{\big\| P_K T_{\mu_2}[b]\big(g,h\big) \big\|}_{L^2}
  & \lesssim {\big\| \what{g} \big\|}_{L^p} {\big\| \what{h} \big\|}_{L^q} \cdot 2^{\max(K,L,M)} %2^{-K+\min(K,L,M)+\med(K,L,M)}
  \cdot 2^{C_0A}
  %+ 2^{-D} \mathcal{D}(g,h).
\end{split}
\end{align}
for some sufficiently large $C_0$.
%Get exact factor in the estimate from Lemma \ref{lemBEan2}
\end{theorem}

\def\factor{2^{\min(K,L,M)}}

\medskip
\begin{theorem}[Bilinear bounds with vectorfields 2]\label{theomu2vf}
Let $\nu_2^1$ %, $\nu_{2,R}^1$ 
be given as in Proposition \ref{Propnu21} and
$\nu_{2\pm}^2$, $\nu_{2,R}^2$ as in Proposition \ref{Propnu22}.
With the same notation and assumptions on $b$ and $(p,q)$ as in Theorem \ref{theomu2}, the following hold:

%Changed $r \mapsto 2$. General $r$ by duality?

\begin{itemize}

\medskip
\item[(i)] Let
\begin{align}\label{theomu2X}
{\bf X}_\pm = \pm \partial_{|\ell|} + \partial_{|m|},
\end{align}
%\begin{align}\label{theomu2TX}
%\begin{split}
%T_{{\bf X}_\pm^a \nu}[b](g,h)(k) 
%  := \mathcal{F}^{-1}_{k\mapsto x} \iint_{\R^3\times\R^3} g(k-\ell) h(m) 
%  \, b(k,\ell,m) \\ \times \big[ \varphi_M(m) {\bf X}^a \nu(k,\ell,m) \big]\, \mathrm{d}\ell \mathrm{d}m,
%\end{split}
%\end{align}
%and define the expression
Then, for $a = 1,2$,
\begin{align}\label{theomu2Xconc1}
{\big\| P_KT[{\bf X}_-^a \nu_2^1; b](g,h) \big\|}_{L^2} 
  & \lesssim {\big\| \what{g} \big\|}_{L^p} {\big\| \what{h} \big\|}_{L^q}
    \cdot 2^{(1-a)\max(K,L,M)}  %2^{-a\min(L,M)} 
    \cdot 2^{(C_0+12)A}, %+ 2^{-D}\mathcal{D}(g,h).
\end{align}
and
\begin{align}\label{theomu2Xconc2}
\begin{split}
& {\big\| P_KT[{\bf X}_\pm^a \nu_{2,\pm}^2; b] (g,h) \big\|}_{L^2}
\\
& + {\big\| P_KT[{\bf X}_\pm^a \nu_{2,R}^2; b] (g,h) \big\|}_{L^2} 
   \lesssim {\big\| \what{g} \big\|}_{L^p} {\big\| \what{h} \big\|}_{L^q} 
   2^{-a\min(L,M)} 2^{\med(K,L,M)} \cdot 2^{(C_0+12)A}. %+ 2^{-D}\mathcal{D}(g,h).
\end{split}
\end{align}

%Check factor of $\med$ and if it is helpful\dots

\medskip
\item[(ii)] Let
\begin{align}\label{theomu2Y}
{\bf Y}_\pm = \partial_{k} \pm \frac{k}{|k|} \big( \frac{\ell}{|\ell|} \cdot \partial_{\ell} \big).
\end{align}
Then, for $a=1,2$, we have
%\footnote{Here we are considering $\nu_2^2 = \nu_2^2(k,\ell,m)$,
%but analogous statements hold by reversing the roles of $\ell$ and $m$
%when one considers $\nu_2^2(k,m,\ell)$, see \eqref{SOmu2}.}
\begin{align}\label{theomu2Yconc1}
{\big\| T[{\bf Y}_+^a \nu_2^1; b](g,h) \big\|}_{L^2} 
  & \lesssim {\big\| \what{g} \big\|}_{L^p} {\big\| \what{h} \big\|}_{L^q}
    \cdot 2^{(1-a)\max(K,L,M)} \cdot 2^{(C_0+12)A};
\end{align}
The same estimate %(actually a stronger one) 
holds for $T[{\bf Y}_+^a\nu^1_{2,R}; b]$.

Moreover, we have
%Think of exchange $\ell\leftrightarrow m$
\begin{align}\label{theomu2Yconc2}
\begin{split}
{\big\| T[{\bf Y}_\mp^a \nu_{2,\pm}^2; b] (g,h) \big\|}_{L^2} 
  & \lesssim {\big\| \what{g} \big\|}_{L^p} {\big\| \what{h} \big\|}_{L^q}
    \cdot 2^{-aL} 2^{\max(K,L,M)} \cdot  2^{(C_0+12)A};
\end{split}
\end{align}
the same estimate %(actually a stronger one) 
holds for $T[{\bf Y}_\mp^a\nu^2_{2,R}; b]$.

\medskip
\item[(iii)] 
Define the cutoffs
\begin{align}\label{theomu2dlcut}
\begin{split}
\chi^\pm(k,\ell,m) := \varphi_{\geq - 10}\Big( \frac{|k| \pm |\ell| - |m|}{ |k| + |\ell| + |m|} \Big). 
%\\ \chi^-(k,\ell,m) := \varphi_{\geq \max(L,M,K) - 10}(|k| - |\ell| - |m|) 
\end{split}
\end{align}
Then, for $a=(a_1,a_2)$ with $1\leq |a| \leq 2$, we have %might need more derivatives...
\begin{align}\label{theomu2dlmconc}
\begin{split}
{\big\| T\big[\nabla_{(\ell,m)}^a (\nu \chi^-); b\big](g,h) \big\|}_{L^2} 
  \lesssim {\big\| \what{g} \big\|}_{L^p} {\big\| \what{h} \big\|}_{L^q}
    \cdot 2^{-|a_1|L} 2^{-|a_2|M} \cdot 2^{\max(K,L,M)} 2^{(C_0+12)A},
    \\ \qquad \mbox{for} \,\, \nu \in \{ \nu^1_2,\nu^2_2\},
\end{split}
\end{align}
\end{itemize}

%Previous separated versions of \eqref{theomu2dlmconc}:
%
%\begin{align}\label{theomu2dlconc}
%\begin{split}
%{\big\| T\big[\partial_\ell^a (\nu \chi^-); b\big](g,h) \big\|}_{L^2} 
%  \lesssim {\big\| \what{g} \big\|}_{L^p} {\big\| \what{h} \big\|}_{L^q}
%    \cdot 2^{-aL} \cdot 2^{\max(K,L,M)} 2^{(C_0+12)A},
%    \\ \qquad \mbox{for} \,\, \nu \in \{ \nu^1_2,\nu^2_{2,-}\},
%\end{split}
%\end{align}
%and
%\begin{align}\label{theomu2dmconc}
%{\big\| T\big[\partial_m^a (\nu^2_{2,+} \chi^+); b\big](g,h) \big\|}_{L^2} 
%  & \lesssim {\big\| \what{g} \big\|}_{L^p} {\big\| \what{h} \big\|}_{L^q}
%    \cdot 2^{-aM} \cdot \factor 2^{(C_0+12)A}. %+ 2^{-D}\mathcal{D}(g,h).
%\end{align}
%The same estimates also hold for $T[\partial_\ell^a \nu^i_{2,R}, b]$ and $T[\partial_m^a \nu^i_{2,R}, b]$, $i=1,2$.

\end{theorem}

\medskip
Theorem \ref{theomu2} gives bilinear estimates for operators involving $\mu_2$
which are analogous to the estimates of Theorems \ref{theomu1} and \ref{theomu1'} for $\mu_1$.

%Could comment/explain more?

%Note how some of the estimates are a little worse than one would expect, comparing with the previous ones;
%for example \eqref{theomu2Xconc2} appears to be a little weaker than \eqref{theomu1Xconc}
%(which is only really used when $|L-M| \leq 10$) when $L \leq M-10$.
%Nevertheless these are sufficient for the nonlinear analysis in \ref{ssecmu23Est}.

To prove Theorem \ref{theomu2} the key ingredient is the following:

\medskip
\begin{lemma}[Bilinear operators restricted to small annuli 2]\label{lemBEan2}
Let $j\geq 1$, $\sigma_1,\sigma_2 \in \{+,-\}$ and consider the bilinear operator
\begin{align}
\label{lemBE21}
\begin{split}
B_{j}^{\sigma_1\sigma_2}[b](g,h)(x) 
  = \mathcal{F}^{-1}_{k\mapsto x} \iint_{\R^3\times\R^3} \what{g}(\ell) \what{h}(m) 
  \, b(k,\ell,m) \, \chi\big(2^j(|k| +\sigma_1 |\ell| +\sigma_2 |m|)\big) \, \mathrm{d}\ell \mathrm{d}m,
\end{split}
\end{align}
where $\chi$ is a Schwartz function and

\setlength{\leftmargini}{2em}
\begin{itemize}

\medskip
\item %For some $A \geq 1$ %and $L \gg -j$ 
The support of $b$ satisfies
\begin{align}\label{lemBEan2b1}
\supp(b) \subseteq \big\{ (k,\ell,m) \in \R^9\,:\, 
  %|k|+|\ell|+|m| \lesssim 2^A, \, 
  |k|\approx 2^K, \,|\ell| \approx 2^L, \,|m| \approx 2^M \big\},
\end{align}
with $-j \ll \max(K,L,M) \leq A$ for some $A\geq 1$; 
%in particular we also have $2^{-j} \ll |m| \approx |\ell| \approx 2^L$.

\medskip 
\item The following estimates hold
\begin{align}\label{lemBEan2b2}
| \nabla_k^a \nabla^\alpha_\ell \nabla^\beta_m b(k,\ell,m)| 
	\lesssim 2^{-K|a|} 2^{-|\alpha|L} 2^{-|\beta|M} \cdot 2^{(|a|+|\alpha|+|\beta|)A},
	\qquad |a|, |\alpha|, |\beta|\leq 4.
\end{align}
\end{itemize}

\medskip
\noindent
Then
\begin{align}
\label{lemBEan2conc}
{\|B_{j}^{\sigma_1\sigma_2}[b](g, h) \|}_{L^2} \lesssim 2^{-j} \cdot 2^{\min(K,L,M) + \med(K,L,M)} 
  \cdot 2^{16 A} \cdot {\| g \|}_{L^2} {\| h \|}_{L^\infty}, 
  %\qquad \frac{1}{p}+\frac{1}{q} = \frac{1}{r},
\end{align}
where $\widehat{f} = \widehat{\mathcal{F}}(f)$ denotes the (flat) Fourier transform of $f$.
Moreover, for all $p,q\in[2,\infty]$ with $1/p+1/q=1/2$, 
\begin{align}
\label{lemBEan2conc'}
\begin{split}
{\|B_{j}^{\sigma_1\sigma_2}[b](g, h) \|}_{L^2} \lesssim 2^{-j/2} \cdot 2^{\min(K,L,M) + (3/2)\med(K,L,M)} 
  \cdot 2^{16 A} \cdot {\| g \|}_{L^p} {\| h \|}_{L^q}.
  %\\ \frac{1}{p}+\frac{1}{q} = \frac{1}{2},
\end{split}
\end{align}
%Use the extra $2^{-j}$  to get better than Holder? don't need it in the argument later on...
%But need $(p,q)$ not just $(2,\infty)$
%Don't need ``dual statement'' here... ?
\end{lemma}

\smallskip
Lemma \ref{lemBEan2} is the analogue of Lemma \ref{lemBEan} which was used to prove the bilinear estimates
for $\mu_1$ in Theorems \ref{theomu1} and \ref{theomu1'}.
The proof is in a similar spirit as the one of Subsection \ref{proofBEpre}
but it is more involved, so we give the details below.

Note how, exactly as in \eqref{lemBE12}, we gain a factor of $2^{-j}$ in \eqref{lemBEan2conc}.
However, in this case, such a gain is not necessary to obtain the boundedness
of the bilinear operator associated to $\mu_2$, see \eqref{SODuhamel}-\eqref{SOmu2}. 
Indeed, from \eqref{Propnu211} in Proposition \ref{Propnu21} (and \eqref{Propnu221} in Proposition \ref{Propnu22})
we see that $\mu_2$ is a linear combination of operators like those in Lemma \ref{lemBEan2}
with coefficients that are uniformly controlled independently of $j$,
in contrast with the case of the more singular distribution $\mu_1$, 
see \eqref{mu1} and \eqref{Propnu+2}. %in Proposition \ref{Propnu+}.

The estimate \eqref{lemBEan2conc} is sharp relative to the dependence on $j$ but 
we only allow the pair of H\"older exponents $(2,\infty)$ there.
Since in the nonlinear estimates for the evolution equation we also need different H\"older pairs $(p,q)$,
we will end up using only \eqref{lemBEan2conc'} in the proof of Theorems \ref{theomu2}, \ref{theomu2vf},
and \ref{theomu3}.
The smaller gain of $2^{-j/2}$ is still more than sufficient to deal with $\mu_2$ and $\mu_3$.
%since, as we observed in the paragraph above, there is no growing factor of $2^j$ 
%that needs to be compensated for in the summations \eqref{Propnu211} and \eqref{Propnu221}.

\medskip
\begin{proof}[Proof of Lemma \ref{lemBEan2}]
We only consider the case $\sigma_1=\sigma_2=-$, since the cases $\sigma_1\sigma_2 = -$ are similar,
and the case $\sigma_1=\sigma_2=+$ is empty.
We denote $B_{j}[b](g,h)(x) = B_{j}^{--}[b](g,h)(x)$, insert cutoffs according to the support restriction
\eqref{lemBEan2b1} and write
\begin{align}
\label{prBEan21}
\begin{split}
\langle B_{j}[b](g,h)(x),f\rangle %= \what{\mathcal{F}}^{-1}_{k\mapsto x} \varphi_{K}(k)
  %\iint_{\R^3\times\R^3} g(\ell) h(m) \, \varphi_{L}(\ell) \varphi_{M}(m) 
  %b(k,\ell,m) \\ \, \times \varphi_{L}(\ell)\varphi_{M}(m)\, \chi(2^j(|k|-|\ell|-|m|)) \, \mathrm{d}\ell \mathrm{d}m,
  %\\
  & = \iiint_{\R^3_x \times \R^3_y \times \R^3_z} A(x,y,z) \, f(x) g(y) h(z) 
  \, \mathrm{d}x\mathrm{d}y\mathrm{d}z
  \\
A(x,y,z) &:= \iiint_{\R^3_k \times \R^3_\ell \times\R^3_m} 
  e^{ix\cdot k} e^{-iy\cdot \ell} e^{-iz\cdot m}  \varphi_K(k) \varphi_{L}(\ell) \varphi_{M}(m) b(k,\ell,m) 
  \\ \, & \qquad \times \, \chi(2^j(|k|-|\ell|-|m|)) \, \mathrm{d}\ell \mathrm{d}m \mathrm{d}k.
\end{split}
\end{align}
Without loss of generality, by symmetry we may assume $|\ell|\geq |m|$ (or $L\geq M$).
In view of $-j \leq \max(K,L) - 10$ and the fact that $\chi$ is Schwartz, we may also assume $|K-L| \leq 5$,
for otherwise the kernel is a regular one and the desired bound follows more easily. 
We also set $b\equiv 1$ for convenience; it will be clear to the reader what minor modification
in the arguments are needed for a general $b$ satisfying the assumptions in the statement.

Given a parameter $\lambda$, we introduce an angular partition of unity of $\mathbb{S}^2$
adapted to polar caps of aperture $\lambda$.
We consider a family $E_\lambda$ of unit vectors $\{e_i\}_{i=1,\dots N}$, $N \approx \lambda^{-2}$,
uniformly spaced, and associated cutoffs
\begin{align*}
q_{e,\lambda}(k):=\frac{1}{\sum_{e'} \varphi_{\leq 0}\big(\big| \frac{k}{|k|} - e' \big|^2 \frac{1}{\lambda^2} \big)} 
  \varphi_{\leq 0}\big( \big| \dfrac{k}{|k|} - e \big|^2 \frac{1}{\lambda^2}\big), \qquad e \in E_\lambda
\end{align*}
and projections given by
\begin{align*}
\what{Q_{e,\lambda} f}(k) := q_{e,\lambda}(k) \what{f}(k).
\end{align*}

Let
\begin{align}\label{l_B}
\lambda_B := 2^{-j/2}2^{-B/2} 
\end{align}
for $B=L,M$ or $K$, and define
$E_{\lambda,L,M,K} := E_{\lambda_K} \times E_{\lambda_L} \times E_{\lambda_M}$.
We write
\begin{align}\label{prBEan22}
\begin{split}
\langle B_{j}[1](g,h)(x),f\rangle = \sum_{(e_0,e_1,e_2) \in E_{\lambda,L,M,K}}
  \langle B_{j,e_0,e_1,e_2} 
  \big(Q_{e_1,5\lambda_L} g, Q_{e_2,5\lambda_M}h \big)(x), Q_{e_0,5\lambda_K}f\rangle
\end{split}
\end{align}
where
\begin{align}\label{prBEan23}
\begin{split}
\langle B_{j,e_0,e_1,e_2} (f_1,f_2),f_0\rangle
  & = \iiint_{\R^3_x \times \R^3_y \times \R^3_z} A_{e_0,e_1,e_2}(x,y,z) \, f_0(x) 
    f_1(y) f_2(z) \, \mathrm{d}x\mathrm{d}y\mathrm{d}z,
  \\
A_{e_0,e_1,e_2}(x,y,z) &:= \iiint_{\R^3_k \times \R^3_\ell \times\R^3_m} 
  e^{ix\cdot k} e^{-iy\cdot \ell} e^{-iz\cdot m}  \varphi_K(k) \varphi_{L}(\ell) \varphi_{M}(m) %b(k,\ell,m) 
  \\ \, & \times  q_{e_0,\lambda_K}(k)q_{e_1,\lambda_L}(\ell)q_{e_2,\lambda_M}(m) 
  \, \chi(2^j(|k|-|\ell|-|m|)) \, \mathrm{d}\ell \mathrm{d}m \mathrm{d}k.
\end{split}
\end{align}
Note that there are $\approx \lambda_M^{-2} \cdot \lambda_K^{-2}$ 
elements in the sum over the parameters $e_0,e_2$,
and therefore
\begin{align}\label{prBEan24}
\begin{split}
\big| \langle B_{j}[1](g,h)(x),f\rangle \big| & \lesssim 2^{2j} 2^{K+M}
\\ 
  &\times \sup_{e_0 \in E_{\lambda_K}, \, e_2 \in E_{\lambda_M}} 
  \sum_{e_1\in E_{\lambda_L}}
  \big| \langle B_{j,e_0,e_1,e_2} \big(Q_{e_1,5\lambda_L} g, Q_{e_2,5\lambda_M}h \big)(x), Q_{e_0,5\lambda_K}f\rangle \big|
\end{split}
\end{align}
We thus fix $e_0,e_2$ and claim that it suffices to prove that the kernel satisfies
\begin{align}\label{prBEan2main}
\begin{split}
\big| A_{e_0,e_1,e_2}(x,y,z) \big| \lesssim \, 
	& \frac{2^{-j+L}}{(1+2^{-j+L}|x-y|^2)^2} \frac{2^{-j+M}}{(1+2^{-j+M}|x-z|^2)^2}
  \\ \cdot \, & \frac{2^L}{1+(2^Le_1\cdot (x-y))^2}
  \frac{2^M}{1+(2^Me_2\cdot (x-z))^2} \cdot 2^{-3j}.
\end{split}
\end{align}
%\begin{align}\label{prBEan2main}
%\begin{split}
%{\big\| A_{e_0,e_1,e_2}(x,y,z) \big\|}_{L^1_y L^1_z} \lesssim 2^{-3j}.
%\end{split}
%\end{align}

Let us assume \eqref{prBEan2main} and show how it implies the desired conclusions. %\eqref{lemBEan2conc}.
Up to rotating the variables $k$ %and $m$ 
(and $x$) %and $z$) 
in \eqref{prBEan23} we may assume that $e_0 = e_1$. 
Then we estimate
\begin{align*}
& \big| \langle B_{j,e_1,e_1,e_2} \big(Q_{e_1,5\lambda} g, Q_{e_2,5\lambda}h \big)(x), Q_{e_1,5\lambda}f\rangle \big| 
  %\\
  %& \lesssim 
  %\iiint_{\R^3_x \times \R^3_y \times \R^3_z} \big| A_{e_1,e_1,e_2}(x,y,z) \big| 
  %\, \big|Q_{e_1,5\lambda}f(x) \big| 
  %\, \big| Q_{e_1,5\lambda} \what{g}(y-x) \big| \, \big|Q_{e_2,5\lambda}\what{h}(z-x) \big| 
  %\, \mathrm{d}x\mathrm{d}y\mathrm{d}z
  \\
  & \lesssim 2^{-3j}
  \iiint_{\R^3_x \times \R^3_y \times \R^3_z}  \big| Q_{e_1,5\lambda}f(x) \big|
  \big|Q_{e_1,5\lambda} g(y+x)\big| \big|Q_{e_2,5\lambda}h(z+x)\big| 
  \\
  & \times \frac{2^{-j+L}}{(1+2^{-j+L}|y|^2)^2}
    \frac{2^{-j+M}}{(1+2^{-j+M}|z|^2)^2}
    \frac{2^L}{1+(2^Le_1\cdot y)^2} \frac{2^M}{1+(2^Me_2\cdot z)^2}\, \mathrm{d}x\mathrm{d}y\mathrm{d}z
  \\
  & \lesssim 2^{-3j}
  {\big\| Q_{e_1,5\lambda}f \big\|}_{L^2} {\big\| Q_{e_1,5\lambda}g \big\|}_{L^p} 
  {\big\| Q_{e_2,5\lambda}h \big\|}_{L^q},
\end{align*}
for $1/p+1/q=1/2$.
Summing over $e_1 \in E_\lambda$, using Cauchy-Schwartz and orthogonality, we have
\begin{align*}
& \sum_{e_1\in E_\lambda}
  \big| \langle B_{j,e_1,e_1,e_2} \big(Q_{e_1,5\lambda}g, Q_{e_2,5\lambda}h \big)(x), 
  Q_{e_1,5\lambda}f\rangle \big| 
  \\
  & \lesssim 2^{-3j}
  \Big[ \sum_{e\in E_\lambda}{\big\| Q_{e,5\lambda}f \big\|}_{L^2}^2 \Big]^{1/2}
  \Big[ \sum_{e\in E_\lambda}{\big\| Q_{e,5\lambda}g\big\|}_{L^2}^2 \Big]^{1/2} {\|h\|}_{L^\infty}
  \\
  & \lesssim 2^{-3j} \cdot {\|f\|}_{L^2} {\|g\|}_{L^2} {\|h\|}_{L^\infty}.
\end{align*}
Together with \eqref{prBEan24} which gives a $2^{2j} 2^{L+M}$ factor we obtain \eqref{lemBEan2conc}.
For \eqref{lemBEan2conc'} we only need to slightly modify the last estimate above using
\begin{align*}
& \sum_{e_1\in E_\lambda} 
%  \big| \langle B_{j,e_1,e_1,e_2} \big(Q_{e_1,5\lambda_L} \what{g},
%    Q_{e_2,5\lambda_M}\what{h} \big)(x), Q_{e_1,5\lambda_K}f\rangle \big| 
%  \\
{\big\| Q_{e_1,5\lambda}f \big\|}_{L^2} {\big\| Q_{e_1,5\lambda}g \big\|}_{L^p} 
  {\big\| Q_{e_2,5\lambda}h \big\|}_{L^q}
  \\
  & \lesssim %2^{-3j}
  \Big[ \sum_{e\in E_\lambda}{\big\| Q_{e,5\lambda_K}f \big\|}_{L^2}^2 \Big]^{1/2}
  \Big[ \sum_{e\in E_\lambda}{\big\| Q_{e,5\lambda_L}g \big\|}_{L^p}^2 \Big]^{1/2} {\|h\|}_{L^q}
  \\
  & \lesssim %2^{-3j} \cdot 
  {\| f \|}_{L^2} {\big\|g\big\|}_{L^p} \big| E_{\lambda_L} \big|^{1/2} {\|h\|}_{L^q}
  \lesssim %2^{-3j} \cdot 
  2^{j/2} 2^{L/2} \cdot {\| f \|}_{L^2} {\|g\|}_{L^p} {\|h\|}_{L^q}.
\end{align*}

\medskip
{\it Proof of \eqref{prBEan2main}}. Rotating $k$ and $m$, we may assume $e_0=e_1=e_2$.
Changing variables $k\mapsto k+\ell+m$, it suffices to look at
\begin{align}
\begin{split}
A_{e_1,e_1,e_1}(x,Y,Z) &= \iiint_{\R^3_k \times \R^3_\ell \times\R^3_m} 
  e^{ix\cdot k} e^{-iY\cdot \ell} e^{-iZ\cdot m}  \varphi_K(k+\ell+m) 
  \varphi_{L}(\ell) \varphi_{M}(m) %b(k,\ell,m) 
  \\ \, & \times  q_{e_1,\lambda_K}(k+\ell+m)q_{e_1,\lambda_L}(\ell)q_{e_1,\lambda_M}(m) 
  \, \chi(2^j(|k+\ell+m|-|\ell|-|m|)) \, \mathrm{d}\ell \mathrm{d}m \mathrm{d}k
\end{split}
\end{align}
and prove
\begin{align}\label{prBEan2main'}
\begin{split}
\big| A_{e_1,e_1,e_1}(x,Y,Z) \big| & \lesssim \frac{2^{-j+L}}{(1+2^{-j+L}Y^2)^2} \frac{2^{-j+M}}{(1+2^{-j+M}Z^2)^2}
  \\ & \times \frac{2^L}{1+(2^Le_1\cdot Y)^2} \frac{2^M}{1+(2^Me_1\cdot Z)^2} 
  \cdot 2^{-3j} .
\end{split}
\end{align}
The idea is to integrate by parts using the vectorfields $2^{-j}\Delta_\ell$ and $2^{-j}\Delta_m$
as well as the scaling vectorfields $\ell \cdot \nabla_\ell$ and $m \cdot \nabla_m$,
taking advantage of the fact that the three frequencies $\ell,m$ and $k+\ell+m$ are essentially aligned.
Observe that
\begin{align}\label{prBEan2d1}
\begin{split}
\Big| 
2^{-j}\Delta_\ell \big[ q_{e_1,\lambda_K}(k+\ell+m) q_{e_1,\lambda_L}(\ell) \chi(2^j(|k+\ell+m|-|\ell|-|m|)) \big] \Big|
  \\ \lesssim 
  2^{-j} 2^{-2K}\lambda_K^{-2} + 2^{-j} 2^{-2L}\lambda_L^{-2} + 2^j(\lambda_K^2 + \lambda_L^2)
  \\
  \approx 2^{-L},
\end{split}
\end{align}
having used that, on the support of the integral, 
\[|(k+\ell+m)_i/|k+\ell+m| - \ell_i/|\ell| |\lesssim \lambda_K+\lambda_L\]
and that $-j/2 \leq K,L$.
A similar estimate holds replacing $\ell$ by $m$ and $L$ by $M$.
%Could have $|m| \leq 2^{-j/2}$\dots treat this case separately?

We also have
\begin{align}\label{prBEan2d2}
\begin{split}
& \ell \cdot \nabla_\ell \, q_{e_1,\lambda_L}(\ell) = 0,
\\
& \Big| \ell \cdot \nabla_\ell \big[ q_{e_1,\lambda_K}(k+\ell+m) \chi(2^j(|k+\ell+m|-|\ell|-|m|)) \big] \Big| 
  \lesssim 1, %2^L
\end{split}
\end{align}
with similar estimates for $m\cdot \nabla_m$.
To see this last inequality, recall the definition $q_{e,\lambda}$ %, the notation $\lambda=2^{-j/2}2^{L}$,
and calculate
\begin{align*}
\ell \cdot \nabla_\ell \frac{(k+\ell+m)_a}{|k+\ell+m|} 
  & = \ell^i \frac{\delta_{ia}}{|k+\ell+m|} - \ell^i \frac{(k+\ell+m)_a(k+\ell+m)_i}{|k+\ell+m|^3}
  \\
  & = \frac{\ell_a}{|k+\ell+m|} - \frac{\ell^i\ell_a\ell_i}{|k+\ell+m||\ell|^2} + 
  O\big((\lambda_K + \lambda_L)\cdot 2^{L-K}\big),
  %=  O(2^{-j/2} 2^L),
\end{align*}
%2^{-j/2}$.
Therefore
\begin{align*}
\big| \ell \cdot \nabla_\ell q_{e_1,\lambda_K}(k+\ell+m) \big| \lesssim 1 + \lambda_L \lambda_K^{-1} \approx 1.
   %\cdot 2^{L-K}
\end{align*}
Similarly, we calculate
\begin{align*}
\ell \cdot \nabla_\ell ( |k+\ell+m|-|\ell|-|m| )
  & = |\ell| \Big[ \frac{\ell}{|\ell|} \cdot \frac{k+\ell+m}{|k+\ell+m|} - 1 \Big]= 
  \\
  & = O(2^L) \sin^2 \angle(k+\ell+m,\ell) = O\big(2^L(\lambda_K + \lambda_L)^2\big) = O(\lambda_K),
\end{align*}
and see that
\begin{align*}
| \ell \cdot \nabla_\ell \chi(2^j(|k+\ell+m|-|\ell|-|m|) | \lesssim 2^j \cdot 2^L (\lambda_K + \lambda_L)^2
  \approx 1.
\end{align*}

Using integration by parts through the identities 
$(1-2^{-j}2^L\Delta_\ell) e^{iY\cdot \ell} = (1+ 2^{-j+L}|Y|^2) e^{iY\cdot \ell} $ %and $2^{-2j}\Delta_m$
and $[1 - (2^L e_1 \cdot \nabla_\ell)^2]e^{iY\cdot \ell} = [1 + (2^L e_1\cdot Y)^2]e^{iY\cdot \ell}$
(and similarly for $e^{iZ\cdot m}$)
together with 
%the fact that $\ell\cdot Y = |\ell| (e_1\cdot Y) + O(2^L 2^{-j/2}) = |\ell| (e_1\cdot Y) + O(2^{-j/2}Y)$, and 
\eqref{prBEan2d1}-\eqref{prBEan2d2},
we obtain
\begin{align}
\begin{split}
\big| A_{e_1,e_1,e_1}(x,Y,Z) \big| \lesssim 
  \frac{1}{(1+2^{-j+L}|Y|^2)^2} \frac{1}{(1+2^{-j+M}|Z|^2)^2} \frac{1}{1+(2^Le_1\cdot Y)^2} 
  \frac{1}{1+(2^Me_1\cdot Z)^2} 
  \\ \times \iiint_{\R^3_k \times \R^3_\ell \times\R^3_m} 
  \varphi_K(k+\ell+m) \varphi_{L}(\ell) \varphi_{M}(m)
  q_{e_1,\lambda}(k+\ell+m)q_{e_1,\lambda}(\ell)q_{e_1,\lambda}(m) 
  \\ \times \, \chi(2^j(|k+\ell+m|-|\ell|-|m|)) \, \mathrm{d}\ell \mathrm{d}m \mathrm{d}k.
\end{split}
\end{align}

We then notice that, at fixed $\ell$ and $m$, the integral over $k$ is taken over a region 
where $|k| \lesssim 2^{-j}$, since 
\begin{align*}
& \big| |\ell|+|m| -|\ell+m| \big| = O((\lambda_M+\lambda_L)^2 2^M) \lesssim 2^{-j}, 
\\
& \big| |k+\ell+m|-|\ell+m|-|k| \big| = O((\lambda_L+\lambda_K)^2) 2^K) \lesssim 2^{-j}.
\end{align*}
Finally, since the integrals in $\ell$ and $m$ contribute the volumes of integration
$\lambda_L^2 2^{3L} = 2^{-j}2^{2L}$ and $\lambda_M^2 2^{3M} = 2^{-j}2^{2M}$ respectively,
we arrive at \eqref{prBEan2main'}.
%\begin{align}\label{prBEan25}
%\end{align}
\end{proof}

%\subsection{Proof of Proposition \ref{theomu2}}\label{secBEmu2}

\medskip
\begin{proof}[Proof of Theorems \ref{theomu2} and \ref{theomu2vf}]
The proof of Theorem \ref{theomu2} follows from the structural Propositions \ref{Propnu21} and \ref{Propnu22}
and Lemma \ref{lemBEan2}. 
We can use the same arguments as those in Subsections \ref{prtheomu1} and \ref{sectheomu1'};
note that, actually, the situation here is substantially better 
since \eqref{Propnu210} and \eqref{Propnu220} 
do not contain any singular contribution like $\nu_0$ in \eqref{Propnu+1}-\eqref{Propnu+1.1};
moreover, the summations in \eqref{Propnu211} and \eqref{Propnu221} have coefficients uniformly bounded in $J$,
unlike \eqref{Propnu+2}.
%
%Furthermore, for convenience, in Theorems \ref{theomu2} and \ref{theomu2vf} we are only claiming $L^2$ bounds,
%which are easier to prove, since we do not need to use duality arguments.
%which means that one can directly sum over $J$...
\end{proof}

\medskip
\subsection{Analysis of $\mu_3$: structure and bilinear estimates}\label{ssecmu3}
Recall the definition
\begin{align}\label{SOmu3'}
\begin{split}
\mu_3(k,\ell,m) = \int \frac{e^{i(-|k|+|\ell|+|m|)|x|}}{|x|^3} 
   \, \bar{\psi_1(x,k)} \psi_1(x,\ell) \psi_1(x,m) \, \mathrm{d}x.
\end{split}
\end{align}
We have the following analogues  of Propositions \ref{Propnu+} and \ref{Propnu21}:

\medskip
\begin{proposition}[Structure of $\mu_3$]\label{Propmu3}
Let $\mu_3$ be defined as in \eqref{SOmu3'} with $\psi_1$ defined by \eqref{psipsi1}. 
%Fix $N_2 \in (5,N_1)\cap\Z$.
%Let $k,\ell,m \in\R^3$ with $|k|\approx 2^K, |\ell|\approx 2^L$ and $|m|\approx 2^M$, 
%and assume that $K,L,M \leq A$ for some $A>0$.
Fix an integer $N_2 \in[10,N_1/4]$. 
Let $k,\ell,m \in\R^3$ with $|k|\approx 2^K, |\ell|\approx 2^L$ and $|m|\approx 2^M$, 
and assume that $K,L,M \leq A$ for some $A>0$.
%Under the same assumptions of Proposition \ref{Propnu21}, 
Then we can write
\begin{align}\label{Propmu30}
\mu_3(k,\ell,m) = \mu_{3,0}(k,\ell,m) + \mu_{3,R}(k,\ell,m),
\end{align}
where:

%\medskip
\begin{itemize}

\medskip
\item[(1)] 
The leading order has the form
\begin{align}\label{Propmu31}
\mu_{3,0}(k,\ell,m) =
  \sum_{i=0}^{N_2} \sum_{J\in\Z} b_{i,J}(k,\ell,m) \cdot K_i\big(2^J(|k|-|\ell|-|m|)\big) 
\end{align}
with $K_i \in \mathcal{S}$, and
\begin{align}\label{Propmu31'}  
\begin{split}
& \big| \varphi_K(k) \varphi_L(\ell) \varphi_M(m) 
  \nabla_k^a \nabla_\ell^\alpha \nabla_m^\beta  b_{i,J}(k,\ell,m) \big| 
  \\ & \lesssim \max(1,2^{-(|a|-1)K})  \max(1,2^{-(|\a|-1)L}) \max(1,2^{-(|\b|-1)M})
  \mathbf{1}_{\{|K-\max(L,M)| < 5\}}
  %\big( 2^{(1-|a|)K} 2^{(1-|\alpha|)L} 2^{(1-|\beta|)M} \big),
\end{split}
\end{align}
for all $K,L,M \leq A$, and $|a| + |\a|+|\b| \leq N_2/2$.

\medskip
\item[(2)] The remainder term satisfies
\begin{align}\label{Propmu3R}
\begin{split}
\big| %\varphi_K(k) \varphi_L(\ell)  \varphi_M(m) 
  \nabla_k^a \nabla_\ell^\alpha \nabla_m^\beta \mu_{3,R}(k,\ell,m) \big| 
  \lesssim %2^{-3\max(L,M,K)} \cdot%2^{-|a|\max(L,M,K)}
  \max(1,2^{-(|a|-1)K})  \max(1,2^{-(|\a|-1)L}) \\ \cdot \max(1,2^{-(|\b|-1)M}) \cdot 2^{4A(|a|+|\a|+|\b|+1)}
\end{split}
\end{align}
for all $K,L,M \leq A$ and $|a|+|\a|+|\b| \leq N_2/2$.
\end{itemize}
\end{proposition}

\medskip
\begin{proof}
Notice how \eqref{SOmu3'} is easier to treat than 
$\mu_1$ and $\mu_2$ since the exponential factor is just a radial function 
and there is no need to apply stationary phase arguments on the sphere.
%compare, for example, with the arguments following \eqref{pr1} and \eqref{nu21pr1}.
We let  $\mu_3 = \mu_3^+ + \mu_3^-$,
where
\begin{align*}%\label{SOmu3'}
\begin{split}
\mu_3^+(k,\ell,m) & := \sum_{J \in [4A,\infty)\cap \mathbb{Z}} \mu_{3,J}(k,\ell,m), \qquad
  \mu_3^-(k,\ell,m) := \sum_{J \in [0,4A) \cap \mathbb{Z}} \mu_{3,J}(k,\ell,m),
  %int \frac{e^{i(-|k|+|\ell|+|m|)|x|}}{|x|^3} \bar{\psi_1(x,k)} \psi_1(x,\ell) \psi_1(x,m) \, \mathrm{d}x = 
\\  
\mu_{3,J}(k,\ell,m) & := \int_0^\infty e^{i(-|k|+|\ell|+|m|)r} 
  \Big( \int_{\mathbb{S}^2} \bar{\psi_1(r\omega,k)} \psi_1(r\omega,\ell) \psi_1(r\omega,m) \, \mathrm{d}\omega \Big)
    \, \varphi_J^{(0)}(r) r^{-1} \, \mathrm{d}r.
\end{split}
\end{align*}
%and proceed to verify the desired properties.

To isolate the leading order $\mu_{3,0}(k,\ell,m)$ within $\mu_3^+$, 
we expand the three $\psi_1$ functions in negative powers of $r$ through Lemma \ref{lemmapsi1}, 
writing, according to \eqref{lemmapsi1exp},
\begin{align*}
\psi_1 = \psi_{N_2} + R_{N_2}.
\end{align*}
All the contributions that do not contain a reminder term $R_{N_2}$, 
give rise to linear combinations of terms of the form
\begin{align}\label{prmu31}
\begin{split}
& \jk^{j_1} \jell^{j_2} \jm^{j_3}
  \int_0^\infty e^{i(-|k|+|\ell|+|m|)r} 
  \Big( \int_{\mathbb{S}^2} \bar{g_{j_1}(\omega,k)} g_{j_2}(\omega,\ell) g_{j_3}(\omega,m) 
    \, \mathrm{d}\omega \Big) \, 
    r^{-j_1-j_2-j_3-1} \varphi(r2^{-J})\, \mathrm{d}r \, 
    \\
    & = \Big( \int_{\mathbb{S}^2} \bar{g_{j_1}(\omega,k)} g_{j_2}(\omega,\ell) g_{j_3}(\omega,m) 
    \, \mathrm{d}\omega \Big) 
    \frac{\jk^{j_1} \jell^{j_2} \jm^{j_3}}{2^{J(j_1+j_2+j_3)}}
    \cdot \chi\big(2^J(|k|-|\ell|-|m|)\big)
\end{split}
\end{align}
where $g_{j_i} \in \mathcal{G}^{N_1-j_i}$ and $\chi=\chi_{j_1+j_2+j_3} =  \mathcal{F}(r^{-j_1-j_2-j_3-1}\varphi)$
is a Schwartz function. %, $0\leq j_1,j_2,j_3\leq N_2$.
We then let 
\begin{align}
b_i (k,\ell,m) = \sum_{j_1+j_2+j_3 = i} 
    \frac{\jk^{j_1} \jell^{j_2} \jm^{j_3}}{2^{J(j_1+j_2+j_3)}}
    \int_{\mathbb{S}^2} \bar{g_{j_1}(\omega,k)} g_{j_2}(\omega,\ell) g_{j_3}(\omega,m) 
    \, \mathrm{d}\omega
\end{align}
for $i=0,\dots,N_2$,
recall that $K,L,M \leq A \leq J/4$, and use the estimate \eqref{lemmapsi1gj} for the $g_{j}$ factors,
to obtain \eqref{Propmu31} and \eqref{Propmu31'}.
All the terms of the form \eqref{prmu31} with $j_1+j_2+j_3 > N_2$ can be absorbed into the remainder terms,
since the losses of $2^J$ factors coming from differentiating $\chi$ can 
be compensated by the factor $2^{-J(j_1+j_2+j_3)}$.

The other terms remaining in $\mu_3^+$ are those containing at least an $R_{N_2}(x,\cdot)$ function,
such as 
\begin{align}\label{prmu33}
\int \frac{e^{i(-|k|+|\ell|+|m|)|x|}}{|x|^3} \, \bar{R_{N_2}(x,k)} \psi_1(x,\ell) \psi_1(x,m) 
  \varphi(2^{-J}x)\, \mathrm{d}x,
\end{align}
and similar terms obtained by exchanging the role of the frequencies, 
or putting $\psi_{N_2}$ instead of $\psi_1$.
Using the estimate for $R_{N_2}$ in \eqref{lemmapsi1R}, and the estimates for $\psi_1$ in \eqref{psi10},
we can directly see that \eqref{prmu33} satisfies the bounds \eqref{Propmu3R}.

Finally, we look at $\mu_3^-$. The estimate for $a=\alpha=\beta=0$ is obvious since $|\mu_{3,J}| \lesssim 1$.
The estimates with derivatives follow directly from the estimates \eqref{psi10} for $\psi_1$
and the bound
%\begin{align*}
$| \nabla^\rho_{q} e^{ir|q|} | \lesssim 2^{4A}(2^{4A (|\rho|-1)} + |q|^{-|\rho|+1})$,
%\end{align*}
which holds for $|r| \lesssim 2^{4A}$.
%and proceed as in the proof of Proposition \ref{Propnu+}
\end{proof}

\medskip
Using Proposition \ref{Propmu3} and Lemma \ref{lemBEan2} we can obtain:

\smallskip
\begin{theorem}[Bilinear bounds for $\mu_3$]\label{theomu3}
Let $T[\mu_3;b]$ be the operator defined according to the notation \eqref{theomu2T},
with $b$ satisfying the same assumptions in Theorem \ref{theomu2}.
Then, with the same assumptions and notation of Theorem \ref{theomu2}, we have
\begin{align}\label{theomu3conc}
\begin{split}
{\big\| T[\mu_3;b]\big(g,h\big) \big\|}_{L^r}
  & \lesssim {\big\| \what{g} \big\|}_{L^p} {\big\| \what{h} \big\|}_{L^q} %\cdot \factor 
  \cdot 2^{C_0A}.
  %+ 2^{-D} \mathcal{D}(g,h).
\end{split}
\end{align}

Moreover, $T[\mu_3;b]$ satisfies the same estimates satisfied by $\nu^1_{2}$ in Theorem \ref{theomu2vf},
namely, the bound \eqref{theomu2Xconc1} for $T[{\bf X}^a_-\mu_3;b]$, 
the bound \eqref{theomu2Yconc1} for $T[{\bf Y}^a_+\mu_3;b]$, 
and the bound \eqref{theomu2dlmconc} for $T[\nabla^a_{\ell,m} (\chi_-\mu_3);b]$.

%\medskip
%\noindent
%Furthermore
%\begin{itemize}
%
%\item[(i)]
%
%\medskip
%\item[(ii)]
%
%\end{itemize}

\end{theorem}

The proof Theorem \ref{theomu3} is essentially the same of Theorem \ref{theomu2}, so we can skip it.

%\medskip
%\begin{lemma}[Bounds for operators restricted to small annuli 3]\label{lemBEa3} 
%\tofill
%\end{lemma}
%% Probably don't need...

%%%%%%%%%%%%%%%%%%%%%%%%%%%%%%%%%%%%%%%%%%%%%%%%%%
%%%%%%%%%%%%%%%%%%%%%%%%%%%%%%%%%%%%%%%%%%%%%%%%%%
%%%%%%%%%%%%%%%%%%%%%%%%%%%%%%%%%%%%%%%%%%%%%%%%%%
%%%%%%%%%%%%%%%%%%%%%%%%%%%%%%%%%%%%%%%%%%%%%%%%%%
%%%%%%%%%%%%%%%%%%%%%%%%%%%%%%%%%%%%%%%%%%%%%%%%%%
%%%%%%%%%%%%%%%%%%%%%%%%%%%%%%%%%%%%%%%%%%%%%%%%%%
%%%%%%%%%%%%%%%%%%%%%%%%%%%%%%%%%%%%%%%%%%%%%%%%%%
%%%%%%%%%%%%%%%%%%%%%%%%%%%%%%%%%%%%%%%%%%%%%%%%%%

\medskip
\section{Weighted estimates for lower order terms}\label{ssecmu23Est} %II: lower order terms for $\mu_2$ and $\mu_3$
In this last section we establish a priori nonlinear estimate for the components of Duhamel's formula involving 
$\mu_2$ and $\mu_3$, that is $\mathcal{D}_2$ and $\mathcal{D}_3$ as defined by \eqref{SODuhamel}-\eqref{SOmu3};
at the end we also discuss the estimates for $\mathcal{D}_0$. %, see \eqref{SODuhamel}.
We want to show that, under the a priori assumptions \eqref{apriori},
\begin{align}\label{SOest1} 
{\| \partial_k \mathcal{D}_i(t) \|}_{L^2_k} %+ {\| \partial_k \mathcal{D}_3(t) \|}_{L^2_k} 
  \lesssim \e_1^2, \qquad i=0,2,3,
\end{align}
and
\begin{align}\label{SOest2} 
{\| \partial_k^2 \mathcal{D}_i(t) \|}_{L^2_k} %+ {\| \partial_k^2 \mathcal{D}_3(t) \|}_{L^2_k} 
  \lesssim \e_1^2 \jt^{1/2 + \delta}, \qquad i=0,2,3.
\end{align}
In view of the bilinear bounds from Theorems \eqref{theomu2} and \eqref{theomu3},
to treat $\mathcal{D}_2$ and $\mathcal{D}_3$ 
we can use arguments similar to those we have used in Section \ref{secdkL2} for the leading order term $\mathcal{D}_1$.
It also suffices to highlight the key steps for $\mathcal{D}_2$.

\medskip
\subsection{Estimates for $\mathcal{D}_2$} %Proof of \eqref{SOest1}-\eqref{SOest2}}
We write 
\begin{align}
\nonumber
& \mathcal{D}_2(t)(f,f) = \mathcal{D}_{2,1}(t)(f,f) + \mathcal{D}_{2,2}(t)(f,f) 
\\
\label{SOest11}
& \mathcal{D}_{2}^1(t)(f,f) :=\int_0^t \iint e^{is\Phi(k,\ell,m)}
  \widetilde{f}(s,\ell) \widetilde{f}(s,m)
  \, \nu_2^1(k,\ell,m) \, \mathrm{d}\ell \mathrm{d}m \,\mathrm{d}s,
\\
\label{SOest12}
& \mathcal{D}_{2}^2(t)(f,f) := 2\int_0^t \iint e^{is\Phi(k,\ell,m)}
  \widetilde{f}(s,\ell) \widetilde{f}(s,m)
  \, \nu_2^2(k,\ell,m) \, \mathrm{d}\ell \mathrm{d}m \,\mathrm{d}s, 
%\\
%\mathcal{D}_{2,3}(t)(f,f) :=\int_0^t \iint e^{is (-|k|^2 + |\ell|^2 + |m|^2 )} \widetilde{f}(s,\ell) \widetilde{f}(s,m)
%  \, \nu_{2,2}(k,\ell,m) \, \mathrm{d}\ell \mathrm{d}m \,\mathrm{d}s, 
\\
& \Phi(k,\ell,m) := -|k|^2 + |\ell|^2 + |m|^2.
\nonumber
\end{align}

%Change variables to highlight the singularity and integration by parts directions,

\smallskip
\noindent
\subsubsection{Estimates for \eqref{SOest11}}
Proposition \ref{Propnu21} and Theorem \ref{theomu2vf} show that $\nu_2^1(k,\ell,m)$
is regular in the direction $\partial_{|\ell|} - \partial_{|m|}$.
We then let $X_1 := (1/2)(\partial_{|\ell|} - \partial_{|m|})$ and calculate
\begin{align}\label{SOest13}
\begin{split}
(|\ell| - |m|) X_1 \Phi & = (|\ell| - |m|)^2 = \Phi + |k|^2 - 2|m||\ell|
  \\ & = \Phi + |\ell|^2 + |m|^2 + (|k|-|\ell|-|m|)(|k|+|\ell|+|m|)
\end{split}
\end{align}
This identity allows us to integrate by parts in \eqref{SOest12} close to the singularity of the kernel
when $||k|-|\ell|-|m|| \ll |\ell|+|m| \approx |k|$ using
\begin{align}\label{SOest14}
\begin{split}
\frac{1}{c_1(k,\ell,m)} \Big( \frac{1}{is} (|\ell|-|m|)X_1 + i \partial_s \Big) e^{is\Phi(k,\ell,m)} = e^{is\Phi(k,\ell,m)}, 
\\
c_1(k,\ell,m) := |k|^2 - 2|\ell||m|, %|\ell|^2 + |m|^2 + (|k|-|\ell|-|m|)(|k|+|\ell|+|m|)
\end{split}
\end{align}
which is analogous to \eqref{PhiXPhiIBP} with the estimates \eqref{PhiXPhisym} replaced by
\begin{align}\label{SOest14'}
\big| \nabla^a_k \nabla_\ell^\alpha \nabla_\ell^\alpha \nabla_m^\beta \frac{1}{c_1(\ell,m)} \big| 
  \lesssim \frac{1}{|\ell|^2+ |m|^2+|k|^2} 
  |k|^{-|a|} |\ell|^{-|\alpha|} |m|^{-|\beta|}.
\end{align}
We can then proceed similarly to in \ref{l=m} 
after observing that: (1) the factor $|\nabla_k \Phi| = |2k| \approx |\ell|+|m|$
can be used to cancel part of the mild singularities introduced by the $1/c_1$ factors,
and
(2) $\partial_k \nu_2^1$ can be handled through the estimates of part (ii) in Theorem \ref{theomu2vf}
using the vectorfield $\mathbf{Y}_+$,
similarly to how this was handled in \ref{ssecd_kN21} using $\mathbf{Y}^\prime$.

Next, we look at the region away from the singularity,
that is, we analyze \eqref{SOest11} after inserting the cutoff $\chi^+$ defined in \eqref{theomu2dlcut}. 
Recall that, without loss of generality, we may assume $|\ell|\geq |m|$,
and note that if $|k| \not\approx |\ell|$ then $|\Phi|\gtrsim \max(|k|^2,|\ell|^2)$, 
with $1/\Phi$ a nice symbol; then integration by parts in time suffices.
For $|k| \approx |\ell| \gtrsim \js^{-1/2+}$, instead we can integrate by parts in $\nabla_\ell$ using
\begin{align}\label{SOest15}
e^{is\Phi} = \frac{1}{2is} \frac{\ell}{|\ell|^2} \cdot \nabla_\ell e^{is\Phi}.
\end{align}
With the H\"older estimates in part (iii) of Theorem \ref{theomu2vf},
and using that $\partial_{k_i}\Phi \, \ell/|\ell|^2$ is a bounded admissible symbol,
we can close our estimates.
When instead $|k| \approx |\ell| \lesssim \js^{-1/2+}$ we can use an $L^{6-} \times L^3$ bound,
and the smallness of $|\partial_{k_i}\Phi| \lesssim |k|$ to obtain the bound \eqref{SOest2} for $\mathcal{D}_2^1$.
%See the arguments after \eqref{dlPsisym} for similar computations.

\subsubsection{Estimates for \eqref{SOest12}}
%The term \eqref{SOest12} can be treated similarly, but some parts of the analysis require more care,
%in view of the slightly weaker bounds \eqref{theomu2Xconc2} and \eqref{theomu2Yconc2}
%compared to \eqref{theomu2Xconc1} and \eqref{theomu2Yconc1}.

Recall that we split $\nu_2^2 = \nu^2_{2,+} + \nu^2_{2,-} + \nu_{2,R}$; see Proposition \ref{Propnu21}.
%Note that here we do not assume $|\ell|\geq |m|$ since we also want to include $\nu_2^2(k,m,\ell)$ in our analysis.
It suffices to treat the leading orders $\nu^2_{2,+}$ and $\nu^2_{2,-}$.

\medskip
\noindent
{\it The case of $\nu^2_{2,-}$.}
When we consider $\nu^2_{2,-}$ %see Theorem \ref{theomu2} 
the singularity is relative to $|\ell| + |m| -|k|$, see \eqref{Propnu221},
%(that is, we are considering $\nu^2_{2,-} + \nu_{2,R}$)
and we would like proceed as in the case of $\nu^1_2$ above.
We need however to be more careful to treat some of the terms coming from
the integration by parts using \eqref{SOest14}
because of the less effective bilinear estimates that involve derivatives of $\nu^2_{2,-}$,
see \eqref{theomu2Xconc2} and \eqref{theomu2Yconc2}, compared to those for derivatives of $\nu_2^1$,
see \eqref{theomu1Xconc} and \eqref{theomu1'conc2}.

More precisely, let us look at $\mathcal{D}^2_{2,-}$, defined in the natural way similarly to \eqref{SOest12}.
Applying $\partial_k^2$ to it, using \eqref{SOest14} to integrate by parts,
and $X_1 = -(1/2)X_-$, see \eqref{theomu2X},
gives
\begin{align}\label{SOest20}
{\| \mathcal{D}^2_{2,-}(t) \|}_{L^2_k} \lesssim  
  {\| A_1(t) \|}_{L^2_k} + {\| A_2(t) \|}_{L^2_k} + {\| A_3(t) \|}_{L^2_k} + \cdots
\end{align}
with
\begin{align}
\label{SOest21}
& A_1(t)(f,f) := \int_0^t \iint e^{is\Phi(k,\ell,m)} \, b(k,\ell,m)
  \, \partial_{|\ell|}\widetilde{f}(s,\ell) \widetilde{f}(s,m)
  \, X_- \nu_{2,-}^2(k,\ell,m) \, \mathrm{d}\ell \mathrm{d}m \,\mathrm{d}s, 
\\
\label{SOest22}
& A_2(t)(f,f) := \int_0^t \iint e^{is\Phi(k,\ell,m)} \, b(k,\ell,m)
  \, \widetilde{f}(s,\ell) \, \partial_{|m|} \widetilde{f}(s,m)
  \, X_- \nu_{2,-}^2(k,\ell,m) \, \mathrm{d}\ell \mathrm{d}m \,\mathrm{d}s, 
\\
\label{SOest23}
& A_3(t)(f,f) := \int_0^t \iint e^{is\Phi(k,\ell,m)} \, b(k,\ell,m)
  \, \widetilde{f}(s,\ell) \widetilde{f}(s,m)
  \, X_-^2 \nu_{2,-}^2(k,\ell,m) \, \mathrm{d}\ell \mathrm{d}m \,\mathrm{d}s,
%\\ \label{SOest24}
%& A_4(t)(f,f) := \int_0^t \iint e^{is\Phi(k,\ell,m)} \, b(k,\ell,m)
%  \, \widetilde{f}(s,\ell) \widetilde{f}(s,m)
%  \, Y^2_+ \nu_{2,-}^2(k,\ell,m) \, \mathrm{d}\ell \mathrm{d}m \,\mathrm{d}s,
\end{align}
where 
\begin{align*}
b(k,\ell,m) := \frac{|k|^2(|\ell|-|m|)^2}{c_1^2(k,\ell,m)}.
\end{align*}
The `$\cdots$' in \eqref{SOest20} are terms easier to bound or similar to the ones obtained when dealing with
$\mathcal{D}_2^1$ and $\mathcal{D}_1$ before.
Note how we have included in the `$\cdots$' also bilinear terms with kernel measure $Y_+^2 \nu_{2,-}^2$,
that are coming from $\partial_k$ hitting $\nu_{2,-}^2$; comparing the estimates \eqref{theomu2Yconc2}
and \eqref{theomu2Xconc2} we see that these are not worse than \eqref{SOest21}-\eqref{SOest23}.
Moreover, note that `$\cdots$' do not contain terms with a measure $Y_+X_-\nu_{2,-}^2$.

Since $b$ is an admissible symbol, we may disregard it in what follows and just assume $b\equiv 1$ for convenience.
We may also assume that the integrals \eqref{SOest21}-\eqref{SOest23} are localized as usual
to $|k|\approx 2^K$, $|\ell|\approx 2^L$, $|m|\approx 2^M$ and $s\approx 2^S$
(omitting some of these localization for lighter notation) %with $K,M\leq L$.
and show a slightly stronger bound than \eqref{SOest2} with a $\delta-$ factor instead of $\delta$.
Let us also assume that $L\leq M$, the other case being similar.

With the same notation of Theorem \ref{theomu2} we write
\begin{align*}
{\| A_1(t)(f,f) \|}_{L^2_k} & \lesssim 
  \int_0^t {\big\| P_K T[X_-\nu^2_{2,-},b]
  \big(e^{is|k|^2}\partial_{|k|} \wt{f}(s), \Ftil u(s) \big) \big\|}_{L^2} \, \tau_S(s) \mathrm{d}s.
\end{align*}
%Note that we have made a little abuse of notation converting $\partial_{|\ell|}$ into $\partial_\ell$;
%this can be done without loss of generality by slightly redefining the symbol $b$.
%Using that $b$ satisfies the hypotheses \eqref{theomu1asb1}-\eqref{theomu1asb2},
Applying \eqref{theomu2Xconc2}, followed by Bernstein,
together with the a priori decay bound \eqref{aprioriL<6},  we obtain
\begin{align*}
\begin{split}
{\| A_1(t)(f,f) \|}_{L^2_k} & \lesssim 
\int_0^t {\| P_{\sim L}\Fhat e^{is|k|^2}\partial_k \wt{f}(s) \|}_{L^6} {\| \Fhat \wt{u}(s) \|}_{L^{3-}}
  \cdot 2^{-L} \cdot 2^{(C_0+13)A} %+ \js^{-3} \e^2\, \big] 
  \, \tau_S(s) \mathrm{d}s
\\
& \lesssim \int_0^t {\| P_{\sim L}\Fhat e^{is|k|^2}\partial_k \wt{f}(s) \|}_{L^2} \cdot \e 2^{-S/2}
  \cdot 2^{(C_0+14)A}
  \, \tau_S(s) \mathrm{d}s
\\
& \lesssim \e^2 2^{S/2 + \delta -}
\end{split}
\end{align*}
having used $2^{A} \leq 2^{\delta_NS}$ with %$\delta_N$ is small enough. 
$(C_0+14)\delta_N < \delta$.

The term \eqref{SOest22} is estimated using again \eqref{theomu2Xconc2}, 
and then Bernstein followed by \eqref{aprioriL<6} on the same input function:
\begin{align*}
\begin{split}
{\| A_2(t)(f,f) \|}_{L^2_k} & \lesssim 
  \int_0^t  {\| P_{\sim L} \Fhat \wt{u}(s) \|}_{L^{\infty-}} 
  {\| \Fhat e^{is|k|^2}\partial_k \wt{f}(s) \|}_{L^{2}}
  \cdot 2^{-L} \cdot 2^{(C_0+13)A} \, \tau_S(s) \mathrm{d}s
\\
& \lesssim 
\int_0^t {\| P_{\sim L} \Fhat \wt{u}(s) \|}_{L^{3-}} \cdot \e \cdot 2^{(C_0+13)A}
  \, \tau_S(s) \mathrm{d}s
\\
& \lesssim 
\int_0^t \e 2^{-S/2} \cdot \e \cdot 2^{(C_0+14)A} \, \tau_S(s) \mathrm{d}s
  \lesssim \e^2 2^{S/2 + \delta-}.
\end{split}
\end{align*}

The last term can be bounded using \eqref{theomu2Xconc2},
Hardy's inequality $2^{-L} {\| P_{\sim L} g \|}_{L^2} \lesssim {\| \partial_\ell \Fhat^{-1} g\|}_{L^2}$
and the usual a priori bounds:
\begin{align*}
\begin{split}
{\| A_3(t)(f,f) \|}_{L^2_k} & \lesssim 
\int_0^t  {\| P_{\sim L}\Fhat \wt{u}(s) \|}_{L^{6}} {\| \Fhat \wt{u}(s) \|}_{L^{3-}}
  \cdot 2^{-2L} \cdot 2^{(C_0+13)A}
  \, \tau_S(s) \mathrm{d}s
\\
& \lesssim \int_0^t 2^{-L} {\| P_{\sim L}\Fhat \wt{u}(s) \|}_{L^2} \cdot \e 2^{-S/2}
  \cdot 2^{(C_0+14)A}
  \, \tau_S(s) \mathrm{d}s
\\
& 
  \lesssim \e^2 2^{S/2 + \delta-}.
\end{split}
\end{align*}

\medskip
\noindent
{\it The case of $\nu^2_{2,+}$.}
Finally, we consider $\nu^2_{2,+}$ whose kernel is singular relative to $-|\ell| + |m| -|k|$.
We observe that the following analogue of \eqref{SOest13} holds:
let the ``good direction'' be $X_2 := (1/2)(\partial_{|\ell|} + \partial_{|m|})$, then
\begin{align}\label{SOest13'}
\begin{split}
\frac{1}{2}(|\ell|+|m|) X_2 \Phi & = \Phi + |k|^2 + 2|\ell||m|
  \\ & = \Phi + |\ell|^2 + |m|^2 + (|k|+|\ell|-|m|)\big(|k|-(|\ell|-|m|)\big).
\end{split}
\end{align}
This gives us an identity as in \eqref{SOest14} with the slightly different coefficient 
$c_2(k,\ell,m):= |k|^2 + 2|\ell||m|$, which still satisfies proper symbol estimates as in \eqref{SOest14'}
in the region $| |k| + |\ell| - |m| | \ll |m| \approx |k|+|\ell|$.
We may again assume $|k| \lesssim \max(|\ell|,|m|)$ for otherwise integration by parts in $s$ suffices;
then, the factor $\nabla_k \Phi = 2k$ which appears 
when we differentiate in $k$ the exponential $e^{is\Phi}$,
can again be used to cancel part of the singularities introduced by factors of $1/c_2 \approx \max(|\ell|,|m|)^{-2}$.
This takes care of the singular region.

In the region away from the singularity of $\nu_{2,+}^2$, 
when $||k| + |\ell| - |m|| \gtrsim \max(|m|,|k|+|\ell|)$, we can proceed exactly as in the case of 
$\nu^1_2$ above,
relying on \eqref{SOest15} with $m$ instead of $\ell$, 
and the bilinear estimate \eqref{theomu2dlmconc} of Theorem \ref{theomu2vf}. %part (iii).

%\medskip
%\subsubsection{Proof of \eqref{SOest2}}
%Eventually we come to
%\begin{align}
%\mathcal{D}_3(t)(f,f) = \int_0^t \iint e^{is (-|k|^2 + |\ell|^2 + |m|^2 )} \widetilde{f}(s,\ell) \widetilde{f}(s,m)
%  \, \mu_3(k,\ell,m) \, \mathrm{d}\ell \mathrm{d}m \,\mathrm{d}s
%\end{align}
%Singularity of $\mu_3(k,\ell,m)$ is relative to $-|k|+|\ell|+|m|$, so
%of the same type treated above, and we have the bilinear bounds from Theorem \ref{theomu2}.

%\tofill

\medskip
\subsection{Nonlinear Estimates for $\mathcal{D}_0$}\label{ssecD0}
At last, we discuss how to estimate the nonlinear term $\mathcal{D}_0$ %defined in \eqref{Duhameldec}
corresponding to the $\delta$ interaction in the expansion \eqref{mudec} of $\mu$.
This is the nonlinearity in the ``flat'' equation, that is, \eqref{NLSV} with $V=0$,
which has been treated %and estimates for this term have been already performed 
%when dealing with the ``flat'' case 
in the works \cite{HNNLS3d} and \cite{GMS2}.
We want to show that, under the a a priori assumptions \eqref{apriori},
\begin{align}\label{SO0est} 
{\| \partial_k \mathcal{D}_0(t) \|}_{L^2_k} 
   + \jt^{-1/2 - \delta} {\| \partial_k^2 \mathcal{D}_2(t) \|}_{L^2_k} \lesssim \e_1^2.
\end{align}

These bounds are obviously easier than those for the other terms we have analyzed.
However, we cannot directly utilize the bounds of \cite{HNNLS3d} and \cite{GMS2} 
because of the different functional spaces we are considering here, and, more precisely,
the different time decay rate of our solution (which is weaker) and time growth rates of our norms.
%see \eqref{space} and a priori bounds \eqref{apriori}.
%and check that the estimates there are compatible with our norms \eqref{space} and 
%a priori bounds \eqref{apriori}.
Nevertheless, bounds for $\mathcal{D}_0$ follow from simple adaptations of the arguments in Section \ref{secdkL2},
so we just give a sketch of the proof.

Let us look at the bounds for the highest weighted norm. Applying $\Delta_k^2$ to the expression 
for $\mathcal{D}_0$ in \eqref{Duhameldec} gives the term
\begin{align}\label{SO01} 
\begin{split}
\int_0^t \iint 4s^2 |\ell|^2\, e^{is \Phi(k,\ell)} \widetilde{f}(s,\ell) 
  \widetilde{f}(s,k-\ell) \, \mathrm{d}\ell \,\mathrm{d}s,
  \\
  \Phi(k,\ell) := -|k|^2 + |\ell|^2 + |k-\ell|^2 = - 2k\cdot \ell + 2|\ell|^2,
\end{split} 
\end{align}
up to simpler terms where the derivatives hit the profile $\wt{f}$.
We notice that 
%\begin{align*}
$(3\ell - k) \cdot \nabla_\ell \Phi = 2(|\ell|^2+|k|^2) + 5\Phi$, and therefore
%\end{align*}
\begin{align}\label{SO0id}
e^{is\Phi(k,\ell)} 
  = \frac{1}{2(|\ell|^2+|k|^2)}\Big( \frac{1}{is}(3\ell - k) \cdot \nabla_\ell +5i \partial_s \Big) e^{is\Phi(k,\ell)}.
\end{align}
This identity plays the same role as the identity \eqref{PhiXPhiIBP} or \eqref{SOest14}
and allows us to integrate by parts in $\ell$ or $s$.
With this, the proof can proceed as for the term \eqref{dk^2N_20} in \ref{ssecdk^2N_12},
with the distribution $\nu_1$ substituted by a delta measure.
We first integrate by parts to obtain terms similar to \eqref{dk^2N_20'}-\eqref{dk^2N_28} 
and then estimate all these through standard product estimate of Coifman-Meyer type 
(see for example Proposition \eqref{proGHW}) in place of the bilinear estimates for $\nu_1$.
%used to handle the terms  \eqref{dk^2N_20'}-\eqref{dk^2N_28}  arising from the integration by parts, 

This concludes the proof of the a priori bounds \eqref{SOest1}-\eqref{SOest2} and gives us the main Proposition
\ref{proBoot} which implies Theorem \ref{maintheo}.

\end{document}